\documentclass[12pt]{amsart}
\usepackage[colorlinks=true,citecolor=purple,linkcolor=blue]{hyperref}
\usepackage{amssymb,amscd,amsmath,graphicx,cleveref,xcolor,extarrows}
\usepackage{enumerate, multirow}
\usepackage{fullpage}
\usepackage{comment}

\numberwithin{equation}{section}
\newtheorem{thm}{Theorem}[section]

\newtheorem{cor}[thm]{Corollary}
\newtheorem{lem}[thm]{Lemma}
\newtheorem{prop}[thm]{Proposition}

\theoremstyle{definition}
\newtheorem{conj}[thm]{Conjecture}

\newtheorem{defn}[thm]{Definition}

\newtheorem{notn}[thm]{Notation}
\newtheorem{quest}[thm]{Question}

\theoremstyle{definition}
\newtheorem{ex}[thm]{Example}
\newtheorem{rem}[thm]{Remark}

\DeclareMathOperator{\Ann}{Ann}
\DeclareMathOperator{\Ass}{Ass}
\DeclareMathOperator{\chara}{char}
\DeclareMathOperator{\conv}{convex\ hull}

\DeclareMathOperator{\Gendeg}{Gendeg}

\DeclareMathOperator{\indeg}{indeg}
\DeclareMathOperator{\ini}{in}
\DeclareMathOperator{\lcm}{lcm}

\DeclareMathOperator{\Min}{Min}
\DeclareMathOperator{\NP}{NP}
\DeclareMathOperator{\Supp}{Supp}
\DeclareMathOperator{\reg}{reg}
\DeclareMathOperator{\SP}{SP}
\DeclareMathOperator{\Tor}{Tor}

\newcommand{\udb}{\underbrace}
\newcommand{\udl}{\underline}

\newcommand{\ZZ}{{\mathbb Z}}
\newcommand{\NN}{{\mathbb N}}
\newcommand{\RR}{{\mathbb R}}
\newcommand{\QQ}{{\mathbb Q}}

\newcommand{\bsc}{\boldsymbol{c}}
\newcommand{\up}[1]{^{\tt #1}\!}
\newcommand{\wti}{\widetilde}

\def\kk{{\Bbbk}}

\def\mm{{\mathfrak m}}
\def\pp{{\mathfrak p}}

\newcommand{\Lcc}{\mathcal{L}}
\newcommand{\bsa}{{\boldsymbol a}}

\renewcommand{\a}{\mathbf{a}}
\def\I{\mathcal{I}}
\def\R{\mathcal{R}}
\def\bv{\textbf{v}}

\begin{document}
	
	\title{Asymptotic regularity of graded families of ideals}
	
	\author{T\`ai Huy H\`a}
	\address{Tulane University \\ Mathematics Department \\
		6823 St. Charles Ave. \\ New Orleans, LA 70118, USA}
	\email{tha@tulane.edu}
	
	\author{Hop D. Nguyen}
	\address{Institute of Mathematics, VAST, 18 Hoang Quoc Viet, Cau Giay, 10307 Hanoi, VIETNAM}
	\email{ngdhop@gmail.com}
	
	\author{Th\'ai Th\`anh Nguy$\tilde{\text{\^E}}$n}
\address{University of Dayton, Department of Mathematics,
	300 College Park, Dayton, Ohio, USA \\
	and University of Education, Hue University, 34 Le Loi, Hue, Vietnam}
\email{tnguyen5@udayton.edu}

	\keywords{Castelnuovo--Mumford regularity, graded family of ideals, Newton--Okounkov body, integral closure, symbolic power, saturated power, sum of powers of ideals, product of ideals, Gr\"obner basis, monomial ideals.}
	\subjclass[2020]{13D45, 13P10, 13D02, 13H15}
	
	\begin{abstract}
	We show that the asymptotic regularity of a graded family $(I_n)_{n \ge 0}$ of homogeneous ideals in a reduced standard graded algebra, i.e., the limit $\lim_{n \rightarrow \infty} \reg I_n/n$, exists in several cases; for example, when the family $(I_n)_{n \ge 0}$ consists of artinian ideals, or Cohen-Macaulay ideals of the same codimension over an uncountable base field of characteristic $0$, or when its Rees algebra is Noetherian. Many applications, including simplifications and generalizations of previously known results on symbolic powers and integral closures of powers of homogeneous ideals, are discussed. We provide a combinatorial interpretation of the limit $\lim_{n \rightarrow \infty} \reg I_n/n$ in terms of the associated Newton--Okounkov region in various situations. We give a negative answer to the question of whether the limits $\lim_{n \rightarrow \infty} \reg (I_1^n + \dots + I_p^n)/n$ and $\lim_{n \rightarrow \infty} \reg (I_1^n \cap \cdots \cap I_p^n)/n$ exist, for $p \ge 2$ and homogeneous ideals $I_1, \dots, I_p$. We also examine ample evidence supporting a negative answer to the question of whether the asymptotic regularity of the family of symbolic powers of a homogeneous ideal always exists. Our work presents explicit Gr\"obner basis construction for ideals of the form $Q^n + (f^k)$, where $Q$ is a monomial ideal, $f$ is a polynomial in the polynomial ring in 4 variables over a field of characteristic 2.
	\end{abstract}
	
\maketitle

\tableofcontents
	
	
\section{Introduction} \label{sec.intro}

\subsection{Asymptotic regularity of graded families of ideals} Let $\kk$ be a field and let $R$ be a standard graded algebra over $\kk$. Let $(I_n)_{n\ge 0}$ be a family of homogeneous ideals in $R$. Familiar examples include \emph{graded families} of the ordinary powers $(I^n)_{n\ge 0}$, symbolic powers $(I^{(n)})_{n\ge 0}$, integral closures of powers $(\overline{I^n})_{n\ge 0}$, and saturation powers $(\widetilde{I^n})_{n\ge 0}$ of a homogeneous ideal, and \emph{$p$-families} of Fr\"obenius powers $(I^{[p^n]})_{n \ge 0}$ of a homogeneous ideal when $\chara \kk = p > 0$. Let $\reg I$ denote the Castelnuovo-Mumford regularity of a homogeneous ideal $I$ in $R$.

The asymptotic behavior of the regularity function $\reg I_n$, $n \in \NN$, has been a highly active research topic over the last few decades. For ordinary powers $(I^n)_{n \ge 0}$ of a homogeneous ideal $I$, it is well known --- proved independently by Cutkosky, Herzog and Trung \cite{CHT99} and Kodiyalam \cite{K00} --- that $\reg I^n$ is asymptotically a linear function in $n$. The study of this linear function has been extended in further work such as \cite{BHT2020, Chardin2013, Chardin2013b, Chardin2015, Ha2011, HNT2023, TW05}. For more general families $(I_n)_{n \ge 0}$, including symbolic powers $(I^{(n)})_{n \ge 0}$, saturation powers $(\widetilde{I^{n}})_{n \ge 0}$, and Fr\"obenius powers $(I^{[p^n]})_{n \ge 0}$ of a homogeneous ideal $I$, the existence of linear bounds for $\reg I_n$ and the limit $\lim_{n \rightarrow \infty} {\reg I_n}/{n}$ have been investigated extensively (cf. \cite{BEL, Bren2005, Cut2000, CEL2001, EHL2022, HeHoTr02, Katz1998, KZ2014}). The limit $\lim_{n \rightarrow \infty} {\reg I_n}/{n}$, when it exists, is referred to as the \emph{asymptotic regularity} of the family $(I_n)_{n \ge 0}$.

A question that has attracted significant attention is: for which family $(I_n)_{n \ge 0}$ does the asymptotic regularity exist? This question, particularly, encompasses long-standing open problems of Herzog, Hoa and Trung \cite{HeHoTr02} and of Katzman \cite{Katz1998}. Specifically,
\begin{enumerate}
	\item Herzog--Hoa--Trung in \cite{HeHoTr02} asked if the asymptotic regularity $\lim_{n \rightarrow \infty} {\reg I^{(n)}}/{n}$ of the graded family $(I^{(n)})_{n \ge 0}$ of symbolic powers of a homogeneous ideal $I$ in a polynomial ring $R = \kk[x_1, \dots, x_r]$ exists; and
	\item Katzman \cite{Katz1998} conjectured that, if $R = \kk[x_1, \dots, x_r]/J$, where $J \subseteq \kk[x_1, \dots, x_r]$ is a homogeneous ideal, and $\text{char } \kk = p > 0$, then the degrees of Gr\"obner bases, with respect to the reverse lexicographic order, of Fr\"obenius powers $I^{[p^n]}$, for an ideal $I \subseteq R$, has a linear growth in $p^n$.
\end{enumerate}
While the question of Herzog--Hoa--Trung arose from an attempt to understand the computational complexity of symbolic powers of an ideal, Katzman's conjecture is closely related to the question of whether \emph{tight closure} localizes. Both of these problems have affirmative answers in various special cases, see \cite{Bren2005, CEL2001, DSMNB22+, DHNT21, HS2005, HeHoTr02, HoaTrung2008, HoaTrung2010, KZ2014}, but remain unresolved in general.


\subsection{Specific questions of interest}
In Example \ref{ex_diverge}, we demonstrate that the function $\reg I_n$ of a graded family $(I_n)_{n \ge 0}$ of homogeneous ideals can exhibit arbitrary behavior. Consequently, one cannot expect the asymptotic regularity to exist in full generality. However, the asymptotic regularity of $(I_n)_{n \ge 0}$ exists for the family of saturated powers of a homogeneous ideal when $R$ defines a \emph{non-singular} variety (by Cutkosky, Ein and Lazarsfeld \cite{CEL2001}), the family of symbolic powers of a \emph{monomial} ideal in a polynomial ring (by Dung, Hien, Nguyen and T.N.\,Trung (henceforth L.X. Dung et al.) \cite{DHNT21}), or if $R$ is a \emph{F-finite} and \emph{F-pure} ring of positive characteristic, and $(I_n)_{n \ge 0}$ is an \emph{F-pure} filtration of ideals in $R$ such that its Rees algebra is Noetherian (by De Stefani, Monta\~no and N\'u\~nez-Betancourt \cite{DSMNB22+}).

DiPasquale, Nguy$\tilde{\text{\^e}}$n and Seceleanu, in the preliminary version of their paper \cite{DNS23}, posed the following refined statement on the existence of asymptotic regularity, which gives more insights to the question of Herzog, Hoa and Trung. 

\begin{quest}
	\label{quest.DSN}
	Let $(I_n)_{n \ge 0}$ be a graded family of homogeneous ideals in $R$ and assume that its Rees algebra is a Noetherian ring. Does the limit $\lim\limits_{n \rightarrow \infty} \dfrac{\reg I_n}{n}$ exist?
\end{quest}

In parallel, efforts to understand Katzman's conjecture have inspired alternative perspectives for the study of asymptotic regularity of families of ideals. Hoa and T.N.\,Trung \cite{HoaTrung2008, HoaTrung2010} considered the following question, which generalizes the asymptotic linearity of the regularity of powers of an ideal (as shown in \cite{CHT99, K00, TW05}), and gave special cases where the question has affirmative answers.

\begin{quest}
	\label{quest.HT}
	Let $p \ge 2$ be an integer, and let $I_1, \dots, I_p$ be homogeneous ideals in $R$. Does the limit $\lim\limits_{n \rightarrow \infty} \dfrac{\reg (I_1^n + \dots + I_p^n)}{n}$ exist?
\end{quest}

Symbolic powers of various ideals of interest are described by the intersection of powers of certain primary components\footnote{This is the case, for example, for ideals without embedded associated primes whose primary components have identical ordinary and symbolic powers, e.g., the primary components in question are either monomial ideals or complete intersections.}. In addition, the sum and intersection of ideals are related through standard exact sequences. Thus, inspired by Questions \ref{quest.DSN} and \ref{quest.HT}, the following question, which is closely connected to previous studies of Bruns and Conca \cite{BC2017} and of Baghari, Chardin and H\`a \cite{BCH13} on the regularity of product of powers of ideals, is also of interest.

\begin{quest}
	\label{quest.Intersection}
	Let $p \ge 2$ be an integer, and let $I_1, \dots, I_p$ be homogeneous ideals in $R$. Does the limit $\lim\limits_{n \rightarrow \infty} \dfrac{\reg (I_1^n \cap \cdots \cap I_p^n)}{n}$ exist?
\end{quest}


In this paper, we investigate and provide satisfactory solutions to Questions \ref{quest.DSN}, \ref{quest.HT} and \ref{quest.Intersection}. Specifically, our results provide an affirmative answer to Question \ref{quest.DSN} and negative answers to Questions \ref{quest.HT} and \ref{quest.Intersection}. We also explore several applications and consequences, and examine an example, which exhibits computational evidence for a negative answer to the open question of Herzog, Hoa and Trung. In various instances, when the asymptotic regularity of a graded family of homogeneous ideals exists, we provide a combinatorial understanding of this invariant via the Newton--Okounkov region associated to the given family.

We shall now outline the main results of the paper and our methods and approaches to obtaining these results.


\subsection{Existence results for Noetherian families of ideals} Our first main result gives an affirmative answer to Question \ref{quest.DSN}, that is, when the Rees algebra of the given family $(I_n)_{n \ge 0}$ is a Noetherian ring,  in the case where the ring $R$ is \emph{reduced}. Write $d(I)$ for the maximal degree of a minimal generating set for $I$.

\medskip

\noindent\textbf{Theorem \ref{thm_limit_noeth_families}.} Let $\I=(I_n)_{n\ge 0}$ be a graded family of homogeneous ideals such that $\Ann_R(I_n) = (0)$ for all $n \in \NN$. Suppose that $R$ is reduced and $\R(\I)$ is a finitely generated $R$-algebra. Then,
\[
\lim \limits_{n\to \infty} \frac{\reg I_n}{n} = \lim \limits_{n\to \infty} \frac{d(I_n)}{n} =\lim \limits_{n\to \infty} \frac{\reg \overline{I_n}}{n} = \lim \limits_{n\to \infty} \frac{d(\overline{I_n})}{n}.
\]

Theorem \ref{thm_limit_noeth_families} generalizes and gives a unified perspective for several known results, including those aforementioned of De Stefani, Monta\~no and N\'u\~nez-Betancourt \cite[Theorem 4.10]{DSMNB22+} and of L.X. Dung et al. \cite[Theorem 1.3]{DHNT21}, which were obtained by different methods. Theorem \ref{thm_limit_noeth_families} also generalizes the study of Hoa \cite{Hoa22b} on the asymptotic maximal generating degree of powers of homogeneous ideals. Theorem \ref{thm_limit_noeth_families} further shows that the equality $\lim_{n \rightarrow \infty} \reg I_n/n = \lim_{n \rightarrow \infty} d(I_n)/n$, known by Cutkosky, Ein and Lazarsfeld \cite[Theorem B]{CEL2001} for saturated powers of an ideal when $R$ defines a non-singular variety, by Hoa \cite[Theorem 2.7]{Hoa22} and Hoa and T.N.\,Trung \cite[Theorem 4.10]{HoaTrung2010} for integral closures of powers of a monomial ideal in a polynomial ring, and by L.X. Dung et al. \cite[Theorem 1.3]{DHNT21} for symbolic powers of a monomial ideal in a polynomial ring, also holds for an arbitrary Noetherian graded family $(I_n)_{n \ge 0}$ in a standard graded $\kk$-algebra.

We remark that the condition $\Ann_R(I_n) = (0)$ in Theorem \ref{thm_limit_noeth_families} cannot be omitted, as illustrated in \Cref{ex.AnnNot0}.
To establish Theorem \ref{thm_limit_noeth_families}, we show that the Noetherian property of the Rees algebra of $(I_n)_{n \ge 0}$ allows us to reduce the proof to considering a family of the form $(I^nM)_{n \ge 0}$, where $I \subseteq R$ is a homogeneous ideal and $M$ is a finitely generated graded faithful $R$-module. Theorem \ref{thm_limit_noeth_families} is then a direct consequence of the following result.

\medskip

\noindent\textbf{Theorem \ref{thm_limit_powers}.} Let $R$ be a reduced, standard graded $\kk$-algebra and $I \subseteq R$ be a homogeneous ideal with $\Ann_R(I) = (0)$. Then, for any finitely generated graded faithful $R$-module $M$, we have
$$\lim \limits_{n\to \infty} \frac{\reg (I^nM)}{n} = \lim \limits_{n\to \infty} \frac{d(I^nM)}{n} =\lim \limits_{n\to \infty} \frac{\reg I^n}{n} =\lim \limits_{n\to \infty} \frac{\reg \overline{I^n}}{n} = \lim \limits_{n\to \infty} \frac{d(\overline{I^n})}{n}.$$

\medskip

To achieve Theorem \ref{thm_limit_powers}, we employ techniques of Trung and Wang \cite{TW05} in realizing the asymptotic regularities of $(I^n)_{n \ge 0}$ and $(I^nM)_{n \ge 0}$ as the maximal generating degree of reductions and $M$-reductions of $I$, respectively. We explore reductions and $M$-reductions of $I$ and, particularly, prove that these reductions are the same when $M$ is a faithful $R$-module.


\subsection{Existence results for non-Noetherian families of ideals}
As indicated by a result of Cutkosky, Ein and Lazarsfeld \cite[Theorem B]{CEL2001}, the asymptotic regularity $\lim_{n \rightarrow \infty} \reg I_n/n$ may exist and the equality $\lim_{n \rightarrow \infty} \reg I_n/n = \lim_{n \rightarrow \infty} d(I_n)/n$ may hold for the family of saturated powers $(\widetilde{I^n})_{n \ge 0}$ of a homogeneous ideal $I$, which is not necessarily Noetherian. We give several sufficient conditions for the asymptotic regularity to exist and for the equality between the asymptotic regularity and the asymptotic maximal generating degree to hold, when the associated Rees algebras are not necessarily Noetherian.

Let $\mm$ denote the maximal homogeneous ideal in $R$. We focus on the following three conditions:
\begin{enumerate}
	\item[(i)] $(I_n)_{n \ge 0}$ consists of $\mm$-primary ideals;
	\item [(ii)] $\kk$ is uncountable, $R/I_n$ is Cohen-Macaulay of the same dimension for all $n \ge 0$; and
	\item [(iii)] $(I_n)_{n \ge 0}$ is of the form $(J^{a_n}I^n)_{n \ge 0}$, where $I, J \subseteq R$ are homogeneous ideals with $\Ann_R(I) = \Ann_R(J) = (0)$, and $(a_n)_{n \ge 0}$ is a sequence of nonnegative integers such that $\lim_{n \rightarrow \infty} {a_n}/{n} = 0$.
\end{enumerate}
Our results show that the asymptotic regularity $\lim_{n \rightarrow \infty} {\reg I_n}/{n}$ exists and the equality $\lim_{n \rightarrow \infty} \reg I_n/n = \lim_{n \rightarrow \infty} d(I_n)/n$ holds under these conditions. We prove the following theorems.

\medskip

\noindent\textbf{Theorems \ref{thm_m-primary}, \ref{thm_CohenMacaulay}, and \ref{thm_limit_varyingpowers}.} Let $(I_n)_{n \ge 0}$ be a graded family of homogeneous ideals in $R$. \begin{enumerate}
	\item If $(I_n)_{n \ge 0}$ satisfies either condition (i) or condition (ii), as described above, then
$$\lim_{n \rightarrow \infty} \frac{\reg I_n}{n} = \lim_{n \rightarrow \infty} \frac{d(I_n)}{n} = \inf_{n \ge 1} \frac{d(I_n)}{n}.$$
	\item If $(I_n)_{n \ge 0} = (J^{a_n}I^n)_{n \ge 0}$ satisfies condition (iii) and in addition, $R$ is reduced, then
	$$\lim_{n \rightarrow \infty} \frac{\reg \overline{I_n}}{n} = \lim_{n \rightarrow \infty} \frac{d(\overline{I_n})}{n} = \lim_{n \rightarrow \infty} \frac{\reg I_n}{n} = \lim_{n \rightarrow \infty} \frac{d(I_n)}{n} = \lim_{n \rightarrow \infty} \frac{\reg I^n}{n}.$$
\end{enumerate}

\medskip

To prove Theorem \ref{thm_m-primary}, we show that for a family of $\mm$-primary ideals, the regularity function is sub-additive. The result then follows by applying Fekete's Lemma. Theorem \ref{thm_CohenMacaulay} is obtained from Theorem \ref{thm_m-primary} by using techniques of \emph{specialization}. The proof of Theorem \ref{thm_limit_varyingpowers} generalizes that of Theorem \ref{thm_limit_noeth_families}, making use of the result of Bagheri, Chardin and H\`a \cite{BCH13} in understanding the nonzero graded Betti numbers of modules of the form $J^tI^sM$, for a graded faithful $R$-module $M$.

In general, when the family $(I_n)_{n \ge 0}$ is not Noetherian, one does not expect the asymptotic regularity to exist, and even if $\lim_{n \rightarrow \infty} \reg I_n/n$ exists, it may not be equal to $\lim_{n \rightarrow \infty} d(I_n)/n$, as illustrated in Example \ref{ex_diverge}. The question of when the asymptotic maximal generating degree $\lim_{n \rightarrow \infty} d(I_n)/n$ of a graded family $(I_n)_{n \ge 0}$ exists, independent of the same question for the asymptotic regularity, has also been studied in, for instance, \cite{Hoa22b}.


\subsection{Asymptotic regularity and Newton--Okounkov region} When the asymptotic regularity exists, it is desirable to compute and understand this invariant. This problem has been considered by Cutkosky \cite{Cut2000}, Cutkosky, Ein and Lazarsfeld \cite{CEL2001}, L.X. Dung et al. \cite{DHNT21}, Hoa \cite{Hoa22}, and Hoa and T.N.\,Trung \cite{HoaTrung2008, HoaTrung2010}.

We proceed to give a combinatorial interpretation of the asymptotic regularity of a graded family $(I_n)_{n \ge 0}$ in various situations. Our approach is via the \emph{Newton--Okounkov region} associated to the given family of ideals. This is an analog of the construction of the \emph{Newton-Okounkov body} that originated from the pioneering work of Okounkov \cite{Ok1996, Ok2003}, and has been systematically developed by Lazarsfeld and Musta\c{t}\v{a} \cite{LM2009}, Kaveh and Khovanskii \cite{KK2012, KK2014}, and Cutkosky \cite{Cut2013, Cut2014}. The theory of Newton-Okounkov body has garnered significant interest and seen many applications in several areas of mathematics.

We now restrict ourselves to a polynomial ring $R = \kk[x_1, \dots, x_r]$. When $\I = (I_n)_{n \ge 0}$ is a family of monomial ideals, its Newton--Okounkov region $\Delta(\I) \subseteq \RR^r_{\ge 0}$ is easy to understand and has been considered in \cite{HaN23, KK2014}, namely,
$$\Delta(\I) = \overline{\bigcup_{n \ge 1} \frac{1}{n} \NP(I_n)},$$
where $\NP(I)$ represents the Newton polyhedron of a monomial ideal $I$.
For a polyhedron $P \subseteq \RR^r_{\ge 0}$, set
$$\delta(P) = \max \{|\bv| ~\big|~ \bv \text{ is a vertex of } P\},$$
where $|\bv| = v_1 + \dots + v_r$ if $\bv = (v_1, \dots, v_r)$. Our next result describes the asymptotic regularity of a Noetherian graded family $\I$ of monomial ideals.

\medskip

\noindent\textbf{Theorem \ref{thm.regOkn}.} Let $\I = (I_n)_{n \ge 0}$ be a Noetherian graded family of monomial ideals in $R = \kk[x_1, \dots, x_r]$ and let $\Delta(\I)$ be its Newton--Okounkov region. Then,
$$\lim \limits_{n\to \infty} \frac{\reg I_n}{n} = \lim \limits_{n\to \infty} \frac{d(I_n)}{n} =\lim \limits_{n\to \infty} \frac{\reg \overline{I_n}}{n} = \lim \limits_{n\to \infty} \frac{d(\overline{I_n})}{n} = \delta(\Delta(\I)).$$

\medskip

The new contribution of Theorem \ref{thm.regOkn} lies at the last equality, while other equalities were already shown in Theorem \ref{thm_limit_noeth_families}. This equality generalizes previous results of Hoa \cite[Theorem 2.7]{Hoa22}, of Hoa and T.N.\,Trung \cite[Theorem 4.10]{HoaTrung2010}, and of L.X. Dung et al. \cite[Theorem 3.6]{DHNT21}, which relate the asymptotic regularity of the family of integral closures of powers $(\overline{I^n})_{n\ge 0}$ and the family of symbolic powers $(I^{(n)})_{n \ge 0}$ of a monomial ideal $I \subseteq R$ to its Newton and symbolic polyhedrons, respectively. We prove Theorem \ref{thm.regOkn} by evoking a recent result \cite[Theorem 3.4]{HaN23}, which characterizes Noetherian families of monomial ideals, to imply that $\Delta(\I) = \frac{1}{c}\NP(I_c)$ for some $c \ge 1$, and reducing the assertion to that of the family of ordinary powers of a monomial ideal, which was already established in \cite[Theorem 2.7]{Hoa22}.

When the graded family $\I$ is not of monomial ideals, we focus on its Newton--Okounkov region constructed from a good valuation that \emph{respects the monomials in $R$}; that is, a valuation $v : \text{QF}(R) \setminus \{0\} \rightarrow \ZZ^r$ such that $v(x_1^{a_1} \dots x_r^{a_r}) = (a_1, \dots, a_r)$ for any monomial $x_1^{a_1} \dots x_r^{a_r}$ in $R$. The Gr\"obner valuation of $R$ is such a valuation. We prove the following result.

\medskip

\noindent\textbf{\Cref{thm.mPrimaryNOvaluation} and \ref{prop.NOofCMfamily}.}
    Let $\I$ be a graded family of homogeneous ideals in $R= \kk[x_1, \dots, x_r]$, and let $\Delta(\I)$ be its Newton--Okounkov region given by a good valuation that respects the monomials in $R$.
    \begin{enumerate}
        \item If $\I$ consists of $\mm$-primary ideals, and $\Delta(\I)$ is a polyhedron that is non-degenerate along the coordinate axes (see \Cref{def.nondegenaxes}), then
        \[
        \lim \limits_{n\to \infty} \frac{\reg I_n}{n} = \lim \limits_{n\to \infty} \frac{d(I_n)}{n} =\lim \limits_{n\to \infty} \frac{\reg \overline{I_n}}{n} = \lim \limits_{n\to \infty} \frac{d(\overline{I_n})}{n} = \delta(\Delta(\I)).
        \]
        \item Assume furthermore that $\kk$ is an uncountable field of characteristic zero. If $R/I_n$ is Cohen-Macaulay of the same dimension, for all $n \in \NN$, and $\Delta({\rm gin}(\I))$ is a polyhedron that is non-degenerate along the coordinate axes (see \Cref{def.nondegenaxes}), then
        \[
        \lim \limits_{n\to \infty} \frac{\reg I_n}{n} = \lim \limits_{n\to \infty} \frac{d(I_n)}{n}=\delta(\Delta({\rm gin}(\I))).
        \]
    \end{enumerate}

As before, the new contribution of \Cref{thm.mPrimaryNOvaluation} and \ref{prop.NOofCMfamily} lies at the last equality, while the other equalities were already shown in Theorems \ref{thm_m-primary} and \ref{thm_CohenMacaulay}. These new equalities are proved by approximation. For \Cref{thm.mPrimaryNOvaluation}, we approximate the given family $(I_n)_{n \ge 0}$ by a graded family of $\mm$-primary monomial ideals $(\widetilde{I_n})_{n \ge 0}$, for which $\widetilde{I_n}$ and $I_n$ have the same Hilbert function, and so the same regularity and Newton--Okounkov region, reducing the assertion to the case of monomial ideals, which is then proved by closely examining the vertices of the corresponding Newton--Okounkov region as limits points of sequences of the vertices of the Newton polyhedrons of ideals in the family; see \Cref{prop.monprimary}. The proof of \Cref{prop.NOofCMfamily} follows a similar line of arguments as that of \cite[Theorem 1.1]{Mayes14}, using induction, replacing $(I_n)_{n \ge 0}$ by a graded family of ideals $(L_n)_{n \ge 0}$ in the polynomial ring $R' = \kk[x_1, \dots, x_{r-c}]$, where $c = \dim R/I_n$, and reducing to the case of $\mm$-primary monomial ideals.


\subsection{Negative answers for Questions \ref{quest.HT} and \ref{quest.Intersection}} We proceed to give negative answers to Questions \ref{quest.HT} and \ref{quest.Intersection}.
We shall focus on the following setting: let $R = \kk[x,y,a,b]$ be a polynomial ring in $4$ variables over a field $\kk$ of characteristic 2, and let
$$Q = (x^3,y^3) \text{ and } f = xya - (x^2+y^2)b.$$
Consider $I_1 = Q$ and $I_2 = (f)$. The negative answers to Questions \ref{quest.HT} and \ref{quest.Intersection} come from the following theorem.

\medskip

\noindent\textbf{Theorem \ref{thm.noLimit}.} Let $Q$ and $f$ be as before. The limits
$$\lim\limits_{n \rightarrow \infty} \dfrac{\reg (Q^n + (f^n))}{n} \text{ and } \lim\limits_{n \rightarrow \infty} \dfrac{\reg (Q^n \cap (f^n))}{n}$$
do not exist.

\medskip

To prove Theorem \ref{thm.noLimit}, we exhibit two infinite sequences $(m_s)_{s \ge 1}$ and $(n_s)_{s \ge 0}$ for which
\begin{align*}\liminf_{s \rightarrow \infty} \dfrac{\reg (Q^{m_s}+(f^{m_s}))}{m_s} & \not= \limsup_{s \rightarrow \infty} \dfrac{\reg (Q^{n_s}+(f^{n_s}))}{n_s},\\
\liminf_{s \rightarrow \infty} \dfrac{\reg (Q^{m_s} \cap (f^{m_s}))}{m_s} & \not= \limsup_{s \rightarrow \infty} \dfrac{\reg (Q^{n_s} \cap (f^{n_s}))}{n_s}.
\end{align*}
In fact, we shall take $m_s= 2^s$ and $n_s = 3\cdot 2^s$, for $s \ge 1$. The standard short exact sequence
$$0 \longrightarrow R/(Q^n \cap (f^n)) \longrightarrow R/Q^n \oplus R/(f^n) \longrightarrow R/(Q^n + (f^n)) \longrightarrow 0,$$
allows us to focus on only one of the two desired limits, namely,
$\lim\limits_{n \rightarrow \infty} \dfrac{\reg (Q^n + (f^n))}{n}.$
The non-existence of this limit and, hence, \Cref{thm.noLimit} will then follow from Theorems \ref{thm_regbound_2powers} and \ref{thm_limregge6_3times2power} below.

\medskip

\noindent\textbf{Theorems \ref{thm_regbound_2powers}.} Consider $n=2^s$, where $s\ge 3$ is an integer. Consider the degree revlex order in which $x>y>a>b$.
\begin{enumerate}
	\item $\ini\left(Q^n+(f^n)\right)$ is given by
	\[
	\begin{cases}
		Q^n+(x^ny^na^n)+b^nx^{2n+1}yQ^\frac{n-2}{3}+b^nx^2y^{2n+1}Q^\frac{n-2}{3}, &\text{if $s$ is odd},\\
		Q^n+(x^ny^na^n)+b^nx^{2n+2}y^2Q^\frac{n-4}{3}+b^n\left(x^2y^{2n}Q^\frac{n-1}{3}+(x^{2n}y^{n+1})\right), &\text{otherwise}.
	\end{cases}
	\]
	\item We have the following containment:
	\begin{align*}
		x^ny^{2n+1}a^{n-1}b^{n-1} & \in \left(\left(Q^n+f^n\right):\mm\right)\setminus (Q^n+f^n), \quad \text{if $s$ is odd},\\
		x^{n+1}y^{2n}a^{n-1}b^{n-1} & \in \left(\left(Q^n+f^n\right):\mm\right)\setminus (Q^n+f^n), \quad \text{otherwise}.
	\end{align*}
	\item We have the following bounds for $\reg(Q^n + (f^n))$:
	\[
	5n\le \reg\left(Q^n+(f^n)\right)\le 5n+2.
	\]
\end{enumerate}

\medskip

\noindent\textbf{\Cref{thm_limregge6_3times2power}.} Consider $n=3\cdot 2^s$, where $s\ge 3$ is an odd integer, and set $t=n/3=2^s$. Then,
\[
x^{2t+1}y^{7t}a^{t-1}b^{8t-1} \in \left(\left(Q^n+(f^n)\right):\mm\right)\setminus \left(Q^n+(f^n)\right).
\]
In particular, we have: $\reg\left(Q^n+(f^n)\right)\ge 6n$.

\medskip

The proofs of Theorems \ref{thm_regbound_2powers} and \ref{thm_limregge6_3times2power} are based on a Gr\"obner basis analysis. Particularly, we describe the Gr\"obner basis and initial ideal of $Q^n + (f^n)$, when $n = 2^s$ or $n = 3 \cdot 2^s$, with respect to the degree reserve lexicographic term order, where $x > y > a > b$; see \Cref{thm_GB_2powers} and \Cref{prop_GB_3times2power}. These results are proved by a careful consideration of $S$-pairs of the desired generators. We hope that our detailed Gr\"obner basis analysis will be useful for the interested reader in finding other applications and in understanding the Gr\"obner basis techniques for sums, intersections and products of powers of ideals.


\subsection{Asymptotic regularity and Herzog--Hoa--Trung's question} It is our belief that the example considered in Theorem \ref{thm.noLimit} in fact also provides a counterexample to the open question of Herzog, Hoa and Trung Particularly, let $R = \kk[x,y,a,b]$ be as before, let $S = R[z]$, and consider $I = Q \cap (f,z) \subseteq S$. Note that $\Ass(I)=\Min(I)$, so symbolic powers of $I$ can be defined either using associated primes or minimal primes of $I$. We conjecture that the graded family of symbolic powers of $I$ does not have an asymptotic regularity.

\begin{conj}
	\label{conj.counterHHT}
	For $I = Q \cap (f,z) \subseteq S$ as above, the limit
	$\lim\limits_{n \rightarrow \infty} \dfrac{\reg I^{(n)}}{n}$ does not exist.
\end{conj}

\medskip

Inspired by the proof of Theorem \ref{thm.noLimit}, we propose to resolve Conjecture \ref{conj.counterHHT} by proving the following two inequalities:
\begin{enumerate}
	\item[(IN1)] $\liminf_{s\to \infty} \dfrac{\reg I^{(2^s)}}{2^s}  \le 6-\epsilon$, for some $\epsilon > 0$; and
	\item[(IN2)] $\limsup_{s\to \infty} \dfrac{\reg I^{(3\cdot 2^s)}}{3\cdot 2^s}  \ge 6.$
\end{enumerate}

While the first inequality (IN1) remains at large, \Cref{thm_limregge6_3times2power} already establishes the inequality (IN2).
We also observe that the inequality (IN1) is a consequence of the following conjectural statement.

\medskip

\noindent\textbf{Conjecture \ref{conj_regbound}.}
There exists $\epsilon >0$ such that the inequality
$$\reg\left(Q^n+(f^k)\right) \le (5-\epsilon)n+k$$
holds for all integers $1\le k\le n$ with $n=2^s$ for some integer $s \in \NN$.

\medskip

Macaulay2 \cite{GS96} computations suggest that one can take $\epsilon=1/2$. We provide partial answer to \Cref{conj_regbound} in \Cref{thm_regbound_2powers} and the following statement.

\medskip

\noindent\textbf{Theorems \ref{thm_regbound_double2powers_odd}.(3) and \ref{thm_regbound_double2powers_even}.(3).}
Let $n=2^s, k=2^u$, where $0 \le u < s$. Then,
\[
3n+3k-1 \le \reg \left(Q^n+(f^k)\right)\le 3n+3k+2.
\]

\medskip

As with Theorems \ref{thm_regbound_2powers} and \ref{thm_limregge6_3times2power}, the proof of Theorems \ref{thm_regbound_double2powers_odd} and \ref{thm_regbound_double2powers_even} are based on a careful analysis of the Gr\"obner basis of ideals of the form $Q^n + (f^k)$. This Gr\"obner basis is found by examining all $S$-pairs of potential generators. While some of these $S$-pairs are easy to handle, many others require significantly more work that call upon several known results. We have not been able to complete the proof of Conjecture \ref{conj_regbound} and, thus, of Conjecture \ref{conj.counterHHT} because computational experiments seem to suggest that the Gr\"obner basis of $Q^n + (f^k)$, when $k$ is not a power of 2, is too complicated to handle.

We point out, as a final remark, that the example provided in \Cref{conj.counterHHT} appears to have worse behavior in characteristic 0. We raise Conjectures \ref{conj.regChar0} and \ref{conj.Quad} to illustrate the quadratic growth of $\reg (Q^n+(f^n))$ and $\reg I^{(n)}$ when $\chara \kk = 0$. In particular, the condition $\chara \kk=2$ is indispensible in Conjecture \ref{conj_regbound}.


\subsection{Outlines of the paper} In this section, we discuss motivations of our work and state the main results of the paper. In Section \ref{sec.prel}, we give basic notations and terminology used in the paper. Section \ref{sec.Noeth} focuses on Noetherian graded families of homogeneous ideals and establishes that the asymptotic regularity of these families exists. Our first main result, Theorem \ref{thm_limit_noeth_families} is proved in this section. In Section \ref{sec.nonNoeth}, we show that if the family consists of $\mm$-primary (artinian) ideals then its asymptotic regularity exists. We also provide additional cases where the families are not necessarily Noetherian, and yet their asymptotic regularities still exist. Theorems \ref{thm_m-primary}, \ref{thm_CohenMacaulay} and \ref{thm_limit_varyingpowers} are established in this section. Section \ref{sec.NObody} gives a combinatorial interpretation of the asymptotic regularity of graded families of ideals in terms of their associated Newton--Okounkov regions in several cases. Theorems \ref{thm.regOkn} and \ref{thm.mPrimaryNOvaluation} and \ref{prop.NOofCMfamily} are proved in this section.

Sections \ref{sec.nonExist}--\ref{sec.2nd} are devoted to illustrating a negative answer to Question \ref{quest.HT}. This is established in Theorem \ref{thm.noLimit} whose proof is split into two statements, expressed in Theorems \ref{thm_regbound_2powers} and \ref{thm_limregge6_3times2power}. The proofs of these results then occupy Section \ref{sec.1st} and Section \ref{sec.2nd}, respectively.

In Sections \ref{sec.HHT}--\ref{sec.regQnfk_Even}, we demonstrate a strong evidence for a counterexample toward the question of Herzog, Hoa and Trung of whether the asymptotic regularity of the family of symbolic powers of a homogeneous ideal always exists. Such a counterexample is established via Conjecture \ref{conj.counterHHT}. In order to resolve this conjecture, we proceed in proving two inequalities (IN1) and (IN2), where (IN2) follows from the statement of Theorem \ref{thm_limregge6_3times2power}. The desired inequality (IN1) is a consequence of a more detailed statement, Conjecture \ref{conj_regbound}. In Theorems \ref{thm_regbound_double2powers_odd} and \ref{thm_regbound_double2powers_even}, we provide partial answer to Conjecture \ref{conj_regbound}.


\section*{Acknowledgments}

Part of this work was done when the second author visited the first author at Tulane University, and when the authors were at the Vietnam Institute for Advanced Study in Mathematics (VIASM) for a research program. The authors thank Tulane University and VIASM for their supports and hospitality. The first author (THH) is partially supported by a Simons Foundation grant. The second author (HDN) is supported by NAFOSTED under the grant number 101.04-2023.30, and by the Vietnam Academy of Science and Technology under the grant number CTTH00.01/25-26.


\section{Preliminaries} \label{sec.prel}

In this section, we provide basic notations and terminology used in the paper, and recall a number of necessary auxiliary results. For unexplained terminology, we refer the reader to standard texts \cite{BH98, Eis95}.

Throughout this paper, $R$ is a standard graded algebra over a field $\kk$ and $\mm$ denotes its maximal homogeneous ideal.


\subsection{Castelnuovo-Mumford regularity and symbolic powers of ideals} We begin with the main algebraic invariant studied in this paper.

\begin{defn}
	\label{def.regularity}
	Let $M$ be a finitely generated graded $R$-module. For $i \ge 0$, set
	$$a^i(M) = \max\{n ~\big|~ \left[H^i_\mm(M)\right]_n \not= 0\}.$$
	The \emph{Castelnuovo-Mumford regularity} (or simply, \emph{regularity}) of $M$ is defined to be
	$$\reg M = \max_{i \ge 0} \{a^i(M)+i\}.$$
\end{defn}

We collect some basic facts about the regularity of homogeneous ideals and graded modules.

\begin{lem}[{Chandler \cite{Ch97}, Conca-Herzog \cite{CH03}}]
\label{lem_reg_lowdim}
Let $I$ be a homogeneous ideals of $R$, and $M$ a finitely generated graded $R$-module. Assume that $\dim(R/I)\le 1$. Then there is an inequality
\[
\reg(IM)\le \reg(I)+\reg(M).
\]
\end{lem}

\begin{lem}[Folklore]
\label{lem_reg_containment}
Let $H\subseteq L$ be proper homogeneous  ideals of $R$, such that $\dim(R/H)$ $=\dim(R/L)$ and $R/L$ is Cohen-Macaulay. Then there is an inequality $\reg H\ge \reg L$.

In particular, if $H\subseteq L$ are $\mm$-primary ideals then $\reg H\ge \reg L$.
\end{lem}
\begin{proof}
Setting $X=L/H$, we have an exact sequence
\[
0\to X \to R/H \to R/L \to 0.
\]
Let $d=\dim(R/L)$. Since $\dim(X)\le \dim(R/H)=d$, $H^{d+1}_\mm(X)=0$. The exact sequence of local cohomology yields
\[
H^d_\mm(R/H) \to H^d_\mm(R/L) \to H^{d+1}_\mm(X)=0
\]
Since $R/L$ is Cohen-Macaulay, its unique non-vanishing local cohomology is $H^d_\mm(R/L)$. Hence the last exact sequence yields
\[
\reg(R/L)=\sup\{d+j: H^d_\mm(R/L)_j\neq 0\} \le \reg(R/H).
\]
The desired conclusion follows.
\end{proof}

\begin{lem}
\label{lem_regular_linform}
Let $I$ be a homogeneous ideal of $R$ and $x\in R_1$ an $(R/I)$-regular linear form. Then the ideals $I\subseteq R$ and $(I,x)/(x) \subseteq R/(x)$ have the same regularity and maximal generating degree.
\end{lem}
\begin{proof}
Let $F$ be the minimal graded free resolution of $I$. Then $F\otimes_R (R/(x))$ is a minimal graded free resolution of $I/xI = I /((x)\cap I) \cong (I,x)/(x)$ over $R/(x)$, where the equality holds since $x$ is $(R/I)$-regular. The desired conclusions follow since we can read the regularity and maximal generating degree of a module from any  minimal graded free resolution.
\end{proof}

Symbolic powers are the main objects in the aforementioned question of Herzog, Hoa and Trung \cite{HeHoTr02}. There are two definitions of symbolic powers existed and equally used in the literature, one defined by \emph{minimal primes} and the other defined by \emph{associated primes}. We shall recall both of these definitions and use distinct notations for clarity.

For an ideal $I \subseteq R$, let $\Min(I)$ and $\Ass(I)$ denote the sets of minimal primes and associated primes of $I$, respectively.

\begin{defn}
	\label{def.symbolic}
	Let $I \subseteq R$ be an ideal and let $n \in \NN$. The \emph{$n$-th \textup{(}associated, resp. minimal\textup{)} symbolic power} of $I$ is given by
	$$\up{a}I^{(n)} = \bigcap_{\pp \in \Ass(I)} (I^n R_\pp \cap R) \text{ and } \up{m}I^{(n)} = \bigcap_{\pp \in \Min(I)} (I^n R_\pp \cap R).$$
\end{defn}

When the ideal $I$ has no embedded primes (such an ideal will be called \emph{unmixed}), these two definitions of $\up{a}I^{(n)}$ and $\up{m}I^{(n)}$ coincide, and we will write simply $I^{(n)}$ for either of them, following the traditional notation.

Herzog--Hoa--Trung's was for symbolic powers defined using minimal primes. One can naturally ask the same question for symbolic powers defined using associated primes. In this paper, we will work mainly with the ``associated" symbolic powers, i.e. the first definition in \Cref{def.symbolic}, because results concerning ``minimal" symbolic powers are usually special cases of the corresponding results for ``associated" symbolic powers in the following sense. 
\begin{rem}
\label{rem.assimpliesmin}
Let $I$ be an ideal in a Noetherian ring $R$. Denote $J={} \up{m}I^{(1)}$, the intersection of minimal primary components of $I$. Note that $J$ is an unmixed ideal and $\Ass(J)=\Min(J)=\Min(I)$. It is not hard to see that $\up{m}I^{(n)}={}\up{a}J^{(n)}$ for all $n\ge 1$.

For completeness, we include an argument for the last claim, no doubt well-known to experts. Let $I=Q_1\cap \cdots \cap Q_s \cap Q_{s+1}\cap \cdots \cap Q_t$ be an irredundant primary decomposition of $I$, such that $\Min(I)=\{\sqrt{Q_1},\ldots,\sqrt{Q_s}\}$. Denote $\pp_i=\sqrt{Q_i}$ for $1\le i\le s$. Since $Q_i$ is unmixed, as mentioned above, the $n$-th ``associated" and ``minimal" symbolic powers $Q_i$ are identical, and can be written as $Q_i^{(n)}$. The crucial observation is $I^nR_{\pp_i}=Q_i^nR_{\pp_i}$ for all $i$ and all $n\ge 1$, thanks to the fact that $\pp_i\in \Min(I)$ and $Q_i$ is primary to $\pp_i$. Hence
\[
\up{m}I^{(n)} = \cap_{i=1}^s (I^nR_{\pp_i}\cap R) =\cap_{i=1}^s (Q_i^nR_{\pp_i}\cap R) = \cap_{i=1}^s Q_i^{(n)}.
\]
In particular, $J={} \up{m}I^{(1)}=Q_1\cap \cdots\cap Q_s$ satisfies $\Ass(J)=\Min(J)=\Min(I)$. Since $J$ is unmixed, $\up{m}J^{(n)}={}\up{a}J^{(n)}$ for all $n\ge 1$. The same arguments as above applied to $J$ yield
\[
\up{a}J^{(n)}=\cap_{i=1}^s Q_i^{(n)}={} \up{m}I^{(n)}.
\]
\end{rem}
We add that related to Herzog--Hoa--Trung's question, the main results of Sections \ref{sec.nonExist} onward are concerned exclusively with unmixed ideals. Hence, \emph{unless otherwise stated, in all of the results concerning symbolic powers in this paper, there is no need to distinguish between the two notions of symbolic powers}.


\subsection{Graded families of ideals} We continue with graded families of ideals and their associated Rees algebras.

\begin{defn}
	\label{def.Rees}
	\quad
	\begin{enumerate}
		\item A collection $(I_n)_{n \ge 0}$ of ideals in $R$ is called a \emph{graded family} of ideals if $I_p \cdot I_q \subseteq I_{p+q}$ for all $p, q \ge 0$. By convention, we will assume that $I_0 = R$ when discussing graded families of ideals. A graded family $(I_n)_{n \ge 0}$ is called a \emph{filtration} of $I_p \supseteq I_{p+1}$ for all $p \ge 0$.
		\item The \emph{Rees algebra} of a graded family $\I = (I_n)_{n \ge 0}$ of ideals in $R$ is defined by
		$$\R(\I) = \bigoplus_{n \ge 0} I_nt^n \subseteq R[t].$$
		\item A graded family $\I$ of ideals in $R$ is said to be \emph{Noetherian} if its Rees algebra is a Noetherian ring.
	\end{enumerate}
\end{defn}

For a graded algebra $A=\bigoplus_{i=0}^\infty A_i$ over an arbitrary ring $A_0$ (not necessarily a field) and an integer $c \ge 1$, the $c$-th Veronese subring  of $A$ is 
$$A^{(c)}=\bigoplus_{i=0}^\infty A_{ic}.$$
The ring $A^{(c)}$ can also be viewed as a graded algebra over $A_0$, by setting $\deg(A_{ic})=i$.
The next result is known to experts; we mostly follow the argument given in \cite{HHT07}.

\begin{prop}
\label{prop_fg_Reesalg}
Let $\I=(I_n)_{n\ge 0}$ be graded family of ideals of $R$. Denote $A=\R(\I)$ the Rees algebra of $\I$. The following statements hold.
\begin{enumerate}[\quad \rm (1)]
 \item The following conditions are equivalent:
\begin{enumerate}[\quad \rm (i)]
\item $A$ is a finitely generated $R$-algebra;
\item There exists an integer $c\ge 1$ such that $A^{(c)}$ is a standard graded $R$-algebra;
 \item There exists an integer $c\ge 1$ such that $I_{nc}=I_c^n$ for all $n\ge 0$.
\end{enumerate}
 \item Assume that $A$ is a finitely generated $R$-algebra, and $c$ is as in \textup{(iii)} of part \textup{(1)}. Then there exist an integer $N\ge 1$ and ideals $L_1,\ldots,L_{c-1}\subseteq R$ such that $I_{nc+i}=L_iI_c^{n-N}$ for all $1\le i\le c-1$ and all $n\ge N$.

 In particular,  setting $\ell=Nc$, we have $I_{n+\ell}=I_nI_\ell$ for all $n\ge \ell$.
\end{enumerate}
\end{prop}
\begin{proof}
(1) The equivalence between (i) and (ii) is contained in \cite[Theorem 2.1]{HHT07}. Condition (iii) is simply a restatement of (ii).

(2) Since $A$ is a finitely generated $R$-algebra that is graded, we can choose $f_1,\ldots,f_p$ to be a finite set of homogeneous generators of $A$ over $R$. Then $A$ is also generated as an $A^{(c)}$-algebra by $f_1,\ldots,f_p$. As $I_i^c\subseteq I_{ic}$ for all $i$, $f_i^c\in A^{(c)}$ for $i=1,\ldots,p$. Hence $f_1,\ldots,f_p$ are integral over $A^{(c)}$, and $A$ is a finitely generated $A^{(c)}$-module. Choose $N$ large enough such that $A$ is generated as an $A^{(c)}$-module by elements in the graded pieces $I_1, \ldots, I_{Nc+c-1}$. For each $1\le i\le c-1$, set $L_i=I_{Nc+i}$.

Consider the ideal $I_{nc+i}$ where $1\le i\le c-1$, $n\ge N$. By the choice of $N$, for each $1\le i\le c-1$, the $A^{(c)}$-submodule
\[
V_{c,i}=\bigoplus_{j=0}^\infty I_{jc+i}
\]
of $A$ (whose grading is given by $\deg(I_{jc+i})=j$), is generated by the graded pieces $I_{jc+i}, 0\le j\le N$. Therefore
\[
I_{nc+i}=\sum_{j=0}^N A^{(c)}_{n-j}I_{jc+i}=\sum_{j=0}^NI_c^{n-j}I_{jc+i}=I_c^{n-N}I_{Nc+i}.
\]
The last equality holds since for $0\le j\le N$, $I_c^{n-j}I_{jc+i}=I_c^{n-N}I_c^{N-j}I_{jc+i}\subseteq I_c^{n-N}I_{Nc+i}$. Hence the required equalities hold with our choice of $N$ and the $L_i$s.

Finally we verify that $I_{n+Nc}=I_nI_{Nc}$ for all $n\ge Nc$. If $c$ divides $n$, this is clear. If not, then $n=pc+i$, where $1\le i\le c-1$. As $pc+i\ge Nc$, we deduce $p\ge N$. Now
\[
I_{n+Nc}=I_{(p+N)c+i}=L_iI_c^p=L_iI_c^{p-N}I_c^N=I_{pc+i}I_{Nc}=I_nI_{Nc},
\]
as desired.
\end{proof}
\begin{rem}
From \Cref{prop_fg_Reesalg}, we see that \cite[Remark (2.4.3)]{Rat79} is incorrect: With $\ell=Nc$, we have $I_{n+\ell}=I_nI_\ell$ for all $n\ge \ell$.  But as  pointed out in \cite[Remark 2.4, p.\,309]{HHT07}, Remark (2.4.4) in \cite{Rat79} is wrong! This remark of  \cite{Rat79} implies that if $(I_n)_{n\ge 0}$ is graded family of ideals with
$$I_0=R \supseteq I_1\supseteq I_2 \supseteq \cdots$$
and $\R(\I)$ is a finitely generated $R$-algebra, then $I_{nh}=I_h^n$ for all $n\ge 0$ and all large enough $h$. An example was given in \cite[Remark 2.4]{HHT07}. Here is another simple example. Let $R=\kk[x_1,\ldots,x_r]$ where $r\ge 1$, $I \subseteq J$ arbitrary non-zero homogeneous ideals of $R$. Choose the family $\I=(I_n)_{n\ge 0}$ as follows: $I_n=I^n$ if $n$ is even and $I_n=J I^n$ if $n$ is odd. Since $I\subseteq J$, $I_{n+1} \subseteq I_n$ for all $n$. Moreover
\[
\R(\I)=R\oplus J I\oplus I^2 \oplus J I^3 \oplus I^4 \oplus \cdots,
\]
is clearly generated by $I_1$ and $I_2$. On the other hand, if $n$ and $h$ are arbitrary odd numbers such that $n\ge 2$,
\[
I_{nh}=J I^{nh} \quad \text{strictly contains} \quad I_h^n = J^n I^{nh},
\]
by the Nakayama's lemma.
\end{rem}

The equivalence of (1), (2), and (3) in the next result was established in \cite[Theorem 3.4]{HaN23} for families of monomial ideals. Their proof carries over almost verbatim to the case of arbitrary families. We include the details for completeness of the exposition and convenience of the reader.

\begin{lem}[{H\`a--Nguy$\tilde{\text{\^e}}$n \cite{HaN23}}]
\label{lem_intclos_family}
Let $\I=(I_n)_{n\ge 0}$ be a graded family of ideals of $R$. Let $\overline{\I}=(\overline{I_n})_{n\ge 0}$ be the corresponding family of integral closures. The following are equivalent:
\begin{enumerate}[\quad \rm (1)]
 \item $\R(\I)$ is a finitely generated $R$-algebra;
 \item There exists an integer $c\ge 1$ such that $\overline{I_{nc}}=\overline{I_c^n}$ for all $n\ge 0$;
 \item $\R(\overline{\I})$ is a finitely generated $R$-algebra.
\end{enumerate}
Moreover, if $R$ is reduced and $\R(\I)$ is a finitely generated $R$-algebra, then there exists integers $d, k \ge 1$ such that for all $n\ge dk+2d$ and for $i\in \{d,d+1,\ldots ,2d-1 \}$ the unique integer such that $n \equiv i$ (mod $d$), we have $\overline{I_n}=I_{n-dk-i}\overline{I_{dk+i}}$.
\end{lem}
\begin{proof}
(1) $\Longrightarrow$ (2): Assume that $\R(\I)$ is a finitely generated $R$-algebra. By \Cref{prop_fg_Reesalg}, there exists an integer $c\ge 1$ such that $I_{nc}=I_c^n$ for all $n\ge 0$. Hence $\overline{I_{nc}}=\overline{I_c^n}$ for all $n\ge 0$.

(2) $\Longrightarrow$ (3): Looking at the $c$-th Veronese subalgebra of $\R(\overline{\I})$, we get
\[
\R(\overline{\I})^{(c)}=\bigoplus_{n\ge 0} \overline{I_{nc}}=\bigoplus_{n\ge 0} \overline{I_c^n}=\overline{\R(I_c)},
\]
where $\R(I_c)$ is the usual Rees algebra of the ideal $I_c$. Since $\R(I_c)$ is a finitely generated algebra over $\kk$, so is $\overline{\R(I_c)}$ by \cite[Theorem 4.6.3]{SH06}. Thus $\R(\overline{\I})^{(c)}$ is finitely generated over $R$, and hence so is $\R(\overline{\I})$, by \Cref{prop_fg_Reesalg}.

(3) $\Longrightarrow$ (1): By \Cref{prop_fg_Reesalg}, there exists an integer $c\ge 1$ such that $\overline{I_{nc}}=\left(\overline{I_c}\right)^n$ for all $n\ge 1$. Using \cite[Remark 1.3.2(4)]{SH06}, we get the first inclusion in the chain
\[
\left(\overline{I_c}\right)^n \subseteq \overline{I_c^n} \subseteq \overline{I_{cn}} =\left(\overline{I_c}\right)^n.
\]
Thus equalities hold through from left to right. In particular, $\overline{I_{nc}}=\overline{I_c^n}$ for all $n\ge 0$. Again the $c$-th Veronese subalgebra of $\R(\overline{\I})$ is
\[
\R(\overline{\I})^{(c)}=\bigoplus_{n\ge 0} \overline{I_{nc}}=\overline{\R(I_c)}.
\]
Hence $\R(\overline{\I})^{(c)}$ is a finitely generated $\R(I_c)$-module. Since $\R(\overline{\I})$ is finitely generated over $R$ (hence also over $\R(\overline{\I})^{(c)}$), and each element of $\R(\overline{\I})$ is integral over $\R(\overline{\I})^{(c)}$, $\R(\overline{\I})$ is module-finite over $\R(\overline{\I})^{(c)}$. Therefore $\R(\overline{\I})$ is also module-finite over $\R(I_c)$. Being an $\R(I_c)$-submodule of  $\R(\overline{\I})$, $\R(\I)$ is also finitely generated over  $\R(I_c)$. This implies that $\R(\I)$ is finitely generated over $R$, as desired.

For the last assertion, assume that $\R(\I)$ is a finitely generated $R$-algebra. Then so is $\R(\overline{\I})$, hence, by \Cref{prop_fg_Reesalg}, there exists an integer $d \ge 1$ such that 
\[
\overline{I_n} = \overline{I_{n-d}}\cdot \overline{I_d}, \text{ for all } n\ge 2d.
\]
Furthermore, by \cite[Corollary 9.2.1 (2)]{SH06}, since $R$ is reduced, it is analytically unramified, hence there exists an integer $k\ge 1$ such that
\[
\overline{I_d^n} = I_d^{n-k}\overline{I_d^k}, \text{ for all } n\ge k.
\]
We now show that $d, k$ are the desired integers in the claim. Indeed, for $n\ge dk+2d$ and for $i\in \{d,d+1,\ldots ,2d-1 \}$ with $n \equiv i$ (mod $d$), we get
\begin{align*}
    \overline{I_n} &=\overline{I_{n-d}}\cdot \overline{I_d} = \overline{I_{n-2d}} \cdot \left( \overline{I_d}\right)^2 = \ldots = \overline{I_{i}}\left( \overline{I_d}\right)^{\frac{n-i}{d}} \subseteq \overline{I_{i}}\cdot \overline{I_d^{\frac{n-i}{d}}} \subseteq \overline{I_iI_d^{\frac{n-i}{d}}} \subseteq \overline{I_iI_{n-i}} \subseteq \overline{I_n}.
\end{align*}
Therefore, all containment above become equalities, in particular,
\[
\overline{I_n} = \overline{I_{i}}\cdot \overline{I_d^{\frac{n-i}{d}}}.
\]

Since $\frac{n-i}{d}\ge \frac{n-2d}{d}\ge k$, the choice of $k$ yields
\[
\overline{I_d^{\frac{n-i}{d}}} = I_d^{\frac{n-i}{d}-k} \overline{I_d^k},
\]
which implies that
\begin{align*}
    \overline{I_n} &= I_d^{\frac{n-i}{d}-k} \overline{I_d^k} \cdot \overline{I_i} \subseteq I_{n-dk-i}\overline{I_d^k\cdot I_i} 
    \subseteq I_{n-dk-i}\overline{I_{dk+i}} \subseteq \overline{I_{n-dk-i}I_{dk+i}} \subseteq \overline{I_n}.
\end{align*}
Therefore, equalities hold throughout, in particular, $\overline{I_n}=I_{n-dk-i}\overline{I_{dk+i}}$ as claimed.

\end{proof}


\subsection{Newton--Okounkov region and Noetherian Rees algebra} We recall the general construction of the Newton--Okounkov region associated to a graded family of ideals.
Let $(R,\mm)$ be either a local domain or a standard graded domain over a field $\kk$ and assume that $\kk \simeq R/\mm$. Let $K$ be the quotient field of $R$. Given a total ordering $>$ in $\ZZ^r$, a \emph{valuation} $v: K \setminus \{0\} \rightarrow \ZZ^r$ is a function satisfying the following conditions:
\begin{enumerate}
	\item $v(fg) = v(f) + v(g)$ for all $f,g \not= 0$ ;
	\item $v(f+g) \ge \min\{v(f), v(g)\}$ for all $f,g \not= 0$;
	\item $v(\lambda) = 0$ for all $0 \not= \lambda \in \kk.$
\end{enumerate}

\begin{defn} \label{def.good}
A valuation $v: K \setminus \{0\} \rightarrow \ZZ^r$ is said to have \emph{one-dimensional leaves} if whenever $v(f) = v(g)$, there exists $\lambda \in \kk$ with $v(g+\lambda f) > v(g)$. A valuation $v$ with one-dimensional leaves is \emph{good} if the following conditions hold:
\begin{enumerate}
	\item the value group $S = v(R \setminus \{0\}) \cup \{0\}$ generates the whole lattice $\ZZ^r$, and its associated cone $C(S)$ is a strongly convex cone;
	\item there exists $r_0 > 0$ and a linear function $\ell: \RR^r \rightarrow \RR$ such that $C(S)$ lies in $\ell^{-1}(\RR_{\ge 0})$ and intersects with $\ell^{-1}(0)$ only at the origin, and for any $f \in R$ if $\ell(v(f)) \ge kr_0$ for some $k > 0$ then $f \in \mm^k$.
\end{enumerate}
\end{defn}

\noindent It is shown in \cite[Theorem 7.6]{KK2014} and \cite{Cut2014} that if $R$ is a regular local ring then $R$ has good valuations.

\begin{defn}
	Let $v$ be a good valuation and let $\I = (I_n)_{n\ge 0}$ be a graded family of ideals in $R$. For $n \in \NN_{>0}$, set
	$$V_n = \{v(f) ~\big|~ f \in I_n \setminus \{0\}\}.$$
	The \emph{limiting region} and \emph{Newton--Okounkov region} of $\I$ is defined to be
	$$\mathcal{C}(\I) = \bigcup_{n \in \NN_{>0}} \conv \left\langle\left\{\dfrac{\a}{n} ~\big|~ \a \in V_n\right\}\right\rangle \subseteq \RR^r \text{ and } \Delta(\I) = \overline{\mathcal{C}(\I)} \subseteq \RR^r.$$
\end{defn}

It is easy to see from definition that the limiting and Newton-Okounkov regions of a graded family of ideals are convex sets. However, they are often not compact, as oppose to the Newton Okounkov body.

When discussing families of ideals in the polynomial ring $R = \kk[x_1, \dots, x_r]$, we shall restrict our attention to valuations that \emph{respect the monomials in $R$}; that is, valuations $v: K \setminus \{0\} \rightarrow \ZZ^r$ such that for any monomial $x^\a = x_1^{a_1} \cdots x_r^{a_r}$ in $R$, where $\a = (a_1, \dots, a_r) \in \ZZ^r_{\ge 0}$, we have $v(x^\a) = \a$. The \emph{Gr\"obner valuation} of $R$ is such a valuation.

\begin{ex}
Fix a monomial order $>$ on $R= \kk[x_1, \dots, x_r]$, and consider the \emph{Gr\"obner valuation} $v: \text{QF}(R) \setminus \{0\} \rightarrow \ZZ^r$ given by
$$v(f) = \text{the multidegree of the least term of } f \text{ with respect to } >,$$ for $f\in R$, and extend to $\text{QF}(R)$ by $v(f/g)=v(f)-v(g)$. Then $v$ is a good valuation that respects the monomials of $R$.
\end{ex}

\begin{rem} For a graded family $\I = (I_n)_{n \ge 0}$ of \emph{monomial} ideals in a polynomial ring $R = \kk[x_1, \dots, x_r]$, the Newton--Okounkov region constructed from the Gr\"obner valuation has a simple presentation, namely,
 	$$\mathcal{C}(\I) = \bigcup_{n \in \NN} \conv\left\langle\left\{ \dfrac{\textbf{a}}{n} ~\big|~ x^\textbf{a} \in I_n \right\}\right\rangle \text{ and } \Delta(\I) = \overline{\mathcal{C}(\I)}.$$
 This presentation of $\mathcal{C}(\I)$ and $\Delta(\I)$ has been considered in \cite{HaN23, KK2014}.
\end{rem}

For the rest of this section, let $R=\kk[x_1,\ldots,x_r]$. We have the following lemma that relates the Newton-Okounkov region of a graded family of ideals $\I = (I_n)_{n\ge 0}$ and that of the graded family of their integral closure $\overline{\I} = (\overline{I_n})_{n\ge 0}$. This generalizes the result for graded families of monomial ideals \cite[Proposition 3.9]{HaN23}.

\begin{lem}
    \label{lem.NOIntClo}
    Let $\Delta(\I)$ and $\Delta(\overline{\I})$ be the Newton--Okounkov regions of the graded families $\I = (I_n)_{n\ge 0}$ and $\overline{\I} = (\overline{I_n})_{n\ge 0}$ of $R$, respectively, constructed from a good valuation $v$. Then $$\Delta(\I)= \Delta(\overline{\I}).$$
\end{lem}

\begin{proof}
    Denote $V_n(\I)=\{v(f) ~\big|~ f \in I_n \setminus \{0\}\}$ and $V_n(\overline{\I})=\{v(f) ~\big|~ f \in \overline{I_n} \setminus \{0\}\}$. Since for each $n$, $I_n\subseteq \overline{I_n}$, we have $V_n(\I)\subseteq V_n(\overline{\I})$, thus $\Delta(\I)\subseteq \Delta(\overline{\I})$. Now for any $n$ and for any $\textbf{a} \in V_n(\overline{\I})$, let $f\in \overline{I_n}$ so that $v(f) = \textbf{a}$. As $f\in \overline{I_n}$, by \cite[Corollary 6.8.2]{SH06} there exists a non-zero divisor $c$ of $R$ such that $cf^s \in I_n^s \subseteq I_{ns}$ for all $s\gg 1$. It follows that for all $s\gg 1$,
    \[
    v(c)+sv(f)\in V_{ns}(\I) \subseteq ns\cdot \Delta(\I),
    \]
    which implies that 
    \[
    \frac{v(c)}{ns}+\frac{v(f)}{n} \in \Delta(\I), \quad \forall s\gg 1.
    \]
    Letting $s\rightarrow \infty$, since $\Delta(\I)$ is a closed set, we get $\frac{\textbf{a}}{n} \in \Delta(\I)$.
    This shows that for each $n$, $\frac{1}{n}\cdot V_n(\overline{\I}) \subseteq \Delta(\I)$ and since $\Delta(\I)$ is convex, we have $\Delta(\overline{\I}) \subseteq \Delta(\I)$.
\end{proof}

 \begin{ex}
\label{ex.NObody}
 Let $I \subseteq R$ be a monomial ideal.
\begin{enumerate}
\item The \emph{Newton polyhedron} of $I$ is defined to be
 $$\NP(I) = \conv\left(\left\{ \textbf{a} \in \NN^r ~\big|~ x^\textbf{a} \in I\right\}\right).$$
 If $\I = (I^n)_{n \ge 0}$ or $\I = (\overline{I^n})_{n \ge 0}$, then $\Delta(\I)=\NP(I)$.
 \item Let $\text{maxAss}(I)$ denote the set of maximal associated primes of $I$ and set $Q_{\pp} = R \cap IR_\pp$ for $\pp \in \text{maxAss}(I).$ The \emph{symbolic polyhedron} of $I$ is defined to be
 $$\SP(I) = \bigcap_{\pp \in \text{maxAss}(I)} \NP(Q_{\subseteq \pp}).$$
 Let $\I = (\up{a}I^{(n)})_{n \ge 0}$, where $\up{a}I^{(n)}$ is defined in terms of associated primes, that is, $\up{a}I^{(n)}=\bigcap_{\pp \in \Ass(I)}(I^nR_{\pp}\cap R)$. Then, $\Delta(\I)=\SP(I)$.
\item Let $\I = (\up{m}I^{(n)})_{n \ge 0}$, where $\up{m}I^{(n)}$ is defined in terms of minimal primes, that is, $\up{m}I^{(n)}=\bigcap_{\pp \in \Min(I)}(I^nR_{\pp}\cap R)$. Denote $J={} \up{m}I^{(1)}$, the intersection of minimal primary components of $I$. Note that $J$ is an unmixed ideal and $\Ass(J)=\Min(J)=\Min(I)$. By \Cref{rem.assimpliesmin}, $\up{m}I^{(n)}={}\up{a}J^{(n)}$ for all $n\ge 1$. Therefore, $\Delta(\I)=\SP(J)$.

Let $I=Q_1\cap \cdots\cap Q_s \cap \cdots \cap Q_t$ be an irredundant primary decomposition of $I$ as an intersection of monomial primary ideals, and assume that $Q_1,\ldots,Q_s$ are the components corresponding to minimal primes of $I$. Then $J=Q_1\cap \cdots \cap Q_s$, and $\Delta(\I)=\SP(J)=\NP(Q_1)\cap \cdots \cap \NP(Q_s)$. The last polyhedron was denoted as $\mathcal{SP}(I)$ in \cite[Page 1489]{DHNT21}.
 	\end{enumerate}
 \end{ex}

The following theorem provides a characterization of the Noetherian property of a Rees algebra of a graded family of monomial ideals via the geometric property of its associated Newton--Okounkov/limiting regions.

 \begin{thm}[{\cite[Theorem 3.4]{HaN23}}]
 	\label{thm.NPpoly}
 	Let $\I = (I_n)_{n \ge 0}$ be a graded family of monomial ideals and $\overline{\I} = (\overline{I_n})_{n \ge 0}$. The following are equivalent:
 	\begin{enumerate}[\quad \rm (1)]
 		\item The limiting region ${\displaystyle \mathcal{C}(\I) = \bigcup_{n \in \NN}\dfrac{1}{n}\NP(I_n)}$ is a polyhedron.
 		\item There exists an integer $c$ such that $\Delta(\I) = \dfrac{1}{c}\NP(I_c)$.
 		\item There exists an integer $c$ such that $\dfrac{1}{c}\NP(I_c) = \dfrac{1}{nc}\NP(I_{nc})$ for all $n \in \NN_{>0}$.
 		\item $\R(\I)$ is Noetherian.
 	\end{enumerate}
 \end{thm}


\section{Noetherian Rees algebras and the existence of asymptotic regularity} \label{sec.Noeth}

In this section, we give an affirmative answer to Question \ref{quest.DSN}. We shall begin with examples illustrating that the regularity function $\reg I_n$ of a graded family $(I_n)_{n \ge 0}$ of homogeneous ideals can have arbitrary behavior in general. This example also demonstrates that the analytic spread of a graded family of ideals studied by Dao and Monta\~no \cite{DM2021} may not be finite, as opposed to that considered by Cutkosky and Sarkar \cite{CS2021}.

\begin{ex}
\label{ex_diverge}
Let $R=\kk[x,y,a,b]$ be a polynomial ring over $\kk$. Let $f: \NN \to \NN$ be any \emph{non-decreasing} function such that $f(n) \ge n$ for all $n\ge 1$. Consider the collection $\I=(I_n)_{n\ge 0}$ of ideals given by $I_0=R$, and
$$I_n=(a^4,a^3b,ab^3,b^4)(x,y)^n+a^2b^2(x,y)^{f(n)} \text{ for } n\ge 1.$$
Let $d(I)$ and $\mu(I)$ denote maximal degree of a minimal homogeneous generating set and the minimal number of generators of a homogeneous ideal $I \subseteq R$, respectively. We claim that:
\begin{enumerate}
 \item $(I_n)_{n\ge 0}$ is a graded filtration of ideals;
 \item $\reg I_n=f(n)+5$ and $d(I_n)=f(n)+4$ for all $n\ge 1$;
 \item $\mu(I_n)=f(n)+4n+5$ for all $n\ge 1$.
\end{enumerate}
Particularly, since the function $f$ can be chosen with arbitrary rate of growth, $\reg I_n$ and $\mu(I_n)$ may exhibit arbitrary orders of growth.
\end{ex}

\begin{proof}[Proof of Claim]
(1) Observe first that $I_n \supseteq I_{n+1}$ for all $n \in \NN$, as $f(n)$ is a non-decreasing function.

We next check that $I_mI_n=(a,b)^8 (x,y)^{m+n} \subseteq I_{m+n}$ for all $m,n\ge 1$. Indeed, since $(a^4, a^3b, ab^3, b^4)^2 = (a,b)^8$, $(a^4,a^3b,ab^3,b^4)a^2b^2 \subseteq (a,b)^8$, $a^4b^4 \subseteq (a,b)^8$, and $f(n) \ge n$, we have
\begin{align*}
I_mI_n &= (a,b)^8 (x,y)^{m+n}+(a^4,a^3b,ab^3,b^4)a^2b^2((x,y)^{m+f(n)}+ (x,y)^{f(m)+n}) + \\
       & \qquad \qquad + a^4b^4(x,y)^{f(m)+f(n)} \\
          & = (a,b)^8 (x,y)^{m+n}.
\end{align*}

(2) It is not hard to verify that
\begin{align*}
(a^4,a^3b,ab^3,b^4)(x,y)^n \cap a^2b^2(x,y)^{f(n)} &= (a^4,a^3b,ab^3,b^4) \cap (x,y)^n \cap (a^2b^2) \cap (x,y)^{f(n)} \\
& =(a^4,a^3b,ab^3,b^4)\cap (a^2b^2) \cap  (x,y)^n \cap (x,y)^{f(n)}\\
& = a^2b^2(a^2,ab,b^2) \cap (x,y)^{f(n)}\\
& = a^2b^2(a^2,ab,b^2) (x,y)^{f(n)}=:U.
\end{align*}
Note that $\reg U=\reg a^2b^2(a^2,ab,b^2)+ \reg (x,y)^{f(n)}=f(n)+6$.

Consider the exact sequence
\[
 0\to U \to (a^4,a^3b,ab^3,b^4)(x,y)^n \oplus a^2b^2(x,y)^{f(n)} \to I_n \to 0.
\]
Then, $\reg (a^4,a^3b,ab^3,b^4)(x,y)^n = n+5$ and $\reg a^2b^2(x,y)^{f(n)} = f(n)+4$, both of which are strictly less than $\reg U = f(n)+6$. Hence,
\[
\reg I_n =\reg U-1=f(n)+5.
\]
The fact that $d(I_n)=f(n)+4$ is clear, as the minimal generators of $a^2b^2(x,y)^{f(n)}$ are also that of $I_n$.

(3) The statement follows from a simple counting.
\end{proof}

It was shown by Hoa \cite{Hoa22} and Hoa and T.N.\,Trung \cite{HoaTrung2010} that, for a monomial ideal $I$ in a polynomial ring, the asymptotic regularity of $(I^n)_{n \ge 0}$ and $(\overline{I^n})_{n \ge 0}$ are the same.
By modifying Example \ref{ex_diverge} slightly, we obtain the following example illustrating that, for a graded family $(I_n)_{n \ge 0}$ in general, even if the asymptotic regularity of $(I_n)_{n \ge 0}$ and $(\overline{I_n})_{n \ge 0}$ exist, these values may not be the same.

\begin{ex}
\label{ex_distinct_lims}
Let $R=\kk[x,y,a,b]$.
Consider the collection of ideals $(I_n)_{n \ge 0}$ given by $I_0 = R$ and, for $n \ge 1$,
$$
I_n=(a^4,a^3b,ab^3,b^4)(x,y)^n+a^2b^2(x^n,y^n).
$$
Similarly to \Cref{ex_diverge}, we shall show that:
\begin{enumerate}
 \item $(I_n)_{n\ge 0}$ is a graded filtration of ideals;
 \item $\reg I_n=2n+4$ and $d(I_n)=n+4$ for all $n\ge 1$;
 \item $\overline{I_n}=(a,b)^4(x,y)^n$, so $\reg \overline{I_n}=d(\overline{I_n})=n+4$ for all $n\ge 1$.
\end{enumerate}
In particular,
\[
\lim \limits_{n\to \infty} \frac{\reg I_n}{n} = 2> 1= \lim \limits_{n\to \infty} \frac{\reg \overline{I_n}}{n}= \lim \limits_{n\to \infty} \frac{d(I_n)}{n}.
\]

Indeed, the assertion (1) is similar to \Cref{ex_diverge}(1), noting that $I_mI_n=(a,b)^8(x,y)^{m+n}$ for all $m,n\ge 1$.
The assertion (2) is similar to \Cref{ex_diverge}(2), with the following statements easily seen:
\begin{itemize}
	\item $(a^4,a^3b,ab^3,b^4)(x,y)^n \cap a^2b^2(x^n,y^n) =a^2b^2(a^2,ab,b^2)(x^n,y^n)$ has regularity $2n+5$,
	\item $\reg (a^4,a^3b,ab^3,b^4)(x,y)^n = n+5$,
	\item $\reg a^2b^2(x^n,y^n) = 2n+3$.
\end{itemize}
To see (3), observe that the ideal $W=(a,b)^4(x,y)^n=(a,b)^4\cap (x,y)^n$ is integrally closed. We also have $I_n \subseteq W$ and $I_nW=W^2$, so $\overline{I_n}=\overline{W}=W$. The remaining statements about limits are obvious.
\end{ex}

The following result shows that the asymptotic regularity of $(I^n)_{n \ge 0}$ and $(I^nM)_{n \ge 0}$ are the same, for a homogeneous ideal $I \subseteq R$ and a faithful $R$-module $M$, generalizing what was known from the work of N.V.\,Trung and Wang \cite{TW05}. Recall that an $R$-module $M$ is \emph{faithful} if $\Ann_R(M) = (0)$. It is easy to see that if an ideal $I \subseteq R$ is such that $\Ann_R(I) = (0)$ and $M$ is a faithful $R$-module, then $IM$ is also a faithful $R$-module.

\begin{thm}
\label{thm_limit_powers}
Let $R$ be a reduced, standard graded $\kk$-algebra and $I$ be a homogeneous ideal of $R$ with $\Ann_R(I) = (0)$. Then, for any finitely generated graded faithful $R$-module $M$, there are equalities
\[
\lim \limits_{n\to \infty} \frac{\reg (I^nM)}{n} = \lim \limits_{n\to \infty} \frac{d(I^nM)}{n} =\lim \limits_{n\to \infty} \frac{\reg I^n}{n} =\lim \limits_{n\to \infty} \frac{\reg \overline{I^n}}{n} = \lim \limits_{n\to \infty} \frac{d(\overline{I^n})}{n}.
\]
\end{thm}

\begin{proof}
Let
$$\rho_M(I)=\min\{d(K) ~\big|~ K\subseteq I \text{ a homogeneous ideal}, I^nM=I^{n-1}KM \text{ for } n \gg 0\}.
$$
By \cite[Theorem 3.2]{TW05}, there exists a constant $e$ such that for $n \gg 0$,
\[
\reg(I^nM)=\rho_M(I)n+e.
\]
We first show that $\rho_R(I)=\rho_M(I)$ for any graded faithful $R$-module $M$. We shall prove a stronger statement; that is, for a homogeneous ideal $K\subseteq I$, the following are equivalent:
\begin{enumerate}
 \item $I^nM=I^{n-1}KM$ for all $n \gg 0$;
 \item $I^n=I^{n-1}K$ for all $n \gg 0$.
\end{enumerate}
Indeed, clearly if $K$ is a reduction of $I$, then $I^nM=I^{n-1}KM$, for $n \gg 0$, for any $R$-module $M$. Conversely, suppose that $I^nM=I^{n-1}KM$ for $n \gg 0$. Then, since $\Ann_R(I) = (0)$, $I^{n-1}M$ is a faithful $R$-module for all $n \ge 1$. Thus, for $n \gg 0$, we have
\[
I\subseteq I^nM : I^{n-1}M = I^{n-1}KM: I^{n-1}M \subseteq \overline{K}.
\]
Hence, $K$ is a reduction of $I$, by \cite[Corollary 1.2.5]{SH06}.

We have shown that $\rho_R(I)=\rho_M(I)=:\rho$. In particular,
\[
\lim \limits_{n\to \infty} \frac{\reg (I^nM)}{n}  =\lim \limits_{n\to \infty} \frac{\reg I^n}{n}=\rho.
\]
Let $\indeg(M)=\inf\{t: M_t\neq 0\}$. Thanks to \cite[Lemma 3.1]{TW05}, we get the second inequality in the chain
\[
\rho n+e = \reg(I^nM) \ge d(I^nM)\ge \rho n+\indeg(M) \text{ for } n \gg 0.
\]
Therefore
\[
\lim \limits_{n\to \infty} \frac{d(I^nM)}{n}=\rho,
\]
completing the first two equalities.

For the remaining two equalities, note that since $R$ is reduced, it is analytically unramified, hence, by \cite[Corrolary 9.2.1]{SH06} for the usual Rees algebra $\R(I)$, there exists an integer $d$ such that $\overline{I^n}=I^{n-d}\overline{I^d}$ for all $n\ge d$. Let $L=\overline{I^d}$. Since $\Ann_R(I) = (0)$, by induction, we have $\Ann_R(I^n) = (0)$ for all $n \in \NN$. It follows that $L$ is a faithful $R$-module. Applying the first two equalities for $M = L$, we obtain
\begin{align*}
 \lim \limits_{n\to \infty} \frac{\reg \overline{I^n}}{n} &= \lim \limits_{n\to \infty} \frac{\reg I^{n-d}L}{n} =\rho,\\
 \lim \limits_{n\to \infty} \frac{d(\overline{I^n})}{n} &= \lim \limits_{n\to \infty} \frac{d(I^{n-d}L)}{n} =\rho.
\end{align*}
This completes the proof.
\end{proof}

We are now ready to establish the first main result of the paper, giving an affirmative answer to Question \ref{quest.DSN}.

\begin{thm}
\label{thm_limit_noeth_families}
Let $\I=(I_n)_{n\ge 0}$ be a graded family of homogeneous ideals such that for all $n \in \NN$, $\Ann_R(I_n) = (0)$. Suppose that $R$ is reduced and $\R(\I)$ is a finitely generated $R$-algebra. Then,
\[
\lim \limits_{n\to \infty} \frac{\reg I_n}{n} = \lim \limits_{n\to \infty} \frac{d(I_n)}{n} =\lim \limits_{n\to \infty} \frac{\reg \overline{I_n}}{n} = \lim \limits_{n\to \infty} \frac{d(\overline{I_n})}{n}.
\]
\end{thm}

\begin{proof}
By \Cref{prop_fg_Reesalg}, there exist integers $c\ge 1, N\ge 1$ and ideals $L_i = I_{Nc+i}$ where $1\le i\le c-1$ such that:
\begin{enumerate}[\quad \rm (i)]
 \item $I_{nc}=I_c^n$ for all $n\ge 0$; and,
 \item $I_{nc+i}=L_iI_c^{n-N}$ for all $1\le i\le c-1$ and all $n\ge N$.
\end{enumerate}
Set $\rho=\min\{d(K) ~\big|~ K\subseteq I_c \text{ a homogeneous ideal}, I_c^n=KI_c^{n-1} \text{ for } n \gg 0\}$.  We claim that
\[
\lim\limits_{q\to \infty} \frac{\reg I_q}{q}=\frac{\rho}{c}.
\]

Observe that, since $\Ann_R(I_n) = (0)$ for all $n \in \NN$, $L_i$ is a faithful $R$-module for all $1 \le i \le c-1$. Now, if $q$ is a multiple of $c$, then
\begin{align*}
\lim\limits_{q\to \infty} \frac{\reg I_q}{q}&=\lim\limits_{n\to \infty} \frac{\reg I_{nc}}{nc}=\lim\limits_{n\to \infty} \frac{\reg I_c^n}{nc} =\frac{\rho}{c}.
\end{align*}
Furthermore, for $q=nc+i$, for some fixed $1\le i\le c-1$, by the choice of $N$, we get
\begin{align*}
\lim\limits_{q\to \infty} \frac{\reg I_q}{q}&=\lim\limits_{n\to \infty} \frac{\reg I_{nc+i}}{nc+i}\\
 &=\lim\limits_{n\to \infty} \frac{\reg (L_i I_c^{n-N})}{nc+i} =\lim\limits_{n\to \infty} \left(\frac{\reg (L_i I_c^{n-N})}{n-N} \cdot \frac{n-N}{nc+i} \right)=\frac{\rho}{c}.
\end{align*}
Here, the last equality holds by \Cref{thm_limit_powers}.

By a similar argument, we also get
\[
\lim\limits_{q\to \infty} \frac{d(I_q)}{q}=\frac{\rho}{c}.
\]
Therefore, the first desired equality is established.

It follows from the last assertion of \Cref{lem_intclos_family} that there exists an integers $d, k \ge 1$ such that $\overline{I_n}=I_{n-dk-i}\overline{I_{dk+i}}$ for all $n\ge dk+2d$ and for $i\in \{d,d+1,\ldots ,2d-1 \}$ with $n \equiv i$ (mod $d$). Furthermore, note that in \Cref{prop_fg_Reesalg}, since the algebra $\R(\I)^{(\ell c)}$ is also a standard graded $R$-algebra for any multiple $\ell c$ with $\ell\ge 1$ and $\ell c$ can play the role of $c$ in other assertions, we can assume that $c$ is a multiple of $d$ (by choosing $\ell$ a multiple of $d$) without loss of generality. Let $c=md$ with $m \ge 1$. For all $n\ge \max\{N,2\}$ (here the number $N$ is given as in \Cref{prop_fg_Reesalg}), since $nc\ge 2d$, we get
\[
 \overline{I_{nc+dk+pd+i}}=I_{nc+pd} \cdot \overline{I_{dk+i}}= L_{pd} \cdot \overline{I_{dk+i}} \cdot I_c^{n-N} \text{ for all } 0\le p \le m-1 \text{ and } i\in \{d, d+1,\ldots ,2d-1 \}.
\]

Here, $L_0:= R$. Observe again that, by the assumption that $\Ann_R(I_n) = (0)$ for all $n \in \NN$, we have $L_{pd}\overline{I_{dk+i}}$ is a faithful $R$-module for all $0\le p \le m-1$ and $i\in \{d, d+1,\ldots ,2d-1 \}$. Hence, the same arguments as above, based on \Cref{thm_limit_powers}, also yield
\[
\lim \limits_{n\to \infty} \frac{\reg \overline{I_n}}{n} = \lim \limits_{n\to \infty} \frac{d(\overline{I_n})}{n}=\frac{\rho}{c}.
\]
The proof is completed.
\end{proof}

The following example demonstrates that the condition $\Ann_R(I_n) = (0)$ in Theorem \ref{thm_limit_noeth_families} is necessary.

\begin{ex} \label{ex.AnnNot0}
Let $R = \kk[x,a,b]/(xa)$ and let $I = (a^5,b^2)$ and $J = (x)$ be ideals in $R$. Consider the collection $\I = (I_n)_{n \ge 0}$ of ideals in $R$ given by $I_0 = R$ and, for $n \ge 1$,
$$I_n = \left\{ \begin{array}{lll} I^n & = \dfrac{(a^5,b^2)^n + (xa)}{(xa)} & \text{if } n \text{ is even} \\
J\cdot I^n & = \dfrac{(xb^{2n}) + (xa)}{(xa)} & \text{if } n \text{ is odd.} \end{array}\right.$$
We claim that:
\begin{enumerate}
\item $\I = (I_n)_{n \ge 0}$ is a graded family of ideals in $R$;
\item the Rees algebra $\R(\I)$ is generated in degrees 1 and 2 over $R$;
\item $\reg I_n = 5n+1$ for $n$ even and $\reg I_n = 2n+1$ for $n$ odd.
\end{enumerate}
Particularly, the asymptotic regularity $\lim\limits_{n \rightarrow \infty} \reg I_n/n$ does not exist.

Indeed, (1) and (2) are obvious from the definition. To see (3), we easily reduce to proving that the regularity of $(a^5,b^2)^n + (xa)$ and $(xb^{2n}) + (xa)$ in $\kk[x,a,b]$ is $5n+1$ and $2n+1$, respectively, for all $n\ge 1$.

Since $(xb^{2n}) + (xa) = (x)(b^{2n},a)$, its regularity is easily seen to be $2n+1$ as claimed. On the other hand, by letting $K = (a^5,b^2)^n + (xa)$, we get that
$$K : x = (a,b^{2n}) \text{ and } (K,x) = (a^5,b^2)^n + (x)$$
are of regularity $2n$ and $5n+1$, respectively. Thus, by \cite[Theorem 4.7]{CH+19}, we have $\reg K = 5n+1$ as claimed.
\end{ex}

Theorem \ref{thm_limit_noeth_families} applies to \emph{rational} powers gives the following example.

\begin{ex}
Let $e \in \NN_{>0}$ be a positive integer, $R$ be a reduced, standard graded $\kk$-algebra and $I$ be a homogeneous ideal of $R$. Consider the graded family $\I=(I_n)_{n\ge 0}$, where $I_n = \overline{I^{n/e}}$, for $n\in \NN$, is the $n/e$-rational power of $I$ defined by
$$\overline{I^{n/e}} = \{ f \in R \ | \ f^e \in \overline{I^n} \}.$$
(See \cite[Section 10.5]{SH06} for more details on rational powers.) Since the $e$-th Veronese subalgebra of $\R(\I)$ is $\R^{(e)}(\I)\cong \overline{\R(I)}$, which is Noetherian, $\R(\I)$ is Noetherian as well. Thus, by Theorem \ref{thm_limit_noeth_families}, we have
\[
\lim \limits_{n\to \infty} \frac{\reg \overline{I^{n/e}}}{n/e} = \lim \limits_{n\to \infty} \frac{d(\overline{I^{n/e}})}{n/e} =\lim \limits_{n\to \infty} \frac{\reg \overline{I^n}}{n} = \lim \limits_{n\to \infty} \frac{d(\overline{I^n})}{n} = \lim \limits_{n\to \infty} \frac{\reg I^n}{n} = \lim \limits_{n\to \infty} \frac{d(I^n)}{n}.
\]

\end{ex}

As an application of \Cref{thm_limit_noeth_families}, we get the following result.

\begin{thm}
\label{thm_limit_intclo}
Let $R$ be a reduced, standard graded $\kk$-algebra and $I, J$ be homogeneous ideals of $R$ such that $\Ann_R(I) = \Ann_R(J) = (0)$. The following equalities hold:
\[
\lim \limits_{n\to \infty} \frac{\reg \overline{JI^n}}{n} = \lim \limits_{n\to \infty} \frac{d(\overline{JI^n})}{n} = \lim \limits_{n\to \infty} \frac{\reg I^n}{n}.
\]
Moreover, there exists an integer $\ell$ such that $\overline{JI^n}=\overline{JI^{n-\ell}}\cdot \overline{I^{\ell}}$ for all $n\ge 2\ell$.
\end{thm}

\begin{proof}
Set
\[
\rho = \lim \limits_{n\to \infty} \frac{\reg I^n}{n}.
\]
Consider the collection of ideals $\Lcc=(L_n)_{n\ge 0}$ given by $L_n=JI^n$ if $n$ is odd, and $L_n=I^n$ if $n$ is even. Concretely,
$$
L_0=R, L_1= JI, L_2= I^2, L_3=JI^3, L_4=I^4,\ldots.
$$
It can be seen from the assumption that $\Ann_R(L_n) = (0)$ for all $n \in \NN$.

One checks easily that $(L_n)_{n\ge 1}$ is a graded family, whose Rees algebra
\[
\R(\Lcc)= R \oplus JI \oplus I^2 \oplus JI^3 \oplus I^4 \oplus \cdots
\]
is generated over $R$ in degrees 1 and 2. Therefore, $\R(\Lcc)$ is finitely generated over $R$, and \Cref{thm_limit_noeth_families} implies to give
\[
 \lim \limits_{n\to \infty} \frac{\reg \overline{JI^{2n+1}}}{2n+1} = \lim \limits_{n\to \infty} \frac{d(\overline{JI^{2n+1}})}{2n+1} = \lim \limits_{n\to \infty} \frac{\reg I^{2n}}{2n} =\rho.
\]
Replacing $J$ by $JI$, we also get
\[
 \lim \limits_{n\to \infty} \frac{\reg \overline{JI^{2n+2}}}{2n+1} = \lim \limits_{n\to \infty} \frac{d(\overline{JI^{2n+2}})}{2n+1} =\rho.
\]
Thus,
\[
\lim \limits_{n\to \infty} \frac{\reg \overline{JI^n}}{n} = \lim \limits_{n\to \infty} \frac{d(\overline{JI^n})}{n} = \rho.
\]

For the remaining assertion, since $\R(\overline{\Lcc})^{(2)}$ equals  the integral closure $\overline{\R(I^2)}$ of the Rees algebra $\R(I^2)$, it is finitely generated over $R$. So there exists an integer $c\ge 1$ such that $\overline{I^{2cn}}=(\overline{I^{2c}})^n$ for all $n\ge 0$. In particular, $\R(\overline{\Lcc})^{(2c)}$ is standard graded over $R$. By \Cref{prop_fg_Reesalg}(2) for $\R(\overline{\Lcc})$, there exists a multiple $d$ of $2c$ such that $\overline{L_n}=\overline{L_{n-d}} \cdot \overline{L_d}$ for all $n\ge 2d$. As a consequence, we get
\begin{equation}
\label{eq_odd}
\overline{JI^n}=\overline{L_n}=\overline{L_{n-d}} \cdot \overline{L_d} =\overline{JI^{n-d}} \cdot \overline{I^d}, \quad \text{for all $n$ odd, $n\ge 2d+1$}.
\end{equation}
Replacing $J$ by $JI$ and repeating the arguments, we then can find a multiple $e$ of $2c$ such that for all $n$ odd greater than $2e$,
\[
\overline{JI^{n+1}} =\overline{JII^n}=\overline{JI^{n+1-e}} \cdot \overline{I^e}.
\]
In other words,
\begin{equation}
\label{eq_even}
\overline{JI^n} =\overline{JI^{n-e}} \cdot \overline{I^e}, \quad \text{for all $n$ even, $n\ge 2e+1$}.
\end{equation}

We claim that for all $n\ge de+d+e+1$,
\[
\overline{JI^n}= \overline{JI^{n-de}} \cdot \overline{I^{de}}.
\]
If $n$ is odd, using \eqref{eq_odd} repeatedly, and the fact that $2c$ divides $d$, we have
\[
\overline{JI^n}=\overline{JI^{n-d}} \cdot \overline{I^d}=\cdots=\overline{JI^{n-de}} \cdot (\overline{I^d})^e=\overline{JI^{n-de}} \cdot \overline{I^{de}}.
\]
If $n$ is even, using \eqref{eq_even} repeatedly, and the fact that $2c$ divides $e$, we have
\[
\overline{JI^n}=\overline{JI^{n-e}} \cdot \overline{I^e}=\cdots=\overline{JI^{n-de}} \cdot (\overline{I^e})^d=\overline{JI^{n-de}} \cdot \overline{I^{de}}.
\]
Hence, by choosing $\ell=de$, we are done.
\end{proof}


\section{Graded families of $\mm$-primary ideals and non-Noetherian Rees algebras} 
\label{sec.nonNoeth}

In this section, we give instances where the Rees algebra of a graded family of ideals is not necessarily Noetherian and yet its asymptotic regularity still exists.


\subsection{$\mm$-primary and Cohen-Macaulay ideals} We shall show that the asymptotic regularity exists for graded families of $\mm$-primary ideals or, more generally, Cohen-Macaulay ideals of the same codimension. Examples of the latter include the family of symbolic powers of the defining ideal of a set of points in a projective space.

\begin{thm}
	\label{thm_m-primary}
	If $\I=(I_n)_{n\ge 0}$ is a graded family of $\mm$-primary homogeneous ideals, then there are equalities
	\[
	\lim \limits_{n\to \infty} \frac{\reg I_n}{n} = \lim \limits_{n\to \infty} \frac{d(I_n)}{n}=\inf\limits_{n\ge 1} \frac{d(I_n)}{n}.
	\]
\end{thm}
\begin{proof}
	Since $I_mI_n\subseteq I_{m+n}$, and these ideals are $\mm$-primary, Lemmas \ref{lem_reg_lowdim} and \ref{lem_reg_containment} yield
	\[
	\reg(I_{m+n})\le \reg(I_mI_n) \le \reg(I_m)+\reg(I_n).
	\]
	By Fekete's Lemma, $\lim \limits_{n\to \infty} \dfrac{\reg I_n}{n}$ exists. Furthermore, we have
	\[
	\lim \limits_{n\to \infty} \frac{\reg I_n}{n} \ge \limsup \limits_{n\to \infty} \frac{d(I_n)}{n} \ge \liminf \limits_{n\to \infty} \frac{d(I_n)}{n} \ge \inf\limits_{n\ge 1} \frac{d(I_n)}{n}=:d.
	\]
	It remains to show that
	\[
	\lim \limits_{n\to \infty} \frac{\reg I_n}{n} \le d.
	\]
	
	Take any $\epsilon >0$. There exists $n_0\ge 1$ such that $d_0:=d(I_{n_0}) < n_0(d+\epsilon)$. For each homogeneous ideal $L$ and each integer $p$, denote by $L_{\ge p}$ the subideal of $L$ generated by elements of degree at least $p$. Note that if $L$ is $\mm$-primary then so is $L_{\ge p}$. Consider the ideals
	\[
	J_n= (I_n)_{\ge d_0\left\lceil \dfrac{n}{n_0}\right\rceil}.
	\]
	The inequality
	$\left\lceil \dfrac{m+n}{n_0} \right\rceil \le \left\lceil \dfrac{m}{n_0} \right\rceil+ \left\lceil\dfrac{n}{n_0}\right\rceil$
	implies that
	$$J_mJ_n\subseteq (I_{m+n})_{\ge d_0\left\lceil \dfrac{m+n}{n_0}\right\rceil}=J_{m+n}.$$
	Hence, $(J_n)_{n\ge 0}$ is a graded family of $\mm$-primary ideals. Note that $J_{n_0}=(I_{n_0})_{\ge d_0}$ is generated in a \emph{single degree} $d_0$. Thus, there exist $N$ and $e$ such that $\reg J_{n_0}^p=pd_0+e$ for all $p\ge N$.
	
	Since $J_{n_0}^p \subseteq J_{pn_0} \subseteq I_{pn_0}$, \Cref{lem_reg_containment} implies
	\[
	\reg I_{pn_0} \le \reg J_{n_0}^p=pd_0+e \qquad \text{for all $p\ge N$}.
	\]
	This yields
	\[
	\dfrac{\reg I_{pn_0}}{pn_0} \le \dfrac{pd_0+e}{pn_0}=\dfrac{d_0}{n_0}+\dfrac{e}{pn_0}.
	\]
	Letting $p$ tends to infinity, we obtain
	\[
	\lim \limits_{n\to \infty} \frac{\reg I_n}{n} \le \dfrac{d_0}{n_0} <d+\epsilon.
	\]
	Since $\epsilon>0$ is arbitrary, the desired conclusion follows.
\end{proof}

If the base field $\kk$ is uncountable then, using the technique of \emph{specialization}, we obtain a slightly more general result than that of Theorem \ref{thm_m-primary}.

\begin{thm}
\label{thm_CohenMacaulay}
Let $\kk$ be an uncountable field. Let $\I=(I_n)_{n\ge 0}$ be a graded family of homogeneous ideals such that for all $n\ge 1$, the ring $R/I_n$ is Cohen-Macaulay of dimension $\dim(R/I_n)=\dim(R/I_1)$. There are equalities
\[
\lim \limits_{n\to \infty} \frac{\reg I_n}{n} = \lim \limits_{n\to \infty} \frac{d(I_n)}{n}=\inf\limits_{n\ge 1} \frac{d(I_n)}{n}.
\]
\end{thm}

To prove Theorem \ref{thm_CohenMacaulay}, we need a result on the existence of simultaneous filter-regular elements for a countable family of cyclic modules.

\begin{defn}
	\label{def.filter-regular} \quad
\begin{enumerate}
\item A homogeneous element $f$ of degree 1 in $R$ is said to be \emph{filter-regular} if $(0:f)_n = 0$ for $n \gg 0$.
\item Let $I \subseteq R$ be a homogeneous ideal. We say that $f \in R_1$ is \emph{filter-regular with respect to} (or simply, on) $R/I$ if $f$ is filter-regular in $R/I$.
\end{enumerate}
\end{defn}

It is easy to see that $f$ is filter-regular on $R/I$ is equivalent to $f$ being regular modulo $I : \mm^\infty$, i.e., $f$ is not a zero-divisor on $R/(I : \mm^\infty)$.

\begin{lem}
	\label{lem_simultaneous_filreg}
	Let $\kk$ be an uncountable field, let $R=\kk[x_1,\ldots,x_r]$ be a polynomial ring of dimension $\ge 1$, and let $(I_n)_{n\ge 0}$ be a countable collection of homogeneous ideals. Then, there exists a linear form $x$ in $R$ such that $x$ is filter-regular with respect to $R/I_n$ for all $n\ge 1$.
\end{lem}
\begin{proof}
	Let $u_1,\ldots,u_r$ be indeterminates and consider $f_u=u_1x_1+\cdots +u_rx_r$ in $R(u)=R\otimes_k k(u_1,...,u_r)$. For a tuple $\bsa=(a_1,\ldots,a_r)\in \kk^r$, set $f_\bsa=a_1x_1+\cdots+a_rx_r\in R$. Let $\mm=(x_1,\ldots,x_r)$ be the graded maximal ideal of $R$.
	
	Clearly, $f_u$ is regular modulo $(I_n:\mm^\infty)R(u)=I_nR(u):(\mm R(u))^\infty$, where the equality holds as the map $R\to R(u)$ is flat. Thus, by  \cite[Proposition 3.2(i) and Proposition 3.6]{NhTr99},
	$f_\bsa \in R$ is regular modulo $I_n:\mm^\infty$ for \emph{almost all} $\bsa\in \kk^r$. That is, there exists a Zariski-open dense subset $U_n \subseteq \kk^r$ such that, for any $\bsa\in \kk^r$, $f_\bsa$ is regular modulo $I_n:\mm^\infty$. Since $\kk$ is uncountable, the intersection $\bigcap_{n\ge 1} U_n$ is also uncountable by, for example, \cite[Lemma 3.1]{Nh07}. Take any $\bsa \in \bigcap_{n\ge 1} U_n$. Then, for every $n\ge 1$, the linear form $f_\bsa\in R$ is regular modulo $I_n:\mm^\infty$, namely, $f_\bsa$ is filter-regular on $R/I_n$.
\end{proof}

\begin{proof}[{\bf Proof of \Cref{thm_CohenMacaulay}}]
	If $\dim(R/I_1)=0$, then we are done by \Cref{thm_m-primary}. Assume that $\dim(R/I_1)\ge 1$.
	
	By \Cref{lem_simultaneous_filreg}, we can choose a linear form $x\in R$ that is filter-regular with respect to $R/I_n$ for all $n\ge 1$. Each $R/I_n$ is Cohen-Macaulay of positive dimension, so $x$ is $(R/I_n)$-regular for all $n$. Using \Cref{lem_regular_linform}, we reduce to considering the ideals $(I_n,x)/(x) \subseteq R/(x)$. These ideals again form a graded family, of Cohen-Macaulay ideals such that $\dim(R/(I_n,x))=\dim(R/I_n)-1=\dim(R/(I_1,x))$. The desired conclusion follows by induction.
\end{proof}


\subsection{Families of mixed products and non-Noetherian Rees algebras} The next class of graded families of ideals that we consider is inspired by \Cref{thm_limit_intclo} whose Rees algebras are not necessarily Noetherian.

\begin{thm}
	\label{thm_limit_varyingpowers}
	Let $R$ be a reduced, standard graded $\kk$-algebra and $I, J$ be homogeneous ideals in $R$ such that $\Ann_R(I) = \Ann_R(J) = (0)$. Let $(a_n)_{n\ge 1}$ be sequence of non-negative integers such that $\lim \limits_{n\to \infty} {a_n}/{n}=0$. There are equalities
	\[
	\lim \limits_{n\to \infty} \frac{\reg (\overline{J^{a_n}I^n})}{n} = \lim \limits_{n\to \infty} \frac{d(\overline{J^{a_n}I^n})}{n} =\lim \limits_{n\to \infty} \frac{\reg \left(J^{a_n}I^n\right)}{n} = \lim \limits_{n\to \infty} \frac{d\left(J^{a_n}I^n\right)}{n}= \lim \limits_{n\to \infty} \frac{\reg I^n}{n}.
	\]
\end{thm}

The proof uses the following lemma, which relies on previous work of Bagheri, Chardin and the first author \cite{BCH13}.
\begin{lem}
\label{lem_regbound_products}
Let $R$ be an arbitrary standard graded $\kk$-algebra, and let $I, J$ be homogeneous ideals in $R$. Let $M$ be an arbitrary finitely generated graded $R$-module. Then there exist non-negative integers $N_1,N_2, C$ such that for all $t_1\ge N_1$ and all $t_2\ge N_2$, there is an inequality
\[
\reg(I^{t_1}J^{t_2}M) \le t_1d(I)+t_2d(J)+C.
\]
\end{lem}
\begin{proof}
Choose a presentation $R=S/L$, where $S=\kk[x_1,\dots,x_r]$ is a standard graded polynomial ring and $L\subseteq (x_1,\ldots,x_r)^2$ a homogeneous ideal. Let $\overline{f_1},\dots,\overline{f_d}$ be a minimal set of homogeneous generators of $I$, and $\overline{g_1},\dots,\overline{g_e}$ be a minimal set of homogeneous generators of $J$, where $f_i, g_j$ are homogeneous elements in $S$, and $\overline{x}$ denotes the residue class in $R$ of $x\in S$. Let $\wti{I}=(f_1,\dots,f_d), \wti{J}=(g_1,\dots,g_e)$ be $S$-ideals, and regard $M$ as an $S$-module. Then, $\reg I^{t_1}J^{t_2}M= \reg \wti{I}^{t_1}\wti{J}^{t_2}M$ for all $t_1,t_2$. The desired conclusion is not affected by replacing $R, I, J$ with $S, \wti{I}, \wti{J}$, respectively, so we may assume that $R=\kk[x_1,\dots,x_r]$ is a standard graded polynomial ring.

Next, we will use a result of Bagheri, Chardin, and H\`a \cite[Theorem 4.6]{BCH13} on the possible non-zero $\ZZ$-graded Betti numbers of $J^{t_2}I^{t_1}M$. For a graded $R$-module $N$, set
	$$
	\Supp N=\{i\in \ZZ \mid N_i \neq 0\}.
	$$
	Apply \cite[Theorem 4.6]{BCH13} for the data: $G=\ZZ$, $A=\kk$, $S=R=\kk[x_1,\ldots,x_r]$, $s=2$, $I_1=I$ and $I_2=J$, and $0\le \ell \le r=\dim R$. Then, there exist integers $m\ge 1, \delta^\ell_p,t^\ell_{p,1}, t^\ell_{p,2}, 1\le p\le m$ and non-empty subsets $E^\ell_{p,i} \subseteq \Gendeg(I_i)$ (the set of minimal generating degrees of $I_i$), where $i=1,2$, such that for all $t_1 \ge \max \{t^\ell_{p,1} \mid 1\le p\le m\}$ and all $t_2 \ge \max \{t^\ell_{p,2} \mid 1\le p\le m\}$, the following containment holds:
	\[
	\Supp \Tor^R_\ell(J^{t_2}I^{t_1}M, \kk) = \bigcup_{p=1}^m \left(\delta^\ell_p+ \bigcup_{\bsc_i \in \ZZ^{\left | E^\ell_{p,i} \right|}_{\ge 0}, |\bsc_i|=t_i-t^\ell_{p,i}}(\bsc_1\cdot E^\ell_{p,1}+ \bsc_2\cdot E^\ell_{p,2}) \right).
	\]
In the last expression, the $\cdot$ operation is defined as follows: For a vector $\bsc=(c_1,\ldots,c_s)\in \ZZ^s$ and a tuple $E=(\nu_1,\ldots,\nu_s)$ of elements in some additive abelian group, we let $\bsc\cdot E:=c_1\nu_1+\cdots+c_s\nu_s$.

	In particular, there exist positive integers $N_1, N_2, C$ such that
	\[
		\reg (I^{t_1}J^{t_2}M) \le t_1 d(I)+t_2 d(J)+C \quad \text{for all $t_1\ge N_1$ and all $t_2\ge N_2$}.
	\]
This concludes the proof.
\end{proof}

\begin{proof}[Proof of \Cref{thm_limit_varyingpowers}] Observe that the assumption $\Ann_R(I) = \Ann_R(J) = (0)$ implies that $J^aI^b$ is a faithful $R$-module for all $a,b \in \ZZ_{\ge 0}$. The proof is divided into 2 steps.
	
	\medskip
	
\noindent\textsf{Step 1: proving the last two desired equalities.} We claim that for any  finitely generated graded faithful $R$-module $M$, and any sequence of positive integers $a_n$ such that $\lim_{n\to \infty} (a_n/n)=0$, the equalities
	\[
	\lim \limits_{n\to \infty} \frac{\reg \left(J^{a_n}I^n M\right)}{n} = \lim \limits_{n\to \infty} \frac{d\left(J^{a_n}I^n M\right)}{n}= \lim \limits_{n\to \infty} \frac{\reg I^n}{n}.
	\]
	hold. Set $\rho = \lim \limits_{n\to \infty} {\reg I^n}/{n}$.
	
	We shall reduce the claim to the case where $\rho=d(I)$. Consider $K\subseteq I$ a homogeneous reduction of $I$ such that $d(K)=\rho$; the existence of $K$ is shown in \cite{K00, TW05}. Then, there exists $c\ge 1$, such that $I^n=I^cK^{n-c}$ for all $n\ge c$. By \Cref{thm_limit_powers},
	\[
	d(K)=\lim \limits_{n\to \infty} \frac{\reg I^n}{n}=\lim \limits_{n\to \infty} \frac{\reg I^cK^{n-c}}{n} =\lim \limits_{n\to \infty} \frac{\reg K^n}{n}.
	\]
	Thus, the claimed equalities can be rewritten as
	\[
	\lim \limits_{n\to \infty} \frac{\reg \left(J^{a_n}K^{n-c} M'\right)}{n} = \lim \limits_{n\to \infty} \frac{d\left(J^{a_n}K^{n-c} M'\right)}{n}= \lim \limits_{n\to \infty} \frac{\reg K^n}{n},
	\]
	where $M'=I^cM$ is a again a faithful $R$-module. Therefore, by shifting the sequence $(a_n)_{n\ge 1}$ suitably and replacing $I$ with $K$, we reduce to the case where $\rho=d(I)$.
	
Thanks to \Cref{lem_regbound_products}, there exist positive integers $N_1, N_2, C$ such that
	\begin{equation}
		\label{eq_reg_lessthan}
		\reg (J^{t_2}I^{t_1}M) \le t_2 d(J)+ t_1 d(I)+C \quad \text{for all $t_1\ge N_1$ and all $t_2\ge N_2$}.
	\end{equation}

Take an arbitrary $\epsilon >0$. We shall show that
\[
\dfrac{\reg (J^{a_n}I^nM)}{n} < \rho+\epsilon, \quad \text{for all $n\gg 0$}.\]
	Indeed, since $\lim\limits_{n\to \infty} (a_n/n)=0$, there exists $N_3$ such that $a_n/n < N_2/N_1$ for all $n\ge N_3$.
	For each $0\le \ell \le N_2-1$, we have by \Cref{thm_limit_powers} that
	\[
	\lim \limits_{n\to \infty} \dfrac{\reg (J^{\ell}I^nM)}{n}=\rho.
	\]
	Thus increasing $N_3$ suitably, we can assume that
	\begin{align}
\dfrac{\reg (J^{\ell}I^nM)}{n} < \rho+\epsilon, \quad \text{for all $n\ge N_3, 0\le  \ell \le N_2-1$}. \label{eq.an1}
\end{align}
	On the other hand, if $a_n\ge N_2$, for some $n \ge N_3$, then $n\ge  a_n(N_1/N_2)\ge N_1$, whence, by \eqref{eq_reg_lessthan}, we have
\begin{align*}	
\frac{\reg (J^{a_n}I^n M)}{n} \le \frac{a_n d(J)+ n \rho +C}{n}=\rho +\frac{a_n d(J) +C}{n}.
\end{align*}
	As $a_n/n \rightarrow 0$, by further increasing $N_3$ suitably, we get that
\begin{align}
	\dfrac{\reg (J^{a_n}I^nM)}{n} < \rho+\epsilon, \quad \text{for all $n\ge N_3$}. \label{eq.an2}
\end{align}
Thus, (\ref{eq.an1}) and (\ref{eq.an2}) prove the assertion that
\[
	\dfrac{\reg (J^{a_n}I^nM)}{n} < \rho+\epsilon, \quad \text{for all $n\gg 0$}.
	\]

To establish the claimed equality, since $d (J^{a_n}I^nM)\le \reg(J^{a_n}I^nM)$, it suffices to show that
	\begin{align}
	\dfrac{d (J^{a_n}I^nM)}{n} \ge \rho+\frac{\indeg(M)}{n}, \quad \text{for all $n\gg 0$}, \label{eq.reg=d.an}
	\end{align}
as this implies
	\[
	\lim \limits_{n\to \infty} \frac{\reg \left(J^{a_n}I^n M\right)}{n} = \lim \limits_{n\to \infty} \frac{d\left(J^{a_n}I^n M\right)}{n}= \rho.
	\]

	Since $d(I)=\rho$, we can write $I=U+H$, where $d(U)\le \rho-1$, and $H=(f\in I \mid \deg f=\rho)$. We claim that $J^{a_n}H^n M \not\subseteq \mm J^{a_n}I^n M$ for all $n\ge 1$. Indeed, as $I^n=H^n+UI^{n-1}$, we have
	\[
	J^{a_n}I^n M = J^{a_n}H^n M+ J^{a_n}UI^{n-1} M.
	\]
If $J^{a_n}H^n M \subseteq \mm J^{a_n}I^n M$, then Nakayama's lemma yields
	\[
	J^{a_n}I^n M = J^{a_n}UI^{n-1} M.
	\]
	In particular, the faithfulness of $J^{a_n}I^{n-1}M$ implies that
	\[
	I\subseteq J^{a_n}UI^{n-1} M: J^{a_n}I^{n-1} M \subseteq \overline{U}.
	\]
	Thus, $U$ is a reduction of $I$ with $d(U)<d(I)$. This contradicts the assumption that $d(I)=\rho=\min\{d(K) \mid K \quad \text{is a homogeneous reduction of $I$}\}$. That is, $J^{a_n}H^n M \not\subseteq \mm J^{a_n}I^n M$ for all $n\ge 1$.
	
	We now conclude that, for any $n \ge 1$, $J^{a_n}I^n M$ has a minimal generator of degree at least $\indeg(J^{a_n}H^n M) \ge \rho n+\indeg(M)$, as $H$ is generated in degree $\rho$. Hence, $d(J^{a_n}I^n M) \ge \rho n+\indeg(M)$, for all $n\ge 1$, and (\ref{eq.reg=d.an}) is proved.
	
	\medskip
	
\noindent\textsf{Step 2: proving the first two desired equalities.} Letting $M=R$ in Step 1, we get
	\[
	\lim \limits_{n\to \infty} \frac{\reg \left(J^{a_n}I^n\right)}{n} = \lim \limits_{n\to \infty} \frac{d\left(J^{a_n}I^n\right)}{n}= \lim \limits_{n\to \infty} \frac{\reg I^n}{n}=\rho.
	\]
	It remains to show that
	\[
	\lim \limits_{n\to \infty} \frac{\reg (\overline{J^{a_n}I^n})}{n} = \lim \limits_{n\to \infty} \frac{d(\overline{J^{a_n}I^n})}{n} =\rho.
	\]
	
	Consider the bigraded Rees algebra
	$$
	\R(I,J)=\bigoplus_{n_1,n_2\ge 0}I^{n_1}J^{n_2}t_1^{n_1}t_2^{n_2} \subseteq R[t_1,t_2]=B.
	$$
	The natural inclusion map of domains $\R(I,J) \to B$ induces an injective ring map of their integral closures in the corresponding fraction fields $\overline{\R(I,J)} \to \overline{B}$. Since $B$ is integrally closed, we have an injection $\overline{\R(I,J)} \to B$. Hence $\overline{\R(I,J)}$ equals the integral closure of $\R(I,J)$ inside $B$. Arguing as for the case of the usual ($\ZZ$-graded) Rees algebra \cite[Proposition 5.2.1]{SH06}, we get
	\[
	\overline{\R(I,J)}= \bigoplus_{n_1,n_2\ge 0}\overline{I^{n_1}J^{n_2}}t_1^{n_1}t_2^{n_2}.
	\]
	Since $\R(I,J)$ is a finitely generated domain over $\kk$, its absolute integral closure (inside the fraction field) is module-finite over $\R(I,J)$ thanks to \cite[Theorem 4.6.3]{SH06}.
	
	Note that $\overline{\R(I,J)}$ is a bigraded $\R(I,J)$-module. Let $f_1,\ldots,f_p$ be bihomogeneous generators of $\overline{\R(I,J)}$ as an $\R(I,J)$-module, where $ \deg f_i =(d_{i1},d_{i2})$.
	Set
	$$d_1=\max\{d_{11},\ldots,d_{p1}\} \text{ and } d_2=\max\{d_{12},\ldots,d_{p2}\}.$$
	Then, for all $n_1\ge d_1$ and all $n_2\ge d_2$, we have
	\[
	\overline{I^{n_1}J^{n_2}} =\sum_{i=1}^p I^{n_1-d_{i1}}J^{n_2-d_{i2}}f_i \subseteq \sum_{i=1}^p I^{n_1-d_{i1}}J^{n_2-d_{i2}}\overline{I^{d_{i1}}J^{d_{i2}}} \subseteq I^{n_1-d_1}J^{n_2-d_2}\overline{I^{d_1}J^{d_2}} \subseteq \overline{I^{n_1}J^{n_2}},
	\]
	where the last two containments hold thanks to the fact that $\overline{U}\cdot \overline{V} \subseteq \overline{UV}$ \cite[Remark 1.3.2(4)]{SH06}. Thus, for all $n_1\ge d_1$ and all $n_2\ge d_2$,
	\begin{align}
	\overline{J^{n_2}I^{n_1}} =J^{n_2-d_2}I^{n_1-d_1}\overline{J^{d_2}I^{d_1}}. \label{eq.UV}
	\end{align}

	Consider an arbitrary $\epsilon>0$. For any $0\le \ell \le d_2$, it follows from \Cref{thm_limit_intclo} that
	\[
	\lim \limits_{n\to \infty} \frac{\reg (\overline{J^{\ell}I^n})}{n} = \lim \limits_{n\to \infty} \frac{d(\overline{J^{\ell}I^n})}{n} =\rho.
	\]
	Hence, there exists a positive integer $N_4$ such that
	\[
	\frac{\reg (\overline{J^{\ell}I^n})}{n}, \frac{d(\overline{J^{\ell}I^n})}{n} \in [\rho-\epsilon, \rho+\epsilon] \quad \text{for all $n\ge N_4$ and all $0\le \ell \le d_2$}.
	\]
	By increasing $N_4$ suitably, we can also assume that $a_n/n < d_2/d_1$ for $n\ge N_4$.
	
	Observe that, for any $n\ge N_4$, if $a_n \le d_2$ then the above containment gives
	\[
	\frac{\reg (\overline{J^{a_n}I^n})}{n}, \frac{d(\overline{J^{a_n}I^n})}{n} \in [\rho-\epsilon, \rho+\epsilon].
	\]
	On the other hand, if $a_n \ge d_2$ then $n>a_n(d_1/d_2)\ge d_1$ and it follows from (\ref{eq.UV}) that $$\overline{J^{a_n}I^{n}} =J^{a_n-d_2}I^{n-d_1}\overline{J^{d_2}I^{d_1}}.$$
	Let $M=\overline{J^{d_2}I^{d_1}}$, and set $b_n=\max\{0,a_n-d_2\}$ for $n\ge 1$. Then, $M$ and $I^{n-d_1}M$ are faithful $R$-modules, and $b_n/n \rightarrow 0$. Thus, by Step 1, we have
	\[
	\lim \limits_{n\to \infty} \frac{\reg \left(J^{b_n}I^{n-d_1} M\right)}{n} = \lim \limits_{n\to \infty} \frac{d\left(J^{b_n}I^{n-d_1} M\right)}{n}= \rho.
	\]
Therefore, by increasing $N_4$ further if necessary, we can assume that
	\[
	\frac{\reg (J^{b_n}I^{n-d_1}M)}{n}, \frac{d(J^{b_n}I^{n-d_1}M)}{n} \in [\rho-\epsilon, \rho+\epsilon]  \text{ for all } n \ge N_4.
	\]
Hence, for any $n\ge N_4$,
	\[
	\frac{\reg (\overline{J^{a_n}I^n})}{n}, \frac{d(\overline{J^{a_n}I^n})}{n} \in [\rho-\epsilon, \rho+\epsilon].
	\]
	This concludes the proof of Step 2 and the theorem.
\end{proof}

Theorem \ref{thm_limit_varyingpowers}, when applying to monomial ideals $I$ and $J$, gives an instance where Question \ref{ques.polyOkn} in the next section has a positive answer; see \Cref{rem.evidenceQ55}. It is also worth to note that the freedom in choosing the sequence $(a_n)_{n \ge 1}$ in Theorem \ref{thm_limit_varyingpowers} gives rise to many examples of non-Noetherian graded families of homogeneous ideals whose asymptotic regularity exists.


\section{Asymptotic regularity via Newton--Okounkov region} \label{sec.NObody}


In this section, we give a combinatorial interpretation for the asymptotic regularity in several instances when this invariant exists. We will focus on ideals in a polynomial ring; that is, when $R = \kk[x_1, \dots, x_r]$. For a vector $\textbf{v}=(v_1,\ldots ,v_r)\in \RR^r$, set $|\textbf{v}| = v_1+\cdots +v_r$, and for a polyhedron $P$, let
\[
\delta(P)=\max \{ |\textbf{v}| \ | \ \textbf{v} \text{ is a vertex of }  P \}.
\]


\subsection{Families of monomial ideals} We shall start by considering families of monomial ideals, as this is a direct generalization of the work of \cite{DHNT21, Hoa22}. Our next result is stated as follows.

\begin{thm}
\label{thm.regOkn}
Let $\I = (I_n)_{n \ge 0}$ be a Noetherian graded family of monomial ideals in $R$, and let $\Delta(\I)$ be its Newton--Okounkov region. Then,
\[
\lim \limits_{n\to \infty} \frac{\reg I_n}{n} = \lim \limits_{n\to \infty} \frac{d(I_n)}{n} =\lim \limits_{n\to \infty} \frac{\reg \overline{I_n}}{n} = \lim \limits_{n\to \infty} \frac{d(\overline{I_n})}{n} = \delta(\Delta(\I)).
\]
\end{thm}

\begin{proof}
By Theorem \ref{thm.NPpoly}, there exists $c\in \NN$, such that the $c-$th Veronese subalgebra $\R^{(c)}(\I)$ is standard graded and that $\Delta(\I) = \frac{1}{c}\NP(I_c)$. By Theorem \ref{thm_limit_noeth_families},
\[
\lim \limits_{n\to \infty} \frac{\reg I_n}{n} = \lim \limits_{n\to \infty} \frac{d(I_n)}{n} =\lim \limits_{n\to \infty} \frac{\reg \overline{I_n}}{n} = \lim \limits_{n\to \infty} \frac{d(\overline{I_n})}{n} = \lim \limits_{n\to \infty} \frac{\reg I_{cn}}{cn} = \lim \limits_{n\to \infty} \frac{\reg I_c^n}{cn}.
\]
Now, by the results in \cite{Hoa22}, we have that
\[
\lim \limits_{n\to \infty} \frac{\reg I_c^n}{cn} = \frac{1}{c} \delta(\NP(I_c)) = \delta\left(\frac{1}{c}\NP(I_c)\right) = \delta(\Delta(\I)).
\]
This finishes the proof.
\end{proof}

\begin{rem}
In \cite{DHNT21}, it is proved that for a monomial ideal $I$, we have $\lim \limits_{n\to \infty} {\reg \up{m}I^{(n)}}/{n} = \lim \limits_{n\to \infty} {d(\up{m}I^{(n)})}/{n} = \delta(\SP(J)))$, where $\up{m}I^{(n)}=\bigcap_{\pp \in \Min(I)}(I^nR_{\pp}\cap R)$ (i.e., the symbolic powers are defined by minimal primes), and $J={}\up{m}I^{(1)}$; see \Cref{ex.NObody}.  As a consequence of Theorem \ref{thm.regOkn}, the same result holds when $\up{a}I^{(n)}$ is defined in terms of associated primes, that is, for $\up{a}I^{(n)}=\bigcap_{\pp \in \Ass(I)}(I^nR_{\pp}\cap R)$.
\end{rem}

\begin{cor}
Let $\kk$ be an infinite field, let $\I = (I_n)_{n\ge 0}$ be a graded family of homogeneous ideals in $R$, and let ${\rm gin}(\I):= ({\rm gin}(I_n))_{n\ge 0}$ be its associated family of generic initial ideals with respect to the reverse lexicographic order. Suppose that ${\rm gin}(\I)$ is a Noetherian family. Then,
\[
\lim \limits_{n\to \infty} \frac{\reg I_n}{n} = \lim \limits_{n\to \infty} \frac{d(I_n)}{n} =\lim \limits_{n\to \infty} \frac{\reg \overline{I_n}}{n} = \lim \limits_{n\to \infty} \frac{d(\overline{I_n})}{n} = \delta(\Delta({\rm gin}(\I))).
\]
\end{cor}
\begin{proof}
This is straightforward from \Cref{thm.regOkn}, as $\reg(I_n)=\reg({\rm gin}(I_n))$ for all $n\ge 0$ by a well-known result of Bayer and Stillman.
\end{proof}

\begin{rem}
For any monomial order in $R$, if $({\rm in}(I_n))_{n\in \NN}$ is a Noetherian graded family, then $(I_n)_{n\in \NN}$ is a Noetherian graded family as well.

Indeed, since $({\rm in}(I_n))_{n\in \NN}$ is Noetherian, there exists an integer $c$ such that the subsequence $({\rm in}(I_{cn}))_{n\in \NN}$ has a standard graded Rees algebra, that is, ${\rm in}(I_{cn}) = ({\rm in}(I_c))^n, \forall n \in \NN$. Therefore, for all $n \in \NN$,
$${\rm in}(I_{cn}) = ({\rm in}(I_c))^n \subseteq {\rm in}(I_c^n) \subseteq {\rm in}(I_{cn}),$$
and, hence, we have the equality ${\rm in}(I_{cn})={\rm in}(I_c^n)$, for all $n \in \NN$. Since $(I_c)^n \subseteq I_{cn}$ and these ideals now have the same Hilbert function, they are equal for all $n\in \NN$. It follows that $(I_n)_{n\in \NN}$ is a Noetherian graded family.
\end{rem}

\begin{quest}\label{ques.polyOkn}
For which non-Noetherian graded families of monomial ideals $\I$ in $R$ with polyhedral Newton--Okounkov regions $\Delta(\I)$ can we extend Theorem \ref{thm.regOkn}?
\end{quest}

Theorem \ref{thm_limit_varyingpowers} provides an example for which Question \ref{ques.polyOkn} has a positive answer if we assume furthermore that $a_m+a_n \ge a_{m+n}$ for all $m,n\ge 1$.

\begin{rem} \label{rem.evidenceQ55}
	 Let $I, J \subseteq R$ be monomial ideals, and consider $\I=(I_n)_{n\in \NN}$, where $I_n=J^{a_n}I^n$ as in Theorem \ref{thm_limit_varyingpowers}. The assumption $a_m+a_n \ge a_{m+n}$ for all $m,n\ge 1$ implies that $\I$ is a graded family. We shall show that $\Delta(\I) = \NP(I)$. This, combined with \cite[Theorem 2.7]{Hoa22} and Theorem \ref{thm_limit_varyingpowers}, yields
	\[
	\lim \limits_{n\to \infty} \frac{\reg I_n}{n} = \lim \limits_{n\to \infty} \frac{d(I_n)}{n} =\lim \limits_{n\to \infty} \frac{\reg \overline{I_n}}{n} = \lim \limits_{n\to \infty} \frac{d(\overline{I_n})}{n} = \delta(\Delta(\I)).
	\]
	
	To see $\Delta(\I) = \NP(I)$, observe that $I_n=J^{a_n}I^n \subseteq I^n$. Thus, $\Delta(\I) \subseteq \NP(I)$. Now, let $\textbf{v}$ be any vertex of $\NP(I)$ and $\textbf{u}$ be any vertex of $\NP(J)$. For each $n\in \NN$, we have $\textbf{v}+ \frac{a_n}{n}\textbf{u} \in \frac{1}{n}\NP(J^{a_n}I^n)$. Since $\textbf{v}+ \frac{a_n}{n}\textbf{u} \rightarrow \textbf{v}$ as $n\rightarrow \infty$, we have $\textbf{v}\in \Delta(\I)$. Note that $\textbf{v}$ is taken to be an arbitrary vertex of $\NP(I)$ and $\Delta(\I)$ is convex, so we conclude that $\NP(I) \subseteq \Delta(\I)$. That is, $\Delta(\I) = \NP(I)$.
\end{rem}

On the other hand, Example \ref{ex_diverge} provides an instance where the asymptotic regularity of a graded family of monomial ideals may not exist although the associated Newton--Okounkov region is a polyhedron. In addition, Example \ref{ex_distinct_lims} shows that even when the asymptotic regularity exists and the Newton--Okounkov region is a polyhedron, the formula $\lim \limits_{n\to \infty} {\reg I_n}/{n} = \delta(\Delta(\I))$ may still fail if $\lim \limits_{n\to \infty} {\reg I_n}/{n} \not = \lim \limits_{n\to \infty} {d(I_n)}/{n}$. We demonstrate these observations below.

\begin{ex}
As in Example \ref{ex_diverge}, consider the collection $\I=(I_n)_{n\ge 0}$ of ideals given by
$$
I_n=(a^4,a^3b,ab^3,b^4)(x,y)^n+a^2b^2(x,y)^{f(n)} \subseteq R=\kk[a,b,x,y].
$$
When $f(n) >n$, one can show that for each $n$,
\[
\frac{1}{n}\NP(I_n) = \conv \left\{\left(\frac{4}{n},0,1,0\right),\left(\frac{4}{n},0,0,1\right),\left(0,\frac{4}{n},1,0\right),\left(0,\frac{4}{n},0,1\right) \right\} + \RR^r_{\ge 0},
\]
and furthermore, the defining hyperplanes of $\frac{1}{n}\NP(I_n)$ are precisely
\[
a\ge 0, b\ge 0, x\ge 0, y\ge 0, a+b\ge \frac{4}{n}, x+y\ge 1.
\]
Therefore,
\[
\Delta(\I)=\overline{\RR^4_{\ge 0} \cap (x+y \ge 1) \cap \left(\bigcup_{n\in \NN} \left(a+b\ge \frac{4}{n}\right) \right)} = \RR^4_{\ge 0} \cap (x+y \ge 1),
\]
which is a polyhedron with two vertices $(0,0,1,0)$ and $(0,0,0,1)$. 

Similarly, as in Example \ref{ex_distinct_lims}, consider the family $\mathcal{J}=(J_n)_{n\in \NN}$ where
$$
J_n =(a^4,a^3b,ab^3,b^4)(x,y)^n+a^2b^2(x^n,y^n).
$$
Then $\frac{1}{n}\NP(J_n)=\frac{1}{n}\NP(I_n)$ for all $n$. Thus, $\Delta(\mathcal{J}) = \Delta(\I)$ is a polyhedron with two vertices $(0,0,1,0)$ and $(0,0,0,1)$. On the other hand, one has
$$\lim \limits_{n\to \infty} \frac{\reg J_n}{n}=2, \lim \limits_{n\to \infty} \frac{d(J_n)}{n}=1 \text{ and } \delta(\Delta(\mathcal{J})) =1.$$

Nevertheless, we have shown that $\overline{J_n} = W_n = (a,b)^4(x,y)^n$ in Example \ref{ex_distinct_lims}, and since $\NP(I_n)=\NP(J_n)$, we have $\overline{I_n}=\overline{J_n}=W_n$. Thus, $\Delta(\mathcal{W})= \Delta(\mathcal{J})= \Delta(\mathcal{I})$, and
\[
\lim \limits_{n\to \infty} \frac{\reg W_n}{n}= \lim \limits_{n\to \infty} \frac{d(W_n)}{n}= \delta(\Delta(\mathcal{W}))=1.
\]
Note that the family $(W_n)_{n\in \NN}$ is non-Noetherian.
\end{ex}

To continue addressing Question \ref{ques.polyOkn}, we introduce the following condition which allows us to identify more classes of families of ideals for which Theorem \ref{thm.regOkn} extends.

\begin{defn}
\label{def.nondegenaxes}
    Let $\Delta(\I)$ and $\mathcal{C}(\I)$ be the Newton--Okounkov and the limiting regions of a graded family of ideals $\I = (I_n)_{n \ge 0}$ constructed from a good valuation that respects the monomials in $R$. We say that $\Delta(\I)$ is \emph{non-degenerate along the coordinate axes} if each of its vertex on any coordinate axis is the limit point of a sequence of points in $\mathcal{C}(\I)$ along the the same coordinate axis.
\end{defn}

 It is desirable to identify graded families $\I$ for which $\Delta(\I)$ is non-degenerate along the coordinate axes. The use of this notion is illustrated in the following result.

\begin{prop}
\label{prop.monprimary}
Let $\I$ be a graded family of $\mm$-primary monomial ideals in $R$, and suppose that $\Delta(\I)$ is a polyhedron that is non-degenerate along the coordinate axes. Then,
\[
\lim \limits_{n\to \infty} \frac{\reg I_n}{n} = \lim \limits_{n\to \infty} \frac{d(I_n)}{n} =\lim \limits_{n\to \infty} \frac{\reg \overline{I_n}}{n} = \lim \limits_{n\to \infty} \frac{d(\overline{I_n})}{n} = \delta(\Delta(\I)).
\]
\end{prop}

\begin{proof}
The first and the third equality have already been shown in Theorem \ref{thm_m-primary}. It remains to prove the second and the last equality. Since $\delta(\Delta(\I))=\delta(\Delta(\overline{\I}))$, where $\overline{\I} = (\overline{I_n})_{n \ge 0}$, by replacing $\I$ with $\overline{\I}$, it suffices to show that $\lim \limits_{n\to \infty} {\reg I_n}/{n} = \delta(\Delta(\I))$.

Define $\I^* = (I_n^*)_{n \ge 0}$ to be the graded family given by
	$$I_n^* = \langle \{x^\textbf{a} ~\big|~ \textbf{a} \in n\Delta(\I) \cap \ZZ^r\}\rangle.$$
	As shown in \cite[Example 3.2]{HaN23}, $\I^*$ is a graded family of monomial ideals and
	$$\bigcup_{n \in \NN} \frac{1}{n}\NP(I_n^*) = \Delta(\I^*) = \Delta(\I).$$
In particular, $\I^*$ is Noetherian by \Cref{thm.NPpoly}. Thus, it follows from Theorem \ref{thm.regOkn}, $\lim \limits_{n\to \infty} {\reg I_n^*}/{n} = \delta(\Delta(\I))$. As $I_n$ and $I_n^*$ are $\mm-$primary, and $I_n \subseteq I_n^*$, we have $\reg(I_n^*) \le \reg(I_n)$ for all $n\in \NN$. Hence,
\begin{align}
 \delta(\Delta(\I))= \lim \limits_{n\to \infty} \frac{\reg I_n^*}{n} \le \lim \limits_{n\to \infty} \frac{\reg I_n}{n}. \label{eq.monPrimary1}
\end{align}

On the other hand, observe that, since $I_n$ is $\mm$-primary, $\RR^r_{\ge 0} \setminus \frac{1}{n}\NP(I_n)$ is a bounded set for all $n \in \NN$. This implies that $\RR^r_{\ge 0} \setminus \Delta(\I)$ is a bounded set. As $\Delta(\I)$ is a polyhedron contained in $\RR^r_{\ge 0}$, its intersection with the $i$-th coordinate axis is a ray emanating from a vertex of $\Delta(\I)$. Denote this vertex by $\textbf{a}_i=(0,0,\ldots,0,a_i,0,\ldots,0)$, where $a_i \in \QQ_{> 0}$. The hyperplane passing through all these vertices has the equation
$$\frac{x_1}{a_1}+\cdots +\frac{x_r}{a_r}=1.$$
Since $\Delta(\I)$ is convex, all its other vertices of it must satisfy
$$\frac{x_1}{a_1}+\cdots +\frac{x_r}{a_r}<1.$$
It is harmless to assume that $a_1 =\max \{a_1,\ldots ,a_r\}$. Consider any vertex $\textbf{v}=(v_1,\ldots ,v_r)$ of $\Delta(\I)$. Then,
\[
\frac{|\textbf{v}|}{a_1} = \frac{v_1}{a_1} + \cdots +\frac{v_n}{a_1} \le \frac{v_1}{a_1} + \cdots +\frac{v_r}{a_r} \le 1.
\]
Thus, $|\textbf{v}| \le a_1$, which implies that $\textbf{a}_1$ is the vertex of $\Delta$ such that $a_1 = |\textbf{a}_1| =  \delta(\Delta(\I))$. 

As $\Delta(\I)$ is non-degenerate along the coordinate axes, there exists a sequence of points $\textbf{v}_{k,1}$ along the $x_1$-axis, where $\textbf{v}_{k,1} \in \frac{1}{i_{k,1}}\NP(I_{i_{k,1}})$ for some index sequence $i_{k,1} \ge k, \forall k \in \NN$, such that $\textbf{v}_{k,1} \rightarrow \textbf{a}_1$ as $k\rightarrow \infty$ in the Euclidean metric. We claim that this sequence of vertices $\textbf{v}_{k,1}$ can be taken to be the sequence of vertices of $\frac{1}{n}\NP(I_n)$, for $n \in \NN$, lying on the $x_1$-axis. A similar argument will show generally that, for $j=1,\ldots ,r$, the sequence of vertices of $\frac{1}{n}\NP(I_n)$ on the $x_j$-axis converges to the vertex $\textbf{a}_j$ on the $x_j$-axis of $\Delta(\I)$.

To prove the claim, denote the vertex of $\frac{1}{n}\NP(I_n)$ on the $x_1$-axis by $\textbf{u}_{n,1}=(f(n),0,\ldots ,0)$. Then, $x_1^{nf(n)}$ is a minimal generator of $I_n$. This implies, since $\I$ is a graded family, that for any $m,n \in \NN$, $x_1^{mf(m)}x_1^{nf(n)}\in I_{m+n}$, and so it is divisible by $x_1^{(m+n)f(m+n)}$. Particularly,
$$mf(m)+nf(n) \ge (m+n)f(m+n) \text{ for all } m,n \ge 1.$$
It follows that $(nf(n))_{n \ge 1}$ is a subadditive sequence and, by Fekete's Lemma,
$\lim\limits_{n \rightarrow \infty} f(n) = \inf\limits_{n \ge 1} \{ f(n) \}$ exist. That is, the limit
$$\textbf{u}_1 = \lim\limits_{n \rightarrow \infty} \textbf{u}_{n,1}$$
exists. Since $\textbf{a}_1$ is the vertex of $\Delta(\I)$ on the $x_1$-axis, we have $\textbf{u}_{1} \ge \textbf{a}_{1}$ (comparing coordinate-wise or the $x_1$-coordinate only as the other coordinates in both $\textbf{u}_1$ and $\textbf{a}_1$ are 0). On the other hand, since $\textbf{v}_{k,1}$ are on the $x_1$-axis, we get $\textbf{v}_{k,1} \ge \textbf{u}_{i_{k,1}}$. Therefore,
$$\textbf{a}_1 = \lim\limits_{k \rightarrow \infty} \textbf{v}_{k,1} \ge \lim\limits_{k \rightarrow \infty} \textbf{u}_{i_{k,1}}= \textbf{u}_1,$$
and the claim is established.

Now, consider an arbitrary $\epsilon>0$. There exists $n_0$ such that $ |\textbf{a}_j| > |\textbf{u}_{n,j}| -\epsilon$ for all $n\ge n_0$ and $j=1,\ldots ,r$. Let $\ell$ be such that $\textbf{u}_{n_0,\ell}$ is the vertex of $\frac{1}{n_0}\NP(I_{n_0})$ with maximal coordinate sum, i.e.,
$$|\textbf{u}_{n_0,\ell}| = \delta\left(\frac{1}{n_0}\NP(I_{n_0})\right).$$
Then,
$$|\textbf{a}_1| \ge |\textbf{a}_{\ell}| > |\textbf{u}_{n_0,\ell}| -\epsilon.$$
Hence,
\[
\delta(\Delta(\I)) > \delta\left(\frac{1}{n_0}\NP(I_{n_0})\right) - \epsilon = \lim \limits_{n\to \infty} \frac{\reg I_{n_0}^n}{n_0 \cdot n} -\epsilon > \lim \limits_{n\to \infty}\frac{\reg I_{n_0 \cdot n}}{n_0 \cdot n} -\epsilon,
\]
where the equality follows from \Cref{thm.regOkn}, and the last inequality holds because $I_{n_0}^n \subseteq I_{n_0 \cdot n}$ and these ideals are $\mm-$primary. Since $\epsilon$ is taken arbitrary, we derive that
\begin{align}
	\delta(\Delta(\I)) \ge \lim \limits_{n\to \infty} \frac{\reg I_n}{n}. \label{eq.monPrimary2}
\end{align}
The result now follows from (\ref{eq.monPrimary1}) and (\ref{eq.monPrimary2}).
\end{proof}

The following example shows that the condition of being non-degenerate along the coordinate axes may not be omitted from \Cref{prop.monprimary}.

\begin{ex}
    \label{ex.mprimarycounter}
    Let $P_n \subset \RR^2$ be a family of convex polyhedra given by
    \[
    P_n = \conv \{(5n,0),(3n+1,1),(0,2n) \} + \RR^2_{\ge 0},
    \]
    and define $I_n:= \langle x^\textbf{a} \ | \  \textbf{a} \in P_n \rangle \in \kk[x,y]$ for $n \ge 1$.

    To see that $\I = (I_n)_{n\ge 0}$ is a graded family, it suffices to show that $P_m + P_n \subseteq P_{m+n}$ for every $m,n \ge 1$. In fact, it is enough to show that all the vertices of $P_m+P_n$ are in $P_{m+n}$. Indeed, observe first that the vertices of the Minkowski sum $P_m+P_n$ are among the sums of a vertex in $P_m$ and a vertex in $P_n$. Moreover, the defining halfplanes of $P_{m+n}$ are given by
    \begin{align*}
    x\ge 0 , y\ge 0 , x-5(m+n)+(2m+2n-1)y &\ge 0,\quad  \text{and},\\
      (2m+2n-1)x+(3m+3n+1)(y-2m-2n) &\ge 0.
    \end{align*}
    It is now straightforward to verify that all Minkowski sums of a vertex in $P_m$ and a vertex in  $P_n$ satisfy the these inequalities. Thus, these sums are in $P_{m+n}$.

    On the other hand,
    \[
    \Delta(\I) = \bigcup_{n \in \NN} \frac{1}{n} P_n = \conv \{(3,0),(0,2) \} + \RR^2_{\ge 0},
    \]
    and so $\delta(\Delta(\I))=3$. However,
    $\lim \limits_{n\to \infty} {\reg I_n}/{n} = \lim \limits_{n\to \infty} {d(I_n)}/{n} = \lim \limits_{n\to \infty} {5n}/{n} =5 \not= 3$.
\end{ex}

Due to Example \ref{ex.mprimarycounter}, one should not expect Theorem \ref{thm.regOkn} to extend to any graded family of monomial ideals in general. However, we shall exhibit a relationship between the asymptotic regularity of a graded family $\I$ of $\mm$-primary monomial ideals and the vertices of a sub-family of $\frac{1}{n}\NP(I_n)$ in the following result.

\begin{thm}
    \label{thm.asymreg=limdegvert}
    Let $\I=(I_n)_{n\ge 0}$ be a graded family of $\mm$-primary monomial ideals in $R$. Then,

\[
\lim \limits_{n\to \infty} \frac{\reg I_n}{n} = \lim \limits_{n\to \infty} \frac{d(I_n)}{n} =\lim \limits_{n\to \infty} \frac{\reg \overline{I_n}}{n} = \lim \limits_{n\to \infty} \frac{d(\overline{I_n})}{n} = \inf \limits_{n\ge 1}\delta\left(\frac{1}{n}\NP(I_{n})\right).
\]
In particular, there exists a sequence of indexes $(e_n)_{n\ge 1}$ such that
$$\lim \limits_{n\to \infty} \frac{\reg I_n}{n} = \lim \limits_{n\to \infty}\delta\left(\frac{1}{e_n}\NP(I_{e_n})\right).$$
\end{thm}

The following constructions and results from \cite{CS2021,HaN23} will be useful in the proof of \Cref{thm.asymreg=limdegvert}. For a graded family $\I = (I_n)_{n\ge 0}$ of homogeneous ideals in $R$ and a positive integer $a$, the \emph{$a$-th truncation} of $\I$ is the graded family $\I_a = (I_{a,n})_{n\ge 0}$, where
$$
I_{a,n} = \left\{\begin{array}{lll} I_n & \text{if} & n \le a, \\ 
\sum\limits_{\substack{i,j > 0 \\ i+j=n}} I_{a,i}I_{a,j} & \text{if} & n > a. 
\end{array}\right.
$$
By the definition, the Rees algebra $\R(\I_a)$ is finitely generated for any $a \in \NN$, thus, there exists an integer $e_a \in \NN$ such that the $e_a$-th Veronese subalgebra $S_a = \R^{(e_a)}(\I_a)$ of $\R(\I_a)$ is standard graded and $\R(\I_a)$ is a finite module over $S_a$. We shall choose $e_a \ge a$.

\begin{defn}
	\label{def.truncation}
	Let $\I = (I_n)_{n\ge 0}$ be a graded family of homogeneous ideals in $R$ and let $a \in \NN$. Define the \emph{$a$-th upper truncation} of $\I$ to be the graded family $\I_a^* = (I_{a,n}^*)_{n\ge 0}$, where
	$$
I_{a,n}^* = \left\{\begin{array}{lll} I_n & \text{if} & n \le a, \\ 
\left(I_{e_a}\right)^r & \text{if} & n = re_a, r \in \NN, \\ 
\sum\limits_{\substack{i,j > 0 \\ i+j=n}} I_{a,i}^* I_{a,j}^* & \text{if} & n \text{ is otherwise}. \end{array}\right.
$$
\end{defn}

It is shown in \cite[Lemma 4.10]{HaN23} that the family $\I^*_a$ constructed in Definition \ref{def.truncation} is a graded family of ideals. The crucial property of the truncation sequences $\I_a$ and $\I^*_a$ is that they are Noetherian graded families. We have the following auxiliary lemma.

\begin{lem}
    \label{lem.regoftruncation}
    Let $\I = (I_n)_{n\ge 0}$ be a graded family of $\mm$-primary homogeneous ideals in $R$. Then,
    \begin{align*}
       \lim \limits_{n\to \infty} \frac{\reg I_n}{n}& = \inf \limits_{a \in \NN}\left(\lim \limits_{n\to \infty} \frac{\reg I_{a,n}}{n}\right) = \lim \limits_{a \to \infty}\lim \limits_{n\to \infty} \frac{\reg I_{a,n}}{n} \\
       & = \inf \limits_{a \in \NN}\left(\lim \limits_{n\to \infty} \frac{\reg I^*_{a,n}}{n}\right) = \lim \limits_{a \to \infty}\lim \limits_{n\to \infty} \frac{\reg I^*_{a,n}}{n}.
    \end{align*}
\end{lem}

\begin{proof}
    From the containment of $\mm$-primary ideals $I_{a,n} \subseteq I^*_{a,n} \subseteq I_n$, by Lemma \ref{lem_reg_containment}, we have
    \[
    \reg I_{a,n}  \ge  \reg I^*_{a,n} \ge \reg I_n,
    \]
    for all $a$ and $n$. This, by the existence of limits shown in Theorem \ref{thm_m-primary}, implies that
    \[
    \lim \limits_{n\to \infty} \frac{\reg I_{a,n}}{n} \ge \lim \limits_{n\to \infty} \frac{\reg I^*_{a,n}}{n} \ge \lim \limits_{n\to \infty} \frac{\reg I_{n}}{n}
    \]
    for all $a\in \NN$. Thus,
    \[
    \inf \limits_{a \in \NN}\left(\lim \limits_{n\to \infty} \frac{\reg I_{a,n}}{n}\right) \ge \inf \limits_{a \in \NN}\left(\lim \limits_{n\to \infty} \frac{\reg I^*_{a,n}}{n}\right) \ge \lim \limits_{n\to \infty} \frac{\reg I_{n}}{n}.
    \]

    On the other hand, one can check from definition of $\I_a, \I^*_a$ and the graded property of $\I$ that for any $a,n\in \NN$, we have $I_{a,n}\subseteq I_{a+1,n}$ and $I^*_{a,n}\subseteq I^*_{a+1,n}$. Particularly, by Lemma \ref{lem_reg_containment}, the sequences $\left(\reg I_{a,n} \right)_{a\in \NN}$ and $\left(\reg I^*_{a,n} \right)_{a\in \NN}$ are non-increasing for each $n\in \NN$. It follows that the two sequences $\left(\lim \limits_{n\to \infty} {\reg I_{a,n}}/{n} \right)_{a\in \NN}$ and $\left(\lim \limits_{n\to \infty} {\reg I^*_{a,n}}/{n} \right)_{a\in \NN}$ are also non-increasing. This ensures the existence of the limits
    \[
    \lim \limits_{a \to \infty}\lim \limits_{n\to \infty} \frac{\reg I_{a,n}}{n} = \inf \limits_{a \in \NN}\left(\lim \limits_{n\to \infty} \frac{\reg I_{a,n}}{n}\right) \text{and  } \lim \limits_{a \to \infty}\lim \limits_{n\to \infty} \frac{\reg I^*_{a,n}}{n} = \inf \limits_{a \in \NN}\left(\lim \limits_{n\to \infty} \frac{\reg I^*_{a,n}}{n}\right).
    \]

   Finally, for any fixed $\epsilon >0$, there exists $n_0$ such that $\dfrac{\reg I_n}{n} < \lim \limits_{p\to \infty} \dfrac{\reg I_p}{p} +\epsilon$ for all $n\ge n_0$. Therefore, for all $a\ge n_0$, we have
    \[
    \lim \limits_{n\to \infty} \frac{\reg I_{a,n}}{n} = \inf \limits_{n\to \infty} \frac{\reg I_{a,n}}{n} \le \frac{\reg I_{a,a}}{a} = \frac{\reg I_{a}}{a} < \lim \limits_{n\to \infty} \frac{\reg I_n}{n} +\epsilon,
    \]
    where the first equality follows from Fekete's Lemma and the proof of Theorem \ref{thm_m-primary} that for any graded family of homogeneous $\mm$-primary ideals $\mathcal{J} = (J_n )_{n\ge 0}$, the sequence $(\reg(J_n))_{n\in \NN}$ is subadditive. This implies that
    $$\lim \limits_{a \to \infty}\lim \limits_{n\to \infty} \frac{\reg I_{a,n}}{n} \le \lim \limits_{n\to \infty} \frac{\reg I_n}{n} +\epsilon$$
    and, since $\epsilon$ is arbitrary,
    $$\lim \limits_{a \to \infty}\lim \limits_{n\to \infty} \frac{\reg I_{a,n}}{n} \le \lim \limits_{n\to \infty} \frac{\reg I_n}{n}$$
    as desired.
\end{proof}

\begin{rem}
    \label{rem.regoftruncation}
    By Theorem \ref{thm_m-primary}, all terms in Lemma \ref{lem.regoftruncation} are also equal to those when we replace the regularity in the formula by the maximal generating degree.

\end{rem}

Now we are ready to prove Theorem \ref{thm.asymreg=limdegvert}.

\begin{proof}[{\bf Proof of Theorem \ref{thm.asymreg=limdegvert}}]
    The first and the third equalities were already established in Theorem \ref{thm_m-primary}. Since $\NP(I_n)=\NP(\overline{I_n})$, it remains to prove that
    $$\lim \limits_{n\to \infty} \frac{\reg I_n}{n} = \inf \limits_{n\ge 1}\delta\left(\frac{1}{n}\NP(I_{n})\right).$$

    By Lemma \ref{lem.regoftruncation} and the proof of Theorem \ref{thm.regOkn}, there exists a sequence $(e_a)_{a\in \NN}$ as in Definition \ref{def.truncation} such that
    \[
    \lim \limits_{n\to \infty} \frac{\reg I_n}{n} = \lim \limits_{a \to \infty}\lim \limits_{n\to \infty} \frac{\reg I^*_{a,n}}{n} = \lim \limits_{a\to \infty}\delta\left(\frac{1}{e_a}\NP(I_{e_a})\right).
    \]
    In particular, $\lim \limits_{n\to \infty} {\reg I_n}/{n} \ge \inf \limits_{n\to \infty}\delta\left(\frac{1}{n}\NP(I_{n})\right)$.

    On the other hand, for any fixed $n\in \NN$, we have $I_n^{k} \subseteq I_{nk}$ for every $k\in \NN$.
    Since these ideals are $\mm$-primary, by Lemma \ref{lem_reg_containment}, we have $\reg I_n^{k} \ge \reg I_{nk}$. Therefore, by Theorem \ref{thm.regOkn},
    \[
    \delta\left(\frac{1}{n}\NP(I_{n})\right) = \lim \limits_{k\to \infty} \frac{\reg I_n^k}{nk} \ge \lim \limits_{k\to \infty} \frac{\reg I_{nk}}{nk} = \lim \limits_{p\to \infty} \frac{\reg I_p}{p}
    \]
    as desired.
\end{proof}

\begin{ex}
    \label{ex.mprimarycounter2}
    Let $P_n \subset \RR^2$ be the family of convex polyhedra and let $(I_n)_{n \ge 0}$ be the graded family of ideals, as considered in Example \ref{ex.mprimarycounter}. Then, $\delta\left(\frac{1}{n}\NP(I_{n})\right) = \delta\left(\frac{1}{n}P_n\right)=5$ for every $n\in \NN$. Hence,
    $$\lim \limits_{n\to \infty} \frac{\reg I_n}{n} = \lim \limits_{n\to \infty} \frac{d(I_n)}{n} = \lim \limits_{n\to \infty} \frac{5n}{n} =5=\inf \limits_{n\ge 1}\delta\left(\frac{1}{n}\NP(I_{n})\right).$$
\end{ex}

\Cref{prop.monprimary} also gives us the following immediate corollary.

\begin{cor}\label{cor.primarypolyhedron}
Let $\I$ be a graded family of $\mm$-primary homogeneous ideals in $R$. Let $>$ be a monomial order on $R$ and let $\ini_>(\I) = (\ini_>(I_n))_{n \ge 0}$ be the family of initial ideals of $\I$. Suppose that $\Delta({\rm in}_>(\I))$ is a polyhedron that is non-degenerate along the coordinate axes. Then,
\[
\lim \limits_{n\to \infty} \frac{\reg I_n}{n} = \lim \limits_{n\to \infty} \frac{d(I_n)}{n} = \delta(\Delta({\rm in}_>(\I))).
\]
\end{cor}

\begin{proof}
For each $n\in \NN$, $I_n$ is $\mm$-primary, so is ${\rm in}_>(I_n)$. Since $R/I_n$ and $R/{\rm in}_>(I_n)$ are $0$-dimensional and have the same Hilbert function, they have the same regularity. Thus, the result follows immediately by Proposition \ref{prop.monprimary}.
\end{proof}


\subsection{$\mm$-primary homogeneous and Cohen-Macaulay ideals} We proceed to our next main result, which generalizes \Cref{prop.monprimary} to families of $\mm$-primary homogeneous ideals or, more generally, Cohen-Macaulay ideals of the same codimension.

\begin{thm}
\label{thm.mPrimaryNOvaluation}
Let $\I = (I_n)_{n\ge 0}$ be a graded family of $\mm$-primary homogeneous ideals in $R$. Let $\Delta(\I)$ be the Newton--Okounkov region of $\I$ constructed from a good valuation that respects the monomials in $R$. Suppose that $\Delta(\I)$ is a polyhedron that is non-degenerate along the coordinate axes. Then,
\[
\lim \limits_{n\to \infty} \frac{\reg I_n}{n} = \lim \limits_{n\to \infty} \frac{d(I_n)}{n} =\lim \limits_{n\to \infty} \frac{\reg \overline{I^n}}{n} = \lim \limits_{n\to \infty} \frac{d(\overline{I^n})}{n}= \delta(\Delta(\I)).
\]
\end{thm}

\begin{proof}
As $\I$ is a family of $\mm$-primary homogeneous ideals, so is the family $\overline{\I}=(\overline{I_n})_{n\ge 0}$. By Lemma \ref{lem.NOIntClo}, using the same good valuation that respects the monomials in $R$, we have $\Delta(\I)=\Delta(\overline{\I})$, in particular, $\delta(\Delta(\I))=\delta(\Delta(\overline{\I}))$. By Theorem \ref{thm_m-primary}, it remains to show that
$$\lim \limits_{n\to \infty} \frac{\reg I_n}{n} = \delta(\Delta(\I)).$$
For each $f\in R$, denote by $\tilde{f}$ the image of $f$ in $R_\mm$. For $k\in \NN$, set
$$\ZZ^r_{=k} = \{v\in \ZZ^r_{\ge 0} ~\big|~ |v|=k \},$$
and for a homogeneous ideal $J = \bigoplus_{k \ge 0}J_k \subset R$, let
$$V^J_{=k} = \{v(\tilde{f}) \mid  f \in J_k\setminus \{0\} \},$$
and $V_J = \{v(\tilde{f}) \mid f \, \text{is a homogeneous element of}\, J \setminus \{0\} \}$.

We shall see that $v$ respects the grading of $R$ in the sense that for every homogeneous element $f\in R\setminus \{0\}$, we have $|v(\tilde{f})|=\deg f$. Indeed, write $f=\sum_{i=1}^q \alpha_i x^{\textbf{m}_i}$, where $\alpha_i\in \kk\setminus \{0\}$ and the vectors $\textbf{m}_i\in \NN^r$ are pairwise distinct. By the property of valuations and the fact that $v(x^{\textbf{m}_i})=\textbf{m}_i$ for any $i$, we conclude that $v(\widetilde{f})=\textbf{m}_i$ for some $i$. This yields $|v(\tilde{f})|=\deg f$. The proof proceeds in two steps. The key idea is to pass to the monomial case where \Cref{prop.monprimary} is in force.

\medskip

\noindent\textsf{Step 1}: We will first show that if $J$ is a $\mm$-primary homogeneous ideal then for each $k\in \NN$,
$$H(R/J,k):=\dim_\kk(R/J)_k = \# \left(\ZZ^r_{=k} \setminus V^J_{=k}\right).$$
Indeed, observe first that, since $v$ is a good valuation and $J$ is $\mm$-primary, the values of elements in $R_\mm \setminus J_\mm$ are bounded (as if $\ell(v(f))\ge pr_0$ then $f\in \mm^p \subset J$ for large $p$). Thus, $\# \left(\ZZ^r \setminus V_J\right)$ is finite. In particular, $\# \left(\ZZ^r_{=k} \setminus V^J_{=k}\right) <\infty$ for every $k \in \NN$.

Fix $k\in \NN$ and suppose that $\left(\ZZ^r_{=k} \setminus V^J_{=k}\right) = \{ v_1, \ldots ,v_s \}$. Let $B= \{b_1,\ldots ,b_s \} \subseteq R$ be a collection of monomials such that $v(\widetilde{b_i})=v_i$ for each $i = 1, \dots, s$. (Here, we use the hypothesis that $v$ respects the monomials in $R$.) As shown before, $v$ also respects the grading of $R$. Thus, $B\subseteq R_k$. Observe also that, for any $\kk$-linear combination $M = c_1b_1 + \cdots + c_sb_s$, by the assumption on $v$, there exists an $i$ such that $v(\widetilde{M}) = v(\widetilde{b_i}) = v_i \not\in V^J_{=k}$. Therefore, $M \not\in J_k$.
In other words, the image $\text{im}(B)$ of $B$ in $(R/J)_k$ is linearly independent.

We shall further show that $\text{im}(B)$ spans $(R/J)_k$. Consider any $f\in R_k \setminus \text{span} \langle J_k,\text{im}(B) \rangle$ with maximum $v(\tilde{f})$. If $v(\tilde{f}) = v(\tilde{b})$ for some $b\in B$, then since $v$ has one-dimensional leaves, there exists $\lambda \in \kk$ such that $v(\widetilde{f}+\lambda\widetilde{b})>v(\widetilde{f})$, which contradicts
the choice of $f$. On the other hand, if $v(\widetilde{f}) \in V^J_{=k}$, then we have $v(\widetilde{f}) = v(\widetilde{a})$ where $a\in J_k$. Again, there exists $\lambda\in \kk$ such that $v(\widetilde{f}+\lambda\widetilde{a})>v(\widetilde{f})$, which contradicts
the choice of $f$. Therefore, such an element $f$ does not exist; that is, $\text{im}(B)$ forms a $\kk$-vector space basis for $(R/J)_k$, which yields the claimed equality
$$H(R/J,k)=\dim_\kk(R/J)_k = \# \left(\ZZ^r_{=k} \setminus V^J_{=k}\right).$$

\medskip

\noindent\textsf{Step 2}:  Now we complete the proof of the theorem. For an $\mm$-primary homogeneous ideal $J\subseteq R$, consider the monomial ideal $\tilde{J}$ generated by  $$\{x^{v(\widetilde{f})} \mid f\, \text{is a homogeneous element of}\, J\setminus \{0\} \}.$$
Since the values of $\tilde{J}$ contain the values of $J$, it is clear that $\tilde{J}$ is also $\mm$-primary. Moreover, since $v$ respects the monomials and the grading of $R$, we have
$$V^J_{=k} = V^{\tilde{J}}_{=k}, \, \text{for each} \, k\in \NN.$$
Therefore, by Step 1,
$$H(R/\tilde{J},k)= \# \left(\ZZ^r_{=k} \setminus V^J_{=k}\right) = H(R/J,k), \, \text{for each} \, k\in \NN.$$

Applying this equality to the graded family $\I$, we get that $R/I_n$ and $R/\widetilde{I_n}$ have the same Hilbert function for all $n \in \NN$. (Here, $\widetilde{I_n}$ is constructed as at the beginning of Step 2.) Since $I_n$ and $\widetilde{I_n}$ are $\mm$-primary, it follows that $\reg I_n =\reg \widetilde{I_n}$ for all $n \in \NN$. Moreover, as $V_{I_n} = V_{\widetilde{I_n}}$ for all $n$, we have $\mathcal{C}(\I) = \mathcal{C}(\widetilde{\I})$ and $\Delta(\I) = \Delta(\widetilde{\I})$ where $\widetilde{\I} = ( \widetilde{I_n})_{n\ge 0}$. Proposition \ref{prop.monprimary} now implies that
\[
\lim \limits_{n\to \infty} \frac{\reg I_n}{n} = \lim \limits_{n\to \infty} \frac{\reg \widetilde{I_n}}{n} = \delta(\Delta(\widetilde{\I})) = \delta(\Delta(\I)).
\]
The theorem is proved.
\end{proof}

Extending Corollary \ref{cor.primarypolyhedron} by means of Theorem \ref{thm_CohenMacaulay}, we obtain the following result.

\begin{thm}
\label{prop.NOofCMfamily}
Let $\kk$ be an uncountable field of characteristic $0$. Let $\I=(I_n)_{n\ge 0}$ be a graded family of homogeneous ideals such that, for all $n\ge 1$, $R/I_n$ is Cohen-Macaulay of the same dimension. Suppose that $\Delta({\rm gin}(\I))$ is a polyhedron that is non-degenerate along the coordinate axes. Then,
\[
\lim \limits_{n\to \infty} \frac{\reg I_n}{n} = \lim \limits_{n\to \infty} \frac{d(I_n)}{n}=\delta(\Delta({\rm gin}(\I))).
\]
\end{thm}

\begin{proof}
Let $\dim(R/I_n)=\dim(R/I_1)=c$. The statement is vacuous if $c = r$. If $c=0$, then the result follows from Corollary \ref{cor.primarypolyhedron}. Suppose that $1\le c\le r-1$. We shall proceed with a similar line of arguments as that of \cite[Theorem 1.1]{Mayes14}.

 Observe that, by Lemma \ref{lem_simultaneous_filreg}, there is a general linear form $h=a_1x_1+\cdots +a_rx_r$ that is filter-regular with respect to $R/I_n$ for all $n\ge 1$. Thus, using induction on $r$ as in \cite[Proposition 3.3]{Mayes14}, we get that there exists a graded family $(L_n)_{n\in \NN}$ in $\kk[x_1,\ldots ,x_{r-c}]$ such that ${\rm gin}(I_n)$ and ${\rm gin}(L_n)$ have the same minimal generators for all $n\in \NN$. Moreover, since $I_n$ is Cohen-Macaulay, so are ${\rm gin}(I_n)$ and ${\rm gin}(L_n)$. This implies that $L_n$ is Cohen-Macaulay for all $n \in \NN$.

 Also, since ${\rm gin}(I_n)$ and ${\rm gin}(L_n)$ are Borel ideals (or strongly stable ideals), we have that $\reg({\rm gin}(I_n)) = \reg({\rm gin}(L_n))$ is equal to the maximal generating degree of these ideals. Therefore, to prove the desired equality of the theorem, it is enough to show that
\[
\lim \limits_{n\to \infty} \frac{\reg L_n}{n} =\delta(\Delta({\rm gin}(\I))).
\]

Observe further that, since ${\rm gin}(I_n)$ and ${\rm gin}(L_n)$ have the same minimal generators, we have $\NP({\rm gin}(I_n)) = \NP({\rm gin}(L_n)) \times \RR^c_{\ge 0}$. Hence,
\[
\Delta({\rm gin}(\I)) = \overline{\bigcup_{n\in \NN} \frac{1}{n} \NP({\rm gin}(I_n))} = \overline{\bigcup_{n\in \NN} \frac{1}{n} \NP({\rm gin}(L_n))} \times \RR^c_{\ge 0}  = \Delta({\rm gin}(\Lcc)) \times \RR^c_{\ge 0},
\]
where $\Lcc =(L_n)_{n\in \NN}$. Since $\Delta({\rm gin}(\I))$ is a polyhedron that is non-degenerate along the coordinate axes in $\RR^r$, so is $\Delta({\rm gin}(\Lcc))$ in $\RR^{r-c}$.

In addition, since ${\rm gin}(L_n)$ is Cohen-Macaulay of codimension $r-c$ for all $n \in \NN$, it follows from \cite[Lemma 3.1]{Mayes14} and \cite[Proposition 3.2]{Mayes14} that $L_n$ is a $(x_1, \cdots, x_{r-c})$-primary ideal in $\kk[x_1,\ldots ,x_{r-c}]$ for all $n \in \NN$. Corollary \ref{cor.primarypolyhedron} then gives
\[
\lim \limits_{n\to \infty} \frac{\reg I_n}{n} = \lim \limits_{n\to \infty} \frac{\reg L_n}{n} =\delta(\Delta({\rm gin}(\Lcc))) = \delta(\Delta({\rm gin}(\I))).
\]
This finishes the proof.
\end{proof}

\begin{quest}
    Let $\I=(I_n)_{n\ge 0}$ be a graded family of homogeneous ideals such that, for all $n\ge 1$, $R/I_n$ is Cohen-Macaulay of the same codimension. What can be said about $\Delta({\rm gin}(\I))$ and $\Delta(\I)$? In particular, when does the equality $\delta(\Delta({\rm gin}(\I)))=\delta(\Delta(\I))$ hold?
\end{quest}



\section{Negative answers for Questions \ref{quest.HT} and \ref{quest.Intersection}} \label{sec.nonExist}

The remaining of the paper is devoted to finding negative answers to Questions \ref{quest.HT} and \ref{quest.Intersection}, and exploring its consequences toward the aforementioned question of Herzog, Hoa and Trung.
We shall fix the following notations.

\begin{notn}
\label{notn_ideals}
Let $\kk$ be a field of characteristic $2$, and let $R=\kk[x,y,a,b]$ be the standard graded polynomial ring in 4 variables over $\kk$.
\begin{enumerate}
	\item $\mm = (x,y,a,b)$ is the maximal homogeneous ideals in $R$.
	\item $Q=(x^3,y^3)$ is an ideal in $R$.
 	\item $f=xya-(x^2+y^2)b$ is an element in $R$.
	\item We will use the \emph{degree reverse lexicographic} order induced by $x>y>a>b$ when discussing Gr\"obner bases of ideals in $R$.
	\item Beside the natural $\ZZ$-graded structure, $R$ is also equipped with a $\ZZ^2$-graded structure where $\deg(x)=\deg(y)=(1,0)$ and $\deg(a)=\deg(b)=(0,1)$. Clearly $Q$ and $f$ are bihomogeneous in this $\ZZ^2$-grading of $R$.
\end{enumerate}
\end{notn}

The main result of this section is the following negative answers to Questions \ref{quest.HT} and \ref{quest.Intersection}.

\begin{thm}
	\label{thm.noLimit}
	Let $Q$ and $f$ be as before. The limits
	$$\lim\limits_{n \rightarrow \infty} \dfrac{\reg (Q^n + (f^n))}{n} \text{ and } \lim\limits_{n \rightarrow \infty} \dfrac{\reg (Q^n \cap (f^n))}{n}$$
	do not exist.
\end{thm}

\begin{proof} The conclusion is achieved by exhibiting two infinite sequences of integers $(m_s)_{s \ge 1}$ and $(n_s)_{s \ge 1}$ for which
	\begin{align}
		\liminf_{s \rightarrow \infty} \dfrac{\reg (Q^{m_s}+(f^{m_s}))}{m_s} & \not= \limsup_{s \rightarrow \infty} \dfrac{\reg (Q^{n_s}+(f^{n_s}))}{n_s}, \text{ and } \label{eq.notEqualLim} \\
		\liminf_{s \rightarrow \infty} \dfrac{\reg (Q^{m_s} \cap (f^{m_s}))}{m_s} & \not= \limsup_{s \rightarrow \infty} \dfrac{\reg (Q^{n_s} \cap (f^{n_s}))}{n_s}. \label{eq.notEqualLim2}
	\end{align}
	Indeed, we shall choose $m_s = 2^s$ and $n_s = 3\cdot 2^s$, for $s \in \NN$.
	
	It can be easily seen that $\reg Q^n = 3n+2$ and $\reg (f^n) = 3n$ for all $n \in \NN$. Consider the short exact sequence
	$$0 \rightarrow R/(Q^n \cap (f^n)) \rightarrow R/Q^n \oplus R/(f^n) \rightarrow R/(Q^n + (f^n)) \rightarrow 0$$
and its associated long exact sequence of local cohomology modules. This, together with Theorems \ref{thm_regbound_2powers} and \ref{thm_limregge6_3times2power} below, implies that for $n = m_s$ and $n = n_s$, $s \in \NN$, we have
	$$\reg (Q^n + (f^n)) = \reg (Q^n \cap (f^n)) - 1.$$
Therefore, (\ref{eq.notEqualLim}) and (\ref{eq.notEqualLim2}) are equivalent. Furthermore, (\ref{eq.notEqualLim}) is a direct consequence of Theorems \ref{thm_regbound_2powers} and \ref{thm_limregge6_3times2power}. Hence, the result is proved by establishing Theorems \ref{thm_regbound_2powers} and \ref{thm_limregge6_3times2power}.
\end{proof}

\begin{thm}
	\label{thm_regbound_2powers}
	Consider $n=2^s$, where $s\ge 3$ is an integer.
	\begin{enumerate}[\quad \rm (1)]
		\item $\ini\left(Q^n + (f^n)\right)$ is given by
		\[
		\begin{cases}
			Q^n+(x^ny^na^n)+b^nx^{2n+1}yQ^\frac{n-2}{3}+b^nx^2y^{2n+1}Q^\frac{n-2}{3}, &\text{if $s$ is odd},\\
			Q^n+(x^ny^na^n)+b^nx^{2n+2}y^2Q^\frac{n-4}{3}+b^n\left(x^2y^{2n}Q^\frac{n-1}{3}+(x^{2n}y^{n+1})\right), &\text{otherwise}.
		\end{cases}
		\]
		\item We have the following containment
		\begin{align*}
			x^ny^{2n+1}a^{n-1}b^{n-1} & \in \left(\left(Q^n+f^n\right):\mm\right)\setminus (Q^n+f^n), \quad \text{if $s$ is odd},\\
			x^{n+1}y^{2n}a^{n-1}b^{n-1} & \in \left(\left(Q^n+f^n\right):\mm\right)\setminus (Q^n+f^n), \quad \text{otherwise}.
		\end{align*}
		\item We have the following bounds for $\reg (Q^n + (f^n))$:
		\[
		5n\le \reg\left(Q^n+(f^n)\right)\le 5n+2.
		\]
	\end{enumerate}
\end{thm}

\begin{thm}
	\label{thm_limregge6_3times2power}
	Consider $n=3\cdot 2^s$, where $s\ge 3$ is an odd integer, and  set $t=n/3=2^s$. Then,
	\[
	x^{2t+1}y^{7t}a^{t-1}b^{8t-1} \in \left(\left(Q^n+(f^n)\right):\mm\right)\setminus \left(Q^n+(f^n)\right).
	\]
	In particular, we have: $\reg\left(Q^n+(f^n)\right)\ge 6n$.
\end{thm}

The proofs of Theorems \ref{thm_regbound_2powers} and \ref{thm_limregge6_3times2power} are presented separately in the next two sections of the paper.


\section{Proof of Theorem \ref{thm_regbound_2powers}} \label{sec.1st}

This section presents the proof of Theorem \ref{thm_regbound_2powers}. We shall keep the notations from Section \ref{sec.nonExist} and Notation \ref{notn_ideals}.
Denote by $\{\alpha\}=\alpha-\lfloor \alpha\rfloor$ the fractional part of a real number $\alpha$.


\subsection{Gr\"obner basis analysis} The proof of Theorem \ref{thm_regbound_2powers} follows from a good understanding of the Gr\"obner basis of $Q^n + (f^n)$, which is the content of our next result.

\begin{thm}
\label{thm_GB_2powers}
Consider $n=2^s$, where $s\ge 3$ is an integer.
\begin{enumerate}[\quad \rm (1)]
 \item Assume that $s$ is odd. Then, $Q^n+(f^n)$ has a Gr\"obner basis consisting of the natural generators of
 \[
\udb{Q^n}_{T_1} \quad  \cup \quad  \udb{(f^n)}_{T_2} \quad \cup \quad \udb{(x^{2n}+y^{2n})xyQ^{\frac{n-2}{3}}b^n}_{T_3} \quad \cup \quad  \udb{x^2y^{2n+1}Q^{\frac{n-2}{3}}b^n}_{T_4}.
\]
\item Assume that $s$ is even. Then, $Q^n+(f^n)$ has a Gr\"obner basis  consisting of the natural generators of
 \[
\udb{Q^n}_{T_1} \quad  \cup \quad  \udb{(f^n)}_{T_2} \quad \cup \quad \udb{(x^{2n}+y^{2n})x^2y^2Q^{\frac{n-4}{3}}b^n}_{T_3} \quad  \cup \quad  \udb{\left(x^2y^{2n}Q^{\frac{n-1}{3}}+(x^{2n}y^{n+1})\right)b^n}_{T_4}.
 \]
 \item In particular, $Q^n+(f^n)$ has a Gr\"obner basis  with $\dfrac{5\cdot 2^s+8+4\left\{\dfrac{s-1}{2}\right\}}{3}$ $\left(\approx \dfrac{5}{3}n+3\right)$ elements whose maximal possible degree  is $4n+1$.
\end{enumerate}
\end{thm}

For this result, and further analysis in the forthcoming sections, we introduce the following notations for $S$-pairs of polynomials. For $g,h\in R$, let $S(g,h)$ denote the $S$-pair of $g$ and $h$, given by
\[
S(g,h):=\frac{\lcm(\ini(g),\ini(h))}{\ini(g)}\cdot g-\frac{\lcm(\ini(g),\ini(h))}{\ini(h)}\cdot h.
\]
\begin{notn}[Notations for the reduction to zero of $S$-pairs]
\label{notn_GB}
Suppose that $G_1,\ldots,G_p$ are finite subsets of $R$, where $p\ge 1$. For each $1\le i\le p$, let $T_i$ be the ideal generated by $G_i$. Let $1\le i,j\le p$ be integers.
\begin{enumerate}
\item We denote by $S(T_i,T_j)$ the set $\{S(f_i,f_j)\mid f_i\in G_i, f_j\in G_j\}$.
 \item We write $S(T_i,T_j)=\{0\}$ if for any $f_i\in G_i$ and any $f_j\in G_j$, the $S$-polynomial $S(f_i,f_j)$ is identically $0$.
 \item A subset $H$ of $R$ is called \emph{reducible with respect to the collection of ideals $T_1,\ldots,T_p$} (\emph{reducible} for short) if there is a finite sequence of (possibly repeated)  elements $\ell_1,\ldots,\ell_s$ of $[p]=\{1,\ldots,p\}$ such that $H=T_{\ell_1}+T_{\ell_2}+\cdots+T_{\ell_s}$, i.e. $H$ consists of finite sums of the form $u_1+u_2+\cdots +u_s$ where $u_i\in T_{\ell_i}$.
 \item We denote by $\Sigma T_1$ the set of elements of $R$ that is a finite sum of \emph{monomials} in $T_1$.
 \item Let $H$ be reducible with respect to the collection of ideals $T_1,\ldots,T_p$, witnessed by the expression $H=T_{\ell_1}+T_{\ell_2}+\cdots+T_{\ell_s}$. We say that an element $g\in R$ has a \emph{standard representation via} $H$ if $g$ can be written as $g_1+\cdots+g_s$ where $\ini(g)=\ini(g_1) > \ini(g_2) >\cdots > \ini(g_s)$ and $g_i\in T_{\ell_i}$ for every $i=1,\ldots,s$.

 In particular, if $g$ has a standard representation via $H$, then the remainder of the division of $g$ to $G_{\ell_1} \cup \cdots \cup G_{\ell_s}$ is zero.

 \item For each $\ell$ in a finite index set $\Lambda$, assume that we have a set $H_\ell$ that is reducible with respect to the collection of ideals $T_1,\ldots,T_p$.  We write $S(T_i,T_j)\subseteq \bigcup_{\ell\in \Lambda} H_\ell$, if for any $f_i\in G_i$ and $f_j\in G_j$, the $S$-polynomial $S(f_i,f_j)$ has a standard representation via $H_\ell$, for some $\ell \in \Lambda$.
 \item We write $S(T_i,T_j)\mathop{\xrightarrow{\qquad \quad}}\limits_{\bigcup_{\ell \in \Gamma} T_\ell} 0$ , for a set $\Gamma \subseteq [p]$, if for any $f_i\in G_i$ and $f_j\in G_j$, the remainder after the division algorithm of  $S(f_i,f_j)$ to  $\bigcup_{\ell \in \Gamma} G_\ell$ is zero.
 \item We shall underline a monomial in a given polynomial to signify its leading term. For example, $\udl{y^3a}+x^3b$ indicates that $y^3a$ is the leading term of the polynomial $y^3a+x^3b$ in a given term order.
\end{enumerate}
We will usually apply these notations in the case the ideals $T_i$ are given together with a natural finite set of minimal generators.
\end{notn}
\begin{rem}[Reduction to zero after division algorithms versus ideal containments]
As an example for the containment in \Cref{notn_GB}(6), in the proof of \Cref{prop_GB_3times2power}, we will see the following containment (see \Cref{subsec.GB3times2power}):

(S34): $S(T_3,T_4) \subseteq \Sigma T_1 \cup (T_1+T_9+\Sigma T_1)$.

Thus, as explained in \Cref{notn_GB}(6), this should be understood as follows: For any minimal generators $f_3\in G_3, f_4\in G_4$ (where for each $i$, $G_i$ is a predescribed set of minimal generators of the ideal $T_i$), the $S$-polynomial $S(f_3,f_4)$ is either a sum of monomials in $T_1$, or else $S(f_3,f_4)=g_1+g_2+g_3$, where $\ini(g_1)>\ini(g_2)>\ini(g_3)$ and $g_1$ is a monomial in $T_1, g_2\in T_9$ and $g_3$ is a sum of monomials in $T_1$.

In particular, the containment (S34) should \textbf{not} be simply read as an \textbf{ideal containment}, but rather it describes how the $S$-polynomials of the form $S(f_3,f_4)$, where $f_3,f_4$ belongs to certain minimal generating set of $T_3,T_4$, \textbf{reduce to zero after division} to the elements of $G_1 \cup G_9$ (the union of minimal generators of $T_1$ and $T_9$).

Another thing worthy of observing is that in containments as (S34) (or more generally, those of the form specified by \Cref{notn_GB}(6)), in each term of the union on the right-hand side,  the involving summands of that term are generally \textbf{not} permutable. For example, one cannot simplify (S34) to $S(T_3,T_4) \subseteq \Sigma T_1 \cup (T_9+\Sigma T_1)$.

Of course, one can simplify (S34) to $S(T_3,T_4) \subseteq T_1+T_9$, but it is not useful enough when the task at hand is proving certain reduction to zero after certain division algorithm.
\end{rem}

We shall need the following simple lemma, whose proof is straightforward and left for the interested reader, to show (non-)containment of elements in $Q^n$.

\begin{lem}
\label{lem_inQn}
Let $n\ge 1$ be an integer. The following statements hold true.
\begin{enumerate}[\quad \rm (1)]
 \item There is a containment
$(x,y)^{3n+2} \subseteq Q^n$.
\item Consider the $\ZZ^2$-grading of $R$ given in \Cref{notn_ideals}. For a bihomogeneous element $g\in R$, denote by $\deg_{\{x,y\}}g$ the total degree of $g$ with respect to the variables $x$ and $y$. Assume moreover that $\deg_{\{x,y\}}g \ge 3n+2$. Then $g\in Q^n$.
\item For non-negative integers $i,j$ such that $i+j\ge 3n$, the following are equivalent:
\begin{enumerate}[\quad \rm (i)]
 \item $x^iy^j \notin Q^n$;
 \item either \textup{(a)} $i+j=3n$ and $ij\not\equiv 0$ \textup{(}mod $3$\textup{)}, or \textup{(b)} $i+j=3n+1$, $i\equiv 2$ \textup{(}mod $3$\textup{)} and $j\equiv 2$ \textup{(}mod $3$\textup{)}.
\end{enumerate}
\end{enumerate}
\end{lem}

The next lemma will also be employed frequently in the sequel.
\begin{lem}
\label{lem_large_xydegree}
Let $\kk$ be a field of arbitrary characteristic, and $R=\kk[x,y,a,b]$. Let $m_1=x^{\ell_1}y^{\ell_2}a^{\ell_3}b^{\ell_4}$ and $m_2=x^{j_1}y^{j_2}a^{j_3}b^{j_4}$ be two monomials in $R$. Denote $q=\deg_{\{x,y\}}m_2$, i.e., $q=j_1+j_2$. Then there is an inequality
\begin{equation}
\label{eq_degxyineq}
\deg_{\{x,y\}} \lcm(m_1,m_2) \ge q+ \max\{\ell_1-j_1, \ell_2-j_2\}.
\end{equation}
Moreover, the inequality $\deg_{\{x,y\}} \lcm(m_1,m_2) \ge q+1$ holds in each of the following cases:
\begin{enumerate}[\quad \rm (1)]
\item either $\ell_1> j_1$ or $\ell_2>j_2$;
\item $\deg_{\{x,y\}}m_1=q$ and either $\ell_1\neq j_1$ or $\ell_2\neq j_2$;
\item $\deg_{\{x,y\}}m_1 \ge q-1$, $\ell_1\neq j_1$ and $\ell_2\neq j_2$.
\end{enumerate}
\end{lem}
\begin{proof}
	Note that $\lcm(m_1,m_2)=x^{\max\{\ell_1,j_1\}}y^{\max\{\ell_2,j_2\}}a^{\max\{\ell_3,j_3\}}b^{\max\{\ell_4,j_4\}}$. Hence
	\[
	\deg_{\{x,y\}} \lcm(m_1,m_2)=\max\{\ell_1,j_1\}+\max\{\ell_2,j_2\}\ge \ell_1+j_2=(\ell_1-j_1)+q.
	\]
	Similarly $ \deg_{\{x,y\}} \lcm(m_1,m_2)\ge (\ell_2-j_2)+q$, so \eqref{eq_degxyineq} is true.
	
	(1) This is immediate from \eqref{eq_degxyineq}.
	
	(2) We only consider the case $\ell_1\neq j_1$; the case $\ell_2\neq j_2$ is similar. If $\ell_1> j_1$ then the conclusion follows from part (1). If $j_1>\ell_1$ then again the conclusion follows from part (1) by exchanging the role of $m_1$ and $m_2$.
	
	(3) Assume the contrary, that $\deg_{\{x,y\}} \lcm(m_1,m_2)=\max\{\ell_1,j_1\}+\max\{\ell_2,j_2\} \le q$, which yields
	\[
	q=j_1+j_2 \le \max\{\ell_1,j_1\}+\max\{\ell_2,j_2\} \le q.
	\]
	This forces $\ell_1\le j_1, \ell_2\le j_2$. Since $\ell_1\neq j_1, \ell_2\neq j_2$, we get
	\[
	q-1 \le \deg_{\{x,y\}}m_1 = \ell_1+\ell_2 \le (j_1-1)+(j_2-1)=q-2,
	\]
	a contradiction. This yields $\deg_{\{x,y\}} \lcm(m_1,m_2) \ge q+1$ and concludes the proof.
\end{proof}

\begin{center}
\begin{table}[h!]
		\caption{The natural generators of 4 types of ideals featuring in the Gr\"obner basis of $Q^n+(f^n)$, where $n=2^s$, $s \ge 3$ is an odd integer.}
		\begin{tabular}{| c | c | c |}
			\hline
			\multirow{2}{*}{$i$} & Degree of generators & \multirow{2}{*}{Form of generators of $T_i$} \\
			&  of $T_i$            &    \\
			\hline
			1 & $3n$     & $f_1=x^{3i_1}y^{3n-3i_1}, 0\le i_1\le n$\\
			\hline
			2 & $3n$     & $f_2=x^ny^na^n+x^{2n}b^n+y^{2n}b^n$\\
			\hline
			\multirow{3}{*}{3} & \multirow{3}{*}{$4n$}   & $f_3=(x^{2n}+y^{2n})x^{3i_3+1}y^{n-1-3i_3}b^n$\\
			  &        & $\qquad = x^{2n+3i_3+1}y^{n-3i_3-1}b^n+ x^{3i_3+1}y^{3n-3i_3-1}b^n, $ \\
			  &        &  $0\le i_3 \le \frac{n-2}{3}$ \\
			\hline
			4 & $4n+1$ & $f_4=x^{3i_4+2}y^{3n-3i_4-1}b^n, 0\le i_4 \le \frac{n-2}{3}$\\
			\hline
		\end{tabular}
		\label{tab_mingens_GB_2powers}
	\end{table}
\end{center}

Recall from \Cref{notn_GB} that the underlined monomial of a polynomial is its leading term. We are ready to present the proof of \Cref{thm_GB_2powers}.

\begin{proof}[{\bf Proof of \Cref{thm_GB_2powers}}]
The last part follows from (1) and (2) by a simple counting, so we will only focus on (1) and (2). Observe that the ideals $T_i, 1\le i\le 4$, are all bihomogeneous in the $\ZZ^2$-grading of $R$ given in \Cref{notn_ideals}.

(1) Below, we use \Cref{notn_GB} on $S$-pairs. We will also use the forms of the minimal generators of the ideals $T_1,\ldots,T_4$ recorded in \Cref{tab_mingens_GB_2powers}.

To prove that the minimal generators of four type of ideals form a Gr\"obner basis for the ideal $Q^n+(f^n)$, it suffices to prove the following statements (which implies that the $S$-pair condition is satisfied):
\begin{enumerate}
\item[(S0)] For any $1\le i,j\le 4$, we have $S(T_i,T_j)=\{0\}$ if $i=j$ or $(i,j)=(1,4)$.
 \item[(S12)] $S(T_1,T_2) \subseteq \Sigma T_1\cup T_3$. 
 \item[(S13)] $S(T_1,T_3) \subseteq T_1 \cup T_4$.
 \item[(S23)] $S(T_2,T_3)\subseteq \Sigma T_1$. 
 \item[(S24)] $S(T_2,T_4)\subseteq \Sigma T_1$.
 \item[(S34)] $S(T_3,T_4)\subseteq T_1$.
\end{enumerate}
The proofs of these statements will imply that $T_3\subseteq T_1+T_2, T_4\subseteq T_1+T_3$, so the polynomials of type $T_1,T_2, T_3, T_4$ are all in $Q^n+(f^n)$.

\textbf{For (S0)}: If $i=j$ or $(i,j)=(1,4)$, then any pair $(f_i,f_j)$ of two (distinct) minimal elements $f_i\in T_i, f_j\in T_j$  are of the form $f_i=gg_1, f_j=gg_2$, where $g,g_1,g_2\in R$ and $g_1,g_2$ are monomials. Hence, their $S$-pair is zero.

\textbf{For (S12)}:  We have
\begin{align*}
S(f_1,f_2)&=S(x^{3i_1}y^{3(n-i_1)},\udl{x^ny^na^n}+(x^{2n}+y^{2n})b^n)=-x^{\alpha-n}y^{\beta-n}(x^{2n}+y^{2n})b^n\\
          &=x^{\alpha-n}y^{\beta-n}(x^{2n}+y^{2n})b^n,
\end{align*}
where $\alpha =\max\{3i_1,n\}, \beta =\max\{3(n-i_1),n\}$. (Note that $\chara \kk=2$ so $-1=1$.)

Assume that $S(f_1,f_2)\notin \Sigma T_1$. Then, since $S(f_1,f_2)$ is bihomogeneous, we deduce from \Cref{lem_inQn} that $\deg_{\{x,y\}}S(f_1,f_2)=\alpha+\beta \le 3n+1$. Hence, $3i_1\le \alpha \le 2n+1$. Since $s$ is odd, $n\equiv 2 \, \text{(mod 3)}$, so $3i_1 \le 2n-1$, and $3(n-i_1)\ge n+1$. This yields $\beta=3(n-i_1)$. Similarly, $3(n-i_1)\le \beta \le 2n+1$ yields $3i_1\ge n+1$, and so $\alpha=3i_1$. Using the fact that $n\equiv 2 \,\text{(mod 3)}$, we get
\begin{align*}
S(f_1,f_2)&=x^{3i_1-n}y^{2n-3i_1}(x^{2n}+y^{2n})b^n \\
          &=(x^{2n}+y^{2n})xy\left(x^{3i_1-n-1}y^{2n-3i_1-1}\right)b^n \in (x^{2n}+y^{2n})xyQ^{\frac{n-2}{3}}b^n = T_3.
\end{align*}
This concludes the proof of (S12). Moreover, the last display shows that $T_3 \subseteq T_1+T_2$. Indeed, for $0\le i_3\le \frac{n-2}{3}$, let $i_1=i_3+\frac{n+1}{3}$. Then, $\frac{n+1}{3}\le i_1\le \frac{2n-1}{3}$ and the last display yields
\[
f_3=(x^{2n}+y^{2n})xy\left(x^{3i_3}y^{3n-3i_3}\right)b^n=S(f_1,f_2) \in T_1+T_2.
\] 

\textbf{For (S13)}: Consider
\begin{align*}
S(f_1,f_3)&=S(x^{3i_1}y^{3(n-i_1)},\udl{x^{2n+3i_3+1}y^{n-3i_3-1}}b^n+x^{3i_3+1}y^{3n-3i_3-1}b^n)=x^{\alpha-2n}y^{\beta+2n}b^n
\end{align*}
where
\[
\alpha=\max\{3i_1, 2n+3i_3+1\}, \beta=\max\{3(n-i_1),n-3i_3-1\}.
\]
We may assume that $S(f_1,f_3)\notin  T_1$. Then \Cref{lem_inQn} yields $\alpha+\beta\le 3n+1$. In particular, as $\alpha\ge 2n+3i_3+1$, and
\[
3(n-i_1) \le \beta \le 3n+1-\alpha  \le n-3i_3.
\]
Since $n\equiv 2 \text{(mod 3)}$, we get $3(n-i_1)\le n-3i_3-2$. Thus, $\beta=n-3i_3-1\ge 3(n-i_1)+1$ and $\alpha=3i_1\ge 2n+3i_3+2$. Consequently
\begin{align*}
S(f_1,f_3)&=b^nx^{\alpha-2n}y^{\beta+2n}=b^nx^{3i_1-2n}y^{n-3i_3-1+2n}\\
          &=b^nx^2y^{2n+1}x^{3i_1-2n-2}y^{n-3i_3-2} \in T_4,
\end{align*}
as desired. Moreover, we have $T_4 \subseteq T_1+T_3$ similar to what shown in (S13).

\textbf{For (S23)}: Consider
\begin{align*}
S(f_2,f_3) &=S(\udl{x^ny^na^n}+(x^{2n}+y^{2n})b^n,\udl{x^{2n+3i_3+1}y^{n-3i_3-1}}b^n+x^{3i_3+1}y^{3n-3i_3-1}b^n)\\
          &=f_2\cdot x^{n+3i_3+1}b^n-f_3\cdot y^{3i_3+1}a^n\\
          &=\udl{x^{3i_3+1}y^{3n}a^nb^n}+(x^{2n}+y^{2n})x^{n+3i_3+1}b^{2n} \in \Sigma T_1,
\end{align*}
thanks to \Cref{lem_inQn}(3), as claimed.

\textbf{For (S24)}: We have, for $\alpha=\max\{n, 3i_4+2\}, \beta=\max\{n,3n-3i_4-1\}$, that
\begin{align*}
S(f_2,f_4) &= S(\udl{x^ny^na^n}+(x^{2n}+y^{2n})b^n,x^{3i_4+2}y^{3n-3i_4-1}b^n)\\
           &=(x^{2n}+y^{2n})b^n\cdot x^{\alpha-n}y^{\beta-n}b^n\\
           &=x^{\alpha-n}y^{\beta-n}(x^{2n}+y^{2n})b^{2n}.
\end{align*}
Assume that $S(f_2,f_4)\notin \Sigma T_1$. Then \Cref{lem_inQn} yields $\alpha+\beta\le 3n+1$. This implies that
\begin{align*}
3i_4+2 \le \alpha &\le 2n+1 \Longrightarrow 3i_4 \le 2n-1 \Longrightarrow \beta =\max\{n,3n-3i_4-1\}=3n-3i_4-1,\\
3n-3i_4-1 \le \beta &\le 2n+1 \Longrightarrow n-2\le 3i_4 \Longrightarrow \alpha =\max\{n, 3i_4+2\}=3i_4+2.
\end{align*}
Hence,
\begin{align*}
S(f_2,f_4) &=x^{\alpha-n}y^{\beta-n}(x^{2n}+y^{2n})b^{2n} = x^{3i_4-n+2}y^{2n-3i_4-1}(x^{2n}+y^{2n})b^{2n} \in  T_1,
\end{align*}
thanks to \Cref{lem_inQn}(3), which is a contradiction. Hence, $S(f_2,f_4)\in \Sigma T_1$ as stated.

\textbf{For (S34)}: We have
\begin{align*}
S(f_3,f_4) &= S(\udl{x^{2n+3i_3+1}y^{n-3i_3-1}b^n}+x^{3i_3+1}y^{3n-3i_3-1}b^n, x^{3i_4+2}y^{3n-3i_4-1}b^n)\\
          &=x^{\alpha-2n}y^{\beta+2n}b^n
\end{align*}
where $\alpha=\max\{2n+3i_3+1,3i_4+2\}, \beta=\max \{n-3i_3-1, 3n-3i_4-1\}$. Assume that $S(f_3,f_4)\notin T_1$. Then, \Cref{lem_inQn} yields $\alpha+\beta\le 3n+1$. Since $\alpha\ge 3i_4+2, \beta \ge 3n-3i_4-1$, this forces the equalities to happen, and
\begin{align*}
\alpha=3i_4+2 \ge 2n+3i_3+1, \beta = 3n-3i_4-1.
\end{align*}
Thus, $3i_4 -2n+1\ge 3i_3\ge 0$, and
\begin{align*}
 S(f_3,f_4) &= x^{\alpha-2n}y^{\beta+2n}b^n = x^{3i_4-2n+2}y^{5n-3i_4-1}b^n\\
            &=x\left(x^{3i_4-2n+1}y^{5n-3i_4-1}\right)b^n \in T_1,
\end{align*}
thanks to \Cref{lem_inQn}(3). This concludes the proof of (S34) and, hence, (1).

(2) The candidates for our Gr\"obner basis of $Q^n+(f^n)$ are
 \[
\udb{Q^n}_{T_1} \quad  \cup \quad  \udb{(f^n)}_{T_2} \quad \cup \quad \udb{b^n(x^{2n}+y^{2n})x^2y^2Q^{\frac{n-4}{3}}}_{T_3} \quad  \cup \quad  \udb{b^n\left(x^2y^{2n}Q^{\frac{n-1}{3}}+(x^{2n}y^{n+1})\right)}_{T_4}.
 \]
To prove that the minimal generators of described four types of ideals form a Gr\"obner basis for $Q^n+(f^n)$, it suffices to  prove the following statements (which implies that the $S$-pair condition is satisfied):
\begin{enumerate}
\item[(S0)] For any $1\le i,j\le 4$, we have $S(T_i,T_j)=\{0\}$ if $j=i$ or $(i,j)=(1,4)$.
 \item[(S12)] $S(T_1,T_2) \mathop{\xrightarrow{\qquad \qquad}}\limits_{T_1\cup T_3\cup T_4} 0$. More precisely, $S(T_1,T_2) \subseteq \Sigma T_1 \cup (T_4+T_1) \cup (T_1+T_4) \cup T_3$.
 \item[(S13)] $S(T_1,T_3) \subseteq T_1 \cup T_4$.
 \item[(S23)] $S(T_2,T_3)\subseteq \Sigma T_1$. 
 \item[(S24)] $S(T_2,T_4)\subseteq \Sigma T_1$.
 \item[(S34)] $S(T_3,T_4)\subseteq T_1$.
\end{enumerate}
The proofs of these statements imply that $T_3\subseteq T_1+T_2, T_4\subseteq T_1+T_3$, so the polynomials of types $T_1,T_2, T_3, T_4$ are all in $Q^n+(f^n)$.

We shall omit the proofs of (S0), (S13), (S23), (S24), (S34), since they are similar to their counterparts in the proof of (1). Let us only show (S12), which is not similar to any of the statements in (1).

Take $f_1=x^{3i_1}y^{3(n-i_1)}$ and $f_2=x^ny^na^n+(x^{2n}+y^{2n})b^n$ as above. Again,
\[
S(f_1,f_2)=S(x^{3i_1}y^{3(n-i_1)},\udl{x^ny^na^n}+(x^{2n}+y^{2n})b^n)=(x^{2n}+y^{2n})x^{\alpha-n}y^{\beta-n}b^n
\]
where $\alpha=\max\{3i_1,n\}, \beta=\max\{3(n-i_1),n\}$. We may assume that $S(f_1,f_2)\notin  \Sigma T_1$. Then,
\begin{align*}
3i_1 &\le \alpha \le 2n+1,\\
3(n-i_1) &\le \beta \le 2n+1 \Longrightarrow n-1 \le 3i_1.
\end{align*}
Since $s$ is even, $n\equiv 1$ (mod 3).

\textbf{Case 1:} $i_1=\frac{n-1}{3}$. In this case, $\alpha=n, \beta=2n+1$, so thanks to \Cref{lem_inQn}(3)
\[
S(f_1,f_2)=(x^{2n}+y^{2n})y^{n+1}b^n = \udb{x^{2n}y^{n+1}b^n}_{\,\in\,T_4} + \udb{y^{3n+1}b^n}_{\,\in\,T_1} \in T_4+T_1.
\]

\textbf{Case 2:} $i_1=\frac{2n+1}{3}$. In this case, $\alpha=2n+1, \beta=n$, so
\[
S(f_1,f_2)=(x^{2n}+y^{2n})x^{n+1}b^n=\udb{x^{3n+1}b^n}_{\,\in\,T_1}+\udb{x^2y^{2n}x^{n-1}b^n}_{\,\in\,T_4} \in T_1+T_4.
\]

\textbf{Case 3:} $n-1 <3i_1<2n+1$. In this case, since $n\equiv 1 \text{(mod 3)}$, we get $n+2\le 3i_1\le 2n-2$. Thus, $\alpha=3i_1, \beta=3(n-i_1)$, and thanks to \Cref{lem_inQn}(3)
\[
S(f_1,f_2)=(x^{2n}+y^{2n})x^{3i_1-n}y^{2n-3i_1}b^n=\udb{(x^{2n}+y^{2n})x^2y^2\left(x^{3i_1-n-2}y^{2n-2-3i_1}\right)b^n}_{\,\in\,T_3}.
\]
This concludes the proof of (2) and, hence, \Cref{thm_GB_2powers}.
\end{proof}


\subsection{Complete proof of Theorem \ref{thm_regbound_2powers}.} To move forward, we shall need a few auxiliary results. The next statement will be used repeatedly to bound regularity of ideals with a special splitting property, while the lemma that follows is used to bound the regularity of the initial ideal of $Q^n + (f^n)$, when $n$ is a power of 2.

\begin{lem}[Caviglia et. al. {\cite[Corollary 3.2]{CH+19}}]
\label{lem_regbound_splittableideals}
Let $R=\kk[x_1,\ldots,x_r]$ be a standard graded ring over an arbitrary field $\kk$. Let $J, L$ be homogeneous ideals and $f\in R$ a homogeneous element of degree $\deg(f)=d \ge 1$ such that $f$ is $(R/J)$-regular. Then, there are inequalities
\[
\min\{\reg J, \reg(J+L)+d \} \le \reg(J+fL) \le \max\{\reg J+d-1, \reg (J+L)+d\}.
\]
\end{lem}

\begin{lem}
\label{lem_regbound_2powers_initial}
Let $\kk$ be a field of arbitrary characteristic, and let $n, q\ge 1$ be integers. In the standard graded polynomial ring $T=\kk[x,y,a]$, consider the ideals $Q = (x^3,y^3)$ and $U = Q^n+(x^q y^q a^q)$. Let $V$ be a monomial ideal generated by monomials in $x$ and $y$ only. Then, the following inequalities always hold:
\begin{enumerate}[\quad \rm (1)]
 \item $\reg U \le 3n+q+2$,
 \item $\reg(U+V)\le 3n+q+2$.
 \end{enumerate}
\end{lem}
\begin{proof}

(1) is a special case of (2) when $V=0$, so it suffices to prove the latter. There is nothing to do if $V=(1)$, so we assume that $V$ is a proper ideal of $T$.

Let $W=Q^n+V$. Then, $U+V=W+a^q(x^qy^q)$. Since $a^q$ is regular on $W$, by \Cref{lem_regbound_splittableideals}, we get the first inequality in the chain
\begin{align*}
\reg(U+V)&\le \max\{\reg W+q-1, \reg(W+(x^qy^q))+q\}\\
         &\le \max\{\reg Q^n+q-1, \reg Q^n+q\}=3n+q+2.
\end{align*}
The second inequality holds since $Q^n\subseteq W\subseteq W+(x^qy^q)$ are $(x,y)$-primary monomial ideals. The assertion is proved.
\end{proof}

\begin{proof}[{\bf Proof of \Cref{thm_regbound_2powers}}]
We shall assume that $s$ is odd (the case when $s$ is even can be treated by similar arguments, and is left to the interested reader).

(1) The assertion follows from the description of a Gr\"obner basis of $Q^n+(f^n)$ in part (1) of \Cref{thm_GB_2powers}, and a simple counting argument.

(2) Set $h=x^ny^{2n+1}a^{n-1}b^{n-1}$. It can be easily checked, using part (1), that the monomial $h$ does not belong to the initial ideal of $Q^n+(f^n)$. Thus, $h \not\in Q^n + (f^n)$.

For the containment $h\in (Q^n+(f^n)):\mm$, observe first that, since $(x,y)^{3n+2}\subseteq Q^n$, clearly $h \cdot (x,y)\subseteq Q^n\subseteq Q^n+(f^n)$. Observe next that $ha=x^ny^{2n+1}a^nb^{n-1}$ and $f^n=x^ny^na^n+(x^{2n}+y^{2n})b^n$. Therefore, using \Cref{lem_inQn}(3), we have
\begin{align*}
ha-y^{n+1}b^{n-1}f^n &=(x^{2n}+y^{2n})y^{n+1}b^{2n-1} \in Q^n.
\end{align*}
Hence, $ha\in Q^n+(f^n)$. Similarly, we have
\begin{align*}
-hb+x^nya^{n-1}f^n&=-x^ny^{2n+1}a^{n-1}b^n+(x^ny^na^n+x^{2n}b^n+y^{2n}b^n)x^nya^{n-1}\\
                 &=x^{2n}y^{n+1}a^{2n-1}+x^{3n}ya^{n-1}b^n \in Q^n.
\end{align*}
It follows that $hb\in Q^n+(f^n)$, as desired.

(3) Part (2) implies to give $\reg(R/(Q^n+(f^n))\ge 5n-1$. Thus, $\reg(Q^n+(f^n))\ge 5n$.  For the remaining inequality, since
\begin{align*}
\reg(Q^n+(f^n))&\le \reg \ini\left(Q^n+(f^n)\right)\\
               &  =\reg \left(Q^n+(x^ny^na^n)+b^nx^{2n+1}yQ^\frac{n-2}{3}+b^nx^2y^{2n+1}Q^\frac{n-2}{3}\right),
\end{align*}
we only need to show that the last value is at most $5n+2$.

Indeed, consider the following ideals as in \Cref{lem_regbound_2powers_initial}, with $q=n$ in our current situation,
\begin{align*}
U &= Q^n+(x^qy^qa^q),\\
V &= Q^\frac{n-2}{3}x^2y(x^{2n-1},y^{2n}).
\end{align*}
It is clear that $\ini\left(Q^n+(f^n)\right)=U+b^nV$. As $b^n$ is regular on $U$, \Cref{lem_regbound_splittableideals} yields the first inequality in the following chain
\begin{align*}
 \reg \ini\left(Q^n+(f^n)\right) =\reg (U+b^nV) &\le \max\{\reg U+n-1, \reg(U+V)+n\} \\
                                                &\le \max\{4n+2+n-1,4n+2+n\}=5n+2.
\end{align*}
The second inequality comes from \Cref{lem_regbound_2powers_initial}. Hence,
$$
\reg\ini\left(Q^n+(f^n)\right)\le 5n+2,
$$
as asserted.
\end{proof}


\section{Proof of Theorem \ref{thm_limregge6_3times2power}} \label{sec.2nd}

This section presents the proof of Theorem \ref{thm_limregge6_3times2power} and, thus, completes the proof of Theorem \ref{thm.noLimit} giving negative answers to Questions \ref{quest.HT} and \ref{quest.Intersection}. We, again, keep the notations in Section \ref{sec.nonExist} and Notation \ref{notn_ideals}.


\subsection{Gr\"obner basis and Theorem \ref{thm_limregge6_3times2power}} As for Theorem \ref{thm_regbound_2powers} in the previous section, the core technical difficulty in establishing Theorem \ref{thm_limregge6_3times2power} is to obtain an explicit description of a Gr\"obner basis for $Q^n+(f^n)$, when $n=3\cdot 2^s, s\ge 1$. 

\begin{notn}
\label{notn_ideals_3times2power}
Throughout this section, we shall employ the following notations:
\begin{align*}
F_3 &=(\udl{x^3y}+xy^3)\udl{a^2}+(x^4+y^4)ab+(x^3y+xy^3)b^2, & F_{10} & = \udl{y^6a^2}+x^3y^3ab+(x^6+y^6)b^2,\\
F_5 & = \udl{y^3a^2}+x^3ab+y^3b^2, & F_{12} & = \udl{y^3a}+x^3b,\\
F_7 &= \udl{xy^5a^3}+y^6a^2b+x^3y^3ab^2+(x^6+x^4y^2+y^6)b^3, & F_{13} & = \udl{x^3a}+y^3b.\\
F_9 &= \udl{y^6a^2}+x^3y^3ab+x^6b^2,
\end{align*}
The indices of these polynomials were deliberately chosen; they match the type indices of the ideals in certain Gr\"obner basis that is described in \Cref{prop_GB_3times2power} below, and this matching should save our mental resources when working with them.
Note also that $F_3= (x^2+y^2)(xa+yb)(ya+xb)$.
\end{notn}

\begin{prop}
\label{prop_GB_3times2power}
Consider $n=3\cdot 2^s$, where $s \ge 3$ is an integer, and set $t=n/3=2^s$.  Assume that $s$ is odd. Then, $Q^n+(f^n)$ has a Gr\"obner basis  consisting of the natural generators of the following ideals \textup{(}15 types of ideals in total\textup{)}:
\begin{align*}
& \udb{Q^{3t}}_{T_1},\qquad \udb{\left(f^{3t}\right)}_{T_2}, \qquad \udb{x^ty^tF_3^tQ^tb^t}_{T_3}, \qquad \udb{(x^{6t}+y^{6t})(xy)^tQ^{\frac{t+1}{3}}a^tb^{2t}}_{T_4}, \qquad \udb{x^{2t+1}y^{2t+4}F_5^tQ^{\frac{2t-4}{3}}b^t}_{T_5}, \\
&  \udb{xy(x^{8t}+y^{8t})Q^{\frac{t-2}{3}}b^{4t}}_{T_6}, \qquad \udb{x^ty^{t+2}F_7^tQ^\frac{t-2}{3}b^t}_{T_7}, \qquad \udb{x^2y^{8t+1}Q^\frac{t-2}{3}b^{4t}}_{T_8}, \qquad \udb{\left(x^ty^{2t+1}F_9^tb^{2t}\right)}_{T_9},\\
&   \udb{x^{t+3}y^{t+3}F_{10}^tQ^\frac{t-5}{3}b^{2t}}_{T_{10}},\qquad \udb{\left(x^{2t+1}y^{7t}a^tb^{4t}\right)}_{T_{11}},\qquad \udb{\left(x^{4t}y^{2t+1}F_{12}^tb^{5t}\right)}_{T_{12}},\qquad  \udb{\left(x^{2t+1}y^{4t}F_{13}^tb^{5t}\right)}_{T_{13}}, \\
&  \udb{x^{4t}y^{4t}(x^{t+1},y^{t+1})b^{7t}}_{T_{14}}, \qquad \udb{x^{2t+1}y^{2t+1}(x^{5t-1},y^{5t-1})b^{8t}}_{T_{15}}.
\end{align*}
This Gr\"obner basis consists of $\dfrac{19n+111}{9}$ elements whose highest possible degree is $\dfrac{17}{3}n+1$.
\end{prop}

We shall leave the proof of \Cref{prop_GB_3times2power} until later and proceed to establishing \Cref{thm_limregge6_3times2power} assuming \Cref{prop_GB_3times2power}.

\begin{proof}[{\bf Proof of \Cref{thm_limregge6_3times2power}}]
Thanks to \Cref{prop_GB_3times2power}, the initial ideal of $Q^n+(f^n)$ is the sum of the initial ideals of 15 types of ideals, that are given below.
\begin{alignat*}{3}
&\ini T_1: && \quad Q^n,   &&\ini T_9: \quad (x^ty^{2t+1}(y^6a^2)^tb^{2t})=(x^ty^{8t+1}a^{2t}b^{2t}),\\
&\ini T_2: && \quad (x^ny^na^n), &&\ini T_{10}: \quad x^{t+3}y^{t+3}(y^6a^2)^tQ^\frac{t-5}{3}b^{2t} \\
&          &&                    &&  \qquad  \qquad =  x^{t+3}y^{7t+3}Q^\frac{t-5}{3}a^{2t}b^{2t},\\
&\ini T_3: && \quad  x^ty^t(x^3ya^2)^tQ^tb^t= x^{4t}y^{2t}Q^ta^{2t}b^t,  \quad &&\ini T_{11}: \quad (x^{2t+1}y^{7t}a^tb^{4t}),\\
&\ini T_4: && \quad x^{7t}y^tQ^\frac{t+1}{3}a^tb^{2t}, &&\ini T_{12}: \quad (x^{4t}y^{2t+1}(y^3a)^tb^{5t})=(x^{4t}y^{5t+1}a^tb^{5t}),\\
&\ini T_5: && \quad x^{2t+1}y^{2t+4}(y^3a^2)^tQ^\frac{2t-4}{3}b^t  &&\ini T_{13}: \quad (x^{2t+1}y^{4t}(x^3a)^tb^{5t}) =(x^{5t+1}y^{4t}a^tb^{5t}), \\
 &         && \quad = x^{2t+1}y^{5t+4}Q^\frac{2t-4}{3}a^{2t}b^t,         && \\
&\ini T_6: && \quad x^{8t+1}yQ^\frac{t-2}{3}b^{4t},  &&\ini T_{14}: \quad x^{4t}y^{4t}(x^{t+1},y^{t+1})b^{7t},\\
&\ini T_7: && \quad x^ty^{t+2}(xy^5a^3)^tQ^\frac{t-2}{3}b^t \quad  &&\ini T_{15}: \quad x^{2t+1}y^{2t+1}(x^{5t-1},y^{5t-1})b^{8t}.\\
&       && \quad = x^{2t}y^{6t+2}Q^\frac{t-2}{3}a^{3t}b^t,     && \\
&\ini T_8: && \quad x^2y^{8t+1}Q^\frac{t-2}{3}b^{4t}, &&
\end{alignat*}

Let $h=x^{2t+1}y^{7t}a^{t-1}b^{8t-1}$. If $h\in Q^n+(f^n)$, then $h$ belongs to its initial ideal, and so $h$ belongs to one of the 15 types of monomial ideals listed above. On the other hand, it can be seen that:
\begin{align*}
h &\notin Q^n \, \text{(the initial ideal of $T_1$)}, \text{ as $x^{2t+1}y^{7t}\notin Q^{3t}$ thanks to \Cref{lem_inQn}(3)}\\
  &  \qquad \qquad \qquad   \text{and the fact that $2t+1\equiv 7t \equiv 2$ (mod 3)},\\
h &\notin \, \text{the initial ideals of $T_2$--$T_5$, $T_7$, $T_9$--$T_{13}$ (by considering degrees of $a$)},\\
h &\notin \, \text{the initial ideals of $T_6$, $T_{14}$ (by considering degrees of $x$)},\\
h &\notin \, \text{the initial ideal of $T_8$ (by considering degrees of $y$)},\\
h &\notin \, \text{the initial ideal of $T_{15}$ (by considering degrees of $b$)}.
\end{align*}
This contradiction shows that $h\notin Q^n+(f^n)$.

For the containment $h \in (Q^n+(f^n)):\mm$, note that
\begin{align*}
h\cdot (x,y) &\subseteq (x,y)^{3n+2} \subseteq Q^n,\\
ha &= x^{2t+1}y^{7t}a^tb^{8t-1} \in T_{11},\\
hb &= x^{2t+1}y^{7t}a^{t-1}b^{8t}\in T_{15}.
\end{align*}
This proves the first statement of the theorem. The remaining statement follows from the first one and simple accounting.
\end{proof}



\begin{center}
	\begin{table}
		\caption{How the $S$-pairs in the Gr\"obner basis of $Q^n+(f^n)$, for $n=3\cdot 2^s$, reduce to zero.}
		\begin{tabular}{  | c | c | c | c |}
			\hline
			\multirow{4}{*}{No.} &                     & \multirow{4}{*}{$(i,j)$, where $1\le i<j\le 15$} & $\Lambda \subseteq [15]$ for which  \\
			&  The number          &                          & $S(T_i,T_j)$ reduces to   \\
			&    of pairs $(i,j)$          &                          & zero  via $\mathop{\bigcup}\limits_{\ell \in \Lambda} T_\ell$ \\
			\hline
			1 &   1                    & (1,2)                    & \{1,3\} \\
			\hline
			2 & 1                      & (1,3)                    & \{1,3,4,5\}\\
			\hline
			\multirow{11}{*}{3} &   & (1,4), (1,5), (1,9), (1,10), (1,12), (1,13), (2,4) &   \multirow{11}{*}{\{1\}} \\
			& & (2,6), (2,7), (2,8), (2,9), (2,10), (2,11),  (2,15), &   \\
			& & (3,5), (3,6), (3,7), (3,8), (3,9), (3,10), (3,11),  & \\
			& & (4,5), (4,7), (4,8), (4,9), (4,10), (4,11), (4,12),  &  \\
			& &  (4,13), (4,14), (5,6), (5,8), (5,9), (5,10), (5,12), & \\
			& 71 & (5,13), (5,14), (6,7), (6,8), (6,9), (6,10), (6,11), & \\
			& & (6,12), (6,13), (6,14), (6,15), (7,8), (7,9), (7,10),   & \\
			& &  (7,12), (7,13), (7,14), (8,10), (8,12), (8,13),   & \\
			&  & (9,10), (9,11), (9,12), (9,13), (9,14), (9,15),   & \\
			&  & (10,11), (10,12), (10,13), (10,14), (10,15),   & \\
			&  & (11,12), (11,13), (12,13), (12,15), (13,15)  & \\
			\hline
			4 & 1& (1,6)                     & \{1,8\} \\
			\hline
			5 & 1& (1,7)                     & \{1,8,9,10\}\\
			\hline
			\multirow{2}{*}{6} & \multirow{2}{*}{10}  & (1,8), (1,11), (1,14), (1,15), (8,11)                 & \multirow{2}{*}{$\emptyset$}\\
			&   & (8,14), (8,15), (11,14), (11,15), (14,15)             &                           \\
			\hline
			7 & 1& (2,3)                     & \{1,2,6,7\} \\
			\hline
			8 & 1& (2,5)                       & \{1,3,8\} \\
			\hline
			9 & 1& (2,12)                       & \{1,3,8,14\} \\
			\hline
			10 &  1& (2,13)                       & \{1,5,6,14\} \\
			\hline
			11 & 1 & (2,14)                       & \{1,6,8,15\} \\
			\hline
			12 & 1& (3,4)                       & \{1,9\} \\
			\hline
			13 & 1& (3,12)                       & \{1,4,14\} \\
			\hline
			14 & 1& (3,13)                       & \{1,14\} \\
			\hline
			15 & 1& (3,14)                       & \{1,8,11,14\} \\
			\hline
			16 & 1& (3,15)                       & \{1,12\} \\
			\hline
			17 & 1& (4,6)                       & \{1,11\} \\
			\hline
			18 & 1& (4,15)                       & \{8\} \\
			\hline
			19 & 1& (5,7)                       & \{1,3\} \\
			\hline
			20 & 1& (5,11)                       & \{13\} \\
			\hline
			21 & 1& (5,15)                       & \{1,13\} \\
			\hline
			22 & 1& (7,11)                       & \{1,14\} \\
			\hline
			23 & 1& (7,15)                       & \{1,14\} \\
			\hline
			24 & 1& (8,9)                       & \{12\} \\
			\hline
			25 & 1& (12,14)                       & \{15\} \\
			\hline
			26 & 1& (13,14)                       & \{15\} \\
			\hline
			&  105 pairs      & & \\
			\hline
		\end{tabular}
		\label{tab_Spairs_3times2power}
	\end{table}
\end{center}

\subsection{Gr\"obner basis analysis and \Cref{prop_GB_3times2power}.}
\label{subsec.GB3times2power} 
It remains to establish \Cref{prop_GB_3times2power}. Again, we will use \Cref{notn_GB} on $S$-pairs. We shall check that the $S$-pairs of the type $S(T_i,T_j)$, where $1\le i\le j\le 15$, reduce to zero, following the patterns depicted in \Cref{tab_Spairs_3times2power} and \Cref{rem_3times2power_S0}. Even more concretely, these $S$-pairs satisfy the following containments. Recall from \Cref{notn_GB} that the notation $\Sigma T_1$ denotes the set of finite sums of monomials in $T_1$; in particular, when viewed as subsets of $R$, $\Sigma T_1=T_1$.
\begin{enumerate}
 \item[(S12)] $S(T_1,T_2) \subseteq \Sigma T_1 \cup (T_3+\Sigma T_1)$,
 \item[(S13)] $S(T_1,T_3) \subseteq \Sigma T_1 \cup (T_5+\Sigma T_1) \cup (T_3+T_1+T_4+\Sigma T_1)$ (3 cases),
 \item[(S16)] $S(T_1,T_6) \subseteq T_1\cup T_8$,
 \item[(S17)] $S(T_1,T_7) \subseteq \Sigma T_1\cup (T_9+T_1+T_8) \cup (T_{10}+T_1),$
 \item[(S23)] $S(T_2,T_3) \subseteq \Sigma T_1 \cup (T_2+\Sigma T_1+T_6) \cup (T_7+\Sigma T_1)$ (3 cases),
 \item[(S25)] $S(T_2,T_5) \subseteq \Sigma T_1 \cup (T_3+\Sigma T_1+T_8)$ (2 cases),
 \item[(S212)] $S(T_2,T_{12}) \subseteq T_3+\Sigma T_1+T_{14}+\Sigma T_1+T_8$,
 \item[(S213)] $S(T_2,T_{13}) \subseteq \Sigma T_1+T_5+T_1+T_{14}+T_1+T_6+\Sigma T_1$,
  \item[(S214)] $S(T_2,T_{14}) \subseteq (\Sigma T_1+T_{15}+\Sigma T_1+T_8) \cup (\Sigma T_1+T_6+\Sigma T_1 +T_{15}+T_1)$,
 \item[(S34)] $S(T_3,T_4) \subseteq \Sigma T_1 \cup (T_1+T_9+\Sigma T_1)$ (2 cases),
  \item[(S312)] $S(T_3,T_{12}) \subseteq \Sigma T_1 \cup (T_1+ T_4+ T_1+T_{14}+T_1)$,
\item[(S313)] $S(T_3,T_{13}) \subseteq \Sigma T_1 \cup (\Sigma T_1+T_{14}+T_1)$,
  \item[(S314)] $S(T_3,T_{14}) \subseteq \Sigma T_1 \cup (\Sigma T_1+T_8+T_{14}+T_1)  \cup (\Sigma T_1+T_{11}+T_{14}+T_1)$ (3 cases),
  \item[(S315)] $S(T_3,T_{15}) \subseteq  \Sigma T_1 \cup (\Sigma T_1+T_{12}+T_1)$,
 \item[(S46)] $S(T_4,T_6) \subseteq T_{11}+T_1$,
  \item[(S415)] $S(T_4,T_{15}) \subseteq T_8$,
 \item[(S57)] $S(T_5,T_7) \subseteq (T_3+\Sigma T_1) \cup \Sigma T_1$,
 \item[(S511)] $S(T_5,T_{11}) \subseteq T_{13}$,
  \item[(S515)] $S(T_5,T_{15}) \subseteq T_{13} \cup \Sigma T_1$,
  \item[(S711)] $S(T_7,T_{11}) \subseteq (\Sigma T_1+T_{14}+T_1) \cup \Sigma T_1$,
 \item[(S715)] $S(T_7,T_{15}) \subseteq (\Sigma T_1+T_{14}+T_1) \cup \Sigma T_1$,
  \item[(S89)] $S(T_8,T_9) \subseteq T_{12}$,
   \item[(S1214)] $S(T_{12},T_{14}) \subseteq T_{15}$,
    \item[(S1314)] $S(T_{13},T_{14}) \subseteq T_{15}$,
  \item[(S0)] $S(T_i,T_j)=\{0\}$ if $i=j$ or ($i< j$ and $i,j\in \{1, 8, 11, 14, 15\}$) (25 pairs),
 \item[(S$T_1$)] $S(T_i,T_j)\subseteq \Sigma T_1$ if $(i,j) \in \{(1,4), (1,5), \ldots, (13, 15)\}$ (totally 71 pairs).
\end{enumerate}

We will treat the first 24 statements (S12), (S13), $\ldots$, (S1314) via Lemmas \ref{lem_3times2power_S12-17}, \ref{lem_3times2power_S23}, \ref{lem_3times2power_S25}, \ref{lem_3times2power_S212-14}, \ref{lem_3times2power_S34-15}, \ref{lem_3times2power_S46-1314} below. The statement (S0) holds because of \Cref{rem_3times2power_S0} below. The statement (S$T_1$) will be treated in \Cref{lem_3times2power_ST1}.

\begin{center}
	\begin{table}
		\caption{The natural generators of 15 types of ideals featuring in the Gr\"obner basis of $Q^n+(f^n)$, where $n=3\cdot 2^s$, $s$ is odd and $\ge 3$.}
		\begin{tabular}{| c | c | c |}
			\hline
			\multirow{2}{*}{$i$} & Degree of generators & \multirow{2}{*}{Form of generators of $T_i$} \\
			&  of $T_i$            &    \\
			\hline
			1 & $3n$     & $f_1=x^{3i_1}y^{3n-3i_1}, 0\le i_1\le n$\\
			\hline
			\multirow{2}{*}{2} & \multirow{2}{*}{$3n$}     & $f_2=x^{3t}y^{3t}a^{3t}+x^{4t}y^{2t}a^{2t}b^t+x^{2t}y^{4t}a^{2t}b^t+x^{5t}y^ta^tb^{2t} +$\\
			&          & \qquad  $+x^ty^{5t}a^tb^{2t}+x^{6t}b^{3t}+x^{4t}y^{2t}b^{3t}+x^{2t}y^{4t}b^{3t}+y^{6t}b^{3t}$ \\
			\hline
			\multirow{5}{*}{3} & \multirow{5}{*}{$4n$}   & $f_3=(x^{3t}y^ta^{2t}+x^ty^{3t}a^{2t}+x^{4t}a^tb^t+y^{4t}a^tb^t+x^{3t}y^tb^{2t}+$\\
			&          &  \qquad  $+x^ty^{3t}b^{2t})x^{3i_3+t}y^{4t-3i_3}b^t$\\
			&          & \qquad = $x^{3i_3+4t}y^{5t-3i_3}a^{2t}b^t + x^{3i_3+2t}y^{7t-3i_3}a^{2t}b^t +$\\
			&          & \qquad $+ x^{3i_3+5t}y^{4t-3i_3}a^t b^{2t} + x^{3i_3+t}y^{8t-3i_3}a^tb^{2t} +$  \\
			&          &    $+ x^{3i_3+4t}y^{5t-3i_3}b^{3t} +x^{3i_3+2t}y^{7t-3i_3}b^{3t}, 0\le i_3\le t$  \\
			\hline
			\multirow{2}{*}{4} & \multirow{2}{*}{$4n+1$}& $f_4=(x^{6t}+y^{6t})x^{3i_4+t}y^{2t+1-3i_4}a^tb^{2t}$\\
			  &        &    $\,=x^{3i_4+7t}y^{2t+1-3i_4}a^tb^{2t}+ x^{3i_4+t}y^{8t+1-3i_4}a^tb^{2t}$, $0\le i_4\le \frac{t+1}{3}$ \\
			\hline
			\multirow{3}{*}{5} & \multirow{3}{*}{$4n+1$}  & $f_5=(y^{3t}a^{2t}+x^{3t}a^tb^t+y^{3t}b^{2t})x^{3i_5+2t+1}y^{4t-3i_5}b^t$ \\
			&          &  $=x^{3i_5+2t+1}y^{7t-3i_5}a^{2t}b^t+ x^{3i_5+5t+1}y^{4t-3i_5}a^tb^{2t}+$ \\
			&          &  \qquad $ + x^{3i_5+2t+1}y^{7t-3i_5}b^{3t}, 0\le i_5 \le \frac{2t-4}{3}$\\
			\hline
			\multirow{2}{*}{6} & \multirow{2}{*}{$4n+t$}    & $f_6=(x^{8t}+y^{8t})x^{3i_6+1}y^{t-1-3i_6}b^{4t}$\\
			&         & $= x^{3i_6+8t+1}y^{t-1-3i_6}b^{4t}+x^{3i_6+1}y^{9t-1-3i_6}b^{4t}, 0\le i_6 \le \frac{t-2}{3}$ \\
			\hline
			\multirow{5}{*}{7} &  \multirow{5}{*}{$4n+t$}  & $f_7=(x^ty^{5t}a^{3t}+y^{6t}a^{2t}b^t+x^{3t}y^{3t}a^tb^{2t}+x^{6t}b^{3t}+ $\\
			&         & \qquad $+x^{4t}y^{2t}b^{3t}+y^{6t}b^{3t})x^{t+3i_7}y^{2t-3i_7}b^t$,      \\
			&         &  \qquad =$x^{2t+3i_7}y^{7t-3i_7}a^{3t}b^t+x^{t+3i_7}y^{8t-3i_7}a^{2t}b^{2t}+$   \\
			&         & \qquad $+x^{4t+3i_7}y^{5t-3i_7}a^tb^{3t}+x^{7t+3i_7}y^{2t-3i_7}b^{4t}$ \\
			&         &  \qquad $+x^{5t+3i_7}y^{4t-3i_7}b^{4t}+x^{t+3i_7}y^{8t-3i_7}b^{4t},$ $0\le i_7\le \frac{t-2}{3}$ \\
			\hline
			8 & $4n+t+1$  & $f_8=x^{3i_8+2}y^{9t-3i_8-1}b^{4t}, 0\le i_8 \le \frac{t-2}{3}$\\
			\hline
			\multirow{2}{*}{9} & \multirow{2}{*}{$4n+t+1$}  & $f_9=(y^{6t}a^{2t}+x^{3t}y^{3t}a^tb^t+x^{6t}b^{2t})x^ty^{2t+1}b^{2t}$\\
			&           &   \qquad = $x^ty^{8t+1}a^{2t}b^{2t}+x^{4t}y^{5t+1}a^tb^{3t}+x^{7t}y^{2t+1}b^{4t}$ \\
			\hline
			\multirow{5}{*}{10} & \multirow{5}{*}{$4n+t+1$}  & $f_{10}=(y^{6t}a^{2t}+x^{3t}y^{3t}a^tb^t+x^{6t}b^{2t}+$\\
			&            &      \qquad \qquad $+ y^{6t}b^{2t})x^{t+3i_{10}+3}y^{2t-2-3i_{10}}b^{2t}$\\
			&           &   = $x^{t+3i_{10}+3}y^{8t-2-3i_{10}}a^{2t}b^{2t} + x^{4t+3i_{10}+3}y^{5t-2-3i_{10}}a^tb^{3t} + $  \\
			&           &   $+x^{7t+3i_{10}+3}y^{2t-2-3i_{10}}b^{4t} + x^{t+3i_{10}+3}y^{8t-2-3i_{10}}b^{4t}$,\\
			&           &   \qquad $0\le i_{10} \le \frac{t-5}{3} $ \\
			\hline
			11 & $4n+2t+1$  & $f_{11}=x^{2t+1}y^{7t}a^tb^{4t}$\\
			\hline
			12 & $5n+1$  & $f_{12}=x^{4t}y^{5t+1}a^tb^{5t}+x^{7t}y^{2t+1}b^{6t}$\\
			\hline
			13 & $5n+1$  & $f_{13}=x^{5t+1}y^{4t}a^tb^{5t}+x^{2t+1}y^{7t}b^{6t}$\\
			\hline
			14 & $5n+t+1$  & $f_{14}=x^{4t+i_{14}}y^{5t+1-i_{14}}b^{7t}$, $i_{14}\in \{0,t+1\}$\\
			\hline
			15 & $5n+2t+1$  & $f_{15}=x^{2t+1+i_{15}}y^{7t-i_{15}}b^{8t}, i_{15} \in \{0, 5t-1\}$\\
			\hline
		\end{tabular}
		\label{tab_mingens_GB_3times2power}
	\end{table}
\end{center}

\begin{rem}
\label{rem_3times2power_S0}
For $1\le i\le 15$, $S(T_i,T_i)=\{0\}$ since any two minimal generators of $T_i$ are of the form $gh_1, gh_2$, where $g$ is a polynomial, $h_1,h_2$ are \emph{monomials} in $R$, and clearly $S(gh_1,gh_2)=0$. Similarly, $S(T_i,T_j)=\{0\}$ for $i<j$ in $\{1, 8, 11, 14, 15\}$, since $T_i$ is a monomial ideal for each $i$ in the last set. Thus, the statement (S0) holds.
\end{rem}

\begin{lem}
\label{lem_3times2power_S12-17}
The following statements hold:
\begin{enumerate}
 \item[\textup{(S12)}] $S(T_1,T_2)\subseteq \Sigma T_1 \cup (T_3+\Sigma T_1)$ and $T_3 \subseteq T_1+T_2$.
\item[\textup{(S13)}] $S(T_1,T_3) \subseteq \Sigma T_1 \cup (T_5+\Sigma T_1) \cup (T_3+T_1+T_4+\Sigma T_1)$ \textup{(3 cases)}.
Moreover, $T_4 \subseteq T_1+T_3$ and $T_5\subseteq T_1+T_3$.
 \item[\textup{(S16)}] $S(T_1,T_6) \subseteq T_1\cup T_8$ and  $T_8 \subseteq T_1+T_6$.
 \item[\textup{(S17)}] $S(T_1,T_7) \subseteq \Sigma T_1 \cup (T_9+T_1+T_8) \cup (T_{10}+T_1)$  and $T_{10} \subseteq T_1+T_7$.
\end{enumerate}
\end{lem}

\begin{proof}
We will use the form for minimal generators of ideals $T_i, 1\le i\le 8$, as recorded by \Cref{tab_mingens_GB_3times2power}. Note that since $t$ is a power of 2 and $\chara \kk=2$,
\begin{align*}
f_2&=f^t\cdot f^{2t}=(x^ty^ta^t+x^{2t}b^t+y^{2t}b^t)(x^{2t}y^{2t}a^{2t}+x^{4t}b^{2t}+y^{4t}b^{2t})\\
   &=x^{3t}y^{3t}a^{3t}+x^{4t}y^{2t}a^{2t}b^t+x^{2t}y^{4t}a^{2t}b^t+x^{5t}y^ta^tb^{2t} +\\
    &  \qquad  +x^ty^{5t}a^tb^{2t}+x^{6t}b^{3t}+x^{4t}y^{2t}b^{3t}+x^{2t}y^{4t}b^{3t}+y^{6t}b^{3t}.
\end{align*}

\textbf{(S12)}: Letting $\alpha = \max\{3i_1, n\},\beta  = \max\{3(n-i_1), n\}$, and keeping the common highest term $x^\alpha y^\beta a^{3t}$ in mind, we get
\begin{align*}
S(f_1,f_2) &= S(x^{3i_1}y^{3n-3i_1}, \udl{x^{3t}y^{3t}a^{3t}}+x^{4t}y^{2t}a^{2t}b^t+x^{2t}y^{4t}a^{2t}b^t+x^{5t}y^ta^tb^{2t} + \\
           & \qquad \qquad + x^ty^{5t}a^tb^{2t}+x^{6t}b^{3t}+x^{4t}y^{2t}b^{3t}+x^{2t}y^{4t}b^{3t}+y^{6t}b^{3t})\\
           & = f_1\cdot x^{\alpha-3i_1}y^{\beta-(3n-3i_1)}a^{3t}-f_2\cdot x^{\alpha-3t}y^{\beta-3t}\\
           & = \udl{x^{\alpha+t} y^{\beta-t}a^{2t}b^t} + x^{\alpha-t}y^{\beta+t}a^{2t}b^t+  x^{\alpha+2t}y^{\beta-2t}a^tb^{2t}+ x^{\alpha-2t}y^{\beta+2t}a^tb^{2t}+ \\
           & \qquad  + x^{\alpha+3t}y^{\beta-3t}b^{3t}+ x^{\alpha+t}y^{\beta-t}b^{3t}+ x^{\alpha-t}y^{\beta+t}b^{3t}+ x^{\alpha-3t}y^{\beta+3t}b^{3t}.
\end{align*}
Assume that $S(f_1,f_2)\notin \Sigma T_1$. Then, $\alpha+\beta \le 3n+1$. This yields $n \le \alpha \le 3i_1+1 \Longrightarrow n\le 3i_1, \alpha=3i_1$, as $n\equiv 0$ (mod 3). Similarly, $n \le \beta \le 3(n-i_1)+1 \Longrightarrow n\le 3(n-i_1), \beta=3(n-i_1)$.

Since $n=3t$, we get $t\le i_1 \le n-t=2t$. From the last expression of $S(f_1,f_2)$, it follows that
\begin{align*}
S(f_1,f_2) &= (x^{3t}y^ta^{2t}+x^ty^{3t}a^{2t}+x^{4t}a^tb^t+y^{4t}a^tb^t+x^{3t}y^tb^{2t}+\\
            &  \qquad  +x^ty^{3t}b^{2t})x^{\alpha-2t}y^{\beta-2t}b^t+  x^{\alpha+3t}y^{\beta-3t}b^{3t}+ x^{\alpha-3t}y^{\beta+3t}b^{3t}\\
            &= (xy)^tF_3^tx^{\alpha-3t}y^{\beta-3t}b^t+x^{\alpha+3t}y^{\beta-3t}b^{3t}+ x^{\alpha-3t}y^{\beta+3t}b^{3t}.
\end{align*}
Thus,
$$
S(f_1,f_2)= \udb{(xy)^tF_3^tx^{3(i_1-t)}y^{3(n-i_1-t)}b^t}_{\in T_3}+ \udb{x^{3(i_1+t)}y^{3(n-i_1-t)}\cdot b^{3t}+ x^{3(i_1-t)}y^{3(n-i_1+t)} \cdot b^{3t}}_{\,\in\, \Sigma T_1}
$$
as asserted.

It remains to show that $T_3 \subseteq T_1+T_2$. Indeed, for $0\le i_3\le t$, let $i_1=t+i_3$. Then, $t\le i_1\le 2t$. The last display yields
\[
f_3=F_3^t(xy)^tb^tx^{3i_3}y^{3(t-i_3)}=S(f_1,f_2)-\udb{\left(\text{other terms}\right)}_{\,\in\, \Sigma T_1} \in T_1+T_2,
\]
as wanted.

\textbf{(S13)}: Letting $\alpha = \max\{3i_1, 3i_3+4t\}, \beta  = \max\{3(n-i_1), 5t-3i_3\}$ and keeping the common highest term $x^\alpha y^\beta a^{2t}b^t$ in mind, we get
\begin{align*}
S(f_1,f_3) &= S(x^{3i_1}y^{3n-3i_1}, \udl{x^{3i_3+4t}y^{5t-3i_3}a^{2t}b^t}+x^{3i_3+2t}y^{7t-3i_3}a^{2t}b^t+x^{3i_3+5t}y^{4t-3i_3}a^tb^{2t}+\\
           & \qquad \qquad +x^{3i_3+t}y^{8t-3i_3}a^tb^{2t}+x^{3i_3+4t}y^{5t-3i_3}b^{3t}+x^{3i_3+2t}y^{7t-3i_3}b^{3t})\\
           & = f_1\cdot x^{\alpha-3i_1}y^{\beta-(3n-3i_1)}a^{2t}b^t-f_3\cdot x^{\alpha-(3i_3+4t)}y^{\beta-(5t-3i_3)}\\
           & = \udl{x^{\alpha-2t}y^{\beta+2t}a^{2t}b^t} + x^{\alpha+t}y^{\beta-t}a^tb^{2t} +x^{\alpha-3t}y^{\beta+3t}a^tb^{2t}+ x^{\alpha}y^{\beta}b^{3t}+x^{\alpha-2t}y^{\beta+2t}b^{3t}.
\end{align*}
Assume that $S(f_1,f_3)\notin \Sigma T_1$. Then, $\alpha+\beta\le 3n+1$. This, together with $t\equiv 2$ (mod 3), yields
\begin{align*}
3i_3+4t &\le \alpha \le 3i_1+ 1\Longrightarrow 3i_3+4t\le 3i_1-1, \alpha= 3i_1, \beta=5t-3i_3.
\end{align*}
Hence,
\begin{align*}
S(f_1,f_3)&=\udl{x^{3i_1-2t}y^{7t-3i_3}a^{2t}b^t} + x^{3i_1+t}y^{4t-3i_3}a^tb^{2t} +x^{3i_1-3t}y^{8t-3i_3}a^tb^{2t}+ x^{3i_1}y^{5t-3i_3}b^{3t}+\\
          &\qquad \qquad +x^{3i_1-2t}y^{7t-3i_3}b^{3t}.
\end{align*}

\textbf{Case 1:} $i_3\le \frac{2t-4}{3}$. Then,
\begin{align*}
S(f_1,f_3)&=(y^{3t}a^{2t}+x^{3t}a^tb^t+y^{3t}b^{2t})x^{3i_1-2t}y^{4t-3i_3}b^t+ x^{3i_1-3t}y^{8t-3i_3}a^tb^{2t}+x^{3i_1}y^{5t-3i_3}b^{3t}\\
          &= \udb{F_5^tx^{2t+1}y^{2t+4}b^t\cdot x^{3i_1-4t-1}y^{2t-4-3i_3}}_{\in T_5}+\udb{x^{3i_1-3t}y^{8t-3i_3}\cdot a^tb^{2t}+x^{3i_1}y^{5t-3i_3}\cdot b^{3t}}_{\,\in\, \Sigma T_1}.
\end{align*}
Here, the containments hold since $3i_1\ge 3i_3+4t+1$, $i_3 \le \frac{2t-4}{3}$ and $t\equiv 2$ (mod 3). Hence, $S(f_1,f_3)$ belongs to $T_5+\Sigma T_1$.

\textbf{Case 2:} $i_3\ge \frac{2t-1}{3}$. As $3i_1\ge 3i_3+4t+1 \ge 6t$, we get $i_1\ge 2t$. Also,
\begin{align*}
S(f_1,f_3) & = \udb{(\udl{x^{3t}y^ta^{2t}}+x^ty^{3t}a^{2t}+x^{4t}a^tb^t+y^{4t}a^tb^t+x^{3t}y^tb^{2t} +x^ty^{3t}b^{2t})x^{3i_1-5t}y^{6t-3i_3}b^t}_{\text{first summand} \,\in\, T_3}+\\
           &\quad +  x^{3i_1-4t}y^{9t-3i_3}a^{2t}b^t+x^{3i_1-t}y^{6t-3i_3}a^tb^{2t}+x^{3i_1-5t}y^{10t-3i_3}a^tb^{2t} + x^{3i_1-4t}y^{9t-3i_3}b^{3t} + \\
           & \qquad \qquad + \left(x^{3i_1+t}y^{4t-3i_3}a^tb^{2t}+x^{3i_1-3t}y^{8t-3i_3}a^tb^{2t}+x^{3i_1}y^{5t-3i_3}b^{3t} \right).
\end{align*}
Here, the containment holds since the first summand equals
\[
x^ty^tF_3^tb^tx^{3i_1-6t}y^{5t-3i_3} \in x^ty^tF_3^tb^tQ^t =T_3,
\]
noting that $(3i_1-6t)+(5t-3i_3)=3i_1-3i_3-t\ge 3t+1$. Subtracting the first summand from $S(f_1,f_3)$, it remains to show that
\begin{align*}
S'(f_1,f_3) &= \udl{x^{3i_1-4t}y^{9t-3i_3}a^{2t}b^t}+x^{3i_1-t}y^{6t-3i_3}a^tb^{2t}+x^{3i_1-5t}y^{10t-3i_3}a^tb^{2t} + x^{3i_1-4t}y^{9t-3i_3}b^{3t} + \\
           & \qquad \qquad + \left(x^{3i_1+t}y^{4t-3i_3}a^tb^{2t}+x^{3i_1-3t}y^{8t-3i_3}a^tb^{2t}+x^{3i_1}y^{5t-3i_3}b^{3t} \right)\\
           &= \udl{x^{3i_1-4t}y^{9t-3i_3}a^{2t}b^t}+x^{3i_1+t}y^{4t-3i_3}a^tb^{2t}+x^{3i_1-t}y^{6t-3i_3}a^tb^{2t} + x^{3i_1-3t}y^{8t-3i_3}a^tb^{2t} + \\
           & \qquad \qquad +x^{3i_1-5t}y^{10t-3i_3}a^tb^{2t}+x^{3i_1}y^{5t-3i_3}b^{3t}+ x^{3i_1-4t}y^{9t-3i_3}b^{3t}
\end{align*}
belongs to $T_1+T_4+\Sigma T_1$. To see this, observe that
\begin{align*}
S'(f_1,f_3) &= \udl{x^{3i_1-4t}y^{9t-3i_3}a^{2t}b^t}+(x^{6t}+y^{6t})x^ty^ta^tb^{2t}x^{3i_1-6t}y^{3t-3i_3}\\
            & \qquad +x^{3i_1-t}y^{6t-3i_3}a^tb^{2t} + x^{3i_1-3t}y^{8t-3i_3}a^tb^{2t} +x^{3i_1}y^{5t-3i_3}b^{3t}+ x^{3i_1-4t}y^{9t-3i_3}b^{3t}\\
            &= \udb{x^{3i_1-4t}y^{9t-3i_3}a^{2t}b^t}_{\text{first summand}\,\in\, T_1}+\udb{(x^{6t}+y^{6t})x^ty^ta^tb^{2t}x^{3i_1-6t}y^{3t-3i_3}}_{\text{second summand}\,\in\, T_4}\\
            & \qquad +\udb{x^{3i_1-t}y^{6t-3i_3}a^tb^{2t}+ x^{3i_1-3t}y^{8t-3i_3}a^tb^{2t} +x^{3i_1}y^{5t-3i_3}b^{3t}+ x^{3i_1-4t}y^{9t-3i_3}b^{3t}}_{\,\in\, \Sigma T_1}.
\end{align*}
Here, the containment of the first and the last terms hold since
\begin{align*}
3i_1-4t-1 & \equiv 0 \, \text{(mod 3)},\\
(3i_1-4t-1)+(9t-3i_3) & \ge 9t,
\end{align*}
the containment for the second term is because
\begin{align*}
3i_1-6t & \ge 0, i_3\le t, \\
(3i_1-6t)+(3t-3i_3) & \ge t+1,
\end{align*}
and the containment for the third term is due to
\begin{align*}
3i_1-t-1 & \equiv 0 \, \text{(mod 3)},\\
(3i_1-t-1)+(6t-3i_3) & \ge 9t.
\end{align*}
Also, we can justify the containment of the fourth and fifth terms similarly. Therefore, $S'(f_1,f_3)$ belongs to $T_1+T_4+\Sigma T_1$, as claimed.

Next, we show that $T_4 \subseteq T_1+T_3$. Indeed, let $0\le i_4 \le \frac{t+1}{3}$ be given. Choose $i_1=2t+i_4 \in [2t,(7t+1)/3]$, and $i_3=\frac{2t-1}{3}+i_4 \in [(2t-1)/3,t]$, then $3i_1-6t=3i_4, 3t-3i_3=t+1-3i_4$. The above argument in Case 2 yields
\[
S(f_1,f_3) = \udb{\left(\text{terms}\right)}_{\,\in\, T_3}+ \udb{\left(\text{terms}\right)}_{\,\in\, T_1} +f_4+ \udb{\left(\text{terms}\right)}_{\,\in\, \Sigma T_1}.
\]
Hence, $f_4 \in T_1+T_3$.

It remains to show that $T_5 \subseteq T_1+T_3$. Indeed, let $0\le i_5 \le \frac{2t-4}{3}$ be given. Choose $i_1=\frac{4t+1}{3}+i_5 \in [(4t+1)/3,2t-1]$, and $i_3=i_5 \in [0,(2t-4)/3]$. Then, $3i_1-4t-1=3i_5, 2t-4-3i_3=2t-4-3i_5$. The above argument in Case 1 yields
\[
S(f_1,f_3) = f_5+ \udb{\left(\text{other terms}\right)}_{\,\in\, \Sigma T_1}.
\]
Hence, $f_5 \in T_1+T_3$.

\textbf{(S16)}: Letting
\begin{align*}
\alpha &= \max\{3i_1, 3i_6+8t+1\},\\
\beta &= \max\{3n-3i_1, t-1-3i_6\},
\end{align*}
and keeping the common highest term $x^\alpha y^\beta b^{4t}$ in mind, we get
\begin{align*}
S(f_1,f_6) &= S(x^{3i_1}y^{3n-3i_1}, \udl{x^{3i_6+8t+1}y^{t-1-3i_6}b^{4t}}+x^{3i_6+1}y^{9t-1-3i_6}b^{4t})= x^{\alpha-8t}y^{\beta+8t}b^{4t}.
\end{align*}
Assume that $S(f_1,f_6)\notin T_1$. Then, $\alpha+\beta \le 3n+1$. As $8t+1 \equiv 2$ (mod 3), this yields
\begin{align*}
3i_6+8t+1 & \le \alpha \le 3i_1+1 \Longrightarrow 3i_6+8t+1 \le 3i_1-1, \alpha=3i_1,\\
\beta &= t-1-3i_6.
\end{align*}
In particular, $t-1-3i_6 \ge 0$, thus, $t-2-3i_6\ge 0$ as $t \equiv 2$ (mod 3).  Hence,
\[
S(f_1,f_6)= x^{\alpha-8t}y^{\beta+8t}b^{4t}=x^{3i_1-8t}y^{9t-1-3i_6}b^{4t}=x^2y^{8t+1}b^{4t}x^{3i_1-8t-2}y^{t-2-3i_6} \in \udb{x^2y^{8t+1}b^{4t}Q^\frac{t-2}{3}}_{T_8}.
\]
Here, the containment holds since $3i_1\ge 3i_6+8t+2$ and, hence,
\[
(3i_1-8t-2)+(t-2-3i_6) \ge t-2.
\]
Therefore, $S(f_1,f_6)$ belongs to $T_1\cup T_8$.

It remains to show that $T_8 \subseteq T_1+T_6$. Indeed, let $0\le i_8 \le \frac{t-2}{3}$ be given. Choose $i_1=\frac{8t+2}{3}+i_8 \in [(8t+2)/3,n]$, and $i_6=i_8 \in [0,(t-2)/3]$, then $3i_1-8t-2=3i_8, t-2-3i_6=t-2-3i_8$. The above argument yields
\[
f_8= S(f_1,f_6) \in T_1+T_6,
\]
as claimed.

\textbf{(S17)}: Letting $\alpha = \max\{3i_1, 2t+3i_7\}, \beta = \max\{3n-3i_1, 7t-3i_7\}$, and keeping the common highest term $x^\alpha y^\beta a^{3t}b^t$ in mind, we get
\begin{align*}
S(f_1,f_7) &= S(x^{3i_1}y^{3n-3i_1}, \udl{x^{2t+3i_7}y^{7t-3i_7}a^{3t}b^t}+x^{t+3i_7}y^{8t-3i_7}a^{2t}b^{2t}+ x^{4t+3i_7}y^{5t-3i_7}a^tb^{3t}\\
           &\qquad \qquad +x^{7t+3i_7}y^{2t-3i_7}b^{4t}+x^{5t+3i_7}y^{4t-3i_7}b^{4t}+x^{t+3i_7}y^{8t-3i_7}b^{4t})\\
           &=f_1\cdot x^{\alpha-3i_1}y^{\beta-(3n-3i_1)}a^{3t}b^t-f_7\cdot x^{\alpha-(2t+3i_7)}y^{\beta-(7t-3i_7)} \\
           &=x^{\alpha-t}y^{\beta+t}a^{2t}b^{2t}+x^{\alpha+2t}y^{\beta-2t}a^{t}b^{3t}+x^{\alpha+5t}y^{\beta-5t}b^{4t}+ x^{\alpha+3t}y^{\beta-3t}b^{4t}+x^{\alpha-t}y^{\beta+t}b^{4t}.
\end{align*}
Assume that $S(f_1,f_7)\notin \Sigma T_1$. Then, $\alpha+\beta \le 3n+1$. This yields
\begin{align*}
3i_1 & \le \alpha \le 2t+3i_7+1 \Longrightarrow 3i_1 \le 2t+3i_7-1, \alpha=2t+3i_7,\\
2t+3i_7=\alpha & \le 3i_1+1 \Longrightarrow 2t+3i_7=3i_1+1, \text{namely} \, 3(i_1-i_7)=2t-1,\\
\beta &= 3n-3i_1=9t-3i_1.
\end{align*}
We have
\begin{align*}
S(f_1,f_7) &= x^{\alpha-t}y^{\beta+t}a^{2t}b^{2t}+x^{\alpha+2t}y^{\beta-2t}a^{t}b^{3t}+x^{\alpha+5t}y^{\beta-5t}b^{4t}+ x^{\alpha+3t}y^{\beta-3t}b^{4t}+x^{\alpha-t}y^{\beta+t}b^{4t}\\
           &= \udl{x^{t+3i_7}y^{10t-3i_1}a^{2t}b^{2t}}+x^{4t+3i_7}y^{7t-3i_1}a^{t}b^{3t}+x^{7t+3i_7}y^{4t-3i_1}b^{4t}+\\
          & \qquad \qquad + x^{5t+3i_7}y^{6t-3i_1}b^{4t}+x^{t+3i_7}y^{10t-3i_1}b^{4t}.
\end{align*}

\textbf{Case 1:} $i_7=0$. Then, $i_1= \frac{2t-1}{3}$, and
\begin{align*}
 S(f_1,f_7) &= \udb{(x^ty^{8t+1}a^{2t}b^{2t}+x^{4t}y^{5t+1}a^tb^{3t}+x^{7t}y^{2t+1}b^{4t})}_{\in T_9} +\udb{x^{5t}y^{4t+1}b^{4t}}_{\in T_1}+\udb{x^ty^{8t+1}b^{4t}}_{\in T_8}.
\end{align*}

\textbf{Case 2:} $i_7\ge 1$. Then, the equality $2t+3i_7=3i_1+1$ yields the second equality in the chain
\begin{align*}
S(f_1,f_7) &= (x^{t+3}y^{7t+3}a^{2t}b^{2t}+x^{4t+3}y^{4t+3}a^tb^{3t}+x^{7t+3}y^{t+3}b^{4t}+x^{t+3}y^{7t+3}b^{4t})x^{3i_7-3}y^{3t-3-3i_1} +\\
           & \quad + x^{5t+3i_7}y^{6t-3i_1}b^{4t}\\
           &= \udb{(x^{t+3}y^{7t+3}a^{2t}b^{2t}+x^{4t+3}y^{4t+3}a^tb^{3t}+x^{7t+3}y^{t+3}b^{4t}+x^{t+3}y^{7t+3}b^{4t})x^{3i_7-3}y^{t-2-3i_7}}_{\text{first summand}} +\\
           & \quad + \udb{x^{5t+3i_7}y^{4t+1-3i_7}b^{4t}}_{\in T_1}= \udb{x^{t+3}y^{t+3}F_{10}^tb^{2t}x^{3i_7-3}y^{t-2-3i_7}}_{\text{first summand} \,\in\, T_{10}} +\udb{x^{5t+3i_7}y^{4t+1-3i_7}b^{4t}}_{\in T_1}.
\end{align*}
Here, the first containment holds since
\begin{align*}
1 & \le i_7 \le \frac{t-2}{3},\\
x^{3i_7-3}y^{t-2-3i_7} &\in Q^\frac{t-5}{3}.
\end{align*}
Therefore, $S(f_1,f_7) $ belongs to $\Sigma T_1 \cup (T_9+T_1+T_8) \cup (T_{10}+T_1)$, as desired.

It remains to show that $T_{10} \subseteq T_1+T_7$. Indeed, given $i_{10} \in [0,(t-5)/3]$, let $i_7=i_{10}+1 \in [1,(t-2)/3]$. Then, $3i_7-3=3i_{10}$, $t-2-3i_7=t-5-3i_{10}$, and the above argument in Case 2 yields $S(f_1,f_7) = f_{10}+\udb{\left(\text{other terms}\right)}_{\,\in\, T_1}$. Hence, $f_{10}\in T_1+T_7$.
The proof of the lemma is completed.
\end{proof}

\begin{lem}[S23]
\label{lem_3times2power_S23}
We have $S(T_2,T_3) \mathop{\xrightarrow{\qquad \qquad \qquad \qquad}}\limits_{T_1\cup T_2 \cup T_6 \cup T_7} 0.$ More concretely,
$$
\quad S(T_2,T_3) \subseteq \Sigma T_1 \cup (T_2+\Sigma T_1+T_6) \cup (T_7+\Sigma T_1).
$$
Moreover, $T_6 \subseteq T_1+T_2+T_3$ and $T_7 \subseteq T_1+T_2+T_3$.
\end{lem}
\begin{proof}
Letting $\beta=\max\{3t, 5t-3i_3\}$, and keeping the common highest term $x^{3i_3+4t}y^\beta a^{3t}b^t$ in mind, we have
\begin{align*}
S(f_2,f_3) &= S(\udl{x^{3t}y^{3t}a^{3t}}+x^{4t}y^{2t}a^{2t}b^t+x^{2t}y^{4t}a^{2t}b^t+x^{5t}y^ta^tb^{2t} + x^ty^{5t}a^tb^{2t}+x^{6t}b^{3t}+ \\
           & \qquad \qquad +x^{4t}y^{2t}b^{3t}+x^{2t}y^{4t}b^{3t}+y^{6t}b^{3t}, \\
           & \qquad \udl{x^{3i_3+4t}y^{5t-3i_3}a^{2t}b^t} +x^{3i_3+2t}y^{7t-3i_3}a^{2t}b^t+x^{3i_3+5t}y^{4t-3i_3}a^tb^{2t}+\\
           & \qquad \qquad +x^{3i_3+t}y^{8t-3i_3}a^tb^{2t}+x^{3i_3+4t}y^{5t-3i_3}b^{3t}+x^{3i_3+2t}y^{7t-3i_3}b^{3t})\\
           &=f_2\cdot x^{3i_3+t}y^{\beta-3t}b^t- f_3\cdot y^{\beta-(5t-3i_3)}a^t\\
           &= \udl{x^{3i_3+2t}y^{\beta+2t}a^{3t}b^t} + x^{3i_3+3t}y^{\beta+t}a^{2t}b^{2t} + x^{3i_3+t}y^{\beta+3t}a^{2t}b^{2t} + x^{3i_3+6t}y^{\beta-2t}a^tb^{3t} +\\
           &\qquad  + x^{3i_3+4t}y^{\beta}a^tb^{3t}+ x^{3i_3+7t}y^{\beta-3t}b^{4t} + x^{3i_3+5t}y^{\beta-t}b^{4t} + x^{3i_3+3t}y^{\beta+t}b^{4t}+\\
           &\qquad \qquad + x^{3i_3+t}y^{\beta+3t}b^{4t}.
\end{align*}
Assume that $S(f_2,f_3)\notin \Sigma T_1$. Then, $3i_3+\beta+4t\le 3n+1=9t+1$. Hence,
\begin{align*}
3t \le \beta \le 5t-3i_3+1 \Longrightarrow 3t\le 5t-3i_3-1 \Longrightarrow i_3 \le \frac{2t-1}{3}, \beta=5t-3i_3.
\end{align*}
The last display yields
\begin{align*}
S(f_2,f_3) &= \udl{x^{3i_3+2t}y^{7t-3i_3}a^{3t}b^t} + x^{3i_3+3t}y^{6t-3i_3}a^{2t}b^{2t} + x^{3i_3+t}y^{8t-3i_3}a^{2t}b^{2t} + x^{3i_3+6t}y^{3t-3i_3}a^tb^{3t} +\\
           &\qquad  + x^{3i_3+4t}y^{5t-3i_3}a^tb^{3t}+ x^{3i_3+7t}y^{2t-3i_3}b^{4t} + x^{3i_3+5t}y^{4t-3i_3}b^{4t} + x^{3i_3+3t}y^{6t-3i_3}b^{4t}+\\
           &\qquad \qquad + x^{3i_3+t}y^{8t-3i_3}b^{4t}.
\end{align*}

\textbf{Case 1:} $i_3\ge \frac{t+1}{3}$. We have
\begin{align*}
S(f_2,f_3) &= (\udl{x^{3t}y^{3t}a^{3t}}+x^{4t}y^{2t}a^{2t}b^t+x^{2t}y^{4t}a^{2t}b^t+x^{5t}y^ta^tb^{2t} + x^ty^{5t}a^tb^{2t}+x^{6t}b^{3t}+ \\
           & \qquad \qquad +x^{4t}y^{2t}b^{3t}+x^{2t}y^{4t}b^{3t}+y^{6t}b^{3t})x^{3i_3-t}y^{4t-3i_3}b^t +\\
           & \qquad \qquad + \udb{x^{3i_3+6t}y^{3t-3i_3}a^tb^{3t}+x^{3i_3}y^{9t-3i_3}a^tb^{3t}}_{\text{second summand}}+ x^{3i_3+7t}y^{2t-3i_3}b^{4t} + x^{3i_3-t}y^{10t-3i_3}b^{4t}\\
           &= \udb{f^{3t}x^{3i_3-t}y^{4t-3i_3}b^t}_{\text{first summand}\,\in\, T_2}+ \udb{x^{3i_3+6t}y^{3t-3i_3}a^tb^{3t}+x^{3i_3}y^{9t-3i_3}a^tb^{3t}}_{\text{second summand}\,\in\, \Sigma T_1}+\\
           &\qquad \qquad + \udb{xy(x^{8t}+y^{8t})x^{3i_3-t-1}y^{2t-1-3i_3}b^{4t}}_{\text{third summand}\,\in\, T_6}.
\end{align*}
Here, the containments hold since $\frac{t+1}{3} \le i_3 \le \frac{2t-1}{3}$ and $(3i_3-t-1)+ (2t-1-3i_3)=t-2$. Hence, $S(f_2,f_3)\in T_2+\Sigma T_1+T_6$.

\textbf{Case 2:} $i_3\le \frac{t-2}{3}$. We have
\begin{align*}
S(f_2,f_3) &= (\udl{x^ty^{5t}a^{3t}}+y^{6t}a^{2t}b^t+x^{3t}y^{3t}a^tb^{2t}+x^{6t}b^{3t}+x^{4t}y^{2t}b^{3t}+y^{6t}b^{3t})x^{3i_3+t}y^{2t-3i_3}b^t+ \\
            &  \qquad + x^{3i_3+3t}y^{6t-3i_3}a^{2t}b^{2t}+ x^{3i_3+6t}y^{3t-3i_3}a^tb^{3t} +  x^{3i_3+3t}y^{6t-3i_3}b^{4t}\\
            = & \udb{x^ty^{t+2}F_7^tb^tx^{3i_3}y^{t-2-3i_3}}_{\,\in\, T_7} +  \udb{x^{3i_3+3t}y^{6t-3i_3}a^{2t}b^{2t}+x^{3i_3+6t}y^{3t-3i_3}a^tb^{3t}+x^{3i_3+3t}y^{6t-3i_3}b^{4t}}_{\,\in\, \Sigma T_1}.
\end{align*}
Therefore, $S(f_2,f_3)\in T_7+\Sigma T_1$.

Next, we show that $T_6 \subseteq T_1+T_2+T_3$. Indeed, given $0\le i_6\le \frac{t-2}{3}$, choose $i_3= \frac{t+1}{3}+i_6 \in [(t+1)/3,(2t-1)/3]$, so that $3i_3-t-1=3i_6$ and $2t-1-3i_3=t-2-3i_6$. The above argument in Case 1 yields $S(f_2,f_3)= \udb{\left(\text{other terms}\right)}_{\,\in\, T_2+\Sigma T_1}+f_6$. It follows that $f_6\in T_1+T_2+T_3$.

It remains to show that $T_7 \subseteq T_1+T_2+T_3$. Indeed, given $0\le i_7\le \frac{t-2}{3}$, choose $i_3= i_7 \in [0,(t-2)/3]$, so that $3i_3=3i_7$ and $t-2-3i_3=t-2-3i_7$. The last display of Case 2 yields $S(f_2,f_3)= f_7+\udb{\left(\text{other terms}\right)}_{\,\in\, \Sigma T_1}$. Hence, $f_7\in T_1+T_2+T_3$. This concludes the proof of the lemma.
\end{proof}

\begin{lem}
\label{lem_3times2power_S25}
We have $S(T_2,T_5) \mathop{\xrightarrow{\qquad \qquad \qquad \qquad}}\limits_{T_1\cup T_3 \cup T_8} 0.$ More concretely,
$$
S(T_2,T_5) \subseteq \Sigma T_1 \cup (T_3+\Sigma T_1+T_8).
$$
\end{lem}
\begin{proof}
Note that $0\le i_5 \le \frac{2t-4}{3}$ implies $7t-3i_5\ge 5t+4 > 3t$. Letting $\alpha=\max\{3t, 3i_5+2t+1\}$ and keeping the common highest term $x^\alpha y^{7t-3i_5}a^{3t}b^t$ in mind, we get
\begin{align*}
S(f_2,f_5) &= S(\udl{x^{3t}y^{3t}a^{3t}}+x^{4t}y^{2t}a^{2t}b^t+x^{2t}y^{4t}a^{2t}b^t+x^{5t}y^ta^tb^{2t} + x^ty^{5t}a^tb^{2t}+x^{6t}b^{3t}+ \\
           & \qquad \qquad +x^{4t}y^{2t}b^{3t}+x^{2t}y^{4t}b^{3t}+y^{6t}b^{3t}, \\
           & \qquad \udl{x^{3i_5+2t+1}y^{7t-3i_5}a^{2t}b^t} +x^{3i_5+5t+1}y^{4t-3i_5}a^tb^{2t}+x^{3i_5+2t+1}y^{7t-3i_5}b^{3t})\\
           &=f_2\cdot x^{\alpha-3t}y^{4t-3i_5}b^t- f_5\cdot x^{\alpha-(3i_5+2t+1)}a^t\\
           &=\udl{x^{\alpha+3t}y^{4t-3i_5}a^{2t}b^{2t}}+ x^{\alpha+t}y^{6t-3i_5}a^{2t}b^{2t} + x^{\alpha-t}y^{8t-3i_5}a^{2t}b^{2t} + x^{\alpha+2t}y^{5t-3i_5}a^tb^{3t} \\
           & \qquad + x^{\alpha}y^{7t-3i_5}a^tb^{3t}+ x^{\alpha-2t}y^{9t-3i_5}a^tb^{3t} + x^{\alpha+3t}y^{4t-3i_5}b^{4t} + x^{\alpha+t}y^{6t-3i_5}b^{4t} \\
           & \qquad \qquad + x^{\alpha-t}y^{8t-3i_5}b^{4t}+ x^{\alpha-3t}y^{10t-3i_5}b^{4t}.
\end{align*}
Assume that $S(f_2,f_5)\notin \Sigma T_1$. Then, $\alpha+7t-3i_5\le 3n+1=9t+1$, so $\alpha\le 3i_5+2t+1$. This yields $\alpha=3i_5+2t+1$. Using $t\equiv 2$ (mod 3), we then get the second implication in the chain $3t\le 3i_5+2t+1 \Longrightarrow t+1 \le 3i_5 \Longrightarrow i_5\ge \frac{t+1}{3}$.

We have
\begin{align*}
S(f_2,f_5) &= (\udl{x^{3t}y^ta^{2t}}+x^ty^{3t}a^{2t}+x^{4t}a^tb^t+y^{4t}a^tb^t+x^{3t}y^tb^{2t}+x^ty^{3t}b^{2t})x^\alpha y^{3t-3i_5}b^{2t}+\\
           & \qquad + x^{\alpha-t}y^{8t-3i_5}a^{2t}b^{2t}+   x^{\alpha+4t}y^{3t-3i_5}a^tb^{3t} +  x^{\alpha+2t}y^{5t-3i_5}a^tb^{3t} +  x^{\alpha-2t}y^{9t-3i_5}a^tb^{3t} \\
           &\qquad \qquad  +  x^{\alpha-t}y^{8t-3i_5}b^{4t} +  x^{\alpha-3t}y^{10t-3i_5}b^{4t} \\
           &= F_3^tx^ty^tx^{3i_5+t+1}y^{2t-3i_5}b^{2t} + x^{3i_5+t+1}y^{8t-3i_5}a^{2t}b^{2t}+   x^{3i_5+6t+1}y^{3t-3i_5}a^tb^{3t} +  \\
           & + x^{3i_5+4t+1}y^{5t-3i_5}a^tb^{3t} +  x^{3i_5+1}y^{9t-3i_5}a^tb^{3t}   +  x^{3i_5+t+1}y^{8t-3i_5}b^{4t} +  x^{3i_5-t+1}y^{10t-3i_5}b^{4t} \\
           & \qquad \qquad  \text{(substituting $\alpha=3i_5+2t+1$)}\\
           &= \udb{F_3^tx^ty^tx^{3i_5+t+1}y^{2t-3i_5}b^{2t}}_{\,\in\, T_3} + \udb{x^{3i_5+t+1}y^{8t-3i_5}a^{2t}b^{2t}+x^{3i_5+6t+1}y^{3t-3i_5}a^tb^{3t}}_{\,\in\, \Sigma T_1} +  \\
           &\qquad  + \udb{x^{3i_5+4t+1}y^{5t-3i_5}a^tb^{3t}+x^{3i_5+1}y^{9t-3i_5}a^tb^{3t}+x^{3i_5+t+1}y^{8t-3i_5}b^{4t}}_{\,\in\, \Sigma T_1} + \\
           & \qquad \qquad  +  \udb{x^2y^{8t+1}x^{3i_5-t-1}y^{2t-1-3i_5}b^{4t}}_{\,\in\, T_8}.
\end{align*}
Here, the first containment holds since
\begin{align*}
3i_5+t+1 & \equiv 0 \quad \text{(mod 3)},\\
(3i_5+t+1)+(2t-3i_5) &= 3t+1.
\end{align*}
The next two containments hold since $\frac{t+1}{3} \le i_5 \le \frac{2t-4}{3}$ and $t\equiv 2$ (mod 3). The last containment holds since
\begin{align*}
3i_5-t-1 & \equiv 0 \quad \text{(mod 3)},\\
(3i_5-t-1)+(2t-1-3i_5) &= t-2.
\end{align*}
Therefore, $S(f_2,f_5) \in T_3 +\Sigma T_1+T_8$, as desired.
\end{proof}

\begin{lem}
\label{lem_3times2power_S212-14}
The following containments hold:
\begin{enumerate}
 \item[\textup{(S212)}] $S(T_2,T_{12}) \subseteq T_3+\Sigma T_1+T_{14}+\Sigma T_1+T_8$,
 \item[\textup{(S213)}] $S(T_2,T_{13}) \subseteq \Sigma T_1+T_5+T_1+T_{14}+T_1+T_6+\Sigma T_1$,
  \item[\textup{(S214)}] $S(T_2,T_{14}) \subseteq (\Sigma T_1+T_{15}+\Sigma T_1+T_8) \cup (\Sigma T_1+T_6+\Sigma T_1 +T_{15}+T_1)$.
\end{enumerate}
\end{lem}

\begin{proof}
We proceed case by case.

\textbf{(S212)}: The common highest term for $f_2=\udl{x^{3t}y^{3t}a^{3t}}+\cdots$ and $f_{12}=\udl{x^{4t}y^{5t+1}a^tb^{5t}}+x^{7t}y^{2t+1}b^{6t}$ is $x^{4t}y^{5t+1}a^{3t}b^{5t}$. Thus, we have
\begin{align*}
S(f_2,f_{12}) &= f_2\cdot x^ty^{2t+1}b^{5t}-f_{12}a^{2t}\\
           &= x^{5t}y^{4t+1}a^{2t}b^{6t}+x^{3t}y^{6t+1}a^{2t}b^{6t}+x^{6t}y^{3t+1}a^tb^{7t}+x^{2t}y^{7t+1}a^tb^{7t} +\\
           &\qquad + x^{7t}y^{2t+1}b^{8t}+x^{5t}y^{4t+1}b^{8t}+x^{3t}y^{6t+1}b^{8t}+x^ty^{8t+1}b^{8t} + \udl{x^{7t}y^{2t+1}a^{2t}b^{6t}}\\
           &= (\udl{x^{3t}y^ta^{2t}}+x^ty^{3t}a^{2t}+x^{4t}a^tb^t+y^{4t}a^tb^t+x^{3t}y^tb^{2t}+x^ty^{3t}b^{2t})x^{4t}y^{t+1}b^{6t}+\\
           &\qquad + x^{3t}y^{6t+1}a^{2t}b^{6t}+ x^{8t}y^{t+1}a^tb^{7t} + x^{6t}y^{3t+1}a^tb^{7t} +x^{4t}y^{5t+1}a^tb^{7t} \\
           &\qquad \qquad + x^{2t}y^{7t+1}a^tb^{7t}+ x^{3t}y^{6t+1}b^{8t} + x^ty^{8t+1}b^{8t} \\
  \quad   & = \udb{F_3^tx^ty^tx^{3t}yb^{6t}}_{\,\in\, T_3}+  \udb{x^{3t}y^{6t+1}a^{2t}b^{6t}+x^{8t}y^{t+1}a^tb^{7t}+x^{6t}y^{3t+1}a^tb^{7t}}_{\,\in\, \Sigma T_1} +\\
           &\qquad \qquad + \udb{x^{4t}y^{5t+1}a^tb^{7t}}_{\,\in\, T_{14}}+ \udb{x^{2t}y^{7t+1}a^tb^{7t}+x^{3t}y^{6t+1}b^{8t}}_{\,\in\, \Sigma T_1}+ \udb{x^ty^{8t+1}b^{8t}}_{\,\in\, T_8}.
\end{align*}
Therefore, $S(f_2,f_{12})\in T_3+\Sigma T_1+T_{14}+\Sigma T_1+T_8.$

\textbf{(S213)}: The common highest term for $f_2=\udl{x^{3t}y^{3t}a^{3t}}+\cdots$ and $f_{13}=\udl{x^{5t+1}y^{4t}a^tb^{5t}}+x^{2t+1}y^{7t}b^{6t}$ is $x^{5t+1}y^{4t}a^{3t}b^{5t}$. Thus, we have
\begin{align*}
S(f_2,f_{13}) &= f_2\cdot x^{2t+1}y^tb^{5t}-f_{13}a^{2t}\\
              &=\udl{x^{6t+1}y^{3t}a^{2t}b^{6t}}+x^{4t+1}y^{5t}a^{2t}b^{6t}+x^{7t+1}y^{2t}a^tb^{7t}+x^{3t+1}y^{6t}a^tb^{7t} +\\
           &\qquad + x^{8t+1}y^tb^{8t}+x^{6t+1}y^{3t}b^{8t}+x^{4t+1}y^{5t}b^{8t}+x^{2t+1}y^{7t}b^{8t} + x^{2t+1}y^{7t}a^{2t}b^{6t}\\
           &=\udb{x^{6t+1}y^{3t}a^{2t}b^{6t}+x^{4t+1}y^{5t}a^{2t}b^{6t}}_{\,\in\, \Sigma T_1}+\udb{(y^{3t}a^{2t}+x^{3t}a^tb^t+y^{3t}b^{2t})x^{2t+1}y^{2t+4}y^{2t-4}b^{6t}}_{\,\in\, T_5}+\\
           &\qquad + \udb{x^{7t+1}y^{2t}a^tb^{7t}}_{\,\in\, T_1}+ \udb{x^{5t+1}y^{4t}a^tb^{7t}}_{\,\in\, T_{14}} + \udb{x^{3t+1}y^{6t}a^tb^{7t}}_{\,\in\, T_1}  + x^{8t+1}y^tb^{8t}+x^{6t+1}y^{3t}b^{8t}+x^{4t+1}y^{5t}b^{8t} \\
    \qquad \qquad \quad &= \text{(element in $\Sigma T_1+T_5+T_1+T_{14}+T_1$)} + \udb{xy(x^{8t}+y^{8t})b^{8t}y^{t-1}}_{\,\in\, T_6} + \udb{x^{6t+1}y^{3t}b^{8t}}_{\,\in\, T_1} + \\
    \qquad \qquad  \quad  &  \qquad  +  \udb{x^{4t+1}y^{5t}b^{8t} +xy^{9t}b^{8t}}_{\,\in\, \Sigma T_1} \in \Sigma T_1+T_5+T_1+T_{14}+T_1+T_6+\Sigma T_1.
\end{align*}

\textbf{(S214)}: For $i_{14} \in \{0, t+1\}$, the common highest term for $f_2=\udl{x^{3t}y^{3t}a^{3t}}+\cdots$ and $f_{14}=x^{4t+i_{14}}y^{5t+1-i_{14}}b^{7t}$ is $x^{4t+i_{14}}y^{5t+1-i_{14}}a^{3t}b^{7t}$. Thus,
\begin{align*}
S(f_2,f_{14}) &= f_2\cdot x^{t+i_{14}}y^{2t+1-i_{14}}b^{7t} - f_{14}\cdot a^{3t}\\
              &= \udl{x^{5t+i_{14}}y^{4t+1-i_{14}}a^{2t}b^{8t}} + x^{3t+i_{14}}y^{6t+1-i_{14}}a^{2t}b^{8t}+x^{6t+i_{14}}y^{3t+1-i_{14}}a^tb^{9t}+ \\
              &\qquad + x^{2t+i_{14}}y^{7t+1-i_{14}}a^tb^{9t} + x^{7t+i_{14}}y^{2t+1-i_{14}}b^{10t} + x^{5t+i_{14}}y^{4t+1-i_{14}}b^{10t} \\
              &\qquad \qquad + x^{3t+i_{14}}y^{6t+1-i_{14}}b^{10t} +x^{t+i_{14}}y^{8t+1-i_{14}}b^{10t}.
\end{align*}
If $i_{14}=0$ then
\begin{align*}
S(f_2,f_{14}) &= \udb{x^{5t}y^{4t+1}a^{2t}b^{8t}+ x^{3t}y^{6t+1}a^{2t}b^{8t}+x^{6t}y^{3t+1}a^tb^{9t}+ x^{2t}y^{7t+1}a^tb^{9t}}_{\,\in\, \Sigma T_1} \\
              &\qquad  + \udb{x^{7t}y^{2t+1}b^{10t}}_{\,\in\, T_{15}} + \udb{x^{5t}y^{4t+1}b^{10t}+ x^{3t}y^{6t+1}b^{10t}}_{\,\in\, \Sigma T_1} +\udb{x^2y^{8t+1}x^{t-2}b^{10t}}_{\,\in\, T_8}.
\end{align*}
If $i_{14}=t+1$ then
\begin{align*}
S(f_2,f_{14}) &= \udl{x^{6t+1}y^{3t}a^{2t}b^{8t}}+ x^{4t+1}y^{5t}a^{2t}b^{8t} + x^{7t+1}y^{2t}a^tb^{9t} +x^{3t+1}y^{6t}a^tb^{9t} + \\
              &\qquad + x^{8t+1}y^tb^{10t}+ x^{6t+1}y^{3t}b^{10t} + x^{4t+1}y^{5t}b^{10t} +x^{2t+1}y^{7t}b^{10t} \\
              &= \udb{x^{6t+1}y^{3t}a^{2t}b^{8t}+ x^{4t+1}y^{5t}a^{2t}b^{8t} + x^{7t+1}y^{2t}a^tb^{9t} +x^{3t+1}y^{6t}a^tb^{9t}}_{\,\in\, \Sigma T_1} + \\
              &\qquad + \udb{xy(x^{8t}+y^{8t})y^{t-1}b^{10t}}_{\,\in\, T_6}+\udb{x^{6t+1}y^{3t}b^{10t} + x^{4t+1}y^{5t}b^{10t}}_{\,\in\, \Sigma T_1} +\udb{x^{2t+1}y^{7t}b^{10t}}_{\,\in\, T_{15}}  + \udb{xy^{9t}b^{10t}}_{\,\in\, T_1}.
\end{align*}
This concludes the proof of the lemma.
\end{proof}

\begin{lem}
\label{lem_3times2power_S34-15}
The following containments hold:
\begin{enumerate}
\item[\textup{(S34)}] $S(T_3,T_4) \subseteq \Sigma T_1 \cup (T_1+T_9+\Sigma T_1)$ and $T_9 \subseteq T_1+T_3+T_4$.
\item[\textup{(S312)}] $S(T_3,T_{12}) \subseteq \Sigma T_1 \cup (T_1+ T_4+ T_1+T_{14}+T_1)$.
\item[\textup{(S313)}] $S(T_3,T_{13}) \subseteq \Sigma T_1 \cup (\Sigma T_1+T_{14}+T_1)$ and $T_{14} \subseteq T_1+T_3+T_4+T_{12}+T_{13}$.
\item[\textup{(S314)}] $S(T_3,T_{14}) \subseteq \Sigma T_1 \cup (\Sigma T_1+T_8+T_{14}+T_1)  \cup (\Sigma T_1+T_{11}+T_{14}+T_1)$.
\item[\textup{(S315)}] $S(T_3,T_{15}) \subseteq \Sigma T_1 \cup (\Sigma T_1+T_{12}+T_1)$.
\end{enumerate}
\end{lem}
\begin{proof}
We proceed case by case.

\textbf{(S34)}: Letting $\beta=\max\{5t-3i_3, 2t+1-3i_4\}$ and keeping in mind the common highest term $x^{3i_4+7t}y^\beta a^{2t}b^{2t}$, we have
\begin{align*}
S(f_3,f_4) &= S(\udl{x^{3i_3+4t}y^{5t-3i_3}a^{2t}b^t}+x^{3i_3+2t}y^{7t-3i_3}a^{2t}b^t+x^{3i_3+5t}y^{4t-3i_3}a^tb^{2t}+\\
           & \qquad \qquad +x^{3i_3+t}y^{8t-3i_3}a^tb^{2t}+x^{3i_3+4t}y^{5t-3i_3}b^{3t}+x^{3i_3+2t}y^{7t-3i_3}b^{3t},\\
           & \qquad \qquad  \udl{x^{3i_4+7t}y^{2t+1-3i_4}a^tb^{2t}}+ x^{3i_4+t}y^{8t+1-3i_4}a^tb^{2t})\\
           &= f_3\cdot x^{3t+3i_4-3i_3}y^{\beta-(5t-3i_3)}b^t-f_4\cdot y^{\beta-(2t+1-3i_4)}a^t\\
           &=\udl{x^{3i_4+5t}y^{\beta+2t}a^{2t}b^{2t}} + x^{3i_4+t}y^{\beta+6t}a^{2t}b^{2t}+x^{3i_4+8t}y^{\beta-t}a^{t}b^{3t}+x^{3i_4+4t}y^{\beta+3t}a^tb^{3t} +\\
           &\qquad + x^{3i_4+7t}y^{\beta}b^{4t}+ x^{3i_4+5t}y^{\beta+2t}b^{4t}.
\end{align*}
Assume that $S(f_3,f_4)\notin \Sigma T_1$. Then $3i_4+\beta+7t \le 3n+1=9t+1$, hence $\beta \le 2t+1-3i_4$. This yields
\begin{align*}
5t-3i_3 \le \beta = 2t+1-3i_4 \Longrightarrow i_3\ge i_4+t \Longrightarrow i_3=t, i_4=0, \beta=2t+1.
\end{align*}
Here, the second implication holds since $i_3\le t, i_4\ge 0$. Thus,
\begin{align*}
S(f_3,f_4) &= \udl{x^{5t}y^{4t+1}a^{2t}b^{2t}} + x^{t}y^{8t+1}a^{2t}b^{2t}+x^{8t}y^{t+1}a^{t}b^{3t}+x^{4t}y^{5t+1}a^tb^{3t} +\\
           &\qquad + x^{7t}y^{2t+1}b^{4t}+ x^{5t}y^{4t+1}b^{4t}\\
           &=\udb{x^{5t}y^{4t+1}a^{2t}b^{2t}}_{\,\in\, T_1} + (y^{6t}a^{2t}+x^{3t}y^{3t}a^tb^t+x^{6t}b^{2t})x^ty^{2t+1}b^{2t} +\\
           &\qquad + x^{8t}y^{t+1}a^{t}b^{3t}+x^{5t}y^{4t+1}b^{4t} \\
           &= \udb{x^{5t}y^{4t+1}a^{2t}b^{2t}}_{\,\in\, T_1} + \udb{x^ty^{2t+1}F_9^tb^{2t}}_{\,\in\, T_9} + \udb{x^{8t}y^{t+1}a^{t}b^{3t}+x^{5t}y^{4t+1}b^{4t}}_{\,\in\, \Sigma T_1},
\end{align*}
as wanted.

It remains to show that $T_9 \subseteq T_1+T_3+T_4$. Indeed, choose $i_3=t, i_4=0$. The last display yields
\[
S(f_3,f_4)= \udb{\left(\text{terms}\right)}_{\,\in\, T_1}+f_9+\udb{\left(\text{terms}\right)}_{\,\in\, \Sigma T_1}.
\]
Hence, $f_9\in T_1+T_3+T_4$.

\textbf{(S312)}: Keeping the common highest term $x^{3i_3+4t}y^{5t+1}a^{2t}b^{5t}$ in mind, we have
\begin{align*}
S(f_3,f_{12}) &= S(\udl{x^{3i_3+4t}y^{5t-3i_3}a^{2t}b^t}+x^{3i_3+2t}y^{7t-3i_3}a^{2t}b^t+x^{3i_3+5t}y^{4t-3i_3}a^tb^{2t}+\\
           & \qquad \qquad +x^{3i_3+t}y^{8t-3i_3}a^tb^{2t}+x^{3i_3+4t}y^{5t-3i_3}b^{3t}+x^{3i_3+2t}y^{7t-3i_3}b^{3t},\\
           &\qquad \qquad \udl{x^{4t}y^{5t+1}a^tb^{5t}}+x^{7t}y^{2t+1}b^{6t}) \\
           &= f_3\cdot y^{3i_3+1}b^{4t} -f_{12}\cdot x^{3i_3}a^t\\
           &= \udl{x^{3i_3+2t}y^{7t+1}a^{2t}b^{5t}}+ x^{3i_3+7t}y^{2t+1}a^{t}b^{6t}+ x^{3i_3+5t}y^{4t+1}a^{t}b^{6t}+  x^{3i_3+t}y^{8t+1}a^{t}b^{6t}+ \\
           & \qquad + x^{3i_3+4t}y^{5t+1}b^{7t}+ x^{3i_3+2t}y^{7t+1}b^{7t}.
\end{align*}
Assume that $S(f_3,f_{12}) \notin \Sigma T_1$. Then, $3i_3+9t+1\le 3n+1=9t+1$. Thus, $i_3=0$, and
\begin{align*}
S(f_3,f_{12}) &= \udl{x^{2t}y^{7t+1}a^{2t}b^{5t}}+ x^{7t}y^{2t+1}a^{t}b^{6t}+ x^{5t}y^{4t+1}a^{t}b^{6t}+  x^{t}y^{8t+1}a^{t}b^{6t}+ \\
           & \qquad + x^{4t}y^{5t+1}b^{7t}+ x^{2t}y^{7t+1}b^{7t}\\
           &= \udb{x^{2t}y^{7t+1}a^{2t}b^{5t}}_{\,\in\, T_1}+ \udb{x^ty^t(x^{6t}+y^{6t})y^{t+1}a^{t}b^{6t}}_{\,\in\, T_4}+ \udb{x^{5t}y^{4t+1}a^{t}b^{6t}}_{\,\in\, T_1}+   \\
           & \qquad + \udb{x^{4t}y^{5t+1}b^{7t}}_{\,\in\, T_{14}}+ \udb{x^{2t}y^{7t+1}b^{7t}}_{\,\in\, T_1}.
\end{align*}
Therefore, $S(f_3,f_{12}) \in \Sigma T_1 \cup (T_1+T_4+T_1+T_{14}+T_1)$.

\textbf{(S313)}: Letting
\begin{align*}
\alpha &= \max\{3i_3+4t, 5t+1\},\\
\beta &= \max\{5t-3i_3, 4t\},
\end{align*}
and keeping the common highest term $x^\alpha y^{\beta}a^{2t}b^{5t}$ in mind, we have
\begin{align*}
S(f_3,f_{13}) &= S(\udl{x^{3i_3+4t}y^{5t-3i_3}a^{2t}b^t}+x^{3i_3+2t}y^{7t-3i_3}a^{2t}b^t+x^{3i_3+5t}y^{4t-3i_3}a^tb^{2t}+\\
           & \qquad \qquad +x^{3i_3+t}y^{8t-3i_3}a^tb^{2t}+x^{3i_3+4t}y^{5t-3i_3}b^{3t}+x^{3i_3+2t}y^{7t-3i_3}b^{3t},\\
           &\qquad \qquad \udl{x^{5t+1}y^{4t}a^tb^{5t}}+x^{2t+1}y^{7t}b^{6t}) \\
           &= f_3\cdot x^{\alpha-(3i_3+4t)}y^{\beta-(5t-3i_3)}b^{4t}- f_{13}\cdot x^{\alpha-(5t+1)}y^{\beta-4t}a^{t}\\
           &= x^{\alpha-2t}y^{\beta+2t}a^{2t}b^{5t}+ x^{\alpha+t}y^{\beta-t}a^{t}b^{6t}+ x^{\alpha}y^{\beta}b^{7t}+ x^{\alpha-2t}y^{\beta+2t}b^{7t}.
\end{align*}
Assume that $S(f_3,f_{13})\notin \Sigma T_1$. Then, $\alpha+\beta \le 9t+1.$ Since $\alpha\ge 5t+1, \beta\ge 4t$, both inequalities are in fact equalities. In particular,
\begin{align*}
S(f_3,f_{13}) &= \udb{x^{3t+1}y^{6t}a^{2t}b^{5t}+ x^{6t+1}y^{3t}a^{t}b^{6t}}_{\,\in\, \Sigma T_1}+ \udb{x^{5t+1}y^{4t}b^{7t}}_{\,\in\, T_{14}}+ \udb{x^{3t+1}y^{6t}b^{7t}}_{\,\in\, T_1},
\end{align*}
as desired.

It remains to prove the containment $T_{14} \subseteq T_1+T_3+T_4+T_{12}+T_{13}$. Letting $i_3=0$, the proof of (S312) yields
\[
S(f_3,f_{12})= \udb{\left(\text{terms}\right)}_{\,\in\, T_1+T_4+T_1}+x^{4t}y^{5t+1}b^{7t}+\udb{\left(\text{terms}\right)}_{\,\in\, T_1}.
\]
Letting $i_3=\frac{t+1}{3}$, the above argument in (S313) yields
\[
S(f_3,f_{13})= \udb{\left(\text{terms}\right)}_{\,\in\, \Sigma T_1}+x^{5t+1}y^{4t}b^{7t}+\udb{\left(\text{terms}\right)}_{\,\in\, T_1}.
\]
Hence, $x^{4t}y^{5t+1}b^{7t},x^{5t+1}y^{4t}b^{7t} \in T_1+T_3+T_4+T_{12}+T_{13}$, as claimed.

\textbf{(S314)}: If $i_{14}=0$, keeping the common highest term $x^{3i_3+4t}y^{5t+1}a^{2t}b^{7t}$ in mind, we have
\begin{align*}
S(f_3,f_{14}) &= S(\udl{x^{3i_3+4t}y^{5t-3i_3}a^{2t}b^t}+x^{3i_3+2t}y^{7t-3i_3}a^{2t}b^t+x^{3i_3+5t}y^{4t-3i_3}a^tb^{2t}+\\
           & \qquad \qquad +x^{3i_3+t}y^{8t-3i_3}a^tb^{2t}+x^{3i_3+4t}y^{5t-3i_3}b^{3t}+x^{3i_3+2t}y^{7t-3i_3}b^{3t},\\
           &\qquad \qquad x^{4t}y^{5t+1}b^{7t}) \\
           &= f_3\cdot y^{3i_3+1}b^{6t}-f_{14}\cdot x^{3i_3}a^{2t} \\
           &= x^{3i_3+2t}y^{7t+1}a^{2t}b^{7t} + x^{3i_3+5t}y^{4t+1}a^{t}b^{8t} + x^{3i_3+t}y^{8t+1}a^{t}b^{8t} + x^{3i_3+4t}y^{5t+1}b^{9t} + \\
           &\qquad +x^{3i_3+2t}y^{7t+1}b^{9t}.
\end{align*}
Assume that $S(f_3,f_{14})\notin \Sigma T_1$. Then, $3i_3+9t+1\le 9t+1$. This forces $i_3=0$ and
\begin{align*}
S(f_3,f_{14})&=\udb{x^{2t}y^{7t+1}a^{2t}b^{7t} + x^{5t}y^{4t+1}a^{t}b^{8t}}_{\,\in\, \Sigma T_1} + \udb{x^{t}y^{8t+1}a^{t}b^{8t}}_{\,\in\, T_8} + \udb{x^{4t}y^{5t+1}b^{9t}}_{\,\in\, T_{14}}  +\udb{x^{2t}y^{7t+1}b^{9t}}_{\,\in\, T_1}.
\end{align*}

If $i_{14}=t+1$, then letting $\alpha=\max\{3i_3+4t, 5t+1\}, \beta=\max\{5t-3i_3, 4t\}$, and keeping the common highest term $x^\alpha y^\beta a^{2t}b^{7t}$ in mind, we get
\begin{align*}
S(f_3,f_{14}) &= S(\udl{x^{3i_3+4t}y^{5t-3i_3}a^{2t}b^t}+x^{3i_3+2t}y^{7t-3i_3}a^{2t}b^t+x^{3i_3+5t}y^{4t-3i_3}a^tb^{2t}+\\
           & \qquad \qquad +x^{3i_3+t}y^{8t-3i_3}a^tb^{2t}+x^{3i_3+4t}y^{5t-3i_3}b^{3t}+x^{3i_3+2t}y^{7t-3i_3}b^{3t},x^{5t+1}y^{4t}b^{7t}) \\
           &= f_3\cdot x^{\alpha-(3i_3+4t)}y^{\beta- (5t-3i_3)}b^{6t}-f_{14}\cdot x^{\alpha-(5t+1)}y^{\beta-4t}a^{2t} \\
           &= x^{\alpha-2t}y^{\beta+2t}a^{2t}b^{7t} + x^{\alpha+t}y^{\beta-t}a^{t}b^{8t} + x^{\alpha-3t}y^{\beta+3t}a^{t}b^{8t} + x^{\alpha}y^{\beta}b^{9t} + x^{\alpha-2t}y^{\beta+2t}b^{9t}.
\end{align*}
Assume $S(f_3,f_{14})\notin \Sigma T_1$. Then, $\alpha+\beta \le 9t+1$. This forces the inequalities $\alpha \ge 5t+1$ and $\beta \ge 4t$ to be equalities. Thus
\begin{align*}
S(f_3,f_{14})&=\udb{x^{3t+1}y^{6t}a^{2t}b^{7t} + x^{6t+1}y^{3t}a^{t}b^{8t}}_{\,\in\, \Sigma T_1} + \udb{x^{2t+1}y^{7t}a^{t}b^{8t}}_{\,\in\, T_{11}} + \udb{x^{5t+1}y^{4t}b^{9t}}_{\,\in\, T_{14}}  +\udb{x^{3t+1}y^{6t}b^{9t}}_{\,\in\, T_1},
\end{align*}
as claimed.

\textbf{(S315)}: If $i_{15}=0$, then keeping the common highest term $x^{3i_3+4t}y^{7t}a^{2t}b^{8t}$ in mind, we have
\begin{align*}
S(f_3,f_{15}) &= S(\udl{x^{3i_3+4t}y^{5t-3i_3}a^{2t}b^t}+x^{3i_3+2t}y^{7t-3i_3}a^{2t}b^t+x^{3i_3+5t}y^{4t-3i_3}a^tb^{2t}+\\
           & \qquad \qquad +x^{3i_3+t}y^{8t-3i_3}a^tb^{2t}+x^{3i_3+4t}y^{5t-3i_3}b^{3t}+x^{3i_3+2t}y^{7t-3i_3}b^{3t},\\
           &\qquad \qquad x^{2t+1}y^{7t}b^{8t}) \\
           &= f_3\cdot y^{3i_3+2t}b^{7t}-f_{15}\cdot x^{3i_3+2t-1}a^{2t} \\
           &= x^{3i_3+2t}y^{9t}a^{2t}b^{8t} + x^{3i_3+5t}y^{6t}a^{t}b^{9t} + x^{3i_3+t}y^{10t}a^{t}b^{9t} + x^{3i_3+4t}y^{7t}b^{10t} + \\
           &\qquad +x^{3i_3+2t}y^{9t}b^{10t} \in \Sigma T_1.
\end{align*}

If $i_{15}=5t-1$, then letting $\beta=\max\{5t-3i_3, 2t+1\}$, and keeping the common highest term $x^{7t} y^\beta a^{2t}b^{8t}$ in mind, we get
\begin{align*}
S(f_3,f_{15}) &= S(\udl{x^{3i_3+4t}y^{5t-3i_3}a^{2t}b^t}+x^{3i_3+2t}y^{7t-3i_3}a^{2t}b^t+x^{3i_3+5t}y^{4t-3i_3}a^tb^{2t}+\\
           & \qquad \qquad +x^{3i_3+t}y^{8t-3i_3}a^tb^{2t}+x^{3i_3+4t}y^{5t-3i_3}b^{3t}+x^{3i_3+2t}y^{7t-3i_3}b^{3t},x^{7t}y^{2t+1}b^{8t}) \\
           &= f_3\cdot x^{3t-3i_3}y^{\beta- (5t-3i_3)}b^{7t}-f_{15}\cdot y^{\beta-(2t+1)}a^{2t} \\
           &= x^{5t}y^{\beta+2t}a^{2t}b^{8t}+x^{8t}y^{\beta-t}a^tb^{9t} + x^{4t}y^{\beta+3t}a^{t}b^{9t} + x^{7t}y^{\beta}b^{10t} + x^{5t}y^{\beta+2t}b^{10t}.
\end{align*}
Assume $S(f_3,f_{15})\notin \Sigma T_1$. Then, $\beta+7t \le 9t+1$. This forces the inequality $\beta \ge 2t+1$ to be equality. Thus,
\begin{align*}
S(f_3,f_{15})&=\udb{x^{5t}y^{4t+1}a^{2t}b^{8t} + x^{8t}y^{t+1}a^{t}b^{9t}}_{\,\in\, \Sigma T_1} + \udb{(x^{4t}y^{5t+1}a^{t}b^{5t}+x^{7t}y^{2t+1}b^{6t})b^{4t}}_{\,\in\, T_{12}} + \udb{x^{5t}y^{4t+1}b^{10t}}_{\,\in\, T_1},
\end{align*}
as claimed.
\end{proof}

\begin{lem}
\label{lem_3times2power_S46-1314}
The following containments hold:
\begin{enumerate}
\item[\textup{(S46)}] $S(T_4,T_6) \subseteq T_{11}+T_1$ and  $T_{11}\subseteq T_1+T_4+T_6$.
\item[\textup{(S415)}] $S(T_4,T_{15}) \subseteq T_8$.
\item[\textup{(S57)}] $S(T_5,T_7) \subseteq (T_3+\Sigma T_1) \cup \Sigma T_1$.
\item[\textup{(S511)}] $S(T_5,T_{11}) \subseteq T_{13}$ and $T_{13} \subseteq T_5+T_{11}$.
\item[\textup{(S515)}] $S(T_5,T_{15}) \subseteq T_{13} \cup \Sigma T_1$.
\item[\textup{(S711)}] $S(T_7,T_{11}) \subseteq (\Sigma T_1+T_{14}+T_1) \cup \Sigma T_1$.
\item[\textup{(S715)}] $S(T_7,T_{15}) \subseteq (\Sigma T_1+T_{14}+T_1) \cup \Sigma T_1$.
\item[\textup{(S89)}] $S(T_8,T_9) \subseteq T_{12}$ and $T_{12} \subseteq T_8+T_9$.
\item[\textup{(S1214)}] $S(T_{12},T_{14}) \subseteq T_{15}$.
\item[\textup{(S1314)}] $S(T_{13},T_{14}) \subseteq T_{15}$ and $T_{15}\subseteq T_{12}+T_{13}+T_{14}$.
\end{enumerate}
\end{lem}
\begin{proof}
We proceed case by case.

\textbf{(S46)}: Keeping the common highest term $x^{3i_6+8t+1}y^{2t+1-3i_4}a^tb^{4t}$ in mind, we have
\begin{align*}
&S(f_4,f_6)  \\
&= S( \udl{x^{3i_4+7t}y^{2t+1-3i_4}a^tb^{2t}}+ x^{3i_4+t}y^{8t+1-3i_4}a^tb^{2t},\udl{x^{3i_6+8t+1}y^{t-1-3i_6}b^{4t}}+x^{3i_6+1}y^{9t-1-3i_6}b^{4t})\\
           &=f_4\cdot x^{3i_6+t+1-3i_4}b^{2t}-f_6\cdot y^{3i_6+t+2-3i_4}a^t \\
           &=x^{3i_6+2t+1}y^{8t+1-3i_4}a^tb^{4t}+x^{3i_6+1}y^{10t+1-3i_4}a^tb^{4t}\\
           &=\udb{x^{2t+1}y^{7t}a^tb^{4t}x^{3i_6}y^{t+1-3i_4}}_{\,\in\,T_{11}}+\udb{y^{9t}x^{3i_6+1}y^{t+1-3i_4}a^tb^{4t}}_{\,\in\,T_{1}}.
\end{align*}
The containments hold since $3i_4\le t+1$.

It remains to show that $T_{11}\subseteq T_1+T_4+T_6$. Indeed, choose $i_4=\frac{t+1}{3}, i_6=0$. The last display yields
\[
S(f_4,f_6) = f_{11}+\udb{\left(\text{terms}\right)}_{\,\in\, T_1}.
\]
Hence $f_{11}\in T_1+T_4+T_6$.

\textbf{(S415)}: Keeping the common highest term $x^{3i_4+7t}y^{7t-i_{15}}a^tb^{8t}$ in mind, we have
\begin{align*}
S(f_4,f_{15}) &= S( \udl{x^{3i_4+7t}y^{2t+1-3i_4}a^tb^{2t}}+ x^{3i_4+t}y^{8t+1-3i_4}a^tb^{2t},x^{2t+1+i_{15}}y^{7t-i_{15}}b^{8t})\\
           &=f_4\cdot y^{5t-1+3i_4-i_{15}}b^{6t}-f_{15}\cdot x^{3i_4+5t-1-i_{15}}a^t \\
           &=x^{3i_4+t}y^{13t-i_{15}}a^tb^{8t}=x^2y^{8t+1}x^{t-2}b^{4t}\cdot a^tb^{4t}x^{3i_4}y^{5t-1-i_{15}} \in T_8.
\end{align*}
The containment holds since $i_{15}\le 5t-1$.

\textbf{(S57)}: Letting $\alpha = \max\{3i_5+2t+1, 2t+3i_7\}, \beta = \max \{7t-3i_5, 7t-3i_7\},$ and keeping the common highest term $x^\alpha y^\beta a^{3t}b^t$ in mind, we have
\begin{align*}
S(f_5,f_7) &= S(\udl{x^{3i_5+2t+1}y^{7t-3i_5}a^{2t}b^t} +x^{3i_5+5t+1}y^{4t-3i_5}a^tb^{2t}+x^{3i_5+2t+1}y^{7t-3i_5}b^{3t},\\
           &\qquad \qquad \udl{x^{2t+3i_7}y^{7t-3i_7}a^{3t}b^t}+x^{t+3i_7}y^{8t-3i_7}a^{2t}b^{2t}+ x^{4t+3i_7}y^{5t-3i_7}a^tb^{3t}+\\
           &\qquad \qquad +x^{7t+3i_7}y^{2t-3i_7}b^{4t}+x^{5t+3i_7}y^{4t-3i_7}b^{4t}+x^{t+3i_7}y^{8t-3i_7}b^{4t}) \\
           &= f_5\cdot x^{\alpha-(3i_5+2t+1)}y^{\beta-(7t-3i_5)}a^t- f_7\cdot x^{\alpha-(2t+3i_7)}y^{\beta-(7t-3i_7)}\\
           &=\udl{x^{\alpha+3t}y^{\beta-3t}a^{2t}b^{2t}} + x^{\alpha-t}y^{\beta+t}a^{2t}b^{2t} + x^{\alpha+2t}y^{\beta-2t}a^{t}b^{3t} + x^{\alpha}y^{\beta}a^{t}b^{3t} + \\
           &\qquad \qquad + x^{\alpha+5t}y^{\beta-5t}b^{4t} + x^{\alpha+3t}y^{\beta-3t}b^{4t} + x^{\alpha-t}y^{\beta+t}b^{4t}.
\end{align*}
Assume that $S(f_5,f_7)\notin \Sigma T_1$. Then, $\alpha+\beta\le 9t+1$ which forces $\alpha=3i_5+2t+1$ and $\beta=7t-3i_5$. Thus,
\begin{align*}
S(f_5,f_7) &= \udl{x^{3i_5+5t+1}y^{4t-3i_5}a^{2t}b^{2t}} + x^{3i_5+t+1}y^{8t-3i_5}a^{2t}b^{2t} + x^{3i_5+4t+1}y^{5t-3i_5}a^{t}b^{3t} +  \\
           &\qquad + x^{3i_5+2t+1}y^{7t-3i_5}a^{t}b^{3t} + x^{3i_5+7t+1}y^{2t-3i_5}b^{4t} + x^{3i_5+5t+1}y^{4t-3i_5}b^{4t} +\\
           & \qquad \qquad + x^{3i_5+t+1}y^{8t-3i_5}b^{4t}
\end{align*}
\begin{align*}
           &= \udb{(x^{3t}y^ta^{2t}+x^ty^{3t}a^{2t}+x^{4t}a^tb^t+y^{4t}a^tb^t+x^{3t}y^tb^{2t}+x^ty^{3t}b^{2t})x^{3i_5+2t+1}y^{3t-3i_5}b^{2t}}_{\text{first summand}} +\\
           &\qquad + x^{3i_5+3t+1}y^{6t-3i_5}a^{2t}b^{2t} + x^{3i_5+t+1}y^{8t-3i_5}a^{2t}b^{2t}+  x^{3i_5+6t+1}y^{3t-3i_5}a^{t}b^{3t}+ \\
           &\qquad \qquad + x^{3i_5+4t+1}y^{5t-3i_5}a^{t}b^{3t} + x^{3i_5+7t+1}y^{2t-3i_5}b^{4t}+ x^{3i_5+3t+1}y^{6t-3i_5}b^{4t}+ \\
           & \qquad \qquad \qquad + x^{3i_5+t+1}y^{8t-3i_5}b^{4t}.
\end{align*}
In other words,
\begin{align*}
S(f_5,f_7) &= \udb{F_3^tx^ty^t x^{3i_5+t+1}y^{2t-3i_5}b^{2t}}_{\text{first summand} \,\in\, T_3} +\\
           & \qquad  + \udb{x^{3i_5+3t+1}y^{6t-3i_5}a^{2t}b^{2t}+ x^{3i_5+t+1}y^{8t-3i_5}a^{2t}b^{2t}+  x^{3i_5+6t+1}y^{3t-3i_5}a^{t}b^{3t}}_{\,\in\, \Sigma T_1}+ \\
           &\qquad + \udb{x^{3i_5+4t+1}y^{5t-3i_5}a^{t}b^{3t}+ x^{3i_5+7t+1}y^{2t-3i_5}b^{4t}+ x^{3i_5+3t+1}y^{6t-3i_5}b^{4t}}_{\,\in\, \Sigma T_1}\\
           &\qquad \qquad + \udb{x^{3i_5+t+1}y^{8t-3i_5}b^{4t}}_{\,\in\, T_1}.
\end{align*}
The containments hold since $i_5 \le \frac{2t-4}{3}$ and $t\equiv 2$ (mod 3). Therefore, $S(f_5,f_7) \in (T_3+\Sigma T_1) \cup \Sigma T_1.$

\textbf{(S511)}: Keeping the common highest term $x^{3i_5+2t+1}y^{7t}a^{2t}b^{4t}$ in mind, we have
\begin{align*}
S(f_5,f_{11}) &= S(\udl{x^{3i_5+2t+1}y^{7t-3i_5}a^{2t}b^t} +x^{3i_5+5t+1}y^{4t-3i_5}a^tb^{2t}+x^{3i_5+2t+1}y^{7t-3i_5}b^{3t},\\
              & \qquad x^{2t+1}y^{7t}a^tb^{4t})\\
              &= f_5\cdot y^{3i_5}b^{3t}-f_{11}x^{3i_5}a^t\\
              &= x^{3i_5+5t+1}y^{4t}a^{t}b^{5t} + x^{3i_5+2t+1}y^{7t}b^{6t} =f_{13}x^{3i_5} \in T_{13}.
\end{align*}
It remains to show that $T_{13} \subseteq T_5+T_{11}$. To achieve this, choosing $i_5=0$, then the last display yields $f_{13}=S(f_5,f_{11})\in T_5+T_{11}$.

\textbf{(S515)}: If $i_{15}=0$, keeping the common highest term $x^{3i_5+2t+1}y^{7t}a^{2t}b^{8t}$ in mind, we have
\begin{align*}
S(f_5,f_{15}) &= S(\udl{x^{3i_5+2t+1}y^{7t-3i_5}a^{2t}b^t} +x^{3i_5+5t+1}y^{4t-3i_5}a^tb^{2t}+x^{3i_5+2t+1}y^{7t-3i_5}b^{3t},\\
              & \qquad x^{2t+1}y^{7t}b^{8t})\\
              &= f_5\cdot y^{3i_5}b^{7t}-f_{15}\cdot x^{3i_5}a^{2t}\\
              &= x^{3i_5+5t+1}y^{4t}a^{t}b^{9t} + x^{3i_5+2t+1}y^{7t}b^{10t}=f_{13}x^{3i_5}b^{4t} \in T_{13}.
\end{align*}
If $i_{15}=5t-1$, then keeping the common highest term $x^{7t}y^{7t-3i_5}a^{2t}b^{8t}$ in mind, we have
\begin{align*}
S(f_5,f_{15}) &= S(\udl{x^{3i_5+2t+1}y^{7t-3i_5}a^{2t}b^t} +x^{3i_5+5t+1}y^{4t-3i_5}a^tb^{2t}+x^{3i_5+2t+1}y^{7t-3i_5}b^{3t},\\
              & \qquad x^{7t}y^{2t+1}b^{8t})\\
              &= f_5\cdot x^{5t-1-3i_5}b^{7t}-f_{15}\cdot y^{5t-1-3i_5}a^{2t}\\
              &= x^{10t}y^{4t-3i_5}a^{t}b^{9t} + x^{7t}y^{7t-3i_5}b^{10t} \in \Sigma T_1,
\end{align*}
since $14t-3i_5 \ge 12t+4> 9t+2$.

\textbf{(S711)}: If $i_7=0$, then keeping the common highest term $x^{2t+1}y^{7t}a^{3t}b^{4t}$ in mind, we get
\begin{align*}
S(f_7, f_{11}) &= S(\udl{x^{2t}y^{7t}a^{3t}b^t}+x^{t}y^{8t}a^{2t}b^{2t}+ x^{4t}y^{5t}a^tb^{3t}+\\
           &\qquad \qquad +x^{7t}y^{2t}b^{4t}+x^{5t}y^{4t}b^{4t}+x^{t}y^{8t}b^{4t},x^{2t+1}y^{7t}a^tb^{4t})\\
           &=f_7\cdot xb^{3t}-f_{11}\cdot a^{2t}\\
           &= \udb{x^{t+1}y^{8t}a^{2t}b^{5t}+x^{4t+1}y^{5t}a^{t}b^{6t}+x^{7t+1}y^{2t}b^{7t}}_{\,\in\,\Sigma T_1}  + \udb{x^{5t+1}y^{4t}b^{7t}}_{ \,\in\, T_{14}} + \udb{x^{t+1}y^{8t}b^{7t}}_{ \,\in\,T_1}.
\end{align*}
If $i_7\ge 1$, then keeping the common highest term $x^{2t+3i_7}y^{7t}a^{3t}b^{4t}$ in mind, we get
\begin{align*}
S(f_7, f_{11}) &= S(\udl{x^{2t+3i_7}y^{7t-3i_7}a^{3t}b^t}+x^{t+3i_7}y^{8t-3i_7}a^{2t}b^{2t}+ x^{4t+3i_7}y^{5t-3i_7}a^tb^{3t}+\\
           &\qquad \qquad +x^{7t+3i_7}y^{2t-3i_7}b^{4t}+x^{5t+3i_7}y^{4t-3i_7}b^{4t}+x^{t+3i_7}y^{8t-3i_7}b^{4t},x^{2t+1}y^{7t}a^tb^{4t})\\
           &=f_7\cdot y^{3i_7}b^{3t}-f_{11}\cdot x^{3i_7-1}a^{2t}\\
           &= x^{t+3i_7}y^{8t}a^{2t}b^{5t} + x^{4t+3i_7}y^{5t}a^{t}b^{6t} + x^{7t+3i_7}y^{2t}b^{7t} + x^{5t+3i_7}y^{4t}b^{7t} + x^{t+3i_7}y^{8t}b^{7t} \in \Sigma T_1.
\end{align*}

\textbf{(S715)}: If $i_{15}=0$, then letting $\alpha=\max\{2t+3i_7, 2t+1\}$ and keeping the common highest term $x^\alpha y^{7t}a^{3t}b^{8t}$ in mind, we have
\begin{align*}
S(f_7,f_{15}) &= S(\udl{x^{2t+3i_7}y^{7t-3i_7}a^{3t}b^t}+x^{t+3i_7}y^{8t-3i_7}a^{2t}b^{2t}+ x^{4t+3i_7}y^{5t-3i_7}a^tb^{3t}+\\
           &\qquad \qquad +x^{7t+3i_7}y^{2t-3i_7}b^{4t}+x^{5t+3i_7}y^{4t-3i_7}b^{4t}+x^{t+3i_7}y^{8t-3i_7}b^{4t},x^{2t+1}y^{7t}b^{8t})\\
           &= f_7\cdot x^{\alpha-(2t+3i_7)}y^{3i_7}b^{7t}-f_{15}\cdot x^{\alpha-(2t+1)}a^{3t}\\
           &= \udb{x^{\alpha-t}y^{8t}a^{2t}b^{9t}+x^{\alpha+2t}y^{5t}a^{t}b^{10t}+x^{\alpha+5t}y^{2t}b^{11t}}_{\,\in\,\Sigma T_1}  + \udb{x^{\alpha+3t}y^{4t}b^{11t}}_{ \,\in\, T_{14}} + \udb{x^{\alpha-t}y^{8t}b^{11t}}_{ \,\in\,T_1}.
\end{align*}
The containments hold since $\alpha-t\ge t+1$, and $t+1$ is divisible by $3$.

If $i_{15}=5t-1$, then keeping the common highest term $x^{7t}y^{7t-3i_7}a^{3t}b^{8t}$ in mind, we get
\begin{align*}
S(f_7, f_{15}) &= S(\udl{x^{2t+3i_7}y^{7t-3i_7}a^{3t}b^t}+x^{t+3i_7}y^{8t-3i_7}a^{2t}b^{2t}+ x^{4t+3i_7}y^{5t-3i_7}a^tb^{3t}+\\
           &\qquad \qquad +x^{7t+3i_7}y^{2t-3i_7}b^{4t}+x^{5t+3i_7}y^{4t-3i_7}b^{4t}+x^{t+3i_7}y^{8t-3i_7}b^{4t},x^{7t}y^{2t+1}b^{8t})\\
           &=f_7\cdot x^{5t-3i_7}b^{7t}-f_{15}\cdot y^{5t-1-3i_7}a^{3t}\\
           = & \udb{x^{6t}y^{8t-3i_7}a^{2t}b^{9t} + x^{9t}y^{5t-3i_7}a^{t}b^{10t} + x^{12t}y^{2t-3i_7}b^{11t} + x^{10t}y^{4t-3i_7}b^{11t} + x^{6t}y^{8t-3i_7}b^{11t}}_{\,\in\, \Sigma T_1}.
\end{align*}
The containments hold since $14t-3i_7\ge 13t+2 > 9t+2$.

\textbf{(S89)}: The common highest term is $x^ty^{9t-3i_8-1}a^{2t}b^{4t}$, and we have
\begin{align*}
S(f_8,f_9) &= S(\udl{x^{3i_8+2}y^{9t-3i_8-1}b^{4t}}, \udl{x^ty^{8t+1}a^{2t}b^{2t}}+x^{4t}y^{5t+1}a^tb^{3t}+x^{7t}y^{2t+1}b^{4t}) \\
          &=f_8\cdot x^{t-2-3i_8}a^{2t}-f_9\cdot y^{t-2-3i_8}b^{2t}\\
          &=(x^{4t}y^{5t+1}a^tb^{5t}+x^{7t}y^{2t+1}b^{6t})y^{t-2-3i_8} =f_{12}y^{t-2-3i_8} \in T_{12},
\end{align*}
as $i_8\le \frac{t-2}{3}$.

It remains to show that $T_{12} \subseteq T_8+T_9$. Indeed, choosing $i_8=\frac{t-2}{3}$, the last display yields $f_{12}=S(f_8,f_9)\in T_8+T_9$.

\textbf{(S1214)}: The common highest term is $x^{4t+i_{14}}y^{5t+1}a^tb^{7t}$, and we have
\begin{align*}
S(f_{12},f_{14}) &= S(\udl{x^{4t}y^{5t+1}a^tb^{5t}}+x^{7t}y^{2t+1}b^{6t},x^{4t+i_{14}}y^{5t+1-i_{14}}b^{7t})\\
                 &=f_{12}\cdot x^{i_{14}}b^{2t}-f_{14}\cdot y^{i_{14}}a^t= x^{7t+i_{14}}y^{2t+1}b^{8t} \in T_{15}.
\end{align*}

\textbf{(S1314)}: The common highest term is $x^{5t+1}y^{5t+1-i_{14}}a^tb^{7t}$, and we have
\begin{align*}
S(f_{13},f_{14}) &= S(\udl{x^{5t+1}y^{4t}a^tb^{5t}}+x^{2t+1}y^{7t}b^{6t},x^{4t+i_{14}}y^{5t+1-i_{14}}b^{7t}) \\
                 &=f_{13}\cdot y^{t+1-i_{14}}b^{2t} -f_{14}\cdot x^{t+1-i_{14}}a^t= x^{2t+1}y^{8t+1-i_{14}}b^{8t} \in T_{15},
\end{align*}
since $i_{14}\le t+1$.

It remains to show the ideal containment $T_{15}\subseteq T_{12}+T_{13}+T_{14}$.
Letting $i_{14}=0$, the proof of (S1214) yields
\[
x^{7t}y^{2t+1}b^{8t}=S(f_{12},f_{14}) \in T_{12}+T_{14}.
\]
Letting $i_{14}=t+1$, the above argument in (S1314) yields
\[
x^{2t+1}y^{7t}b^{8t}=S(f_{13},f_{14}) \in T_{13}+T_{14},
\]
as claimed.  The proof of the lemma is concluded.
\end{proof}

\begin{lem}
\label{lem_3times2power_ST1}
The following containments hold:
\begin{enumerate}
 \item[\textup{(S$T_1$)}] $S(T_i,T_j)\subseteq \Sigma T_1$ if $(i,j) \in \{(1,4), (1,5), \ldots, (13, 15)\}$ \textup{(totally 71 pairs; see \Cref{tab_Spairs_3times2power})}.
\end{enumerate}
\end{lem}
\begin{proof}
We will use the form for minimal generators of ideals $T_i, 1\le i\le 8$, as recorded by \Cref{tab_mingens_GB_3times2power}. Despite their seemingly large quantity, these  71 cases are grouped in just a dozen general cases, whose treatment require only simple applications of Lemmas \ref{lem_inQn} and \ref{lem_large_xydegree}. The numbers at the beginning of each general case together form a renumbering of the whole 71 cases in question. 

\textbf{General case 1} (Cases 1--6): \textbf{The cases $S(T_1,T_j)$} where $j\in \{4,5,9,10,12,13\}$. Let $m_1=f_1=x^{3i_1}y^{3(n-i_1)}$, and $m_j=\ini(f_j)$. We get
\begin{align*}
 \deg_{\{x,y\}}m_1 &=3n=9t,\\
 \deg_{\{x,y\}}m_j &=9t+1,\\
 m_j &= x^{j_1}y^{j_2}a^{j_3}b^{j_4},
\end{align*}
where $j_1\equiv 2$ (mod 3) and $j_2\equiv 2$ (mod 3). In particular, $j_1\neq 3i_1, j_2\neq 3(n-i_1)$. Applying \Cref{lem_large_xydegree}(3), we conclude that
\[
\deg_{\{x,y\}}S(f_1,f_j)= \deg_{\{x,y\}}\lcm(f_1,\ini(f_j))=\deg_{\{x,y\}}\lcm(m_1,m_j)\ge \deg_{\{x,y\}}m_j+1=9t+2.
\]
Thanks to \Cref{lem_inQn}, this yields $S(f_1,f_j)\in \Sigma T_1$.

\textbf{General case 2} (Cases 7--11): \textbf{The cases $S(T_2,T_{j})$} where $j\in \{4, 8, 9, 10, 11\}$: Let $m_2=\ini(f_2)=x^{3t}y^{3t}a^{3t}$, $m_j=\ini(f_j)$. We have
\begin{align*}
 \deg_{\{x,y\}}m_j &=9t+1,\\
 m_j &= x^{j_1}y^{j_2}a^{j_3}b^{j_4},
\end{align*}
where either $3t>j_1$ or $3t>j_2$. Applying \Cref{lem_large_xydegree}(1), we get
$$
\deg_{\{x,y\}}S(f_2,f_{j}))= \deg_{\{x,y\}}\lcm(m_2,m_j) \ge \deg_{\{x,y\}}m_j+1= 9t+2.
$$
Hence, by \Cref{lem_inQn}, $S(f_2,f_j)\in \Sigma T_1$.

\textbf{General case 3} (Cases 12--14):  \textbf{The cases $S(T_i,T_{6})$} where $i\in \{2,3,5\}$. Let $m_i=\ini(f_i)=x^{\ell_1}y^{\ell_2}a^{\ell_3}b^{\ell_4}$, $m_6=\ini(f_6)=x^{3i_6+8t+1}y^{t-1-3i_6}b^{4t}$. Applying the inequality \eqref{eq_degxyineq} in \Cref{lem_large_xydegree}, we get
\begin{align*}
\qquad \deg_{\{x,y\}}S(f_6,f_{i})) &= \deg_{\{x,y\}}\lcm(m_6,m_i) \ge (3i_6+8t+1)-\ell_1+ \deg_{\{x,y\}}m_i \\
                            &= \begin{cases}
                                (3i_6+8t+1)-3t+6t, &\text{if $i=2$},\\
                                (3i_6+8t+1)-(3i_3+4t)+9t, &\text{if $i=3$},\\
                                (3i_6+8t+1)-(3i_5+2t+1)+(9t+1), &\text{if $i=5$},
                           \end{cases}
\end{align*}
\begin{align*}
              &= \begin{cases}
                                3i_6+11t+1, &\text{if $i=2$},\\
                                3i_6+13t+1-3i_3, &\text{if $i=3$},\\
                                3i_6+15t+1-3i_5, &\text{if $i=5$},
                           \end{cases} \\
              &\ge 10t+1 > 9t+2.
\end{align*}
The penultimate inequality holds since $i_6\ge 0, i_3\le t, i_5\le \frac{2t-4}{3}$. Hence, $S(f_i,f_6)\in \Sigma T_1$.

\textbf{General case 4} (Cases 15--18):  \textbf{The cases $S(T_i,T_{7})$} where $i\in \{2,3,4,6\}$. Let $m_i=\ini(f_i)=x^{\ell_1}y^{\ell_2}a^{\ell_3}b^{\ell_4}$, $m_7=\ini(f_7)=x^{2t+3i_7}y^{7t-3i_7}a^{3t}b^t$. We have
\begin{align*}
 \deg_{\{x,y\}}m_7 &=9t,\\
 \ell_1-(2t+3i_7) & \ge \min\{3t,3i_3+4t,3i_4+7t,3i_6+8t+1\}-(2t+3i_7)\ge 2,
\end{align*}
since $i_7\le \frac{t-2}{3}.$ Applying the inequality \eqref{eq_degxyineq} in \Cref{lem_large_xydegree}, we get
$$
\deg_{\{x,y\}}S(f_i,f_{7})= \deg_{\{x,y\}}\lcm(m_i,m_7) \ge  \ell_1-(2t+3i_7)+\deg_{\{x,y\}}m_7 \ge 9t+2.
$$
Hence, by \Cref{lem_inQn}, $S(f_i,f_7)\in \Sigma T_1$.

\textbf{General case 5} (Case 19): \textbf{The case $S(T_2,T_{15})$}.  If $i_{15}=0$, then
\begin{align*}
\lcm(\ini(f_2),\ini(f_{15})) &= \lcm(x^{3t}y^{3t}a^{3t},x^{2t+1}y^{7t}b^{8t})=x^{3t}y^{7t} a^{3t}b^{8t},\\
\deg_{\{x,y\}}(S(f_2,f_{15})) &= 10t > 9t+2.
\end{align*}
Hence, $S(f_2,f_{15})\in \Sigma T_1$. Similarly, if $i_{15}=5t-1$, then
\begin{align*}
\lcm(\ini(f_2),\ini(f_{15})) &= \lcm(x^{3t}y^{3t}a^{3t},x^{7t}y^{2t+1}b^{8t})=x^{7t}y^{3t} a^{3t}b^{8t},\\
\deg_{\{x,y\}}(S(f_2,f_{15})) &= 10t > 9t+2.
\end{align*}
Hence, $S(f_2,f_{15})\in \Sigma T_1$.

\textbf{General case 6} (Cases 20--24): \textbf{The cases $S(T_3,T_{j})$} where $j\in \{5,8,9,10,11\}$. Consider the monomials $m_3=\ini(f_3)=x^{3i_3+4t}y^{5t-3i_3}a^{2t}b^t$, and $m_j=\ini(f_j)$. We have
\begin{align*}
 \deg_{\{x,y\}}m_j &=9t+1,\\
 m_j &= x^{j_1}y^{j_2}a^{j_3}b^{j_4},
\end{align*}
where $j_1\le 4t-3 < 3i_3+4t$. Applying \Cref{lem_large_xydegree}(1), we get
$$
\deg_{\{x,y\}}S(f_3,f_{j}))= \deg_{\{x,y\}}\lcm(m_3,m_j) \ge \deg_{\{x,y\}}m_j+1= 9t+2.
$$
Hence, by \Cref{lem_inQn}, $S(f_3,f_j)\in \Sigma T_1$.

\textbf{General case 7} (Cases 25--32): \textbf{The cases $S(T_4,T_{j})$} where $j\in \{5, 8, 9,10, 11, 12, 13, 14\}$. Let $m_4=\ini(f_4)=x^{3i_4+7t}y^{2t+1-3i_4}a^tb^{2t}$, $m_j=\ini(f_j)$. We have
\begin{align*}
 \deg_{\{x,y\}}m_j &=9t+1,\\
 m_j &= x^{j_1}y^{j_2}a^{j_3}b^{j_4},
\end{align*}
where $j_1\le 5t+1 < 3i_4+7t$. Applying \Cref{lem_large_xydegree}(1), we get
$$
\deg_{\{x,y\}}S(f_4,f_{j}))= \deg_{\{x,y\}}\lcm(m_4,m_j) \ge \deg_{\{x,y\}}m_j+1= 9t+2.
$$
Hence, by \Cref{lem_inQn}, $S(f_4,f_j)\in \Sigma T_1$.

\textbf{General case 8} (Cases 33--38): \textbf{The cases $S(T_5,T_j)$} where $j\in \{8,9,10,12,13,14\}$. We have
\begin{align*}
m_5=\ini(f_5) &= x^{3i_5+2t+1}y^{7t-3i_5}a^{2t}b^{t},\\
m_j=\ini(f_j) &=x^{j_1}y^{9t+1-j_1}a^{j_3}b^{j_4},
\end{align*}
where $0\le j_1 \le 9t+1, j_3, j_4\ge 0$ and most crucially $j_1\neq 3i_5+2t+1$. More specifically,
\[
\begin{cases}
j_1 \le 2t-2 < 3i_5+2t+1, &\text{if $j\in \{8,9,10\}$},\\
j_1 \ge 4t > 3i_5+2t+1, &\text{if $j\in \{12,13,14\}$}.
\end{cases}
\]
Thus, \Cref{lem_large_xydegree}(2) yields $\deg_{\{x,y\}}S(f_5,f_j)=\deg_{\{x,y\}}\lcm(m_5,m_j)\ge \deg_{\{x,y\}}m_j+1=9t+2$. Hence, $S(f_5,f_j)\in \Sigma T_1$.

 \textbf{General case 9} (Cases 39--46):  \textbf{The cases $S(T_6,T_{j})$} where $8\le j\le 15$. Consider the monomials $m_6=\ini(f_6)=x^{3i_6+8t+1}y^{t-1-3i_6}b^{4t}$, and $m_j=\ini(f_j)$. We have
\begin{align*}
 \deg_{\{x,y\}}m_j &=9t+1,\\
 m_j &= x^{j_1}y^{j_2}a^{j_3}b^{j_4},
\end{align*}
where $j_1\le 7t < 3i_6+8t+1$. Applying \Cref{lem_large_xydegree}(1), we get
$$
\deg_{\{x,y\}}S(f_6,f_{j}))= \deg_{\{x,y\}}\lcm(m_6,m_j) \ge \deg_{\{x,y\}}m_j+1= 9t+2.
$$
Hence, by \Cref{lem_inQn}, $S(f_6,f_j)\in \Sigma T_1$.

\textbf{General case 10} (Cases 47--52):  \textbf{The cases $S(T_7,T_{j})$} where $j\in \{8,9, 10, 12, 13, 14\}$. Let $m_7=\ini(f_7)=x^{2t+3i_7}y^{7t-3i_7}a^{3t}b^{t}$, $m_j=\ini(f_j)$. We have
\begin{align*}
 \deg_{\{x,y\}}m_7 &=9t,\\
 \deg_{\{x,y\}}m_j &=9t+1,\\
 m_j &= x^{j_1}y^{j_2}a^{j_3}b^{j_4},
\end{align*}
where $j_1\neq 2t+3i_7$ and $j_2\neq 7t-3i_7$. More precisely,
\[
\begin{cases}
j_1 \le 2t-2 < 2t+3i_7, j_2 > 7t-3i_7, &\text{if $j\in \{8,9, 10\}$},\\
j_1 \ge 4t > 2t+3i_7, j_2  < 7t-3i_7, &\text{if $j\in \{12,13, 14\}$}.
\end{cases}
\]
Applying \Cref{lem_large_xydegree}(3), we get
$$
\deg_{\{x,y\}}S(f_7,f_{j}))= \deg_{\{x,y\}}\lcm(m_7,m_j) \ge \deg_{\{x,y\}}m_j+1= 9t+2.
$$
Hence, by \Cref{lem_inQn}, $S(f_7,f_j)\in \Sigma T_1$.

\textbf{General case 11} (Cases 53--71): \textbf{Cases $S(T_i,T_j)$} where $(i,j)\in \{(8,10), (8,12), (8,13)$, $(9,10), (9,11), \ldots, (9,15)$, $ (10,11),\ldots, (10,15)$, $(11,12), (11,13), (12,13), (12,15), (13,15)\}$. In these cases, we always have
\begin{align*}
\ini(f_i)&=x^{\ell_1}y^{9t+1-\ell_1}a^{\ell_3}b^{\ell_4}, \\
\ini(f_j)&=x^{j_1}y^{9t+1-j_1}a^{j_3}b^{j_4},
\end{align*}
where $0\le \ell_1, j_1\le 9t, \ell_3, \ell_4, j_3,j_4\ge 0$. Most importantly, we have $\ell_1\neq j_1$.  Indeed, there are strict inequalities
\begin{align*}
\ell_1 < j_1, & \quad  \text{if $(i,j)\in \{(8,10), (8,12), (8,13), (9,10)$, $\ldots, (9,15), (10,11),\ldots,(10,15),$}\\
                 & \qquad \qquad \qquad (11,12), (11,13), (12,13)\},\\
\ell_1 < j_1, & \quad \text{if ($(i,j)=(12,15)$ and $f_{15}=x^{7t}y^{2t+1}b^{8t}$) or}\\
              & \qquad \text{ ($(i,j)=(13,15)$ and $f_{15}=x^{7t}y^{2t+1}b^{8t}$)},\\
\ell_1 > j_1 , & \quad \text{if ($(i,j)=(12,15)$ and $f_{15}=x^{2t+1}y^{7t}b^{8t}$) or},\\
               &\qquad \text{($(i,j)=(13,15)$ and $f_{15}=x^{2t+1}y^{7t}b^{8t}$)}.
\end{align*}
 Thus, by \Cref{lem_large_xydegree}(2),
 \[
\deg_{\{x,y\}}S(f_i,f_{j}))= \deg_{\{x,y\}}\lcm(m_i,m_j) \ge \deg_{\{x,y\}}m_j+1= 9t+2.
 \]
 Therefore, $S(f_i,f_j)\in \Sigma T_1$. The proof is concluded.
\end{proof}

We now end the proof of Theorem \ref{thm_limregge6_3times2power} by completing \Cref{prop_GB_3times2power}.

\begin{proof}[{\bf Proof of \Cref{prop_GB_3times2power}}]
There are two conditions we have to verify: (i) the polynomials of type $T_i$, for $1\le i\le 15$, are elements of $Q^n+(f^n)$; and (ii) these polynomials satisfy the $S$-pair condition.

We will verify (ii) first. This is done by combining \Cref{rem_3times2power_S0} (that handles 25 cases), Lemmas \ref{lem_3times2power_S12-17}, \ref{lem_3times2power_S23}, \ref{lem_3times2power_S25}, \ref{lem_3times2power_S212-14}, \ref{lem_3times2power_S34-15}, \ref{lem_3times2power_S46-1314} (which handle 24 cases), and \Cref{lem_3times2power_ST1} (which handles 71 cases). Thereby, we have verified totally 25+24+71=120 cases, which exhaust all the possibilities for the $S$-pairs $S(T_i,T_j)$, where $1\le i\le j\le 15$.

The verification of (i), i.e., polynomials of type $T_i$, for $1\le i\le 15$, are elements of $Q^n+(f^n)$, is also a consequence of the above $S$-pair checks. Clearly, $T_1, T_2\subseteq Q^n+(f^n)$. Recall that we have the following containments of ideals, thanks to \Cref{lem_3times2power_S12-17}, parts (S12), (S13), (S16), (S17), \Cref{lem_3times2power_S23}(S23),  \Cref{lem_3times2power_S34-15}, parts (S34), (S313) and \Cref{lem_3times2power_S46-1314}, parts (S46), (S511), (S89) and (S1314):
\begin{alignat*}{5}
&T_3 &&\subseteq T_1+T_2 \quad && \text{(S12)},  \qquad && T_{10} \subseteq T_1+T_7 \quad && \text{(S17)}, \\
&T_4 &&\subseteq T_1+T_3 \quad && \text{(S13)},  && T_{11} \subseteq T_1+T_4+T_6 \quad && \text{(S46)}, \\
&T_5 &&\subseteq T_1+T_3 \quad &&\text{(S13)},  && T_{12} \subseteq T_8+T_9 \quad &&\text{(S89)}, \\
&T_6 &&\subseteq T_1+T_2+T_3 \quad &&\text{(S23)},  &&T_{13} \subseteq T_5+T_{11} \quad &&\text{(S511)}, \\
&T_7 &&\subseteq T_1+T_2+T_3 \quad &&\text{(S23)},  &&T_{14} \subseteq T_1+T_3+T_4+T_{12}+T_{13} \quad &&\text{(S313)}, \\
&T_8 &&\subseteq T_1+T_6 \quad &&\text{(S16)},      &&T_{15} \subseteq T_{12}+T_{13}+T_{14} \quad &&\text{(S1314)}.\\
&T_9 && \subseteq T_1+T_3+T_4 \quad &&\text{(S34)},     &&  \qquad \qquad  &&
\end{alignat*}
Therefore, all the ideals $T_i$ are contained in $Q^n+(f^n)$. This concludes the proof of \Cref{prop_GB_3times2power}.
\end{proof}


\section{Asymptotic regularity and  Herzog--Hoa--Trung's question} \label{sec.HHT}

In this section, we give evidence illustrating that the open question of Herzog, Hoa and Trung on the existence of $\lim\limits_{n \rightarrow \infty} {\reg I^{(n)}}/{n}$, for a homogeneous ideal $I$, is expected to have a negative answer. The ideal of interest will be of the form
$$I = Q \cap (f,z) \subseteq S = R[z],$$
where $Q$ is an ideal and $f$ is an element in $R$.

We begin with a lemma that allows us to compute the regularity of symbolic powers of $I$ in terms of those of ideals of the form $Q^n + (f^k)$.

\begin{lem}
	\label{lem_symbpowofintersect_regformula}
	Let $R=\kk[x_1,\ldots,x_r]$ be a polynomial ring over a field $\kk$. Let $Q\subseteq R$ be a homogeneous primary ideal, and let $f\in R$ be a homogeneous element of positive degree. Assume that the following conditions are satisfied:
	\begin{enumerate}[\quad \rm (i)]
		\item $Q\not\subseteq (g)$ for any irreducible factor $g$ of $f$, and,
		\item $Q^n=Q^{(n)}$ for all $n\ge 1$.
	\end{enumerate}
	Consider the ideal $I=Q\cap (f,z)$ in $S = R[z]$.
	\begin{enumerate}[\quad \rm (1)]
		\item For all $n \ge 1$, we have $I^{(n)}=Q^n\cap (f,z)^n$.
		\item Let $n\ge 1$ be such that $\reg(Q^n+(f^n))> \max\{\reg Q^n, n\deg(f)\}$. Then,
		\[
		\reg I^{(n)}=\reg(Q^n+(f,z)^n)+1 = \max\limits_{1\le k\le n} \left\{\reg \left(Q^n+(f^k)\right)+n-k\right\}+1.
		\]
	\end{enumerate}
\end{lem}
\begin{proof}
	(1) Let $f=f_1f_2\cdots f_d$ be the factorization of $f$ as product of powers of distinct irreducible polynomials in $R$. Then $(f,z)=(f_1,z) \cap \cdots \cap (f_d,z)$, so we obtain a primary decomposition
	\[
	I=QS\cap (f_1,z) \cap \cdots \cap (f_d,z).
	\]
	Since $Q$ is primary, the hypothesis $Q\not\subseteq \sqrt{(f_i)}$, for any $1\le i\le d$, implies that  $\Ass(I)=\{\sqrt{Q}, \sqrt{(f_1,z)},\ldots,\sqrt{(f_d,z)}\}$ contains no embedded primes. Thanks to the hypothesis (ii), we obtain the second equality in the chain
	\begin{align*}
		I^{(n)} &=(QS)^{(n)} \cap (f_1,z)^{(n)} \cap \cdots \cap (f_d,z)^{(n)}=Q^nS\cap (f_1\cdots f_d,z)^{(n)}\\
		&=Q^n \cap (f,z)^n.
	\end{align*}
	Also, the last equality holds since $(f,z)$ is a complete intersection.
	
	(2)
	Denote $p=\deg f$. We prove the following claims:
	\begin{enumerate}[\qquad \rm (a)]
		\item $\reg(Q^n+(f,z)^n)=\max\limits_{1\le k\le n} \left\{\reg(Q^n+(f^k))+n-k\right\}$.
		\item $\reg(Q^n+(f,z)^n) >\max\{\reg Q^n, \reg (f,z)^n\}=\max\{\reg Q^n, pn\}.$
	\end{enumerate}
	For (a): The $R$-module $S/(Q^n+(f,z)^n)$ is finite over $S/(z^n)$ and, hence, finite over $R$ itself. Since $Q,(f)\subseteq R$, there is a direct sum decomposition of $R$-modules
	\[
	Q^n+(f,z)^n =(Q^n+(f^n))\oplus (Q^n+(f^{n-1}))z \oplus \cdots \oplus (Q^n+(f))z^{n-1} \oplus \bigoplus_{i=n}^\infty Rz^i.
	\]
	Thus, there is a direct sum decomposition of $R$-modules
	\[
	\frac{S}{Q^n+(f,z)^n} \cong \bigoplus_{k=1}^n \frac{R}{Q^n+(f^k)}z^{n-k}.
	\]
	This implies (a).
	
	For (b): Note that the proof of (a) is valid even for $Q=(0)$, and yields $\reg (f,z)^n=n\deg f$. The hypothesis $\reg (Q^n+(f^n))> \max\{\reg Q^n, n\deg(f)\}$ and the fact that $\reg(Q^n+(f,z)^n)\ge \reg (Q^n+(f^n))$ yield claim (b).
	
	Claim (a) and the short exact sequence
	\[
	0\to I^{(n)} \to Q^n\oplus (f,z)^n \to Q^n+(f,z)^n\to 0
	\]
	imply that $\reg I^{(n)}=\reg \left(Q^n+(f,z)^n\right)+1$, and the desired conclusion follows.
\end{proof}

To proceed, as in the last few sections, we shall employ the notations in Section \ref{sec.nonExist} and Notation \ref{notn_ideals}, i.e.,
\begin{enumerate}
\item $R = \kk[x,y,a,b]$ is a polynomial ring with the degree reserve lexicographic term order induced from $x > y > a > b$,
\item $Q = (x^3,y^3), \quad f = xya - (x^2+y^2)b,$
\item $S = R[x] \text{ and } I = Q \cap (f,z) \subseteq S.$
\end{enumerate}

The following lemma is valid independent of the characteristic of $\kk$.
\begin{lem}
	\label{lem_easy_regboud}
	For all $n\ge 3$, we have
	\[
	\reg(Q^n+(f^n))\ge 4n-1.
	\]
\end{lem}

\begin{proof}
	Take $g\in Q^n:f^n$ a homogeneous polynomial. Consider the grading of $R$, in which $\deg x=\deg y=1$ and $\deg a=\deg b=0$. Then, $\deg f^n=2n$ and $Q^n \subseteq (x,y)^{3n}$ is generated in degree $3n$. Therefore, $\deg g\ge n$. In particular, $g$ is a homogeneous polynomial in $x,y$ of degree at least $n$. It follows that, in the usual standard grading of $R$, we also have $\deg g\ge n$. In particular, $\reg(Q^n:f^n)\ge n$.
	
	From the exact sequence
	\[
	0 \to \dfrac{R}{Q^n:(f^n)}(-3n) \xrightarrow{\cdot f^n} \frac{R}{Q^n} \to \frac{R}{Q^n+(f^n)} \to 0,
	\]
	we get
	\[
	\reg \dfrac{R}{Q^n:(f^n)}(-3n)= \reg \dfrac{R}{Q^n:(f^n)}+3n \ge 4n-1 > 3n+1 = \reg \frac{R}{Q^n}.
	\]
	Hence, the exact sequence yields $\reg(Q^n+(f^n))= \reg \dfrac{R}{Q^n:(f^n)}+3n \ge 4n-1$.
\end{proof}

When the field $\kk$ is of characteristic 2, we raise the following conjecture, asserting an upper bound for the regularity of $Q^n + (f^n)$.

\begin{conj}
	\label{conj_regbound}
	There exists $\epsilon > 0$ such that the inequality
	$$\reg (Q^n + (f^k)) \le (5-\epsilon)n + k$$
	holds for all integers $1 \le k \le n$ with $n = 2^s$ for some $s \in \NN$.
\end{conj}

\begin{rem}
	An affirmative answer to \Cref{conj_regbound} will give a negative answer to the question of Herzog, Hoa and Trung by confirming \Cref{conj.counterHHT}. Particularly, in this case, for $I=Q\cap (f,z) \subseteq \kk[x,y,z,a,b]$, the limit $\lim\limits_{n\to \infty} {\reg I^{(n)}}/{n}$ does not exist. This is because, by \Cref{lem_symbpowofintersect_regformula}, \Cref{thm_limregge6_3times2power} and \Cref{conj_regbound}, we have
	\[
	 \begin{cases}
	\liminf\limits_{n\to \infty} \dfrac{\reg I^{(n)}}{n} 	\le 6-\epsilon, &\text{if $n=2^s, s\ge 1$},\\
	\limsup\limits_{n\to \infty} \dfrac{\reg I^{(n)}}{n} 	\ge 6, &\text{if $n=3\cdot 2^s, s\ge 0$}.
	\end{cases}
	\]
\end{rem}

We shall give a partial answer to Conjecture \ref{conj_regbound} and, hence, Conjecture \ref{conj.counterHHT} in the following result.

\begin{thm}
	\label{thm_regQnfk}
	Consider $n=2^s$ and $k = 2^u$, where $1 \le u \le s$ are integers. Then,
		\begin{align*}
			3n+3k-1 &\le \reg \left(Q^n+(f^k)\right)\le 3n+3k+2, \quad \text{if $s\ge u+1$},\\
			5n &\le \reg(Q^n+(f^n))\le 5n+2, \quad \text{if $s= u$ \textup{(}namely, $n=k$\textup{)}}.
		\end{align*}
\end{thm}

Note that Theorem \ref{thm_regQnfk}, in the case when $s = u$, is already established in Theorem \ref{thm_regbound_2powers}. We shall focus on the case where $u < s$. To establish Theorem \ref{thm_regQnfk} in this case, we consider whether $u$ is odd or even, and provide an explicit description of the Gr\"obner basis for $Q^n + (f^k)$. These are done separately in the next two sections.


\section{Proof of Theorem \ref{thm_regQnfk} when $u$ is odd} \label{sec.Qnfk_Odd}

Theorem \ref{thm_regQnfk} follows from the following more detailed statement.

\begin{thm}
	\label{thm_regbound_double2powers_odd}
	Consider $n=2^s$ and $k=2^u$, where $1 \le u < s$ are integers and $u$ is odd.
	\begin{enumerate}[\quad \rm (1)]
		\item The initial ideal of $Q^n+(f^k)$ is given by
		\begin{align*}
			\ini(Q^n+(f^k)) = \,\,  & Q^n+(x^ky^ka^k)+x^{2k+1}yQ^{n-\frac{2k+2}{3}}b^k+x^2y^{3n-2k+3}Q^\frac{2k-4}{3}b^k+\\
			&+x^{4k+2}y^2(x^{6k},y^{6k})^{\frac{n}{2k}-1}Q^\frac{2k-4}{3}b^{2k}+x^{6k+1}y^{2k+1}(x^{6k},y^{6k})^{\frac{n}{2k}-2}Q^\frac{4k-2}{3}b^{2k}+\\
			&+xy^{3n-k+1}Q^\frac{k-2}{3}a^kb^k+x^2y^{3n-k+3}Q^\frac{k-5}{3}a^kb^{2k}.
		\end{align*}
		(By convention, if $\frac{n}{2k}<2$ then the summand involving $(x^{6k},y^{6k})^{\frac{n}{2k}-2}$ is zero.)
		\item There is a containment
		$$
		(x^{3n-k}y^k+x^ky^{3n-k})a^{k-1}b^{2k-1} \in \left(\left(Q^n+(f^k)\right):\mm\right)\setminus \left(Q^n+(f^k)\right).
		$$
		\item We have inequalities
		$$
		3n+3k-1 \le \reg \left(Q^n+(f^k)\right)\le 3n+3k+2.
		$$
	\end{enumerate}
\end{thm}


\subsection{Gr\"obner basis analysis} In order to establish Theorem \ref{thm_regbound_double2powers_odd}, we need to have a good understanding of the Gr\"obner basis of $Q^n + (f^k)$, which is accomplished in the next statement.

\begin{prop}
	\label{prop_GB_double2powers_odd}
	Consider $n=2^s$ and $k=2^u$, where $3 \le u < s$ are integers and $u$ is odd.
	\begin{enumerate}[\quad \rm (1)]
		\item The ideal $Q^n+(f^k)$ has a Gr\"obner basis  consisting of the natural generators of the following ideals \textup{(}8 types of ideals in total\textup{)}:
		\begin{align*}
			&\udb{Q^n}_{T_1}, \qquad \udb{(f^k)}_{T_2}, \qquad  \udb{(x^{2k}+y^{2k})xyQ^{n-\frac{2k+2}{3}}b^k}_{T_3}, \qquad \udb{x^2y^{3n-2k+3}Q^\frac{2k-4}{3}b^k}_{T_4}, \\
			&\udb{(x^{6k\tau+4k}+y^{6k\tau+4k})x^2y^{3n-6k+2-6k\tau}Q^\frac{2k-4}{3}b^{2k}}_{T_5} \qquad \text{\textup{(}where $0\le \tau\le \dfrac{n}{2k}-1$\textup{)},}\\
			&\udb{(x^{6k\ell}+y^{6k\ell})xy^{3n-4k+1-6k\ell}Q^\frac{4k-2}{3}b^{2k}}_{T_6} \qquad \text{\textup{(}where $1\le \ell \le \dfrac{n}{2k}-1$\textup{)},}\\
			&\udb{(y^{3k}a^k+x^{3k}b^k)xy^{3n-4k+1}Q^\frac{k-2}{3}b^k}_{T_7}, \qquad  \udb{(y^ka^k+x^kb^k)x^2y^{3n-2k+3}Q^\frac{k-5}{3}b^{2k}}_{T_8}.
		\end{align*}
		\item In particular, $Q^n+(f^k)$ has a Gr\"obner basis consisting of $3n+1-\dfrac{2k-1}{3}$ elements, whose maximal possible degree is $3(n+k)$.
	\end{enumerate}
\end{prop}

\begin{center}
	\begin{table}[ht!]
		\caption{How the $S$-pairs in the Gr\"obner basis of $Q^n+(f^k)$, where $n, k$ are powers of 2, $n\ge 2k$, $\log_2 k$ is odd and $\ge 3$, reduce to zero.}
		\begin{tabular}{ | c | c | c | c |}
			\hline
			\multirow{2}{*}{No.} & The number of pairs   & \multirow{2}{*}{$(i,j)$} & $\Lambda \subseteq [8]$ for which $S(T_i,T_j)$ \\
			& $(i,j)$ where $1\le i\le j\le 8$         &            & reduces to zero  via $\mathop{\bigcup}\limits_{\ell \in \Lambda} T_\ell$ \\
			\hline
			1 & 1 & (1,2)                    & \{1,3\} \\
			\hline
			2 &  1 & (1,3)                    & \{1,3,4\}\\
			\hline
			3 & 7 & $(i,i)$ ($i\notin \{5,6\}$), (1,4)                    & $\emptyset$\\
			\hline
			4 & 1 &  (1,5)                    & \{1,4\}\\
			\hline
			\multirow{5}{*}{5} & \multirow{5}{*}{18} & (1,6), (1,7), (2,4), (2,8) &   \multirow{5}{*}{\{1\}} \\
			&   & (3,4), (3,5), (3,7), (3,8) &   \\
			&   & (4,5), (4,6), (4,8), (5,5), & \\
			&   &(5,6), (5,7), (5,8), (6,6), & \\
			&   &  (6,7), (6,8)              & \\
			\hline
			6 &   1& (1,8)                     & \{1,4\} \\
			\hline
			7 &   1& (2,3)                     & \{1,2,3,5,6,7\}\\
			\hline
			8 &   1& (2,5)                     & \{1,2,3,6,8\} \\
			\hline
			9 &   1& (2,6)                       & \{1,2,3,5,6,7\} \\
			\hline
			10 &   1& (2,7)                       & \{1,3\} \\
			\hline
			11 &   1& (3,6)                       & \{1,3,4\} \\
			\hline
			12 &   1& (4,7)                       & \{1,3\} \\
			\hline
			13 &   1& (7,8)                       & \{1,3\} \\
			\hline
			&  \textbf{Totally}: 36 pairs &                             &\\
			\hline
		\end{tabular}
		\label{tab_Spairs_double2powers}
	\end{table}
\end{center}

Below, again we will use \Cref{notn_GB} on $S$-pairs.  We check that the $S$-pairs of the type $S(T_i,T_j)$ reduce to zero according to  \Cref{tab_Spairs_double2powers}. Even more concretely, these $S$-pairs satisfy the following containments:
\begin{enumerate}
	\item[(S12)] $S(T_1,T_2) \subseteq \Sigma T_1 \cup T_3$,
	\item[(S13)] $S(T_1,T_3) \subseteq T_1 \cup T_4 \cup (T_3+T_1)$,
	\item[(S15)] $S(T_1,T_5) \subseteq T_1\cup T_4$,
	\item[(S18)] $S(T_1,T_8) \subseteq T_1\cup T_4,$
	\item[(S23)] $S(T_2,T_3) \subseteq (T_7+T_1) \cup (T_2+T_5) \cup (T_2+T_6+T_3) \cup (T_2+T_5+T_3+T_5) \cup (T_2+T_6+T_3+T_6)$ (5 cases),
	\item[(S25)] $S(T_2,T_5) \subseteq (T_2+T_1+T_3+T_6+T_1) \cup (T_2+T_1+T_3+T_1) \cup (T_8+T_1+T_3+T_6) \cup (T_8+T_1+T_3)$ (4 cases),
	\item[(S26)] $S(T_2,T_6) \subseteq (T_7+T_1+T_3+T_6) \cup (T_7+T_1+T_3) \cup (T_2+T_1+ T_3+T_5+T_1) \cup (T_2+T_1+T_3+T_6+T_1+T_3)$ (4 cases),
	\item[(S27)] $S(T_2,T_7) \subseteq \Sigma T_1 \cup (T_3+T_1)$,
	\item[(S36)] $S(T_3,T_6) \subseteq \Sigma T_1 \cup (T_1+T_3+T_1) \cup (T_1+T_4)$ (3 cases),
	\item[(S47)] $S(T_4,T_7) \subseteq T_1 \cup (T_3+T_1) $,
	\item[(S78)] $S(T_7,T_8) \subseteq \Sigma T_1 \cup T_3$,
	\item[(S0)] $S(T_i,T_j)=\{0\}$ if $i=j \notin \{5,6\}$ or $(i,j)=(1,4)$ (7 pairs),
	\item[(S$T_1$)] $S(T_i,T_j)\subseteq \Sigma T_1$ where $(i,j)$ belongs to $\{(1,6), (1,7), (2,4), (2,8), (3,4),$ $(3,5), (3,7)$, $(3,8), (4,5), (4,6), (4,8), (5,5), (5,6), (5,7), (5,8),$ $(6,6), (6,7), (6,8)$ (totally 18 pairs).
\end{enumerate}
We will prove the first four statements (S12), (S13), (S15), (S18) in \Cref{lem_double2powers_S12-18}, and treat the remaining statements separately in \Cref{lem_double2powers_S23}--\Cref{lem_double2powers_trivialSpairs}. The most complicated statements about these $S$-pairs are the six statements (S23), (S25), (S26), (S27), (S36), and (S47), which will be treated in Lemmas \ref{lem_double2powers_S23}--\ref{lem_double2powers_S47}.

In the following results, we use the forms of the minimal generators of $T_i, 1\le i\le 8$, as in \Cref{tab_mingens_GB_double2powers}.

\begin{lem}
	\label{lem_double2powers_S12-18}
	The following statements hold:
	\begin{enumerate}
		\item[\textup{(S12)}] $S(T_1,T_2) \subseteq \Sigma T_1 \cup T_3$ and  $T_3\subseteq T_1+T_2$.
		\item[\textup{(S13)}] $S(T_1,T_3) \subseteq T_1 \cup T_4 \cup (T_3+T_1)$ and $T_4\subseteq T_1+T_3$.
		\item[\textup{(S15)}] $S(T_1,T_5) \subseteq T_1\cup T_4$.
		\item[\textup{(S18)}] $S(T_1,T_8) \subseteq T_1\cup T_4$.
	\end{enumerate}
\end{lem}

\begin{proof}
	We proceed case by case.
	
	\textbf{(S12)}: We have
	\begin{align*}
		S(f_1,f_2) &= S(x^{3i_1}y^{3n-3i_1},\udl{x^ky^ka^k}+x^{2k}b^k+y^{2k}b^k)\\
		&=(\udl{x^{2k}}+y^{2k})x^{\alpha-k}y^{\beta-k}b^k
	\end{align*}
	where $\alpha=\max\{3i_1, k\}, \beta=\max\{3n-3i_1,k\}$.
	
	Assume that $S(f_1,f_2)\notin \Sigma T_1$. Then, $\alpha+\beta \le 3n+1$. This, together with $k\equiv 2$ (mod 3), yields
	\begin{align*}
		k\le \alpha &\le 3i_1+ 1 \Longrightarrow k\le 3i_1-1 \Longrightarrow \alpha=3i_1,\\
		k\le \beta  &\le 3n-3i_1+1 \Longrightarrow k\le 3n-3i_1-1 \Longrightarrow \beta=3n-3i_1.
	\end{align*}
	
	\begin{center}
	\begin{table}
		\caption{The natural generators of 8 types of ideals featuring in the Gr\"obner basis of $Q^n+(f^k)$, where $n,k$ are power of $2$, $n\ge 2k$, $\log_2 k$ is odd and $\ge 3$.}
		\begin{tabular}{| c | c | c |}
			\hline
			\multirow{2}{*}{$i$} & Degree of generators & \multirow{2}{*}{Form of generators of $T_i$} \\
			&  of $T_i$            &    \\
			\hline
			1 & $3n$     & $f_1=x^{3i_1}y^{3n-3i_1}, 0\le i_1\le n$\\
			\hline
			2 & $3k$     & $f_2=x^ky^ka^k+x^{2k}b^k+y^{2k}b^k$\\
			\hline
			\multirow{2}{*}{3} & \multirow{2}{*}{$3n+k$}   & $f_3=(x^{2k}+y^{2k})x^{3i_3+1}y^{3n-2k-3i_3-1}b^k, 0\le i_3 \le n-\frac{2k+2}{3}$\\
			  &          & $=x^{2k+3i_3+1}y^{3n-2k-3i_3-1}b^k+x^{3i_3+1}y^{3n-3i_3-1}b^k$ \\
			\hline
			4 & $3n+k+1$ & $f_4=x^{3i_4+2}y^{3n-3i_4-1}b^k, 0\le i_4 \le \frac{2k-4}{3}$\\
			\hline
			\multirow{3}{*}{5} & \multirow{3}{*}{$3n+2k$}  & $f_5=(x^{6k\tau+4k}+y^{6k\tau+4k})x^{3i_5+2}y^{3n-6k\tau-4k-3i_5-2}b^{2k}$ \\
						&          &  $=x^{6k\tau+4k+3i_5+2}y^{3n-6k\tau-4k-3i_5-2}b^{2k}+x^{3i_5+2}y^{3n-3i_5-2}b^{2k}$ \\
			&          &  $0\le \tau \le \frac{n}{2k}-1, 0\le i_5 \le \frac{2k-4}{3}$ \\
			\hline
			\multirow{3}{*}{6} & \multirow{3}{*}{$3n+2k$}  & $f_6=(x^{6k\ell}+y^{6k\ell})x^{3i_6+1}y^{3n-6k\ell-3i_6-1}b^{2k}$\\
			                   &                           & $=x^{6k\ell+3i_6+1}y^{3n-6k\ell-3i_6-1}b^{2k}+x^{3i_6+1}y^{3n-3i_6-1}b^{2k},$ \\
			&          &  $1\le \ell \le \frac{n}{2k}-1, 0\le i_6 \le \frac{4k-2}{3}$ \\
			\hline
			\multirow{2}{*}{7} &  \multirow{2}{*}{$3n+2k$} & $f_7=(y^{3k}a^k+x^{3k}b^k)x^{3i_7+1}y^{3n-3k-3i_7-1}b^k$\\
			                   &                           & $=x^{3i_7+1}y^{3n-3i_7-1}a^kb^k+x^{3i_7+3k+1}y^{3n-3i_7-3k-1}b^{2k}, 0\le i_7 \le \frac{k-2}{3}$ \\
			\hline
			\multirow{2}{*}{8} & \multirow{2}{*}{$3n+3k$}  & $f_8=(y^ka^k+x^kb^k)x^{3i_8+2}y^{3n-k-3i_8-2}b^{2k}$\\
			                   &                           & $=x^{3i_8+2}y^{3n-3i_8-2}a^kb^{2k}+x^{3i_8+k+2}y^{3n-k-3i_8-2}b^{3k}, 0\le i_8 \le \frac{k-5}{3}$ \\
			\hline
		\end{tabular}
		\label{tab_mingens_GB_double2powers}
	\end{table}
\end{center}

\noindent Hence, thanks to \Cref{lem_inQn}(3), we have
	\begin{align*}
		S(f_1,f_2) &= (x^{2k}+y^{2k})x^{3i_1-k}y^{3n-3i_1-k}b^k=(x^{2k}+y^{2k})xyx^{3i_1-k-1}y^{3n-3i_1-k-1}b^k \in T_3.
	\end{align*}
	
	It remains to show that $T_3\subseteq T_1+T_2$. Given $i_3\in [0,n-(2k+2)/3]$, choose $i_1=i_3+\frac{k+1}{3} \in [(k+1)/3, n-(k+1)/3]$. The above argument shows that
	\[
	f_3=S(f_1,f_2) \in T_1+T_2.
	\]

	\textbf{(S13)}: We have
	\begin{align*}
		S(f_1,f_3) &= S(x^{3i_1}y^{3n-3i_1},\udl{x^{2k+3i_3+1}y^{3n-2k-3i_3-1}b^k}+x^{3i_3+1}y^{3n-3i_3-1}b^k)\\
		&=x^{\alpha-2k}y^{\beta+2k}b^k
	\end{align*}
	where $\alpha=\max\{3i_1, 2k+3i_3+1\}, \beta=\max\{3n-3i_1,3n-2k-3i_3-1\}$.
	
	Assume that $S(f_1,f_3)\notin T_1$. Then, $\alpha+\beta\le 3n+1$. This, together with $k\equiv 2$ (mod 3), yields
	\[
	2k+3i_3+1 \le \alpha \le 3i_1+1 \Longrightarrow 2k+3i_3+1\le 3i_1-1.
	\]
	Thus, $\frac{2k+2}{3}+i_3\le i_1,\, \alpha=3i_1$. Also,
	\[
	3n-2k-3i_3-1\le \beta \le 3n-3i_1+1 \Longrightarrow i_1 \le i_3+\frac{2k+2}{3}.
	\]
	In particular, we get an equality $i_1=i_3+\frac{2k+2}{3}$, which yields $\beta=3n-2k-3i_3-1$. It follows that
	\begin{align*}
		S(f_1,f_3) &= x^{3i_1-2k}y^{3n-3i_3-1}b^k=x^{3i_3+2}y^{3n-3i_3-1}b^k.
	\end{align*}

	\textbf{Case 1:} $i_3\le \frac{2k-4}{3}$. Then, $S(f_1,f_3)\in T_4$.
	
	\textbf{Case 2:} $i_3\ge \frac{2k-1}{3}$. Then, using \Cref{lem_inQn}(3), we get
	\begin{align*}
		S(f_1,f_3) &= x^{3i_3+2}y^{3n-3i_3-1}b^k = (x^{2k}+y^{2k})x^{3i_3-2k+2}y^{3n-3i_3-1}b^k+x^{3i_3-2k+2}y^{3n+2k-3i_3-1}b^k\\
		&=\udb{(x^{2k}+y^{2k})xyx^{3i_3-2k+1}y^{3n-3i_3-2}b^k}_{\in T_3}  + \udb{x^{3i_3-2k+2}y^{3n+2k-3i_3-1}b^k}_{\in T_1},
	\end{align*}
	as desired.
	
	It remains to show that $T_4\subseteq T_1+T_3$. Given $i_4\in [0,(2k-4)/3]$, choose $i_3=i_4$. The above argument (for Case 1) shows that
	\[
	f_4=S(f_1,f_3) \in T_1+T_3.
	\]
	
	\textbf{(S15)}: We have
	\begin{align*}
		S(f_1,f_5) &= S(x^{3i_1}y^{3n-3i_1},\udl{x^{6k\tau+4k+3i_5+2}y^{3n-6k\tau-4k-3i_5-2}b^{2k}}+x^{3i_5+2}y^{3n-3i_5-2}b^{2k})\\
		&=x^{\alpha-6k\tau-4k}y^{\beta+6k\tau+4k}b^{2k}
	\end{align*}
	where $\alpha=\max\{3i_1, 6k\tau+4k+3i_5+2\}, \beta=\max\{3n-3i_1,3n-6k\tau-4k-3i_5-2\}$.

	Assume that $S(f_1,f_5)\notin T_1$. Then, $\alpha+\beta\le 3n+1$. This, together with $-4k-2\equiv 2$ (mod 3), implies that
	\begin{align*}
		3n-6k\tau-4k-3i_5-2 \le \beta \le 3n-3i_1+1 \Longrightarrow 3n-6k\tau-4k-3i_5-2 \le 3n-3i_1-1.
	\end{align*}
	Thus,
	\[
	\beta=3n-3i_1, \quad \text{and}\quad 3i_1+1 \le 6k\tau+4k+3i_5+2,
	\]
	which in turn yields $\alpha=6k\tau+4k+3i_5+2$. Moreover, $\alpha+\beta\le 3n+1$ also forces $6k\tau+4k+3i_5+2 \le \alpha\le 3i_1+1$. Therefore, there are equalities
	\[
	6k\tau+4k+3i_5+2 = \alpha = 3i_1+1.
	\]
	Now,
	\[
	S(f_1,f_5)=x^{\alpha-6k\tau-4k}y^{\beta+6k\tau+4k}b^{2k}=x^2y^{3n-2k+3}x^{3i_5}y^{2k-4-3i_5}b^{2k} \in T_4,
	\]
	as desired.
	
	\textbf{(S18)}: We have
	\begin{align*}
		S(f_1,f_8) &= S(x^{3i_1}y^{3n-3i_1},\udl{x^{3i_8+2}y^{3n-3i_8-2}a^kb^{2k}}+x^{3i_8+k+2}y^{3n-k-3i_8-2}b^{3k})\\
		&=x^{\alpha+k}y^{\beta-k}b^{3k}
	\end{align*}
	where $\alpha=\max\{3i_1, 3i_8+2\}, \beta=\max\{3n-3i_1,3n-3i_8-2\}$.
	
	Assume that $S(f_1,f_8)\notin T_1$. Then, $\alpha+\beta\le 3n+1$. This implies that
	\begin{align*}
		3i_8+2 &\le \alpha \le 3i_1+1 \Longrightarrow 3i_8+2\le 3i_1-1 \Longrightarrow i_8+1\le i_1, \alpha=3i_1,\\
		3n-3i_8-2 &\le \beta \le 3n-3i_1+1 \Longrightarrow i_1\le i_8+1.
	\end{align*}
	Thus, there is an equality $i_1=i_8+1$ and, therefore, $\beta=3n-3i_8-2=3n-3i_1+1$. Now,
	\begin{align*}
		S(f_1,f_8) &= x^{3i_1+k}y^{3n-3i_8-k-2}b^{3k}=x^{3i_8+k+3}y^{3n-3i_8-k-2}b^{3k}\\
		&=x^2y^{3n-2k+3}x^{3i_8+k+1}y^{k-5-3i_8}b^{3k} \in T_4,
	\end{align*}
	noting that $i_8+\frac{k+1}{3} \le \frac{2k-4}{3}$. Hence, the desired containment $S(T_1,T_8) \subseteq T_1\cup T_4$ holds.
\end{proof}

We treat the remaining statements (S23)--(ST1) in the following results \ref{lem_double2powers_S23}--\ref{lem_double2powers_trivialSpairs}. Lemmas \ref{lem_double2powers_S23}, \ref{lem_double2powers_S25}--\ref{lem_double2powers_S47} deal with the 6 most complicated kinds of $S$-pairs among of the 36 possible kinds in  \Cref{prop_GB_double2powers_odd}.
\begin{lem}
	\label{lem_double2powers_S23}
	We have $S(T_2,T_3) \mathop{\xrightarrow{\qquad \qquad \qquad \qquad}}\limits_{T_1\cup T_2 \cup T_3 \cup T_5 \cup T_6 \cup T_7} 0.$ More concretely,
	$$
	S(T_2,T_3) \subseteq (T_7+T_1) \cup (T_2+T_5) \cup (T_2+T_6+T_3) \cup (T_2+T_5+T_3+T_5) \cup (T_2+T_6+T_3+T_6).
	$$
\end{lem}
\begin{proof}
	Letting
	\begin{equation}
		\label{eq_def_beta_S23}
		\beta:=\max \{k, 3n-2k-3i_3-1\},
	\end{equation}
	we have
	\begin{align*}
		S(f_2,f_3) &= S(\udl{x^ky^ka^k}+x^{2k}b^k+y^{2k}b^k, \udl{x^{2k+3i_3+1}y^{3n-2k-3i_3-1}b^k}+x^{3i_3+1}y^{3n-3i_3-1}b^k)\\
		&=f_2\cdot x^{k+3i_3+1}y^{\beta-k}b^k-f_3\cdot y^{\beta-(3n-2k-3i_3-1)}a^k\\
		&=\udl{x^{3i_3+1}y^{\beta+2k}a^kb^k}+x^{3k+3i_3+1}y^{\beta-k}b^{2k}+x^{k+3i_3+1}y^{\beta+k}b^{2k}\\
		&=(y^{3k}a^k+x^{3k}b^k)x^{3i_3+1}y^{\beta-k}b^k+x^{k+3i_3+1}y^{\beta+k}b^{2k}.
	\end{align*}
	
	\textbf{Case 1:} $\beta \ge 3(n-k)+1$. We show that $S(f_2,f_3)$ is in $T_7+T_1$. Indeed, $x^{k+3i_3+1}y^{\beta+k}\in T_1$ since $k+3i_3+1\equiv 0$ (mod 3) and
	\[
	(k+3i_3+1)+(\beta+k)=\beta+2k+3i_3+1\ge 3n \quad \text{by definition \eqref{eq_def_beta_S23}}.
	\]
	Moreover, the first summand in the last expression of $S(f_2,f_3)$ is
	\[
	(y^{3k}a^k+x^{3k}b^k)x^{3i_3+1}y^{\beta-k}b^k=\udb{(y^{3k}a^k+x^{3k}b^k)xy^{3n-4k+1}x^{3i_3}y^{\beta-3(n-k)-1}b^k}_{\in T_7}
	\]
	since the sum of two non-negative numbers $3i_3$ and $\beta-3(n-k)-1$ is
	$$
	3i_3+(\beta-3(n-k)-1)=\beta+3k+3i_3-3n-1 \ge k-2 \quad \text{by \eqref{eq_def_beta_S23},}
	$$
	and $3i_3\equiv 0$ (mod 3). Thus, $S(f_2,f_3)$ is in $T_7+T_1$.
	
	\textbf{Case 2:} $3n-5k+2 \le \beta \le 3(n-k)$. We show that $S(f_2,f_3)$ is in $T_2+T_5$.
	Indeed, note that $3n-2k-3i_3-1 \le \beta \le 3(n-k)$, namely, $3i_3\ge k-1$. This yields $i_3 \ge \frac{k+1}{3}$ since $k\equiv 2$ (mod 3). This justifies the containment in the second line of the following display
	\begin{align*}
		S(f_2,f_3) &= \udl{x^{3i_3+1}y^{\beta+2k}a^kb^k}+x^{3k+3i_3+1}y^{\beta-k}b^{2k}+x^{k+3i_3+1}y^{\beta+k}b^{2k}\\
		&=\udb{(x^ky^ka^k+x^{2k}b^k+y^{2k}b^k)x^{3i_3+1-k}y^{\beta+k}b^k}_{\in T_2}+x^{3k+3i_3+1}y^{\beta-k}b^{2k}+x^{3i_3-k+1}y^{\beta+3k}b^{2k}\\
		&=\udb{f^kx^{3i_3+1-k}y^{\beta+k}b^k}_{\in T_2}+\udb{(x^{4k}+y^{4k})x^2y^{3n-6k+2}b^{2k}x^{3i_3-k-1}y^{\beta-3n+5k-2}}_{\in T_5}.
	\end{align*}
	The last containment holds since the non-negative number $3i_3-k-1$ is divisible by $3$, and its sum with the non-negative number $\beta-3n+5k-2$ is
	\[
	(3i_3-k-1)+(\beta-3n+5k-2)=\beta+3i_3-3n+4k-3 \ge 2k-4  \quad \text{by definition \eqref{eq_def_beta_S23}}.
	\]
	Therefore, $S(f_2,f_3)$ is in $T_2+T_5$.

	\textbf{Case 3:} $3n-9k+1 \le \beta \le 3n-5k+1$. We show that $S(f_2,f_3)$ is in $T_2+T_6+T_3$.
	Indeed, observe first that $3n-2k-3i_3-1\le \beta \le 3n-5k+1$, so $3k-2\le 3i_3$, which yields $i_3\ge k$.
	
	From the expressions of $S(f_2,f_3)$ in the previous case, it suffices to show that
	\[
	S'(f_2,f_3):=S(f_2,f_3)-f^kx^{3i_3+1-k}y^{\beta+k}b^k= \udl{x^{3k+3i_3+1}y^{\beta-k}b^{2k}}+x^{3i_3-k+1}y^{\beta+3k}b^{2k} \in \text{$T_6+T_3$}.
	\]
	Now,
	\begin{align*}
		S'(f_2,f_3) &=\udl{x^{3k+3i_3+1}y^{\beta-k}b^{2k}}+x^{3i_3-k+1}y^{\beta+3k}b^{2k} \\
		&=(x^{6k}+y^{6k}) x^{3i_3-3k+1}y^{\beta-k} b^{2k} +(x^{2k}+y^{2k})x^{3i_3-3k+1}y^{\beta+3k}b^{2k}\\
		&=\udb{(x^{6k}+y^{6k})xy^{3n-10k+1} x^{3(i_3-k)}y^{\beta-3n+9k-1} b^{2k}}_{\in T_6} +\udb{(x^{2k}+y^{2k})xyx^{3(i_3-k)}y^{\beta-1+3k}b^{2k}}_{\in T_3}.
	\end{align*}
	Here, the first containment holds since $i_3\ge k$, $\beta-3n+9k-1\ge 0$, and $3(i_3-k)+(\beta-3n+9k-1) \ge 4k-2$ by definition \eqref{eq_def_beta_S23}.
	The second containment holds since
	\[
	3(i_3-k)+(\beta-1+3k) \ge 3n-(2k+2), \text{ again by definition \eqref{eq_def_beta_S23}}.
	\]
	Therefore, $S(f_2,f_3)$ is in $T_2+T_6+T_3$.
	
	It remains to consider the situation where $\beta \le 3n-9k$. Choose the maximal number $\ell\ge 1$ such that $\beta \le 3n-(6\ell+3)k$, i.e.
	\[
	\ell=\left\lfloor \dfrac{3n-3k-\beta}{6k} \right\rfloor,
	\]
	then $3n-(6\ell+9)k+1\le \beta$. There are two cases left to examine, according to whether or not $\beta \ge 3n-(6\ell+5)+2$.
	
	\textbf{Case 4:} $3n-(6\ell+5)k+2 \le \beta \le 3n-(6\ell+3)k$. We show that $S(f_2,f_3)$ is in $T_2+T_5+T_3+T_5$. As in Case 3, it suffices to show that
	$$
	S'(f_2,f_3)=\udl{x^{3k+3i_3+1}y^{\beta-k}b^{2k}}+x^{3i_3-k+1}y^{\beta+3k}b^{2k} \in \text{$T_5+T_3+T_5$}.
	$$
	
Observe that $3n-2k-3i_3-1\le \beta \le 3n-(6\ell+3)k$, so $3i_3\ge (6\ell+1)k-1$, which yields
	\begin{equation}
		\label{eq_lowerboundi3}
		i_3 \ge \frac{(6\ell+1)k+1}{3}=2k\ell +\frac{k+1}{3}.
	\end{equation}
	We have
	\begin{align*}
		S'(f_2,f_3) &=(x^{6k\ell+4k}+y^{6k\ell+4k})x^{3i_3-6k\ell-k+1}y^{\beta-k}b^{2k} +x^{3i_3-k+1}y^{\beta+3k}b^{2k}+\\
		&\qquad \qquad \qquad + x^{3i_3-6k\ell-k+1}y^{\beta+6k\ell+3k}b^{2k}\\
		&= \udb{(x^{6k\ell+4k}+y^{6k\ell+4k})x^{3i_3-6k\ell-k+1}y^{\beta-k}b^{2k}}_{\text{first summand}} \\
		&\qquad +\udb{(x^{2k}+y^{2k})x^{3i_3-3k+1}y^{\beta+3k}b^{2k}}_{\text{second summand}}+\left(x^{3i_3-3k+1}y^{\beta+5k}+x^{3i_3-6k\ell-k+1}y^{\beta+6k\ell+3k}\right)b^{2k}
	\end{align*}
	\begin{align*}
		\qquad \qquad    &= \udb{(x^{6k\ell+4k}+y^{6k\ell+4k})x^2y^{3n-6k+2-6k\ell}x^{3i_3-6k\ell-k-1}y^{\beta-3n+(6\ell+5)k-2}b^{2k}}_{\text{first summand}}\\
		&\qquad +\udb{(x^{2k}+y^{2k})xyx^{3i_3-3k}y^{\beta+3k-1}b^{2k}}_{\text{second summand}}+\udb{\left(x^{6k\ell-2k}+y^{6k\ell-2k}\right)x^{3i_3-6k\ell-k+1}y^{\beta+5k}b^{2k}}_{\text{third summand}}\\
		&= \udb{(\text{terms})}_{\text{first summand}\, \in \, T_5} +\udb{(\text{terms})}_{\text{second summand} \, \in \, T_3}+\\
		& \qquad \qquad  +\udb{\left(x^{6k(\ell-1)+4k}+y^{6k(\ell-1)+4k}\right)x^2y^{3n-6k\ell+2}x^{3i_3-6k\ell-k-1}y^{\beta-3n+(6\ell+5)k-2}b^{2k}}_{\text{third summand} \, \in \, T_5}.
	\end{align*}
	Here, the first and third containment hold for the same reason: $3i_3-6k\ell-k-1$ is non-negative (by \eqref{eq_lowerboundi3}), divisible by 3, and its sum with $\beta-3n+(6\ell+5)k-2 \ge 0$ is
	$$(3i_3-6k\ell-k-1) +(\beta-3n+(6\ell+5)k-2) =\beta-3n+3i_3+4k-3 \ge 2k-4 \quad \text{by definition \eqref{eq_def_beta_S23}}.$$
	The second containment holds since $3i_3-3k$ is non-negative (by \eqref{eq_lowerboundi3}), divisible by 3, and
	\[
	(3i_3-3k)+(\beta+3k-1)=\beta+3i_3-1 \ge 3n-(2k+2) \quad \text{by definition \eqref{eq_def_beta_S23}}.
	\]
	Therefore, $S(f_2,f_3)$ is in $T_2+T_5+T_3+T_5$.
	
	\textbf{Case 5:} $3n-(6\ell+9)k+1 \le \beta \le 3n-(6\ell+5)k+1$. We show that $S(f_2,f_3)$ is in $T_2+T_6+T_3+T_6$.
	As in Case 3, it suffices to show that  $$S'(f_2,f_3)=\udl{x^{3k+3i_3+1}y^{\beta-k}b^{2k}}+x^{3i_3-k+1}y^{\beta+3k}b^{2k} \in \text{$T_6+T_3+T_6$}.$$
	
	Note that $3n-2k-3i_3-1\le \beta \le 3n-(6\ell+5)k+1$, so $3i_3\ge (6\ell+3)k-2$.  This yields
	\begin{equation}
		\label{eq_lowerboundi3_case5}
		i_3 \ge (2\ell+1)k.
	\end{equation}
	We have
	\begin{align*}
		S'(f_2,f_3) &=(x^{6k(\ell+1)}+y^{6k(\ell+1)})x^{3i_3-(6\ell+3)k+1}y^{\beta-k}b^{2k} +\\
		& \qquad +x^{3i_3-k+1}y^{\beta+3k}b^{2k}+x^{3i_3-(6\ell+3)k+1}y^{\beta+(6\ell+5)k}b^{2k}\\
		&= \udb{(x^{6k(\ell+1)}+y^{6k(\ell+1)})x^{3i_3-(6\ell+3)k+1}y^{\beta-k}b^{2k}}_{\text{first summand}}+ \udb{(x^{2k}+y^{2k})x^{3i_3-3k+1}y^{\beta+3k}b^{2k}}_{\text{second summand}}\\
		&\qquad \qquad  +\udb{\left(x^{3i_3-3k+1}y^{\beta+5k}+x^{3i_3-(6\ell+3)k+1}y^{\beta+(6\ell+5)k}\right)b^{2k}}_{\text{third summand}}. \\
	\end{align*}
	In other words,
	\begin{align*}
		S'(f_2,f_3) &= \udb{(x^{6k(\ell+1)}+y^{6k(\ell+1)})xy^{3n-4k+1-6k(\ell+1)}x^{3(i_3-(2\ell+1)k)}y^{\beta-3n+(6\ell+9)k-1}b^{2k}}_{\text{first summand} \, \in \, T_6}\\
		&\qquad +\udb{(x^{2k}+y^{2k})xyx^{3i_3-3k}y^{\beta+3k-1}b^{2k}}_{\text{second summand} \, \in \, T_3}+\udb{\left(x^{6k\ell}+y^{6k\ell}\right)x^{3i_3-(6\ell+3)k+1}y^{\beta+5k}b^{2k}}_{\text{third summand}}\\
		&= \udb{(\text{terms})}_{\text{first summand}\, \in \, T_6} +\udb{(\text{terms})}_{\text{second summand} \, \in \, T_3}+\\
		& \qquad \qquad \qquad +\udb{\left(x^{6k\ell}+y^{6k\ell}\right)xy^{3n-4k+1-6k\ell}x^{3(i_3-(2\ell+1)k)}y^{\beta-3n+(6\ell+9)k-1}b^{2k}}_{\text{third summand} \, \in \, T_6}.
	\end{align*}
	Here, the first and third summands in the last display belong to $T_6$ for the same reason:  $3(i_3-(2\ell+1)k)$ is non-negative (by \eqref{eq_lowerboundi3_case5}), divisible by 3, and its sum with $\beta-3n+(6\ell+9)k-1 \ge 0$ is
	\[
	3(i_3-(2\ell+1)k) +(\beta-3n+(6\ell+9)k-1) =\beta-3n+3i_3+6k-1 \ge 4k-2 \quad \text{by definition \eqref{eq_def_beta_S23}}.
	\]
	The second command is in $T_3$ since $3i_3-3k$ is non-negative (by \eqref{eq_lowerboundi3_case5}), divisible by 3, and
	\[
	(3i_3-3k)+(\beta+3k-1)=\beta+3i_3-1 \ge 3n-(2k+2) \quad \text{by definition \eqref{eq_def_beta_S23}}.
	\]
	Therefore, $S(f_2,f_3)$ is in $T_2+T_6+T_3+T_6$. This concludes the proof.
\end{proof}

The proof of \Cref{lem_double2powers_S23} implies the following corollary.
\begin{cor}
	\label{cor_double2powers_T567containments}
	There are ideal containments:
	\begin{enumerate}[\quad \rm (1)]
		\item $T_5\subseteq T_2+T_3$,
		\item $T_6\subseteq T_2+T_3$,
		\item $T_7\subseteq T_1+T_2+T_3$.
	\end{enumerate}
\end{cor}
\begin{proof}
	
	(1) Given $0\le \tau \le \frac{n}{2k}-1$, $0\le i_5 \le \frac{2k-4}{3}$ and
	\begin{align*}
		f_5&:=(x^{6k\tau+4k}+y^{6k\tau+4k})x^2y^{3n-6k+2-6k\tau}x^{3i_5}y^{2k-4-3i_5}b^{2k}.
	\end{align*}
	\textbf{Case 1a:} $\tau=0$. Let $i_3=i_5+\frac{k+1}{3} \in [(k+1)/3,k-1]$. We check that $\beta$ given in \eqref{eq_def_beta_S23}, i.e. $\beta:=\max \{k, 3n-2k-3i_3-1\}$, satisfies
	\begin{enumerate}[\quad \rm (i)]
		\item  $\beta =3n-2k-3i_3-1=3n-3k-2-3i_5,$
		\item $3n-5k+2  \le \beta \le 3(n-k)$.
	\end{enumerate}
	Indeed, (i) is equivalent to
	\begin{align*}
		3n-3k-2-3i_5 & \ge k \Longleftrightarrow 3i_5 \le 3n-4k-2,
	\end{align*}
	which holds since $3i_5\le 2k-4$, and $3n-4k-2 \ge 6k-4k-2=2k-2$.
	
	Since $\beta=3n-3k-2-3i_5$, (ii) is equivalent to
	\begin{align*}
		3n-5k+2  &\le 3n-3k-2-3i_5 \le 3(n-k) \Longleftrightarrow 0\le 3i_5+2 \le 2k-2,
	\end{align*}
	which is true.
	
	Arguing as for Case 2 in the proof of \Cref{lem_double2powers_S23} above, we then get
	\[
	f_5=S(f_2,f_3)+\udb{(\text{terms})}_{\,\in\, T_2} \in T_2+T_3.
	\]
	
	\textbf{Case 1b:} $\tau \ge 1$. Choose $i_3=i_5+2k\tau+\frac{k+1}{3}$ and $\ell=\tau$. We check that $\beta$ given in \eqref{eq_def_beta_S23}, i.e. $\beta:=\max \{k, 3n-2k-3i_3-1\}$, satisfies
	\begin{enumerate}[\quad \rm (i)]
		\item  $\beta =3n-2k-3i_3-1=3n-(6\tau+5)k+2k-2-3i_5$,
		\item $\ell=\left\lfloor \dfrac{3n-3k-\beta}{6k} \right\rfloor$,
		\item $3n-(6\tau+5)k+2  \le \beta \le 3n-(6\tau+3)k$.
	\end{enumerate}
	
	Indeed, (i) is equivalent to
	\begin{align*}
		3n-2k-3i_3-1 =3n-(6\tau+5)k+2k-2-3i_5 \ge k \Longleftrightarrow 3i_5 \le 3n-6k\tau-4k-2.
	\end{align*}
	Since $\tau\le \frac{n}{2k}-1$, we get
	\begin{align*}
		3n-6k\tau-4k-2 & \ge 2k-2 > 3i_5,
	\end{align*}
	implying (i).
	
	As $\ell=\tau$ and $\beta=3n-(6\tau+5)k+2k-2-3i_5$, (ii) is equivalent to
	\begin{align*}
		\tau &\le \dfrac{3n-3k-\beta}{6k} < \tau+1 \Longleftrightarrow \tau \le \dfrac{3i_5+6k\tau+2}{6k} < \tau+1 \\
		\Longleftrightarrow 0 &\le 3i_5+2 < 6k,
	\end{align*}
	which is true. Thus (ii) is confirmed.
	
	Next, (iii) is equivalent to
	\begin{align*}
		3n-(6\tau+5)k+2  &\le 3n-(6\tau+5)k+2k-2-3i_5 \le 3n-(6\tau+3)k \\
		\Longleftrightarrow 0 &\le 3i_5+2 \le 2k-2,
	\end{align*}
	which is true.
	
	Now, let $f'_5$ be obtained from $f_5$ by replacing $\tau$ with $\tau-1$, i.e.,
	\[
	f'_5= (x^{6k(\tau-1)+4k}+y^{6k(\tau-1)+4k})x^2y^{3n-6k+2-6k(\tau-1)}x^{3i_5}y^{2k-4-3i_5}b^{2k}.
	\]
	The argument for Case 4 in the proof of \Cref{lem_double2powers_S23} implies that
	\[
	S(f_2,f_3) =\udb{(\text{terms})}_{\,\in\, T_2}+ f_5+ \udb{(\text{terms})}_{\,\in\, T_3}+ f'_5 \Longrightarrow f_5+f'_5 \in T_2+T_3.
	\]
	By induction on $\tau$, we deduce that $f_5\in T_2+T_3$, as desired. Thus, the containment $T_5\subseteq T_2+T_3$ holds.

	(2) The proof is quite similar to part (1); so we will skip some minor details. Given $1\le \ell \le \frac{n}{2k}-1$, $0\le i_6 \le \frac{4k-2}{3}$ and
	\[
	f_6 :=(x^{6k\ell}+y^{6k\ell})xy^{3n-4k+1-6k\ell}x^{3i_6}y^{4k-2-3i_6}b^{2k}.
	\]
	
	\textbf{Case 2a:} $\ell=1$. Choose $i_3=i_6+k$. We check that $\beta$ given by \eqref{eq_def_beta_S23}, i.e. $\beta:=\max \{k, 3n-2k-3i_3-1\}$, satisfies
	\begin{enumerate}[\quad \rm (i)]
		\item $\beta= 3n-2k-1-3i_3 = 3n-5k-1-3i_6$,
		\item $3n-9k+1  \le \beta \le 3n-5k+1$.
	\end{enumerate}
	Indeed, (i) is equivalent to
	\begin{align*}
		3n-5k-1-3i_6 \ge k \Longleftrightarrow 3i_6 \le 3n-6k-1 \Longleftrightarrow i_6 \le n-2k-1.
	\end{align*}
	Since $\frac{n}{2k}-1 \ge \ell =1$, we get $n\ge 4k$, so
	\[
	n-2k-1 \ge 2k-1 > \frac{4k-2}{3} \ge i_6,
	\]
	implying (i). Part (ii) can be checked easily.
	
	The argument for Case 3 in the proof of \Cref{lem_double2powers_S23} then implies that
	\[
	S(f_2,f_3)=\udb{(\text{terms})}_{\,\in\, T_2}+ f_6+ \udb{(\text{terms})}_{\,\in\, T_3},
	\]
	hence $f_6\in T_2+T_3$.

	\textbf{Case 2b:} $\ell\ge 2$. Choose $\ell':=\ell-1$ and $i_3:=i_6+(2\ell-1)k \le \frac{4k-2}{3}+(2\ell-1)k \le n-\frac{2k+2}{3}$ (noting that $\ell \le \frac{n}{2k}-1$). We check that $\beta$ given by \eqref{eq_def_beta_S23}, i.e. $\beta:=\max \{k, 3n-2k-3i_3-1\}$, satisfies
	\begin{enumerate}[\quad \rm (i)]
		\item $\beta= 3n-2k-1-3i_3 = 3n-(6\ell+3)k+4k-1-3i_6$,
		\item $\ell'= \left\lfloor \dfrac{3n-3k-\beta}{6k} \right \rfloor$,
		\item $3n-(6\ell'+9)k+1 \le \beta \le 3n-(6\ell'+5)k+1$.
	\end{enumerate}
	Indeed, (i) is equivalent to $i_6\le n-2k\ell-1$. Since $n\ge 2k(\ell+1)$, we get
	\[
	n-2k\ell-1 \ge 2k-1 > \frac{4k-2}{3}\ge i_6.
	\]
	(ii) is equivalent to
	\begin{align*}
		6k\ell-6k & \le 3n-3k-\beta < 6k\ell \Longleftrightarrow 6k\ell-6k \le 6k\ell+1-4k+3i_6 < 6k\ell \\
		\Longleftrightarrow 0 &\le 3i_6 +2k+1 < 6k,
	\end{align*}
	which is true.
	
	Let $f'_6$ be obtained from $f_6$ by replacing $\ell$ with $\ell-1$, i.e.,
	\[
	f'_6 = (x^{6k(\ell-1)}+y^{6k(\ell-1)})xy^{3n-4k+1-6k(\ell-1)}x^{3i_6}y^{4k-2-3i_6}b^{2k}.
	\]
	Note that $1\le \ell-1.$ The argument for Case 5 in the proof of \Cref{lem_double2powers_S23} implies that
	\[
	S(f_2,f_3)=\udb{(\text{terms})}_{\,\in\, T_2}+ f_6+ \udb{(\text{terms})}_{\,\in\, T_3}+f'_6.
	\]
	Therefore, $f_6+f'_6\in T_2+T_3$. By induction on $\ell$, we get $f_6\in T_2+T_3$. Hence the containment $T_6\subseteq T_2+T_3$ holds.
	
	(3) Given $i_7\in [0,\frac{k-2}{3}]$ and
	\[
	f_7:=(y^{3k}a^k+x^{3k}b^k)xy^{3n-4k+1}x^{3i_7}y^{k-2-3i_7}b^k.
	\]
	Choose $i_3=i_7$. We check that $\beta$ given by \eqref{eq_def_beta_S23}, i.e. $\beta:=\max \{k, 3n-2k-3i_3-1\}$, satisfies
	\[
	\beta = 3n-2k-3i_3-1= 3n-2k-3i_7-1\ge 3(n-k)+1.
	\]
	Indeed, the last inequality holds since $i_7\le \frac{k-2}{3}$. The first equality is equivalent to
	\begin{align*}
		3n-2k-3i_7-1 \ge k \Longleftrightarrow i_7 \le n-k-1,
	\end{align*}
	which is true since $n-k-1 \ge k-1 > \frac{k-2}{3} \ge i_7$.
	
	The argument for Case 1 in the proof of \Cref{lem_double2powers_S23} then implies that
	\[
	S(f_2,f_3) = f_7+ \udb{(\text{terms})}_{\,\in\, T_1}.
	\]
	Hence, $f_7\in T_1+T_2+T_3$. The proof is concluded.
\end{proof}

\begin{lem}
	\label{lem_double2powers_S25}
	We have $S(T_2,T_5) \mathop{\xrightarrow{\qquad \qquad \qquad \qquad}}\limits_{T_1\cup T_2 \cup T_3 \cup T_6 \cup T_8} 0.$ More concretely,
	$$
	S(T_2,T_5) \subseteq (T_2+T_1+T_3+T_6+T_1) \cup (T_2+T_1+T_3+T_1) \cup (T_8+T_1+T_3+T_6) \cup (T_8+T_1+T_3).
	$$
	In addition, $T_8\subseteq T_1+T_2+T_3+T_5+T_6$.
\end{lem}
\begin{proof}
	Letting
	\begin{equation}
		\label{eq_def_beta_S25}
		\beta:=\max \{k, 3n-6k\tau-4k-3i_5-2\},
	\end{equation}
	and with the common highest term $x^{6k\tau+4k+3i_5+2}y^\beta a^kb^{2k}$ in mind, we have
	\begin{align*}
		S(f_2,f_5) &= S(\udl{x^ky^ka^k}+x^{2k}b^k+y^{2k}b^k,\udl{x^{6k\tau+4k+3i_5+2}y^{3n-6k\tau-4k-3i_5-2}b^{2k}}+x^{3i_5+2}y^{3n-3i_5-2}b^{2k})\\
		&=f_2\cdot x^{6k\tau+3k+3i_5+2}y^{\beta-k}b^{2k}-f_5\cdot y^{\beta-(3n-6k\tau-4k-3i_5-2)}a^k\\
		&= \udl{x^{3i_5+2}y^{\beta+6k\tau+4k}a^kb^{2k}}+x^{6k\tau+5k+3i_5+2}y^{\beta-k}b^{3k}+x^{6k\tau+3k+3i_5+2}y^{\beta+k}b^{3k}.
	\end{align*}
	
	\textbf{Case 1:} $i_5\ge \frac{k-2}{3}$. We show that $S(f_2,f_5)$ belongs to $T_2+T_1+T_3+T_6+T_1$ if  $\tau\ge 1$, and $T_2+T_1+T_3+T_1$ if $\tau=0$. Indeed,
	\begin{align*}
		S(f_2,f_5) &= \udl{x^{3i_5+2}y^{\beta+6k\tau+4k}a^kb^{2k}}+x^{6k\tau+5k+3i_5+2}y^{\beta-k}b^{3k}+x^{6k\tau+3k+3i_5+2}y^{\beta+k}b^{3k} \\
		&=\udb{(x^ky^ka^k+x^{2k}b^k+y^{2k}b^k)x^{3i_5-k+2}y^{\beta+6k\tau+3k}b^{2k}}_{\text{first summand}}+ \udb{x^{6k\tau+5k+3i_5+2}y^{\beta-k}b^{3k}}_{\text{second summand}}\\
		& \qquad \qquad + (x^{6k\tau+3k+3i_5+2}y^{\beta+k}+x^{3i_5+k+2}y^{\beta+6k\tau+3k}+x^{3i_5-k+2}y^{\beta+6k\tau+5k})b^{3k}\\
		&= \udb{(x^ky^ka^k+x^{2k}b^k+y^{2k}b^k)x^{3i_5-k+2}y^{\beta+6k\tau+3k}b^{2k}}_{\text{first summand} \,\in\, T_2}+ \udb{x^{6k\tau+5k+3i_5+2}y^{\beta-k}b^{3k}}_{\text{second summand} \, \in \, T_1} +\\
		&  \qquad \qquad + (x^{6k\tau+3k+3i_5+2}y^{\beta+k}+x^{3i_5+k+2}y^{\beta+6k\tau+3k}+x^{3i_5-k+2}y^{\beta+6k\tau+5k})b^{3k}.
	\end{align*}
	The second summand belongs to $T_1$ since $6k\tau+5k+3i_5+2 \equiv 0$ (mod 3) and
	\[
	(6k\tau+5k+3i_5+2)+(\beta-k)= \beta+6k\tau+4k+3i_5+2 \ge 3n \quad \text{by definition \eqref{eq_def_beta_S25}}.
	\]
	Subtracting the first two summands from $S(f_2,f_5)$, it remains to show that
	\[
	S'(f_2,f_5) := (x^{6k\tau+3k+3i_5+2}y^{\beta+k}+x^{3i_5+k+2}y^{\beta+6k\tau+3k}+x^{3i_5-k+2}y^{\beta+6k\tau+5k})b^{3k}
	\]
	belongs to $T_3+T_6+T_1$ if $\tau\ge 1$ and $T_3+T_1$ if $\tau=0$.
	
	To achieve this, observe that
	\begin{align*}
		S'(f_2,f_5) & =(x^{2k}+y^{2k})x^{6k\tau+k+3i_5+2}y^{\beta+k}b^{3k}+\\
		& \qquad \qquad  \left(x^{6k\tau+k+3i_5+2}y^{\beta+3k}+x^{3i_5+k+2}y^{\beta+6k\tau+3k}+x^{3i_5-k+2}y^{\beta+6k\tau+5k}\right)b^{3k}\\
		& =\udb{(x^{2k}+y^{2k})x^{6k\tau+k+3i_5+2}y^{\beta+k}b^{3k}}_{\text{first summand}}+\\
		& \qquad \qquad  \udb{\left(x^{6k\tau}+y^{6k\tau}\right)x^{3i_5+k+2}y^{\beta+3k}b^{3k}}_{\text{second summand}}+\udb{x^{3i_5-k+2}y^{\beta+6k\tau+5k}b^{3k}}_{\text{third summand}}
	\end{align*}
	\begin{align*}
		& =\udb{(x^{2k}+y^{2k})xyx^{6k\tau+k+3i_5+1}y^{\beta+k-1}b^{3k}}_{\text{first summand}\, \in \, T_3}+\\
		& \qquad \qquad  \udb{\left(x^{6k\tau}+y^{6k\tau}\right)x^{3i_5+k+2}y^{\beta+3k}b^{3k}}_{\text{second summand}}+\udb{x^{3i_5-k+2}y^{\beta+6k\tau+5k}b^{3k}}_{\text{third summand}\, \in \, T_1}.
	\end{align*}
	The first summand belongs to $T_3$ since $6k\tau+k+3i_5+1\equiv 0$ (mod 3) and
	\[
	(6k\tau+k+3i_5+1)+(\beta+k-1)= \beta+6k\tau+2k+3i_5 \ge 3n-(2k+2) \quad \text{by definition \eqref{eq_def_beta_S25}}.
	\]
	The third summand belongs to $T_1$ since $3i_5-k+2$ is $\ge 0$ and divisible by 3, and
	\[
	(3i_5-k+2)+(\beta+6k\tau+5k)= \beta+6k\tau+4k+3i_5+2 \ge 3n \quad \text{again by definition \eqref{eq_def_beta_S25}}.
	\]
	Therefore, we need to show that the second command
	\[
	\left(x^{6k\tau}+y^{6k\tau}\right)x^{3i_5+k+2}y^{\beta+3k}b^{3k}
	\]
	belongs to $T_6$ if $\tau\ge 1$ and equals to 0 if $\tau=0$.
	
	This is clear for $\tau=0$ as $\chara \kk=2$. If $\tau\ge 1$, then
	\[
	\left(x^{6k\tau}+y^{6k\tau}\right)x^{3i_5+k+2}y^{\beta+3k}b^{3k}=\udb{\left(x^{6k\tau}+y^{6k\tau}\right)xy^{3n-4k+1-6k\tau}x^{3i_5+k+1}y^{\beta-3n+6k\tau+7k-1}b^{3k}}_{\in T_6}.
	\]
	Here, for the last containment, it suffices to note that
	\begin{align*}
		3i_5+k+1 &\equiv 0 \quad (\text{mod 3}),\\
		\beta-3n+6k\tau+7k-1 &\ge (3n-6k\tau-4k-3i_5-2)-3n+6k\tau+7k-1 \\
		& =3(k-1-i_5)  >0,\\
		( 3i_5+k+1)+(\beta-3n+6k\tau+7k-1) &= \beta-3n+6k\tau+8k+3i_5 \ge 4k-2 \\
		 & \quad \text{by definition \eqref{eq_def_beta_S25}}.
	\end{align*}
	Hence, we are done in Case 1.
	
	\textbf{Case 2:} $i_5\le \frac{k-5}{3}$. We show that $S(f_2,f_5)$ belongs to $T_8+T_1+T_3+T_6$ if  $\tau \ge 1$, and $T_8+T_1+T_3$ if $\tau=0$.
	
	We have
	\begin{align*}
		S(f_2,f_5) &= \udl{x^{3i_5+2}y^{\beta+6k\tau+4k}a^kb^{2k}}+x^{6k\tau+5k+3i_5+2}y^{\beta-k}b^{3k}+x^{6k\tau+3k+3i_5+2}y^{\beta+k}b^{3k} \\
		&=\udb{(y^ka^k+x^kb^k)x^{3i_5+2}y^{\beta+6k\tau+3k}b^{2k}}_{\text{first summand}}+\udb{x^{6k\tau+5k+3i_5+2}y^{\beta-k}b^{3k}}_{\text{second summand}}+\\
		& \qquad \qquad + x^{6k\tau+3k+3i_5+2}y^{\beta+k}b^{3k} + x^{3i_5+k+2}y^{\beta+6k\tau+3k}b^{3k} \\
		&=\udb{(y^ka^k+x^kb^k)x^2y^{3n-2k+3}x^{3i_5}y^{\beta-3n+6k\tau+5k-3}b^{2k}}_{\text{first summand}\, \in \, T_8}+\udb{x^{6k\tau+5k+3i_5+2}y^{\beta-k}b^{3k}}_{\text{second summand} \, \in \, T_1}+ \\
		&\qquad \qquad +\udb{(x^{6k\tau+3k+3i_5+2}y^{\beta+k}+x^{3i_5+k+2}y^{\beta+6k\tau+3k})b^{3k}}_{\text{third summand}}.
	\end{align*}
	Here, the first summand belongs to $T_8$ because of the following:
	\begin{align*}
		3i_5 &\equiv 0 \quad \text{(mod 3)},\\
		\beta-3n+6k\tau+5k-3  &\ge (3n-6k\tau-4k-3i_5-2)-3n+6k\tau+5k-3\\
		&= k-5-3i_5 \ge 0, \text{ and hence,}\\
		3i_5+\left(\beta-3n+6k\tau+5k-3\right) & \ge k-5.
	\end{align*}
	The second summand belongs to $T_1$ because of the following:
	\begin{align*}
		6k\tau+5k+3i_5+2 &\equiv 0 \quad \text{(mod 3)},\\
		(6k\tau+5k+3i_5+2)+(\beta-k) &= \beta+6k\tau+4k+3i_5+2 \ge 3n \quad \text{by definition \eqref{eq_def_beta_S25}}.
	\end{align*}
	Thus, it remains to show that subtracting the first two summands from $S(f_2,f_5)$, we get
	\[
	S''(f_2,f_5):=(x^{6k\tau+3k+3i_5+2}y^{\beta+k}+x^{3i_5+k+2}y^{\beta+6k\tau+3k})b^{3k}
	\]
	belongs to $T_3+T_6$ if  $\tau\ge 1$, and $T_3$ if $\tau=0$.
	
	To see this, observe that
	\begin{align*}
		S''(f_2,f_5)&=\left((x^{2k}+y^{2k})x^{6k\tau+k+3i_5+2}y^{\beta+k}+(x^{6k\tau}+y^{6k\tau})x^{3i_5+k+2}y^{\beta+3k}\right)b^{3k}\\
		&=\udb{(x^{2k}+y^{2k})xyx^{6k\tau+k+3i_5+1}y^{\beta+k-1}b^{3k}}_{\text{first summand} \, \in \, T_3} \\
		&\qquad \qquad +\udb{(x^{6k\tau}+y^{6k\tau})xy^{3n-4k+1-6k\tau}x^{3i_5+k+1}y^{\beta-3n+7k+6k\tau-1}b^{3k}}_{\text{second summand}}.
	\end{align*}
	The first summand of $S''(f_2,f_5)$ is nothing but the first summand of $S'(f_2,f_5)$ in Case 1 and, by the same arguments as above, it belongs to $T_3$. Also, arguing as in Case 1, the second summand of $S''(f_2,f_5)$, which is the second summand of $S'(f_2,f_5)$, belongs to $T_6$ if $\tau \ge 1$ and equals to 0 if $\tau =0$. This concludes the proof of the first assertion.
	
	It remains to show that $T_8\subseteq T_1+T_2+T_3+T_5+T_6.$ Given $i_8\in [0,\frac{k-5}{3}]$ and
	\[
	f_8=(y^ka^k+x^kb^k)x^{3i_8+2}y^{3n-k-3i_8-2}b^{2k}.
	\]
	Choose $i_5=i_8$, $\tau=\dfrac{n}{2k} -1$. Note that $\tau \ge 0$ since $n\ge 2k$. Let $\beta$ be defined as in \eqref{eq_def_beta_S25}, i.e., $\beta:=\max \{k, 3n-6k\tau-4k-3i_5-2\}$. We check that
	\[
	\beta = 3n-6k\tau-4k-3i_5-2=3n-6k\tau-4k-3i_8-2.
	\]
	Indeed, this is equivalent to
	\begin{align*}
		&3n-6k\tau-4k-3i_8-2 \ge k \Longleftrightarrow 6k\tau \le 3n-5k-2-3i_8\\
		&\Longleftrightarrow 3n-6k \le 3n-5k-2-3i_8
	\end{align*}
	which is true since $3i_8\le k-2$. The above argument for Case 2 implies that $S(f_2,f_5)-f_8\in T_1+T_3+T_6$ and, hence, $f_8\in T_1+T_2+T_3+T_5+T_6$. This finishes the proof.
\end{proof}

\begin{lem}
	\label{lem_double2powers_S26}
	We have $S(T_2,T_6) \mathop{\xrightarrow{\qquad \qquad \qquad \qquad}}\limits_{T_1\cup T_2 \cup T_3 \cup T_5 \cup T_6 \cup T_7} 0.$ More concretely,
	$$
	S(T_2,T_6) \subseteq (T_7+T_1+T_3+T_6) \cup (T_7+T_1+T_3) \cup (T_2+T_1+T_3+T_5+T_1) \cup (T_2+T_1+T_3+T_6+T_1+T_3).
	$$
\end{lem}

\begin{proof}
	Consider
	\begin{align*}
		S(f_2,f_6) &= S(\udl{x^ky^ka^k}+x^{2k}b^k+y^{2k}b^k,\udl{x^{6k\ell+3i_6+1}y^{3n-6k\ell-3i_6-1}b^{2k}}+x^{3i_6+1}y^{3n-3i_6-1}b^{2k}).
	\end{align*}
	Since $1\le \ell \le \frac{n}{2k}-1$ and $0\le i_6 \le \frac{4k-2}{3}$,
	\begin{align*}
		6k\ell+3i_6+1 & \ge 6k+1 >k,\\
		3n-6k\ell-3i_6-1 &\ge 3n-6k\left(\frac{n}{2k}-1\right)-(4k-2)-1=2k+1 >k,
	\end{align*}
	so with the common highest term $x^{6k\ell+3i_6+1}y^{3n-6k\ell-3i_6-1}a^kb^{2k}$ in mind, we have
	\begin{align*}
		S(f_2,f_6) &= S(\udl{x^ky^ka^k}+x^{2k}b^k+y^{2k}b^k,\udl{x^{6k\ell+3i_6+1}y^{3n-6k\ell-3i_6-1}b^{2k}}+x^{3i_6+1}y^{3n-3i_6-1}b^{2k})\\
		&= f_2\cdot x^{6k\ell+3i_6-k+1}y^{3n-6k\ell-3i_6-k-1}b^{2k} - f_6\cdot a^k\\
		&= \udl{x^{3i_6+1}y^{3n-3i_6-1}a^kb^{2k}}+x^{6k\ell+3i_6+k+1}y^{3n-6k\ell-3i_6-k-1}b^{3k}+\\
		&\qquad \qquad + x^{6k\ell+3i_6-k+1}y^{3n-6k\ell-3i_6+k-1}b^{3k}.
	\end{align*}

	\textbf{Case 1:} $i_6\le \frac{k-2}{3}$. We show that $S(f_2,f_6)$ belongs to $T_7+T_1+T_3+T_6$ if $\ell\ge 2$ and $T_7+T_1+T_3$ if $\ell=1$. Indeed, we have
	\begin{align*}
		S(f_2,f_6) &= \udl{x^{3i_6+1}y^{3n-3i_6-1}a^kb^{2k}}+x^{6k\ell+3i_6+k+1}y^{3n-6k\ell-3i_6-k-1}b^{3k}+\\
		&\qquad \qquad + x^{6k\ell+3i_6-k+1}y^{3n-6k\ell-3i_6+k-1}b^{3k}\\
		&=\udb{(y^{3k}a^k+x^{3k}b^k)x^{3i_6+1}y^{3n-3i_6-3k-1}b^{2k}}_{\text{first summand}} + \udb{x^{6k\ell+3i_6+k+1}y^{3n-6k\ell-3i_6-k-1}b^{3k}}_{\text{second summand}}\\
		&\qquad \qquad + \udb{(x^{6k\ell+3i_6-k+1}y^{3n-6k\ell-3i_6+k-1}+x^{3i_6+3k+1}y^{3n-3i_6-3k-1})b^{3k}}_{\text{third summand}}\\
		&=\udb{(y^{3k}a^k+x^{3k}b^k)xy^{3n-4k+1}x^{3i_6}y^{k-2-3i_6}b^{2k}}_{\text{first summand}\,\in \, T_7} + \udb{x^{6k\ell+3i_6+k+1}y^{3n-6k\ell-3i_6-k-1}b^{3k}}_{\text{second summand}\, \in \, T_1}\\
		&\qquad \qquad + \udb{(x^{6k\ell+3i_6-k+1}y^{3n-6k\ell-3i_6+k-1}+x^{3i_6+3k+1}y^{3n-3i_6-3k-1})b^{3k}}_{\text{third summand}}.
	\end{align*}
	The first containment holds since $i_6\le \frac{k-2}{3}$. The second one holds since
	\begin{align*}
		6k\ell+3i_6+k+1 &\equiv 0 \quad \text{(mod 3)},\\
		(6k\ell+3i_6+k+1)+(3n-6k\ell-3i_6-k-1)&=3n.
	\end{align*}
	Thus, it remains to show that the remaining summand
	\[
	S'(f_2,f_6)=(x^{6k\ell+3i_6-k+1}y^{3n-6k\ell-3i_6+k-1}+x^{3i_6+3k+1}y^{3n-3i_6-3k-1})b^{3k}
	\]
	belongs to $T_3+T_6$ if $\ell\ge 2$ and $T_3$ if $\ell=1$. Toward this, observe that
	\begin{align*}
		S'(f_2,f_6)&=(x^{6k\ell+3i_6-k+1}y^{3n-6k\ell-3i_6+k-1}+x^{3i_6+3k+1}y^{3n-3i_6-3k-1})b^{3k}\\
		&=\udb{(x^{2k}+y^{2k})x^{6k\ell+3i_6-3k+1}y^{3n-6k\ell-3i_6+k-1}b^{3k}}_{\text{first summand}} +\\
		&\qquad \qquad + \udb{(x^{6k(\ell-1)}+y^{6k(\ell-1)})x^{3i_6+3k+1}y^{3n-6k\ell-3i_6+3k-1}b^{3k}}_{\text{second summand}}\\
		&=\udb{(x^{2k}+y^{2k})xyx^{6k\ell+3i_6-3k}y^{3n-6k\ell-3i_6+k-2}b^{3k}}_{\text{first summand} \,\in \, T_3} +\\
		&\qquad \qquad + \udb{(x^{6k(\ell-1)}+y^{6k(\ell-1)})xy^{3n-4k+1-6k(\ell-1)}x^{3(i_6+k)}y^{k-2-3i_6}b^{3k}}_{\text{second summand}}.
	\end{align*}
	Here, the first containment holds since
	\[
	(6k\ell+3i_6-3k)+(3n-6k\ell-3i_6+k-2)=3n-(2k+2).
	\]
	It is also clear that the second summand belongs to $T_6$ if $\ell\ge 2$ and equals to 0 if $\ell=1$, as claimed.

	\textbf{Case 2:} $\frac{k+1}{3} \le i_6 \le k-1$. We show that $S(f_2,f_6)\in $ $T_2+T_1+T_3+T_5+T_1$.
	Indeed, it can be seen that
	\begin{align*}
		S(f_2,f_6) &= \udl{x^{3i_6+1}y^{3n-3i_6-1}a^kb^{2k}}+x^{6k\ell+3i_6+k+1}y^{3n-6k\ell-3i_6-k-1}b^{3k}+\\
		&\qquad \qquad \qquad \qquad \qquad \qquad \qquad + x^{6k\ell+3i_6-k+1}y^{3n-6k\ell-3i_6+k-1}b^{3k}\\
		&=\udb{(x^ky^ka^k+x^{2k}b^k+y^{2k}b^k)x^{3i_6-k+1}y^{3n-3i_6-k-1}b^{2k}}_{\text{first summand}\, \in \, T_2}+\\
		&\quad  + \udb{x^{6k\ell+3i_6+k+1}y^{3n-6k\ell-3i_6-k-1}b^{3k}}_{\text{second summand} \,\in\, T_1}+\\
		&\quad + \left(x^{6k\ell+3i_6-k+1}y^{3n-6k\ell-3i_6+k-1}+x^{3i_6+k+1}y^{3n-3i_6-k-1}+x^{3i_6-k+1}y^{3n-3i_6+k-1} \right)b^{3k}
	\end{align*}
	\begin{align*}
		\qquad \qquad    &=\udb{(\text{terms})}_{\text{first summand}\, \in \, T_2}  + \udb{(\text{terms})}_{\text{second summand} \,\in\, T_1}+ \udb{(x^{2k}+y^{2k})x^{6k\ell+3i_6-3k+1}y^{3n-6k\ell-3i_6+k-1}b^{3k}}_{\text{third summand}}\\
		&\quad + \left(x^{6k\ell+3i_6-3k+1}y^{3n-6k\ell-3i_6+3k-1}+x^{3i_6+k+1}y^{3n-3i_6-k-1}+x^{3i_6-k+1}y^{3n-3i_6+k-1} \right)b^{3k}\\
		&=\udb{(\text{terms})}_{\text{first summand}\, \in \, T_2}  + \udb{(\text{terms})}_{\text{second summand} \,\in\, T_1}+ \udb{(x^{2k}+y^{2k})xyx^{6k\ell+3i_6-3k}y^{3n-6k\ell-3i_6+k-2}b^{3k}}_{\text{third summand} \,\in\, T_3}\\
		&\quad + \left(x^{6k\ell+3i_6-3k+1}y^{3n-6k\ell-3i_6+3k-1}+x^{3i_6+k+1}y^{3n-3i_6-k-1}+x^{3i_6-k+1}y^{3n-3i_6+k-1} \right)b^{3k}.
	\end{align*}
	Here, the second summand belongs to $T_1$ as explained in Case 1. The third summand belongs to $T_3$ since
	\begin{align*}
		(6k\ell+3i_6-3k)+(3n-6k\ell-3i_6+k-2) =3n-(2k+2).
	\end{align*}
	It remains to show that the remaining summand
	\[
	S''(f_2,f_6)=\left(x^{6k\ell+3i_6-3k+1}y^{3n-6k\ell-3i_6+3k-1}+x^{3i_6+k+1}y^{3n-3i_6-k-1}+x^{3i_6-k+1}y^{3n-3i_6+k-1} \right)b^{3k}
	\]
	belongs to $T_5+T_1$. Toward this, observe that
	\begin{align*}
		S''(f_2,f_6)&=\udb{\left(x^{6k\ell+3i_6-3k+1}y^{3n-6k\ell-3i_6+3k-1}+x^{3i_6-k+1}y^{3n-3i_6+k-1}\right)b^{3k}}_{\text{first summand}}\\
		&\qquad \qquad +x^{3i_6+k+1}y^{3n-3i_6-k-1}b^{3k}\\
		&=\udb{(x^{6k(\ell-1)+4k}+y^{6k(\ell-1)+4k})x^{3i_6-k+1}y^{3n-6k\ell-3i_6+3k-1}b^{3k}}_{\text{first summand}}\\
		&\qquad +\udb{x^{3i_6+k+1}y^{3n-3i_6-k-1}b^{3k}}_{\text{second summand}}\\
		&=\udb{(x^{6k(\ell-1)+4k}+y^{6k(\ell-1)+4k})x^2y^{3n-6k+2-6k(\ell-1)}x^{3i_6-k-1}y^{3(k-1-i_6)}b^{3k}}_{\text{first summand} \,\in\, T_5}\\
		&\qquad +\udb{x^{3i_6+k+1}y^{3n-3i_6-k-1}b^{3k}}_{\text{second summand} \,\in\, T_1}.
	\end{align*}
	The first summand belongs to $T_5$ since
	\begin{align*}
		\min\{3i_6-k-1, 3(k-1-i_6)\} &\ge 0,\\
		3i_6-k-1 &\equiv 0 \quad \text{(mod 3)},\\
		(3i_6-k-1)+(3(k-1-i_6)) &= 2k-4.
	\end{align*}
	The second summand belongs to $T_1$ since
	\begin{align*}
		3i_6+k+1 &\equiv 0 \quad \text{(mod 3)},\\
		(3i_6+k+1)+(3n-3i_6-k-1) &= 3n.
	\end{align*}
	Hence, $S''(f_2,f_6)$ belongs to $T_5+T_1$, as claimed.

	\textbf{Case 3:} $i_6\ge k$. We show that $S(f_2,f_6) \in$ $T_2+T_1+T_3+T_6+T_1+T_3$.
	Arguing as in Case 2, we reduce to showing that
	\[
	S''(f_2,f_6)=\left(x^{6k\ell+3i_6-3k+1}y^{3n-6k\ell-3i_6+3k-1}+x^{3i_6+k+1}y^{3n-3i_6-k-1}+x^{3i_6-k+1}y^{3n-3i_6+k-1} \right)b^{3k}
	\]
	belongs to $T_6+T_1+T_3$.
	
	To achieve this, observe that
	\begin{align*}
		S''(f_2,f_6)&=\udb{(x^{6k\ell}+y^{6k\ell})x^{3i_6-3k+1}y^{3n-6k\ell-3i_6+3k-1}b^{3k}}_{\text{first summand}}+\udb{x^{3i_6+k+1}y^{3n-3i_6-k-1}b^{3k}}_{\text{second summand}}\\
		&\qquad + \udb{(x^{3i_6-k+1}y^{3n-3i_6+k-1}+x^{3i_6-3k+1}y^{3n-3i_6+3k-1})b^{3k}}_{\text{third summand}}\\
		&=\udb{(x^{6k\ell}+y^{6k\ell})xy^{3n-4k+1-6k\ell}x^{3(i_6-k)}y^{7k-2-3i_6}b^{3k}}_{\text{first summand} \,\in\, T_6}+\udb{x^{3i_6+k+1}y^{3n-3i_6-k-1}b^{3k}}_{\text{second summand} \,\in\, T_1}\\
		&\qquad + \udb{(x^{2k}+y^{2k})xyx^{3(i_6-k)}y^{3n-3i_6+k-2}b^{3k}}_{\text{third summand} \,\in\, T_3}.
	\end{align*}
	Here, the first summand belongs to $T_6$ since
	\begin{align*}
		\min\{3(i_6-k), 7k-2-3i_6\} &\ge 0,\\
		3(i_6-k)+(7k-2-3i_6) &=4k-2.
	\end{align*}
	The second summand belongs to $T_1$ as explained in Case 2. The third summand belongs to $T_3$ since
	\begin{align*}
		3(i_6-k) &\ge 0,\\
		3(i_6-k)+(3n-3i_6+k-2) &= 3n-(2k+2).
	\end{align*}
	The proof is concluded.
\end{proof}

\begin{lem}
	\label{lem_double2powers_S27}
	We have $S(T_2,T_7) \mathop{\xrightarrow{\qquad \qquad}}\limits_{T_1\cup T_3} 0.$ More concretely,
	$$
	S(T_2,T_7) \subseteq \Sigma T_1 \cup (T_3+T_1).
	$$
\end{lem}
\begin{proof}
	Consider
	\[
	S(f_2,f_7)=S(\udl{x^ky^ka^k}+x^{2k}b^k+y^{2k}b^k,\udl{x^{3i_7+1}y^{3n-3i_7-1}a^kb^k}+x^{3i_7+3k+1}y^{3n-3i_7-3k-1}b^{2k}).
	\]
	Since $i_7\le \frac{k-2}{3}$, $3n-3i_7-1\ge 3n-k+1\ge 5k+1$, keeping the common highest term $x^ky^{3n-3i_7-1}a^kb^k$ in mind, we get
	\begin{align*}
		S(f_2,f_7) &=f_2\cdot y^{3n-3i_7-k-1}b^k -f_7\cdot x^{k-3i_7-1} \\
		&=\left(\udl{x^{4k}y^{3n-3i_7-3k-1}}+x^{2k}y^{3n-3i_7-k-1}+y^{3n-3i_7+k-1}\right)b^{2k}.
	\end{align*}
	Assume that $S(f_2,f_7)\notin \Sigma T_1$. Then, by \Cref{lem_inQn}, $3n-3i_7+k-1\le 3n+1$. Equivalently, $i_7\ge \frac{k-2}{3}$. Since $i_7\le \frac{k-2}{3}$, the equality holds. Thus,
	\begin{align*}
		S(f_2,f_7) &=  \left(\udl{x^{4k}y^{3n-4k+1}}+x^{2k}y^{3n-2k+1}+y^{3n+1}\right)b^{2k}\\
		&=\udb{(x^{2k}+y^{2k})xyx^{2k-1}y^{3n-4k}b^{2k}}_{\in \, T_3}+ \udb{y^{3n+1}b^{2k}}_{\in\, T_1}.
	\end{align*}
	Here, the first containment  holds since $2k-1\equiv 0$ (mod 3), and $(2k-1)+(3n-4k)=3n-(2k+1)> 3n-(2k+2)$. Therefore, $S(f_2,f_7)\in$ $T_3+T_1$, as claimed.
\end{proof}

\begin{lem}
	\label{lem_double2powers_S36}
	We have $S(T_3,T_6) \mathop{\xrightarrow{\qquad \qquad}}\limits_{T_1\cup T_3 \cup T_4} 0.$ More concretely,
	$$
	S(T_3,T_6) \subseteq \Sigma T_1 \cup (T_1+T_3+T_1) \cup (T_1+T_4).
	$$
\end{lem}
\begin{proof}
	Consider
	\begin{align*}
		S(f_3,f_6) &= S(\udl{x^{2k+3i_3+1}y^{3n-2k-3i_3-1}b^k}+x^{3i_3+1}y^{3n-3i_3-1}b^k,\\
		& \qquad \qquad \udl{x^{6k\ell+3i_6+1}y^{3n-6k\ell-3i_6-1}b^{2k}}+x^{3i_6+1}y^{3n-3i_6-1}b^{2k}).
	\end{align*}
	Letting
	\begin{align*}
		\alpha &= \max\{2k+3i_3+1, 6k\ell+3i_6+1\},\\
		\beta &= \max\{3n-2k-3i_3-1, 3n-6k\ell-3i_6-1\},
	\end{align*}
	we get
	\begin{align*}
		S(f_3,f_6) &= f_3\cdot x^{\alpha-(2k+3i_3+1)}y^{\beta-(3n-2k-3i_3-1)}b^k -f_6\cdot x^{\alpha-(6k\ell+3i_6+1)}y^{\beta-(3n-6k\ell-3i_6-1)}\\
		&= \udl{x^{\alpha-2k}y^{\beta+2k}b^{2k}}+x^{\alpha-6k\ell}y^{\beta+6k\ell}b^{2k}.
	\end{align*}
	Assume that $S(f_3,f_6)\notin \Sigma T_1$. Then, $\alpha+\beta \le 3n+1$. This and considerations modulo 3 yield
	\begin{align*}
		6k\ell+3i_6+1 &\le \alpha\le 2k+3i_3+2  \Longrightarrow 6k\ell+3i_6+1\le 2k+3i_3,\\
		2k+3i_3+1 &\le \alpha \le 6k\ell+3i_6+2 \Longrightarrow 2k+3i_3 \le 6k\ell+3i_6+1.
	\end{align*}
	Therefore,
	\begin{equation}
		\label{eq_condition_S36}
		2k+3i_3 = 6k\ell+3i_6+1
	\end{equation}
	and thus
	\[
	\beta=3n-6k\ell-3i_6-1=3n-2k-3i_3, \alpha=2k+3i_3+1.
	\]
	Hence,
	\begin{align*}
		S(f_3,f_6) &=x^{\alpha-2k}y^{\beta+2k}b^{2k}+x^{\alpha-6k\ell}y^{\beta+6k\ell}b^{2k}\\
		&=\udb{x^{3i_3+1}y^{3n-3i_3}b^{2k}}_{\in T_1}+x^{2k+3i_3-6k\ell+1}y^{3n-2k-3i_3+6k\ell}b^{2k}.
	\end{align*}
	We claim that the remaining summand in the last expression belongs to either $T_3+T_1$ or $T_4$.
	
	\textbf{Case 1:} $i_3\ge 2k\ell$. The remaining summand is
	\begin{align*}
		& x^{2k+3i_3-6k\ell+1}y^{3n-2k-3i_3+6k\ell}b^{2k}\\
		& \qquad = \udb{(x^{2k}+y^{2k})x^{3i_3-6k\ell+1}y^{3n-2k-3i_3+6k\ell}b^{2k}}_{\text{first summand}}+ \udb{x^{3i_3-6k\ell+1}y^{3n-3i_3+6k\ell}b^{2k}}_{\in T_1} \\
		& \qquad = \udb{(x^{2k}+y^{2k})xyx^{3(i_3-2k\ell)}y^{3n-2k-3i_3+6k\ell-1}b^{2k}}_{\text{first summand} \,\in\, T_3} + \udb{x^{3i_3-6k\ell+1}y^{3n-3i_3+6k\ell}b^{2k}}_{\in T_1}.
	\end{align*}
	
	\textbf{Case 2:} $i_3\le 2k\ell-1$. The remaining summand is
	\begin{align*}
		x^{2k+3i_3-6k\ell+1}y^{3n-2k-3i_3+6k\ell}b^{2k}= \udb{x^2y^{3n-2k+3}x^{2k+3i_3-6k\ell-1}y^{6k\ell-3i_3-3}}_{\in T_4}.
	\end{align*}
	This containment holds since $2k+3i_3-6k\ell-1$ is $\ge 0$ (per \eqref{eq_condition_S36}), divisible by 3, and
	\[
	(2k+3i_3-6k\ell-1)+(6k\ell-3i_3-3) = 2k-4.
	\]
	This concludes the proof.
\end{proof}

\begin{lem}
	\label{lem_double2powers_S47}
	We have $S(T_4,T_7) \mathop{\xrightarrow{\qquad \qquad}}\limits_{T_1\cup T_3} 0.$ More concretely,
	$$
	S(T_4,T_7) \subseteq  T_1 \cup (T_3+T_1).
	$$
\end{lem}
\begin{proof}
	Consider
	$$
	S(f_4,f_7) = S(x^{3i_4+2}y^{3n-3i_4-1}b^k,\udl{x^{3i_7+1}y^{3n-3i_7-1}a^kb^k}+x^{3i_7+3k+1}y^{3n-3i_7-3k-1}b^{2k}).
	$$
	Letting $\alpha = \max\{3i_4+2, 3i_7+1\},\beta = \max\{3n-3i_4-1, 3n-3i_7-1\},$ we get
	\begin{align*}
		S(f_4,f_7) &= f_4\cdot x^{\alpha-3i_4-2}y^{\beta-(3n-3i_4-1)}a^k-f_7\cdot x^{\alpha-(3i_7+1)}y^{\beta-(3n-3i_7-1)}\\
		&=x^{\alpha+3k}y^{\beta-3k}b^{2k}.
	\end{align*}
	Assume that $S(f_4,f_7)\notin T_1$. Then, $\alpha+\beta \le 3n+1$. On the other hand,
	\[
	3n+1\ge \alpha+\beta \ge (3i_4+2)+(3n-3i_4-1)=3n+1,
	\]
	so equalities happen from left to right. Thus,
	\begin{align*}
		3i_7+1 &\le \alpha = 3i_4+2,\\
		3n-3i_7-1 & \le \beta=3n-3i_4-1,
	\end{align*}
	which force $i_4=i_7 \in [0,(k-2)/3]$. We get
	\begin{align*}
		S(f_4,f_7) &= x^{\alpha+3k}y^{\beta-3k}b^{2k}= x^{3i_4+3k+2}y^{3n-3i_4-3k-1}b^{2k} \\
		&= \udb{(x^{2k}+y^{2k})x^{3i_4+k+2}y^{3n-3i_4-3k-1}b^{2k}}_{\text{first summand}}+\udb{x^{3i_4+k+2}y^{3n-3i_4-k-1}b^{2k}}_{\text{second summand}}\\
		&=\udb{(x^{2k}+y^{2k})xyx^{3i_4+k+1}y^{3n-3i_4-3k-2}b^{2k}}_{\text{first summand} \,\in\, T_3}+\udb{x^{3i_4+k+2}y^{3n-3i_4-k-1}b^{2k}}_{\text{second summand} \, \in \, T_1}.
	\end{align*}
	Therefore, $S(f_4,f_7)\in T_3+T_1$, concluding the proof.
\end{proof}

The last part of the proof of \Cref{prop_GB_double2powers_odd} is concerned with the 26 simplest kinds of $S$-pairs among of the possible 36 kinds.
\begin{lem}
	\label{lem_double2powers_trivialSpairs}
	The following statements hold:
	\begin{enumerate}
		\item[\textup{(S78)}] $S(T_7,T_8) \subseteq \Sigma T_1 \cup T_3$,
		\item[\textup{(S0)}] $S(T_i,T_j)=\{0\}$ if $i=j \notin \{5,6\}$ or $(i,j)=(1,4)$ \textup{(}7 pairs\textup{)},
		\item[\textup{(S$T_1$)}] $S(T_i,T_j)\subseteq \Sigma T_1$ if $(i,j) \in \{(1,6), (1,7),\ldots, (6,7), (6,8)\}$ \textup{(}totally 18 pairs; see \Cref{tab_Spairs_double2powers}\textup{)}.
	\end{enumerate}
\end{lem}

\begin{proof}
The proof employs repeatedly \Cref{lem_large_xydegree}.
	
	\textbf{(S78)}: Consider
	\begin{align*}
		S(f_7,f_8)&= S(\udl{x^{3i_7+1}y^{3n-3i_7-1}a^kb^k}+x^{3i_7+3k+1}y^{3n-3i_7-3k-1}b^{2k}, \\
		& \udl{x^{3i_8+2}y^{3n-3i_8-2}a^kb^{2k}}+x^{3i_8+k+2}y^{3n-k-3i_8-2}b^{3k}).
	\end{align*}
	Letting
	\begin{align*}
		\alpha &= \max\{3i_7+1, 3i_8+2\},\\
		\beta &= \max\{3n-3i_7-1, 3n-3i_8-2\},
	\end{align*}
	we get
	\begin{align*}
		S(f_7,f_8) &= f_7\cdot x^{\alpha-(3i_7+1)}y^{\beta-(3n-3i_7-1)}b^k-f_8\cdot x^{\alpha-(3i_8+2)}y^{\beta-(3n-3i_8-2)}\\
		&= \udl{x^{\alpha+3k}y^{\beta-3k}b^{3k}}+x^{\alpha+k}y^{\beta-k}b^{3k}.
	\end{align*}
	Assume that $S(f_7,f_8)\notin \Sigma T_1$. Then, $\alpha+\beta\le 3n+1$. This yields
	\begin{align*}
		3i_8+2 & \le \alpha \le 3i_7+2 \Longrightarrow i_8\le i_7,\\
		3i_7+1 &\le \alpha \le  3i_8+3 \Longrightarrow i_7\le i_8.
	\end{align*}
	Hence, $i_7=i_8, \alpha=3i_7+2, \beta=3n-3i_7-1$. It follows that
	\begin{align*}
		S(f_7,f_8) &=(x^{2k}+y^{2k})x^{\alpha+k}y^{\beta-3k}b^{3k}= (x^{2k}+y^{2k})x^{3i_7+k+2}y^{3n-3i_7-3k-1}b^{3k}\\
		&= (x^{2k}+y^{2k})xyx^{3i_7+k+1}y^{3n-3i_7-3k-2}b^{3k} \in T_3.
	\end{align*}
	Here, the containment holds since
	\begin{align*}
		3i_7+k+1 &\equiv 0 \quad \text{(mod 3)},\\
		3n-3i_7-3k-2 &\ge 3n-(k-2)-3k-2=3n-4k\ge 2k>0,\\
		(3i_7+k+1)+(3n-3i_7-3k-2) &= 3n-(2k+1) > 3n-(2k+2).
	\end{align*}
	Hence, $S(T_7,T_8) \subseteq \Sigma T_1 \cup T_3$, as desired.
	
	\textbf{(S0)}: This is clear, since $T_1, T_4$ are monomial ideals, and every minimal generator of $T_i$, $i\in [8]\setminus \{5,6\}$ is a monomial multiple of a fixed polynomial (depending only on $T_i$).
	
	\textbf{(S16)}: Consider
	\[
	S(f_1,f_6)= S(x^{3i_1}y^{3n-3i_1}, \udl{x^{6k\ell+3i_6+1}y^{3n-6k\ell-3i_6-1}b^{2k}}+x^{3i_6+1}y^{3n-3i_6-1}b^{2k}).
	\]
	Letting
	\begin{align*}
		\alpha &= \max\{3i_1, 6k\ell+3i_6+1\},\\
		\beta &= \max\{3n-3i_1, 3n-6k\ell-3i_6-1\},
	\end{align*}
	we get
	\begin{align*}
		S(f_1,f_6) &= f_1\cdot x^{\alpha-3i_1}y^{\beta-(3n-3i_1)}b^{2k}-f_6\cdot x^{\alpha-(6k\ell+3i_6+1)}y^{\beta-(3n-6k\ell-3i_6-1)}\\
		&=x^{\alpha-6k\ell}y^{\beta+6k\ell}b^{2k}.
	\end{align*}
	Assume that $S(f_1,f_6)\notin T_1$. Then, $\alpha+\beta\le 3n+1$. This yields
	\begin{align*}
		3i_1 &\le \alpha \le 6k\ell+3i_6+2 \Longrightarrow i_1 \le 2k\ell+i_6,\\
		6k\ell+3i_6+1 &\le \alpha \le 3i_1+1 \Longrightarrow 2k\ell+i_6 \le i_1.
	\end{align*}
	Hence,
	\[
	i_1=2k\ell+i_6, \alpha=6k\ell+3i_6+1=3i_1+1, \beta  = 3n-3i_1=3n-6k\ell-3i_6.
	\]
	But then $S(f_1,f_6) =x^{3i_6+1}y^{3n-3i_6}b^{2k} \in T_1,$ a contradiction. Therefore, $S(T_1,T_6)\subseteq T_1$.
	
	\textbf{(S17)}: Consider
	\[
	S(f_1,f_7)= S(x^{3i_1}y^{3n-3i_1}, \udl{x^{3i_7+1}y^{3n-3i_7-1}a^kb^k}+x^{3i_7+3k+1}y^{3n-3i_7-3k-1}b^{2k}).
	\]
	Letting
	\begin{align*}
		\alpha &= \max\{3i_1, 3i_7+1\},\\
		\beta &= \max\{3n-3i_1, 3n-3i_7-1\},
	\end{align*}
	we get
	\begin{align*}
		S(f_1,f_7) &= f_1\cdot x^{\alpha-3i_1}y^{\beta-(3n-3i_1)}a^kb^k-f_7\cdot x^{\alpha-(3i_7+1)}y^{\beta-(3n-3i_7-1)}\\
		&=x^{\alpha+3k}y^{\beta-3k}b^{2k}.
	\end{align*}
	Assume that $S(f_1,f_7)\notin T_1$. Then, $\alpha +\beta \le 3n+1$. Arguing as for (S16), we deduce $i_1=i_7$, $\alpha=3i_1+1, \beta=3n-3i_1$. But then $S(f_1,f_7)=x^{3i_1+3k+1}y^{3n-3i_1-3k}b^{2k}\in T_1$, a contradiction. Therefore, $S(T_1,T_7)\subseteq T_1$.
	
	\textbf{(S24)}: Consider
	\[
	S(f_2,f_4)= S(\udl{x^ky^ka^k}+x^{2k}b^k+y^{2k}b^k, x^{3i_4+2}y^{3n-3i_4-1}b^k).
	\]
	Letting
	\begin{align*}
		\alpha &= \max\{k, 3i_4+2\},\\
		\beta &= \max\{k, 3n-3i_4-1\},
	\end{align*}
	we get
	\begin{align*}
		S(f_2,f_4) &= f_2\cdot x^{\alpha-k}y^{\beta-k}b^k-f_4\cdot x^{\alpha-(3i_4+2)}y^{\beta-(3n-3i_4-1)}a^k\\
		&= x^{\alpha+k}y^{\beta-k}b^{2k}+x^{\alpha-k}y^{\beta+k}b^{2k}.
	\end{align*}
	Assume that $S(f_2,f_4)\notin \Sigma T_1$. Then, $\alpha+\beta\le 3n+1$. The definitions of $\alpha$ and $\beta$ force $\alpha=3i_4+2$ and $\beta=3n-3i_4-1$. Thus,
	\begin{align*}
		S(f_2,f_4) &= x^{3i_4+k+2}y^{3n-3i_4-k-1}b^{2k}+x^{3i_4-k+2}y^{3n-3i_4+k-1}b^{2k} \in \Sigma T_1.
	\end{align*}
	This contradiction yields the desired claim.
	
	\textbf{(S28), (S34), (S37), (S38)}, namely, \textbf{$S(T_i,T_j)$} where $(i,j)\in \{(2,8), (3,4), (3,7), (3,8)\}$: In these cases, we always have
	\begin{align*}
		m_i:=\ini(f_i)&=x^{\ell_1}y^{\ell_2}a^{\ell_3}b^{\ell_4}, \\
		m_j:=\ini(f_j)&=x^{j_1}y^{j_2}a^{j_3}b^{j_4},\\
		\deg_{\{x,y\}}m_j &= j_1+j_2\ge 3n,
	\end{align*}
	where $0\le \ell_1, j_1\le 3n, \ell_2, \ell_3, \ell_4, j_2, j_3,j_4\ge 0$. Most importantly, we have $\ell_1- j_1\ge 2$. Thus, by \eqref{eq_degxyineq} in \Cref{lem_large_xydegree}, there is an inequality
	\[
	\deg_{\{x,y\}}S(f_i,f_{j}))= \deg_{\{x,y\}}\lcm(m_i,m_j) \ge (\ell_1-j_1)+\deg_{\{x,y\}}m_j\ge 3n+2.
	\]
	Thanks to \Cref{lem_inQn}, we get $S(f_i,f_j)\in \Sigma T_1$.

	\textbf{(S35)}: Consider
	\begin{align*}
		S(f_3,f_5)&= S(\udl{x^{2k+3i_3+1}y^{3n-2k-3i_3-1}b^k}+x^{3i_3+1}y^{3n-3i_3-1}b^k,\\
		& \qquad \qquad \udl{x^{6kj+4k+3i_5+2}y^{3n-6kj-4k-3i_5-2}b^{2k}}+x^{3i_5+2}y^{3n-3i_5-2}b^{2k}).
	\end{align*}
	Letting
	\begin{align*}
		\alpha &= \max\{2k+3i_3+1, 6kj+4k+3i_5+2\},\\
		\beta &= \max\{3n-2k-3i_3-1, 3n-6kj-4k-3i_5-2\},
	\end{align*}
	we get
	\begin{align*}
		S(f_3,f_5) &=f_3\cdot x^{\alpha-(2k+3i_3+1)}y^{\beta-(3n-2k-3i_3-1)}b^k-f_5\cdot x^{\alpha-(6kj+4k+3i_5+2)}y^{\beta-(3n-6kj-4k-3i_5-2)}\\
		&=x^{\alpha-2k}y^{\beta+2k}b^{2k} + x^{\alpha-6kj-4k}y^{\beta+6kj+4k}b^{2k}.
	\end{align*}
	Assume that $S(f_3,f_5)\notin \Sigma T_1$. Then, $\alpha+\beta\le 3n+1$. Arguing as for (S16), we deduce
	\begin{align*}
		2k+3i_3 &= 6kj+4k+3i_5+2, \, \text{equivalently,} \, 3i_3=6kj+2k+3i_5+2,\\
		\alpha &= 2k+3i_3+1,\\
		\beta &=3n-2k-3i_3.
	\end{align*}
	Therefore,
	\begin{align*}
		S(f_3,f_5) &= x^{\alpha-2k}y^{\beta+2k}b^{2k} + x^{\alpha-6kj-4k}y^{\beta+6kj+4k}b^{2k}\\
		&=\udb{x^{3i_3+1}y^{3n-3i_3}b^{2k} +x^{3i_3-6kj-2k+1}y^{3n+6kj+2k-3i_3}b^{2k}}_{\,\in\, \Sigma T_1}.
	\end{align*}
	This contradiction shows that $S(T_3,T_5)\subseteq \Sigma T_1$ in any case, as desired.

	\textbf{(S45), (S46)}: We have to consider $S(T_i,T_j)$ where $(i,j)\in \{(4,5), (4,6)\}$.  In these cases, we always have
	\begin{align*}
		m_i:=\ini(f_i)&=x^{\ell_1}y^{\ell_2}a^{\ell_3}b^{\ell_4}, \\
		m_j:=\ini(f_j)&=x^{j_1}y^{j_2}a^{j_3}b^{j_4},\\
		\deg_{\{x,y\}}m_i &= \ell_1+\ell_2= 3n+1,
	\end{align*}
	where $0\le \ell_1, j_1\le 3n, \ell_2, \ell_3, \ell_4, j_2, j_3,j_4\ge 0$. Crucially, using $i_4\le \frac{2k-4}{3}$, we have
	$$
	j_1- \ell_1\ge 4k+2-(3i_4+2) \ge 4k+2-(2k-2)=2k+4.
	$$
	Thus, by \eqref{eq_degxyineq} in \Cref{lem_large_xydegree}, we get the first inequality in the chain
	\[
	\deg_{\{x,y\}}S(f_i,f_{j}))= \deg_{\{x,y\}}\lcm(m_j,m_i) \ge (j_1-\ell_1)+\deg_{\{x,y\}}m_i \ge 3n+1+2k+4 > 3n+2.
	\]
	Thanks to \Cref{lem_inQn}, we get $S(f_i,f_j)\in \Sigma T_1$.

	\textbf{(S48)}: Consider
	\begin{align*}
		S(f_4,f_8)&= S(x^{3i_4+2}y^{3n-3i_4-1}b^k,\udl{x^{3i_8+2}y^{3n-3i_8-2}a^kb^{2k}}+x^{3i_8+k+2}y^{3n-k-3i_8-2}b^{3k}).
	\end{align*}
	Letting
	\begin{align*}
		\alpha &= \max\{3i_4+2, 3i_8+2\},\\
		\beta &= \max\{3n-3i_4-1, 3n-3i_8-2\},
	\end{align*}
	and keeping the common highest term $x^\alpha y^\beta a^kb^{2k}$ in mind, we get
	\begin{align*}
		S(f_4,f_8)&= f_4\cdot x^{\alpha-(3i_4+2)}y^{\beta-(3n-3i_4-1)}a^kb^k-f_8\cdot x^{\alpha -(3i_8+2)}y^{\beta- (3n-3i_8-2)} \\
		&= x^{\alpha+k}y^{\beta-k}b^{3k}.
	\end{align*}
	Assume that $S(f_4,f_8)\notin T_1$. Then, $\alpha+\beta \le 3n+1$. The definitions of $\alpha$ and $\beta$ force $\alpha=3i_4+2, \beta=3n-3i_4-1$. Thus,
	\[
	S(f_4,f_8)=x^{3i_4+k+2}y^{3n-3i_4-k-1}b^{3k}\in T_1,
	\]
	by \Cref{lem_inQn}, which is a contradiction. Hence, $S(T_4,T_8)\subseteq T_1$ always holds.
	
	\textbf{(S55)}: For $0\le \tau, \tau'\le \frac{n}{2k}-1, 0\le i_5, i'_5\le \frac{2k-4}{3}$, consider
	\begin{align*}
		S(f_5,f'_5)&= S(\udl{x^{6k\tau+4k+3i_5+2}y^{3n-6k\tau-4k-3i_5-2}b^{2k}}+x^{3i_5+2}y^{3n-3i_5-2}b^{2k},\\
		& \qquad \qquad \udl{x^{6k\tau'+4k+3i'_5+2}y^{3n-6k\tau'-4k-3i'_5-2}b^{2k}}+x^{3i'_5+2}y^{3n-3i'_5-2}b^{2k}).
	\end{align*}
	Letting
	\begin{align*}
		\alpha &= \max\{6k\tau+4k+3i_5+2, 6k\tau'+4k+3i'_5+2\},\\
		\beta &= \max\{3n-6k\tau-4k-3i_5-2, 3n-6k\tau'-4k-3i'_5-2\},
	\end{align*}
	and keeping the common highest term $x^\alpha y^\beta b^{2k}$ in mind, we get
	\begin{align*}
		& S(f_5,f'_5)\\
		&= f_5\cdot x^{\alpha-(6k\tau+4k+3i_5+2)}y^{\beta-(3n-6k\tau-4k-3i_5-2)}-f'_5\cdot x^{\alpha-(6k\tau'+4k+3i'_5+2)}y^{\beta-(3n-6k\tau'-4k-3i'_5-2)} \\
		&=x^{\alpha-(6k\tau+4k)}y^{\beta+6k\tau+4k}b^{2k}+x^{\alpha-(6k\tau'+4k)}y^{\beta+6k\tau'^+4k}b^{2k}.
	\end{align*}
	Assume that $S(f_5,f'_5)\notin \Sigma T_1$. Then, $\alpha+\beta\le 3n+1$. Thus,
	\begin{align*}
		6k\tau+4k+3i_5+2 &\le \alpha \le (6k\tau'+4k+3i'_5+2)+1,\\
		6k\tau'+4k+3i'_5+2 &\le \alpha \le (6k\tau+4k+3i_5+2)+1.
	\end{align*}
	Since $6k\tau+4k+3i_5+2$ and $6k\tau'+4k+3i'_5+2$ are congruent modulo 3, this yields equalities
	\begin{align*}
		6k\tau+4k+3i_5+2 &= 6k\tau'+4k+3i'_5+2, \text{ and equivalently,} \ 2k\tau+i_5=2k\tau '+i'_5.
	\end{align*}
	
	We claim that $\tau=\tau', i_5=i'_5$ and, thus, $f_5=f'_5$. Indeed, it is harmless to assume that $\tau\ge \tau'$. If the equality does not hold, then
	\[
	2k \le 2k(\tau-\tau')= i'_5-i_5 \le \frac{2k-4}{3} < 2k,
	\]
	a contradiction. Therefore, $\tau=\tau'$ and, hence, $i_5=i'_5$. But then $S(f_5,f'_5)=0 \in T_1$, a contradiction. We conclude that $S(T_5,T_5)\subseteq \Sigma T_1$.

	\textbf{(S56)}: Consider
	\begin{align*}
		S(f_5,f_6)&= S(\udl{x^{6k\tau+4k+3i_5+2}y^{3n-6k\tau-4k-3i_5-2}b^{2k}}+x^{3i_5+2}y^{3n-3i_5-2}b^{2k},\\
		&\qquad \qquad \udl{x^{6k\ell+3i_6+1}y^{3n-6k\ell-3i_6-1}b^{2k}}+x^{3i_6+1}y^{3n-3i_6-1}b^{2k}).
	\end{align*}
	Letting
	\begin{align*}
		\alpha &= \max\{6k\tau+4k+3i_5+2, 6k\ell+3i_6+1\},\\
		\beta &= \max\{3n-6k\tau-4k-3i_5-2, 3n-6k\ell-3i_6-1\},
	\end{align*}
	we get
	\begin{align*}
		S(f_5,f_6)&= f_5\cdot x^{\alpha-(6k\tau+4k+3i_5+2)}y^{\beta-(3n-6k\tau-4k-3i_5-2)} -f_6\cdot x^{\alpha-(6k\ell+3i_6+1)}y^{\beta-(3n-6k\ell-3i_6-1)},\\
		&= x^{\alpha-6k\tau-4k}y^{\beta+6k\tau+4k}b^{2k}+x^{\alpha-6k\ell}y^{\beta+6k\ell}b^{2k}.
	\end{align*}
	Assume that $S(f_5,f_6)\notin \Sigma T_1$. Then, $\alpha+\beta\le 3n+1$. This yields
	\begin{align*}
		6k\tau+4k+3i_5+2 & \le \alpha \le (6k\ell+3i_6+1)+1,\\
		6k\ell+3i_6+1 & \le \alpha \le (6k\tau+4k+3i_5+2)+1.
	\end{align*}
	Since $6k\tau+4k+3i_5+2$ and $6k\ell+3i_6+1$ are congruent modulo 3, they must be equal. So we get
	\begin{align*}
		6k\tau+4k+3i_5+1 &= 6k\ell+3i_6.
	\end{align*}
	
	\textbf{Case 1}: $i_5\ge i_6$. Then,
	\[
	1\le 3(i_5-i_6)+1 = 6k\ell-6k\tau-4k  \le 3i_5+1 \le 2k-3.
	\]
	This is a contradiction, since $6k\ell-6k\tau-4k$ is divisible by $2k$.
	
	\textbf{Case 2}: $i_6\ge i_5+1$. Then
	\[
	6k\tau+4k-6k\ell+1=3(i_6-i_5) \ge 3
	\]
	which yields $\tau\ge \ell$. But then
	\[
	4k-2 \ge 3i_6 \ge 3(i_6-i_5)= 6k\tau+4k-6k\ell+1 \ge 4k+1,
	\]
	a contradiction. Hence, $S(T_5,T_6)\subseteq \Sigma T_1$.
	
	\textbf{(S57), (S58), (S67), (S68)}: Consider $S(T_i,T_j)$ where $(i,j)\in \{(5,7), (5,8), (6,7)$, $(6,8)\}$.  In these cases, we always have
	\begin{align*}
		m_i:=\ini(f_i)&=x^{\ell_1}y^{\ell_2}a^{\ell_3}b^{\ell_4}, \\
		m_j:=\ini(f_j)&=x^{j_1}y^{j_2}a^{j_3}b^{j_4},\\
		\deg_{\{x,y\}}m_j &= j_1+j_2= 3n,
	\end{align*}
	where $0\le \ell_1, j_1\le 3n, \ell_2, \ell_3, \ell_4, j_2, j_3,j_4\ge 0$. Crucially, we have
	$$
	\ell_1-j_1 \ge 4k+2-(k-1)=3k+3.
	$$
	Thus, by \eqref{eq_degxyineq} in \Cref{lem_large_xydegree}, we get the first inequality in the chain
	\[
	\deg_{\{x,y\}}S(f_i,f_{j}))= \deg_{\{x,y\}}\lcm(m_i,m_j) \ge (\ell_1-j_1)+\deg_{\{x,y\}}m_j \ge (3k+3)+3n > 3n+2.
	\]
	Thanks to \Cref{lem_inQn}, we get $S(f_i,f_j)\in \Sigma T_1$.
	
	\textbf{(S66)}: For $1\le \ell, \ell'\le \frac{n}{2k}-1, 0\le i_6, i'_6\le \frac{4k-2}{3}$, consider
	\begin{align*}
		S(f_6,f'_6)&= S(\udl{x^{6k\ell+3i_6+1}y^{3n-6k\ell-3i_6-1}b^{2k}}+x^{3i_6+1}y^{3n-3i_6-1}b^{2k},\\
		& \qquad \qquad \udl{x^{6k\ell'+3i'_6+1}y^{3n-6k\ell'-3i'_6-1}b^{2k}}+x^{3i'_6+1}y^{3n-3i'_6-1}b^{2k}).
	\end{align*}
	Letting
	\begin{align*}
		\alpha &= \max\{6k\ell+3i_6+1, 6k\ell'+3i'_6+1\},\\
		\beta &= \max\{3n-6k\ell-3i_6-1, 3n-6k\ell'-3i'_6-1\},
	\end{align*}
	we get
	\begin{align*}
		S(f_6,f'_6)
		&= f_6\cdot x^{\alpha-(6k\ell+3i_6+1)}y^{\beta-(3n-6k\ell-3i_6-1)}-f'_6\cdot x^{\alpha-(6k\ell'+3i'_6+1)}y^{\beta-(3n-6k\ell'-3i'_6-1)} \\
		&= x^{\alpha-6k\ell}y^{\beta+6k\ell}b^{2k}+x^{\alpha-6k\ell'}y^{\beta+6k\ell'}b^{2k}.
	\end{align*}
	
	Assume that $S(f_6,f'_6)\notin \Sigma T_1$. Then, $\alpha+\beta\le 3n+1$. Arguing as for (S55), we deduce $6k\ell+3i_6+1=6k\ell'+3i'_6+1$, namely $2k\ell+i_6=2k\ell'+i'_6$. We then deduce $\ell=\ell'$ and $i_6=i'_6$, using the fact that $0\le i_6, i'_6\le \frac{4k-2}{3} <2k$. But then $f_6=f'_6$ and $S(f_6,f'_6)=0$. This contradiction shows that $S(T_6, T_6)\subseteq \Sigma T_1$ always holds. The proof is concluded.
\end{proof}

\begin{proof}[{\bf Proof of \Cref{prop_GB_double2powers_odd}}]
	We have to verify two conditions: that the polynomials of type $T_i$, where $1\le i\le 8$, belong to $Q^n+(f^k)$, and that they satisfy the $S$-pair condition. The $S$-pair check follows by combining Lemmas \ref{lem_double2powers_S12-18}, \ref{lem_double2powers_S23}, \ref{lem_double2powers_S25}, \ref{lem_double2powers_S26}, \ref{lem_double2powers_S27}, \ref{lem_double2powers_S36}, \ref{lem_double2powers_S47}, and \ref{lem_double2powers_trivialSpairs}.
	
	To check the first condition only requires a closer look at the proof of these $S$-pair statements, and the fact that $T_1, T_2$ are clearly subsets of $Q^n+(f^k)$. Recall from \Cref{lem_double2powers_S12-18}, parts (S12), (S13), \Cref{cor_double2powers_T567containments}, \Cref{lem_double2powers_S25} that we have the following containments:
	\begin{alignat*}{2}
		&T_3 \subseteq T_1+ T_2, \quad &&T_6 \subseteq T_2+T_3, \\
		&T_4 \subseteq T_1+T_3, \quad &&T_7 \subseteq T_1+T_2+T_3,\\
		&T_5 \subseteq T_2+T_3, \quad &&T_8 \subseteq T_1+T_2+T_3+T_5+T_6.
	\end{alignat*}
	Therefore, all the ideals $T_i$ are subsets of $Q^n+(f^k)$. This completes the proof.
\end{proof}


\subsection{Regularity bounds and \Cref{thm_regbound_double2powers_odd}.} To finish the proof of \Cref{thm_regbound_double2powers_odd}, the crucial point is to have a formula for the initial ideal of $Q^n+(f^k)$. Once such a formula has been established, we will see that, despite its cumbersome appearance, the initial ideal admits a very simple and efficient regularity bound, thanks to \Cref{lem_regbound_splittableideals}.
We shall need some auxiliary lemmas for bounding regularity of the initial ideal of $Q^n+(f^k)$.

\begin{lem}
	\label{lem_easy_regboud_double2powers}
	Let $\kk$ be an arbitrary field, $n, k\ge 1$ be arbitrary integers. In the standard graded polynomial ring $T=\kk[x,y,a]$, let $J_1, L_1$ be monomial ideals such that $J_1$ and $L_1$ are generated by monomials in $x$ and $y$ only. Consider the following monomial ideals:
	\[
	Q = (x^3,y^3), \quad U  =Q^n+(x^ky^ka^k),\quad W_1  = J_1+a^kL_1.
	\]
	Then there are inequalities:
	\begin{enumerate}[\quad \rm (1)]
		\item $\reg U \le 3n+k+2$.
		\item $\reg(U+W_1)\le 3n+k+2$.
	\end{enumerate}
\end{lem}
\begin{proof}
	The assertion is a consequence of the regularity bound for ideals with special splittings (\Cref{lem_regbound_splittableideals}). Part (1) was proved in \Cref{lem_regbound_2powers_initial}, but it is also a special case of part (2), when $J_1=L_1=(0)$, so in any event, it suffices to prove (2). There is nothing to do if $J_1=(1)$, so we assume that $J_1$ is a proper ideal of $T$.
	
	Since $U+W_1=(Q^n+J_1)+a^k((x^ky^k)+L_1)$, and $a^k$ is regular on $Q^n+J_1$ by the hypothesis, we may apply \Cref{lem_regbound_splittableideals} to get
	\begin{align*}
		\reg(U+W_1) & \le \max\{\reg (Q^n+J_1)+k-1, \reg (Q^n+J_1+(x^ky^k)+L_1)+k\} \\
		&\le \max \{\reg Q^n+k-1, \reg Q^n+k\} =3n+k+2.
	\end{align*}
	The second inequality holds since $Q^n\subseteq Q^n+J_1 \subseteq Q^n+J_1+(x^ky^k)+L_1$ are monomial $(x,y)$-primary ideals. The proof is completed.
\end{proof}

\begin{lem}
	\label{lem_regbound_double2powers_initial}
	Let $\kk$ be an arbitrary field, $n, k\ge 1$ be arbitrary integers. In the standard graded polynomial ring $R=\kk[x,y,a,b]$, consider the monomial ideals $Q, U$ with the same generators as in \Cref{lem_easy_regboud_double2powers}, namely $Q=(x^3,y^3)$ and $U=Q^n+(x^ky^ka^k)$. For $i\in \{1,2\}$, let $J_i, L_i$ monomial ideals such that $J_i$ and $L_i$ are generated by monomials in $x$ and $y$ only.
	
	Denote $W_i=J_i+a^kL_i$ for $i\in \{1,2\}$ and $L=W_1+b^kW_2$. Then there are inequalities:
	\begin{enumerate}[\quad \rm (1)]
		\item $\reg(U+L)\le 3n+2k+2$.
		\item $\reg(U+b^kL)\le 3n+3k+2$.
	\end{enumerate}
\end{lem}

\begin{proof}
	The assertion is a consequence of the regularity bound for ideals with special splittings (\Cref{lem_regbound_splittableideals}) and \Cref{lem_easy_regboud_double2powers}.
	
	(1) If $J_1=(1)$, then $\reg(U+L)=\reg R=0$, so we may assume that $J_1$ is a proper ideal of $R$.  We have $U+L=(U+W_1)+b^kW_2$ and $b^k$ is regular on $U+W_1$. Thus, by \Cref{lem_regbound_splittableideals} we get the first inequality in the chain
	\begin{align*}
		\reg (U+L)  &\le \max\{\reg (U+W_1)+k-1, \reg (U+W_1+W_2)+k\}\\
		&\le \max\{3n+k+2+k-1,3n+k+2+k\}=3n+2k+2.
	\end{align*}
	The second inequality is a consequence of \Cref{lem_easy_regboud_double2powers}.
	
	(2) follow from similar arguments, using the splitting the corresponding ideal as a sum of $U$ and $b^kL$, and (1) as well as the bound $\reg U\le 3n+k+2$ in \Cref{lem_easy_regboud_double2powers}. (Note that $b$ is regular on $U$.)
\end{proof}

\begin{proof}[{\bf Proof of \Cref{thm_regbound_double2powers_odd}}]
	(1) From \Cref{prop_GB_double2powers_odd} and its degenerate case, when $u=1$, in \Cref{prop_GB_double2powers_12} below, we deduce that
	\begin{align*}
		\ini(Q^n+(f^k))=\quad &  Q^n+(x^ky^ka^k)+x^{2k+1}yQ^{n-\frac{2k+2}{3}}b^k+x^2y^{3n-2k+3}Q^\frac{2k-4}{3}b^k+\\
		&+\left(x^{6kj+4k+2}y^{3n-6k+2-6kj}Q^\frac{2k-4}{3}\mid \quad  0\le j\le  \dfrac{n}{2k}-1\right)b^{2k}+\\
		&+\left(x^{6k\ell+1}y^{3n-4k+1-6k\ell}Q^\frac{4k-2}{3}\mid \quad  1\le \ell \le  \dfrac{n}{2k}-1\right)b^{2k}+\\
		&+xy^{3n-k+1}Q^\frac{k-2}{3}a^kb^k+x^2y^{3n-k+3}Q^\frac{k-5}{3}a^kb^{2k}.
	\end{align*}
	Thus, the desired formula follows from a simple counting.
	
	(2) We rewrite the components of the initial ideal of $Q^n+(f^k)$ according to types as follows
	\begin{align*}
		\ini(Q^n+(f^k))=\quad &  \udb{Q^n}_{\ini T_1}+\udb{(x^ky^ka^k)}_{\ini T_2}+\udb{x^{2k+1}yQ^{n-\frac{2k+2}{3}}b^k}_{\ini T_3}+\udb{x^2y^{3n-2k+3}Q^\frac{2k-4}{3}b^k}_{\ini T_4}+\\
		&+\udb{x^{4k+2}y^2(x^{6k},y^{6k})^{\frac{n}{2k}-1}Q^\frac{2k-4}{3}b^{2k}}_{\ini T_5}+\udb{x^{6k+1}y^{2k+1}(x^{6k},y^{6k})^{\frac{n}{2k}-2}Q^\frac{4k-2}{3}b^{2k}}_{\ini T_6}+\\
		&+\udb{xy^{3n-k+1}Q^\frac{k-2}{3}a^kb^k}_{\ini T_7}  + \udb{x^2y^{3n-k+3}Q^\frac{k-5}{3}a^kb^{2k}}_{\ini T_8}.
	\end{align*}
	Let $h=(x^{3n-k}y^k+x^ky^{3n-k})a^{k-1}b^{2k-1}$.
	
	To see the non-containment $h\notin Q^n+(f^k)$, assume the contrary. Then, the leading term of $h$ satisfies $h^*=x^{3n-k}y^ka^{k-1}b^{2k-1}\in \ini(Q^n+(f^k))$. In particular, $h^*$ belongs to one of the ideals $\ini T_j, 1\le j\le 8$. On the other hand,
	\begin{align*}
		h^* &\notin \udb{Q^n}_{\ini T_1} \quad \text{(following from $x^{3n-k}y^k\notin Q^n$, $k\equiv 2\, \text{(mod 3)}$ and \Cref{lem_inQn}(3))},\\
		h^* &\notin \ini T_2, \ini T_7, \ini T_8,\quad \text{(inspecting degrees of $a$)},\\
		h^* &\notin \ini T_5,\ini T_6,\quad \text{(inspecting degrees of $b$)},\\
		h^* &\notin \ini T_4,\quad \text{(inspecting degrees of $y$)}.
	\end{align*}
	Thus, the only possibility is that $h^* \in \ini T_3=x^{2k+1}yQ^{n-\frac{2k+2}{3}}b^k$, which yields
	\[
	x^{3n-3k-1}y^{k-1} \in Q^{n-\frac{2k+2}{3}}.
	\]
	This is impossible as $k\equiv 2$ (mod 3). This contradiction shows that $h\notin Q^n+(f^k)$.
	
	For the containment $h\in \left(Q^n+(f^k)\right):\mm$, we first show that $hx \in T_3+\Sigma T_1$. Indeed,
	\begin{align*}
		hx &= x^{3n-k+1}y^ka^{k-1}b^{2k-1}+x^{k+1}y^{3n-k}a^{k-1}b^{2k-1} \\
		&=\udb{(x^{2k}+y^{2k})xyx^{3n-3k}y^{k-1}a^{k-1}b^{2k-1}}_{\in T_3}+\udb{x^{3n-3k+1}y^{3k}a^{k-1}b^{2k-1}+x^{k+1}y^{3n-k}a^{k-1}b^{2k-1}}_{\,\in\, \Sigma T_1},
	\end{align*}
	where the containments follow by \Cref{lem_inQn}. Similarly, $hy \in \Sigma T_1+T_3$ as
	\begin{align*}
		hy &= x^{3n-k}y^{k+1}a^{k-1}b^{2k-1}+ x^{k}y^{3n-k+1}a^{k-1}b^{2k-1}\\
		&= \udb{x^{3n-k}y^{k+1}a^{k-1}b^{2k-1}+x^{3k}y^{3n-3k+1}a^{k-1}b^{2k-1}}_{\,\in\, \Sigma T_1}+\udb{(x^{2k}+y^{2k})xyx^{k-1}y^{3n-3k}a^{k-1}b^{2k-1}}_{\,\in\, T_3}.
		\end{align*}
		
	Next, we prove that $ha\in T_2+T_1+T_3+T_6$. Indeed, as
	\begin{align*}
		ha &=(x^{3n-2k}+y^{3n-2k})x^ky^ka^kb^{2k-1},\\
		(x^{3n-2k}+y^{3n-2k})f^kb^{2k-1} &= (x^{3n-2k}+y^{3n-2k})(x^ky^ka^k+(x^{2k}+y^{2k})b^k)b^{2k-1},
	\end{align*}
	it follows that
	\begin{align*}
		ha + \udb{(x^{3n-2k}+y^{3n-2k})f^kb^{2k-1}}_{\,\in\, T_2} &= (x^{3n-2k}+y^{3n-2k})(x^{2k}+y^{2k})b^{3k-1}\\
		&=\udb{(x^{3n}+y^{3n})b^{3k-1}}_{\,\in\,  T_1}+(x^{3n-4k}+y^{3n-4k})x^{2k}y^{2k}b^{3k-1}.
	\end{align*}
	It remains to show that the last summand belongs to $T_3+T_6$. To see this containment, let $\ell=\frac{n}{2k}-1$. Then, $6k\ell=3n-6k$ and $3n-4k+1-6k\ell=2k+1$. Therefore, we get the third equality in the following display
	\begin{align*}
		&(x^{3n-4k}+y^{3n-4k})x^{2k}y^{2k}b^{3k-1}\\
		&= x^{3n-6k}(x^{4k}y^{2k}+x^{2k}y^{4k})b^{3k-1}+(x^{3n-6k}+y^{3n-6k})x^{2k}y^{4k}b^{3k-1}\\
		&=\udb{(x^{2k}+y^{2k})xyx^{3n-4k-1}y^{2k-1}b^{3k-1}}_{\,\in\,  T_3}+(x^{3n-6k}+y^{3k-6k})x^{2k}y^{4k}b^{3k-1}
		\\
		&= \udb{(\text{terms})}_{\,\in\,  T_3}+\udb{(x^{6k\ell}+y^{6k\ell})xy^{3n-4k+1-6k\ell}x^{2k-1}y^{2k-1}b^{3k-1}}_{\,\in\, T_6}.
	\end{align*}
	
	Finally, we show $hb\in T_5$. Indeed, set $j=\frac{n}{2k}-1$. Then, $6kj=3n-6k$ and $3n-6k+2-6kj=2$. Therefore, we get the third equality in the following display
	\begin{align*}
		hb &= (x^{3n-k}y^k+x^ky^{3n-k})a^{k-1}b^{2k}= (x^{3n-2k}+y^{3n-2k})x^ky^ka^{k-1}b^{2k} \\
		&= \udb{(x^{6kj+4k}+y^{6kj+4k})x^2y^{3n-6k+2-6kj}x^{k-2}y^{k-2}a^{k-1}b^{2k}}_{\,\in\, T_5}.
	\end{align*}
	This concludes the proof of (2).
	
	(3) The inequality on the left follows immediately from (2) as $\deg h=3n+3k-2$. For the remaining inequality, it suffices to show that $\reg \ini(Q^n+(f^k))\le 3n+3k+2$.
	
	The ideal $\ini(Q^n+(f^k))$ is of the form $U+b^kL$, where $U$ and $L$ satisfy the hypotheses of \Cref{lem_regbound_double2powers_initial}. Concretely, let
	\begin{alignat*}{4}
		&U &&= Q^n+(x^ky^ka^k), \quad && L_2 &&= x^2y^{3n-k+3}Q^\frac{k-5}{3},\\
		&J_1 &&= x^{2k+1}yQ^{n-\frac{2k+2}{3}}+x^2y^{3n-2k+3}Q^\frac{2k-4}{3}, \quad && W_1 &&=J_1+a^kL_1,\\
		&L_1 &&= xy^{3n-k+1}Q^\frac{k-2}{3}, \quad &&  W_2 &&=J_2+a^kL_2,\\
		&J_2 &&= x^{4k+2}y^2(x^{6k},y^{6k})^{\frac{n}{2k}-1}Q^\frac{2k-4}{3}+x^{6k+1}y^{2k+1}(x^{6k},y^{6k})^{\frac{n}{2k}-2}Q^\frac{4k-2}{3}, \quad &&  L &&= W_1+b^kW_2.
	\end{alignat*}
	Then, $\ini(Q^n+(f^k)) =U+b^k J_1+ a^kb^kL_1+b^{2k}J_2 +a^kb^{2k}L_2=U+b^kW_1+b^{2k}W_2=U+b^kL$, and the ideals $J_1, J_2, L_1, L_2$ satisfy all the hypotheses of \Cref{lem_regbound_double2powers_initial}.
	
	Hence, $\reg \ini(Q^n+(f^k))= \reg (U+b^kL)\le 3n+3k+2$ by the same result. This concludes the proof.
\end{proof}


\section{Proof of Theorem \ref{thm_regQnfk} when $u$ is even} \label{sec.regQnfk_Even}

Again, Theorem \ref{thm_regQnfk} follows from the following more detailed statement.

\begin{thm}
	\label{thm_regbound_double2powers_even}
Consider $n=2^s$ and $k=2^u$, where $0 \le u < s$ are integers and $u$ is even.
	\begin{enumerate}[\quad \rm (1)]
		\item The initial ideal of $Q^n+(f^k)$ is given by
		\begin{align*}
			\ini(Q^n+(f^k)) = \,\,  & Q^n+(x^ky^ka^k)+x^{2k+2}y^2Q^{n-\frac{2k+4}{3}}b^k+x^2y^{3n-2k+1}Q^\frac{2k-2}{3}b^k+\\
			&+x^{4k+1}y(x^{6k},y^{6k})^{\frac{n}{2k}-1}Q^\frac{2k-2}{3}b^{2k}+x^{6k+2}y^{2k+2}(x^{6k},y^{6k})^{\frac{n}{2k}-2}Q^\frac{4k-4}{3}b^{2k}+\\
			&+x^2y^{3n-k+2}Q^\frac{k-4}{3}a^kb^k+xy^{3n-k+3}Q^\frac{k-4}{3}a^kb^{2k}.
		\end{align*}
		By convention, if $\frac{n}{2k}<2$ then the summand involving $(x^{6k},y^{6k})^{\frac{n}{2k}-2}$ is zero.
		\item There is a containment
		$$
		(x^{3n-k}y^k+x^ky^{3n-k})a^{k-1}b^{2k-1} \in \left(\left(Q^n+(f^k)\right):\mm\right)\setminus \left(Q^n+(f^k)\right).
		$$
		\item We have inequalities
		$$
		3n+3k-1 \le \reg \left(Q^n+(f^k)\right)\le 3n+3k+2.
		$$
	\end{enumerate}	
\end{thm}

As seen in Theorem \ref{thm_regbound_double2powers_odd}, in order to establish Theorem \ref{thm_regbound_double2powers_even}, we need to understand the Gr\"obner basis of $Q^n + (f^k)$. This is accomplished in the following statement, whose proof is similar to that of  \Cref{prop_GB_double2powers_odd}; the reduction to zero of the relevant S-pairs are dictated by \Cref{tab_Spairs_double2powers_ieven}. The details are left for the interested reader.

\begin{prop}
	\label{prop_GB_double2powers_even}
	Consider $n=2^s$ and $k=2^u$, where $2 \le u < s$ are integers and $u$ is even.
	\begin{enumerate}[\quad \rm (1)]
		\item The ideal $Q^n+(f^k)$ has a Gr\"obner basis  consisting of the natural generators of the following ideals \textup{(}8 types of ideals in total\textup{)}:
		\begin{align*}
			&\udb{Q^n}_{T_1}, \qquad \udb{(f^k)}_{T_2},\qquad  \underbrace{(x^{2k}+y^{2k})x^2y^2Q^{n-\frac{2k+4}{3}}b^k}_{T_3}, \qquad \underbrace{x^2y^{3n-2k+1}Q^\frac{2k-2}{3}b^k}_{T_4}, \\
&\underbrace{(x^{6k\tau+4k}+y^{6k\tau+4k})xy^{3n-6k+1-6k\tau}Q^\frac{2k-2}{3}b^{2k}}_{T_5} \qquad \textup{(}\text{where $0\le \tau \le \dfrac{n}{2k}-1$}\textup{)},\\
&\underbrace{(x^{6k\ell}+y^{6k\ell})x^2y^{3n-4k+2-6k\ell}Q^\frac{4k-4}{3}b^{2k}}_{T_6} \qquad \textup{(}\text{where $1\le \ell \le \dfrac{n}{2k}-1$}\textup{)},\\
&\underbrace{(y^{3k}a^k+x^{3k}b^k)x^2y^{3n-4k+2}Q^\frac{k-4}{3}b^k}_{T_7}, \qquad  \underbrace{(y^ka^k+x^kb^k)xy^{3n-2k+3}Q^\frac{k-4}{3}b^{2k}}_{T_8}.
\end{align*}
		\item In particular, $Q^n+(f^k)$ has a Gr\"obner basis consisting of $3n+1-\dfrac{2k-2}{3}$ elements, whose maximal possible degree is $3(n+k)$.
	\end{enumerate}
\end{prop}

Finally, the following statement covers the remaining situations of \Cref{thm_regQnfk}, which can be viewed as degenerate cases of the corresponding statements in  \Cref{prop_GB_double2powers_even} when $u=0$ and in \Cref{prop_GB_double2powers_odd} when $u=1$, respectively. The proofs of these assertions are also similar to that of  \Cref{prop_GB_double2powers_odd} and, therefore, are again left to the interested reader.

\begin{prop}
	\label{prop_GB_double2powers_12}
	Let $n=2^s$, where $s\ge 2$ is an integer. There is a Gr\"obner basis description for $Q^n+(f^k)$, in the case $k\in \{1,2\}$, as follows.
	\begin{enumerate}[\quad \rm (1)]
		\item If $k=1$, then the ideal $Q^n+(f)$ has a Gr\"obner basis with elements being the natural generators of the following ideals
		\begin{align*}
			&\udb{Q^n}_{T_1}, \qquad \udb{(f)}_{T_2}, \qquad  \udb{(x^2+y^2)x^2y^2Q^{n-2}b}_{T_3}, \qquad \udb{x^2y^{3n-1}b}_{T_4}, \\
			&\udb{(x^{6\tau+4}+y^{6\tau+4})xy^{3n-5-6\tau}b^2}_{T_5} \qquad \textup{(}\text{where $0\le \tau \le \dfrac{n}{2}-1$}\textup{)},\\
			&\udb{(x^{6\ell}+y^{6\ell})x^2y^{3n-2-6\ell}b^2}_{T_6} \qquad \textup{(}\text{where $1\le \ell \le \dfrac{n}{2}-1$}\textup{)}.
		\end{align*}
		\item If $k=2$, then the ideal $Q^n+(f^2)$ has a Gr\"obner basis with elements being the natural generators of the following ideals
		\begin{align*}
			&\udb{Q^n}_{T_1}, \qquad \udb{(f^2)}_{T_2}, \qquad  \udb{(x^4+y^4)xyQ^{n-2}b^2}_{T_3}, \qquad \udb{x^2y^{3n-1}b^2}_{T_4}, \\
			&\udb{(x^{12\tau+8}+y^{12\tau+8})x^2y^{3n-10-12\tau}b^4}_{T_5} \qquad \textup{(}\text{where $0\le \tau \le \dfrac{n}{4}-1$}\textup{)},\\
			&\udb{(x^{12\ell}+y^{12\ell})xy^{3n-7-12\ell}Q^2b^4}_{T_6} \qquad \textup{(}\text{where $1\le \ell \le \dfrac{n}{4}-1$}\textup{)},\qquad \udb{(y^6a^2+x^6b^2)xy^{3n-7}b^2}_{T_7}.
		\end{align*}
	\end{enumerate}
\end{prop}

\begin{center}
	\begin{table}[ht!]
		\caption{How the $S$-pairs in the Gr\"obner basis of $Q^n+(f^k)$, where $n, k$ are powers of 2, $n\ge 2k$, $\log_2 k$ is even and $\ge 2$, reduce to zero.}
		\begin{tabular}{ | c | c | c | c |}
			\hline
			\multirow{2}{*}{No.} & The number of pairs   & \multirow{2}{*}{$(i,j)$} & $\Lambda \subseteq [8]$ for which $S(T_i,T_j)$ \\
			& $(i,j)$ where $1\le i\le j\le 8$         &            & reduces to zero  via $\mathop{\bigcup}\limits_{\ell \in \Lambda} T_\ell$ \\
			\hline
			1 & 1 & (1,2)                    & \{1,3,4\} \\
			\hline
			2 &  1 & (1,3)                    & \{1,3,4\}\\
			\hline
			3 & 7 & $(i,i)$ ($i\notin \{5,6\}$), (1,4)                    & $\emptyset$\\
			\hline
			4 & 1 &  (1,5)                    & \{1,4\}\\
			\hline
			\multirow{5}{*}{5} & \multirow{5}{*}{19} & (1,6), (1,7), (2,4), (2,7), &   \multirow{5}{*}{\{1\}} \\
			&   & (2,8), (3,4), (3,5), (3,7), &   \\
			&   & (3,8), (4,5), (4,6), (4,8), & \\
			&   &(5,5), (5,6), (5,7), (5,8), & \\
			&   &(6,6), (6,7), (6,8) & \\
			\hline
			6 &   1& (1,8)                     & \{1,4\} \\
			\hline
			7 &   1& (2,3)                     & \{1,2,3,5,6,7\}\\
			\hline
			8 &   1& (2,5)                     & \{1,2,3,6,8\} \\
			\hline
			9 &   1& (2,6)                       & \{1,2,3,5,6,7\} \\
			\hline
			10 &   1& (3,6)                       & \{1,3,4\} \\
			\hline
			11 &   1& (4,7)                       & \{1,3\} \\
			\hline
			12 &   1& (7,8)                       & \{3\} \\
			\hline
			&  \textbf{Totally}: 36 pairs &                             &\\
			\hline
		\end{tabular}
		\label{tab_Spairs_double2powers_ieven}
	\end{table}
\end{center}


\section{Non-linear asymptotic behavior of regularity} \label{sec.nonLinear}

Macaulay 2 \cite{GS96} computations seem to suggest that, in characteristic 0, the example provided in \Cref{sec.HHT} exhibits worse behavior than what has been seen in characteristic 2 in previous sections. We conclude the paper with the following conjectures illustrating the quadratic growth of the regularity of symbolic powers of homogeneous ideals in general.

As before, fix the following notations:
\begin{enumerate}
\item $\kk$ is a field of characteristic 0,
\item $R = \kk[x,y,a,b]$,
\item $Q = (x^3, y^3)$, $f = xya - (x^2+y^2)b$,
\item $S = R[z]$ and $I = Q \cap (f,z) \subseteq S$.
\end{enumerate}

\begin{conj} \label{conj.regChar0}
Let $g: \NN \rightarrow \NN$ be the following numerical function:
$$g(n) = \left\{ \begin{array}{ll} \Big\lfloor \dfrac{n}{2} \Big\rfloor^2 + 5 \Big\lfloor \dfrac{n}{2} \Big\rfloor + 1 & \text{if } n \text{ is odd,}\\
& \\
\Big\lfloor \dfrac{n}{2}\Big\rfloor^2 + 3 \Big\lfloor \dfrac{n}{2} \Big\rfloor - 1 & \text{if } n \text{ is even.}
\end{array}\right.$$
Then, for all $n \ge 1$, we have
$$\reg (Q^n + (f^n)) = g(n) + 3n + 1.$$
\end{conj}

\begin{conj} \label{conj.Quad}
Let $g: \NN \rightarrow \NN$ be the function given in \Cref{conj.regChar0}. Then, for all $n \ge 1$, we have
$$\reg I^{(n)} \ge g(n) + 3n + 2.$$
\end{conj}
\begin{rem}
Note that if \Cref{conj.regChar0} is true, then so is \Cref{conj.Quad}. Indeed, for $n=1$, by direct inspection, $I$ contains the minimal generator $xy^4a-x^2y^3b-y^5b$ of degree 6, so $\reg I\ge  6= g(1)+3\cdot 1+2.$ For $n\ge 2$, since $\reg(Q^n+(f^n))=g(n)+3n+1 > 3n+2$, using \Cref{lem_symbpowofintersect_regformula}, we get
\[
\reg I^{(n)}\ge \reg(Q^n+(f^n))+1=g(n)+3n+2.
\]
\end{rem}


\end{document}